%% file: main.tex
\documentclass[11pt, twoside]{Thesis} 

\graphicspath{{Pictures/}} 

\pdfoutput=1

\PassOptionsToPackage{linktocpage}{hyperref}

\usepackage[square, numbers, comma, sort&compress]{natbib} 
\hypersetup{urlcolor=blue, colorlinks=true} 
\title{\ttitle} 

\begin{document}

\frontmatter 

\doublespacing
\fancyhead{} 
\rhead{\thepage} 
\lhead{} 

\pagestyle{fancy} 

\newcommand{\HRule}{\rule{\linewidth}{0.5mm}} 

\hypersetup{pdftitle={\ttitle}}
\hypersetup{pdfsubject=\subjectname}
\hypersetup{pdfauthor=\authornames}
\hypersetup{pdfkeywords=\keywordnames}


\begin{titlepage}
\begin{center}

\textsc{\LARGE University of Oxford}\\[1.5cm] 

\HRule \\[0.4cm] 
{\huge \bfseries Higgs bundles, Lagrangians \\and mirror symmetry}\\[0.4cm] 
\HRule \\[1.5cm] 
 
\begin{minipage}{0.4\textwidth}
\begin{flushleft} \large
\emph{Author:}\\
{Lucas Castello Branco} 
\end{flushleft}
\end{minipage}
\begin{minipage}{0.4\textwidth}
\begin{flushright} \large
\emph{Supervisors:} \\
{Nigel Hitchin\\ Frances Kirwan} 
\end{flushright}
\end{minipage}\\[3cm]
 
\includegraphics{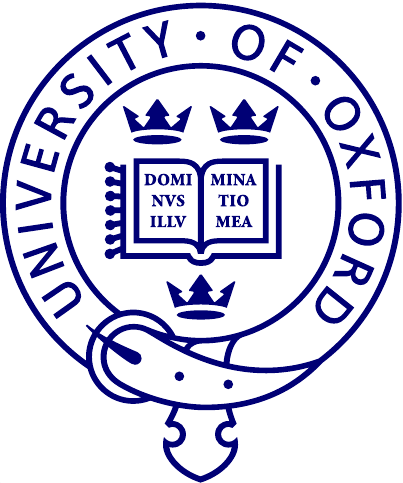} 

\vfill\vfill\vfill\vfill\vfill\vfill\null
\textsc{\large A thesis submitted for the degree of}\\
\textit{\large Doctor of Philosophy}\\[0.5cm]
Jesus College 
 
{\large 2017} 
 
\vfill
\end{center}

\end{titlepage}

\newpage
\thispagestyle{empty}
\mbox{}
\pagebreak


\pagestyle{empty} 

\null\vfill 

\textit{To Kl\'{e}bia, Jos\'{e} and Levi}


\vfill\vfill\vfill\vfill\vfill\vfill\null 

\clearpage 

\newpage
\thispagestyle{empty}
\mbox{}
\pagebreak

\abstract{This thesis is dedicated to the study of certain loci of the Higgs bundle moduli space on a compact Riemann surface. After recalling basic facts about $G$-Higgs bundles for a reductive group $G$, we begin the first part of the work, which deals with Higgs bundles for the real forms $G_0 = SU^*(2m)$, $SO^*(4m)$, and $Sp(m,m)$ of $G = SL(2m,\C)$, $SO(4m,\C)$ and $Sp(4m,\C)$, respectively. The second part of the thesis deals with Gaiotto Lagrangian subvarieties.  

Motivated by mirror symmetry, we give a detailed description of the fibres of the $G$-Hitchin fibration containing $G_0$-Higgs bundles for the real groups $G_0 = SU^*(2m)$, $SO^*(4m)$ and $Sp(m,m)$. The spectral curves associated to these fibres are examples of ribbons and our description is done in two different ways, one in term of objects on the reduced scheme associated to the spectral curve, while the other in terms of the (non-reduced) spectral curve. A link is also provided between the two approaches. We use this description to give a proposal for the support of the dual BBB-brane inside the moduli space of Higgs bundles for the Langlands dual group of $G$, corroborating the conjectural picture. 

In the second part of the thesis we discuss Gaiotto Lagrangian subvarieties inside the moduli spaces of $G$-Higgs bundles, where $G$ is a complex reductive group. These Lagrangians are obtained from a symplectic representation of $G$ and we discuss some of its general properties. We then focus our attention to the Gaiotto Lagrangian for the standard representation of the symplectic group. This is an irreducible component of the nilpotent cone for the symplectic Hitchin fibration. We describe this component by using the usual Morse function on the Higgs bundle moduli space, which is the norm squared of the Higgs field restricted to the Lagrangian in question. 

Lastly, we discuss natural questions and applications of the ideas developed in this thesis. In particular, we say a few words about the hyperholomorphic bundle, how to generalize the Gaiotto Lagrangian to vector bundles which admit many sections and give an analogue of the Gaiotto Lagrangian for the orthogonal group.  

}
\newpage
\mbox{}
\pagebreak



\acknowledgements{\addtocontents{toc}{\vspace{1em}} 

First and foremost, I wish to thank my supervisor Nigel Hitchin for his inspiration, encouragement, time and patience and for introducing me to some beautiful mathematics. I am deeply grateful and honoured for the opportunity to work with him and it has been a remarkable experience, one which I will always treasure. 

The Mathematical Institute has been my home for the last four years and I warmly thank all its members for creating such a pleasant and stimulating environment. A special thanks to my College advisor Andrew Dancer and my second supervisor Frances Kirwan. 

I am very grateful to my Master’s supervisor Marcos Jardim, who has been particularly important in my formation as a mathematician, for all the support through the years. Also, a warm thanks to Laura Schaposnik for all the encouragement and willingness to help during the course of my postgraduate studies.

I am indebted to all mathematicians who have generously shared their expertise with me. In particular, Sergei Gukov, Steven Bradlow, David Baraglia and Oscar Garc\'{i}a-Prada. I am also very thankful to my colleagues and friends Brent Pym and Jakob Blaavand for the many illuminating discussions.
  
A big thanks is due to Marcelo de Martino, for proofreading this thesis, for the many interesting conversations over coffee, for helping me with my various questions on representation theory, and for being such a good friend. Also, a warm thanks to my officemates Matthias Wink and Renee Hoekzema. 

A special thanks goes to Alex Duarte, S\'{e}rgio Kogikoski and Carolina Matt\'{e} Gregory for their friendship and for being part of so many amazing moments which made these four years in Oxford so dear to me.   

My sincere thanks for the generous financial support I received from CNPq.
 
Last but not least, the greatest thanks of all to my parents Jos\'{e} and Kl\'{e}bia and to my brother Levi. I owe them everything I am, everything I have and everything I will ever achieve in life. This thesis is dedicated to you.  

}

\newpage
\mbox{}
\pagebreak




\pagestyle{fancy} 

\lhead{\emph{Contents}} 
\tableofcontents 


\mainmatter 

\pagestyle{fancy} 



\chapter*{\huge \bfseries Introduction}
\thispagestyle{plain}
\label{Introduction}
\addcontentsline{toc}{chapter}{Introduction}

Classically, a Higgs bundle over a compact Riemann surface is a pair consisting of a holomorphic vector bundle $V$ and a holomorphic section of the endomorphism bundle $\End (V)$ twisted by the canonical bundle of the curve, the so-called Higgs field. These objects emerged in the late 1980’s in Hitchin's study of dimensionally reduced self-duality equations of Yang-Mills gauge theory \cite{hitchin1987self} and in Simpson's work on nonabelian Hodge theory \cite{si1}. Since then, Higgs bundles have become central objects of study in K\"{a}hler and hyperk\"{a}hler geometry, nonabelian Hodge theory, representation theory, integrable systems and, most recently, mirror symmetry and Langlands duality.

When $G$ is a complex reductive group, it is well known that the smooth locus of the moduli space $\calM (G)$ of $G$-Higgs bundles on a compact Riemann surface $\Sigma$ of genus $g \geq 2$ has the structure of a hyperk\"ahler manifold. Furthermore, it comes equipped with an algebraically completely integrable Hamiltonian system through the Hitchin map; a proper map to a vector space of half dimension of the moduli space, whose generic fibre is an abelian variety (\cite{hitchin1987stable, falt, donagi1993decomposition}). As pointed out by Kapustin and Witten \cite{kapustin2006electric}, $\calM (G)$ and $\calM (\lan{G})$ are a mirror pair (in the sense of Strominger, Yau, and Zaslow), with the fibrations being the Hitchin fibrations. This manifests itself, for example, in the statement that the non-singular fibres of $\calM (G)$ and $\calM (\lan{G})$ are dual abelian varieties. This was first proven for the special linear group by Hausel and Thaddeus \cite{hausel2003mirror}, who also showed, by calculating Hodge theoretical invariants, that topological mirror symmetry holds for these pairs (see also \cite{MR2927840, gothen2017topological} for an extension of these results for parabolic Higgs bundles). The case of $G_2$ was treated in \cite{hitchin2007langlands} and, in greater generality, Donagi and Pantev \cite{donagi2012langlands} extended the picture for semisimple groups. 

Since the moduli space of Higgs bundles for a complex group has a hyperk\"{a}hler structure, one can speak of branes of type A or B with respect to any of the structures. Typical examples of BAA-branes are moduli spaces of flat connections on a Riemann surface with holonomy in a real form $G_0$ of a complex group $G$, whereas hyperholomorphic bundles supported on a hyperk\"ahler subvariety provide examples of BBB-branes. A hyperholomorphic bundle is a bundle with connection whose curvature is of type $(1,1)$ with respect to all complex structures. These are generalizations of instantons in four dimensions and are very interesting mathematical objects. 

This thesis is essentially divided into two parts. In the first part, motivated by mirror symmetry, we describe certain singular fibres of the Hitchin map associated to the semisimple non-compact connected (real) Lie groups $G_0 = SU^*(2m)$, $SO^*(4m)$ and $Sp(m, m)$. Higgs bundles for $G_0 = SU^*(2m)$, $SO^*(4m)$ and $Sp(m,m)$ viewed inside the Higgs bundle moduli space for $G = SL(2m,\C)$, $SO(4m,\C)$ and $Sp(4m,\C)$, respectively, provide examples of BAA-branes which lie entirely over the discriminant locus of the base of the Hitchin fibration. After describing the whole fibre of the integrable system for the complex group containing Higgs bundles for the real form, we give a proposal for the support of the dual BBB-brane. In particular, our proposal is in accordance with the conjectural picture (see, e.g., \cite{baraglia2013real}) that the dual brane is supported on the moduli space of Higgs bundles for a complex group $\lan{G_0} $ associated to the real form. This group is a complex subgroup of the Langlands dual group $\lan{G}$ and was obtained by Nadler \cite{nadler2005perverse} in the study of perverse sheaves on real loop Grassmannians. In the second part of the thesis we discuss how to construct BAA-branes from symplectic representations of a complex group $G$. This construction has recently appeared in the work of Gaiotto \cite{gaiotto2016s} and Hitchin \cite{spinors}. Also, these so-called Gaiotto Lagrangians have appeared in the literature in the context of derived geometry \cite{ginzburg2017gaiotto} and loop groups \cite{li2017gaiotto}. We finish the second part by describing concretely the Gaiotto Lagrangian for the standard representation of the symplectic group. The following is a more detailed description of each chapter. 

We begin Chapter \ref{Chapter1} by recalling some basic facts about $G$-Higgs bundles for a reductive group $G$. After giving examples of these objects and recalling the main properties of the Higgs bundle moduli space, we discuss Simpson's compactified Jacobian. We then explain the isomorphism between fibres of the Hitchin fibration and the Simpson moduli space of semi-stable sheaves of rank $1$ on the associated spectral curve (which, in ascending order of generality, is due to \cite{hitchin1987stable, bnr, schaub, si3}).   

In Chapter \ref{Chapter2}, we introduce extensions of Higgs bundles as short exact sequences in the abelian category of Higgs bundles and show how these are related to the first hypercohomology group of a certain complex of sheaves. Generalizing the result for vector bundles in \cite[Criterion 2.1]{rmks} we obtain the following (which corresponds to Proposition \ref{themext}). 

\begin{prop} An extension $(V, \Phi)$ of $(W^*, -\trans{\phi})$ by $(W, \phi)$ has an orthogonal (respectively, symplectic) structure turning it into an orthogonal (respectively, symplectic) Higgs bundle and with respect to which $W$ is a maximally isotropic subbundle if and only if it is strongly equivalent to an extension of Higgs bundles whose hypercohomology class belongs to $\K^1 (\Sigma, \Lambda^2\mathsf{W})$ (respectively, $\K^1 (\Sigma, \Sym^2\mathsf{W})$). 
\end{prop}   

In Chapter \ref{Chapter3} we give a concrete description of the fibre of the $SL(2m,\C)$-Hitchin fibration containing $SU^*(2m)$-Higgs bundles. In particular, these fibres have associated non-reduced spectral covers. Our description, which can be found in Theorem \ref{fibreforsu*}, is given in terms of objects on the reduced spectral curve\footnote{The reader is referred to Chapter \ref{Chapter3} for a more detailed description of the objects contained in the theorem.}.
 
\begin{thm} Let $p(\lambda) = \lambda^m + \pi^*a_2\lambda^{m-2} + \ldots + \pi^*a_m$ be a section of the line bundle $\pi^*K^m$ on the total space of the cotangent bundle of $\Sigma$ whose divisor is a smooth curve $S$. The fibre $h^{-1}(p^2)$ of the $SL(2m,\C)$-Hitchin fibration is a disjoint union 
$$h^{-1}(p^2) \cong N \cup \bigcup_{d=1}^{g_{_S}-1} A_d$$
where 
\begin{itemize}
\item $N$ consists of semi-stable rank $2$ vector bundles on $S$ with the property that the determinant bundle of $\pi_*E$ is trivial.
\item $A_d$ is an affine bundle on 
$$Z_d = \{ (D,L) \in S^{(\bar{d})} \times \Pic^{d_2}(S) \ | \ L^2(D) \pi^*K^{1-2m} \in \Prym \}$$
modeled on the vector bundle $E_d \to Z_d$, whose fibre at $(D,L) \in Z_d$ is isomorphic to $H^1(S,K_S^{-1}(D))$. Here, $\bar{d} = -2d + 2 (g_{_S} -1)$.
\item The natural map $Z_d \to S^{(\bar{d})}$ is a fibration, whose fibres are modeled on $2^{2g}$ copies of the Prym variety $\Prym$.
\item Each stratum has dimension $(4m^2 - 1)(g-1)$ and the irreducible components of $h^{-1}(p^2) $ are precisely the Zariski closures of $A_d$, $1 \leqslant d \leqslant g_{_S} -1$, and $N$.
\end{itemize} 
\end{thm} 

Chapter \ref{Chapter4} deals with the description of the fibres of the $G$-Hitchin fibration containing $G_0$-Higgs bundles, for the real forms $G_0 = SO^*(4m)$ and $Sp(m,m)$ of $G=SO(4m,\C)$ and $Sp(4m,\C)$, respectively. In particular, both these real forms are subgroups of $SU^*(4m)$ and now we have an involution $\sigma$ on the reduced spectral curve. Denote by $\rho$ the ramified $2$-fold covering from the reduced spectral curve $S$ to the quotient curve $\bar{S} = S/\sigma$. We adapt the ideas from the last chapter to obtain the following characterizations (corresponding to Theorems \ref{descriptionofthefibreforso} and \ref{descriptionofthesibreforsp}). 

\begin{thm} Let $p(\lambda) = \lambda^{2m} + \pi^*b_2\lambda^{2m-2} + \ldots + \pi^*b_{2m}$ be a section of the line bundle $\pi^*K^{2m}$ on the total space of the cotangent bundle of $\Sigma$ whose divisor is a smooth curve $S$. The fibre $h^{-1}(p^2)$ of the $SO(4m,\mathbb{C})$-Hitchin fibration is a disjoint union 
$$h^{-1}(p^2) \cong N \cup \bigcup_{d=1}^{2(g_{_{\bar{S}}}-1)} A_d$$
where 
\begin{itemize}
\item $N$ has $2$ connected components and it corresponds to the locus of rank $2$ vector bundles on $S$ (of degree $4m(2m-1)(g-1)$) which admit an isomorphism $\psi : \sigma^*E \to E^* \otimes \pi^*K^{2m-1}$ satisfying $\tp{(\sigma^*\psi)} = - \psi$. 
\item $A_d$ is an affine bundle on 
$$Z_d  =  \{ (\bar{D}, L) \in \bar{S}^{(\bar{d})} \times \Pic^{d^\prime}(S) \ | \ L \sigma^*L \cong \rho^* (K_{\bar{S}}^2(-\bar{D}))  \} $$
modeled on the vector bundle $E_d \to Z_d$, whose fibre at $(\bar{D}, L)$ is isomorphic to $H^1(\bar{S}, K_{\bar{S}} (-\bar{D}) )$. Here, $d^\prime = d + 2 (g_{_{\bar{S}}}-1)$ and $\bar{d} = -d + 2 (g_{_{\bar{S}}}-1)$.
\item The natural map $Z_d \to S^{(\bar{d})}$ is a fibration, whose fibres are modeled on the Prym variety $\PS$.
\item Each stratum has dimension $\dim SO(4m,\mathbb{C})(g-1)$ and the irreducible components of $h^{-1}(p^2) $ are precisely the Zariski closures of $A_d$, $1 \leqslant d \leqslant 2(g_{_{\bar{S}}}-1)$, and $N$.
\end{itemize} 
\end{thm} 

\begin{thm} Let $p(\lambda) = \lambda^{2m} + \pi^*b_2\lambda^{2m-2} + \ldots + \pi^*b_{2m}$ be a section of the line bundle $\pi^*K^{2m}$ on the total space of the cotangent bundle of $\Sigma$ whose divisor is a smooth curve $S$. The fibre $h^{-1}(p^2)$ of the $Sp(4m,\mathbb{C})$-Hitchin fibration is a disjoint union 
$$h^{-1}(p^2) \cong N \cup \bigcup_{d=1}^{g_{_{S}}-1} A_d$$
where 
\begin{itemize}
\item $N$ is connected and it corresponds to the locus of rank $2$ semi-stable vector bundles on $S$ (of degree $4m(2m-1)(g-1)$) which admit an isomorphism $\psi : \sigma^*E \to E^* \otimes \pi^*K^{2m-1}$ satisfying $\tp{(\sigma^*\psi)} =  \psi$. 
\item $A_d$ is an affine bundle on 
$$Z_d  =  \{ (\bar{D}, L) \in \bar{S}^{(\bar{d})} \times \Pic^{d^\prime}(S) \ | \ L \sigma^*L \rho^*\mathcal{O}_{\bar{S}}(\bar{D}) \cong \pi^*K^{4m-1} \}$$ 
modeled on the vector bundle $E_d \to Z_d$, whose fibre at $(\bar{D}, L)$ is isomorphic to $H^1(\bar{S}, K_{\bar{S}} (-\bar{D}) )$. Here, $d^\prime = d + 2 (g_{_{\bar{S}}}-1)$ and $\bar{d} = -d + 4m^2 (g-1)$.
\item The natural map $Z_d \to S^{(\bar{d})}$ is a fibration, whose fibres are modeled on the Prym variety $\PS$.
\item Each stratum has dimension $\dim Sp(4m,\mathbb{C})(g-1)$ and the irreducible components of $h^{-1}(p^2) $ are precisely the Zariski closures of $A_d$, $1 \leqslant d \leqslant g_{_{S}}-1$, and $N$.
\end{itemize} 
\end{thm} 

In Chapter \ref{Chapter5} we look at these fibres as sub-loci of semi-stable sheaves of rank $1$ on a ribbon. After providing a dictionary between the two points-of-view, we obtain the following alternative characterizations in terms of Prym varieties of non-reduced curves. For the $SU^*(2m)$ case we obtain\footnote{Proposition \ref{ch5su} is a combination of Propositions \ref{osdoad}, \ref{basesch5}, \ref{ch5seg} and \ref{ch5ter}}:  
\begin{prop}\label{ch5su} Let $1 \leq d \leq g_{_S}-1$ and consider the locus $\calA_d$ inside the fibre $h^{-1}(p^2)$ of the $SL(2m,\C)$-Hitchin fibration consisting of (stable) generalized line bundles of index $\bar{d} = -2d + 2 (g_{_{S}}-1)$. The locus $\calA_d$ is isomorphic to the total space of the affine bundle $A_d$ and the natural map
\begin{align*}
\upsilon : \calA_d & \to \upsilon (\calA_d) \subset \Pic (S) \times S^{(\bar{d})}\\
\calE & \mapsto (D_{\calE}, \bar{\calE})
\end{align*} 
sending a generalized line bundle $\calE$ to its associated effective divisor $D_{\calE}$ together with the line bundle $\bar{\calE} = \calE |_{S}/\text{torsion}$ on $S$ corresponds precisely to the affine bundle $A_d \to Z_d$. Moreover, the fibre of the projection $\calA_d \to S^{(\bar{d})}$ at an effective divisor $D$ is a torsor for the Prym variety $\PXp$, where $X$ is the (non-reduced) spectral curve and $X^\prime$ is the blow-up of $X$ at $D$ (where the Cartier divisor $D$ is considered as a closed subscheme of $X$). In particular, the Prym variety $\PX$ is isomorphic to $\calA_{g_{_S}-1}$ and it has $2^{2g}$ connected components. 
\end{prop}
For the other two cases we have the following (which can be found in Propositions \ref{ribbonsp} and \ref{ribbonso}):   
\begin{prop}\label{ribbonsp1} Let $1 \leq d \leq g_{_S}-1$ and consider the locus $\calA_d$ inside the fibre $h^{-1}(p^2)$ of the $Sp(4m,\C)$-Hitchin fibration consisting of (stable) generalized line bundles of index $\bar{d} = -d + 4m^2(g-1)$. The locus $\calA_d$ is isomorphic to $A_d$ (as defined in Section \ref{Sp(m,m)}) and the natural projection
$$\calA_d \to \bar{S}^{(\bar{d})}$$ 
is a fibration, whose fibre at a divisor $\bar{D}$ is modeled on the Prym variety $\PBl$, where $X$ is the (non-reduced) spectral curve, $X^\prime = \Bl_DX$ and $\bar{X}^\prime$ is a ribbon on the quotient (non-singular) curve $\bar{S}$. In particular, choosing a square root of the canonical bundle of $\Sigma$, the locus $\calA_{g_{_S}-1}$ can be identified with $\PRib \subset \PX$, where $\bar{X}$ is the ribbon $\bar{X}^\prime$ corresponding to $\bar{D}=0$. 
\end{prop}
\begin{prop} Let $1 \leqslant d \leqslant 2(g_{_{\bar{S}}}-1)$ and consider the locus $\calA_d$ inside the fibre $h^{-1}(p^2)$ of the $SO(4m,\C)$-Hitchin fibration consisting of (stable) generalized line bundles of index $\bar{d} = -d + 2 (g_{_{\bar{S}}}-1)$. The locus $\calA_d$ is isomorphic to $A_d$ (as defined in Section \ref{sectiononso}) and the natural projection 
$$\calA_d \to \bar{S}^{(\bar{d})}$$ 
is a fibration, whose fibre at a divisor $\bar{D}$ is modeled on the Prym variety $\PBl$, where again $X$, $X^\prime$ and $\bar{X}^\prime$ are defined as in Proposition \ref{ribbonsp1}. 
\end{prop} 
In the last chapter of the first part we observe the following\footnote{This is contained in the text, but it is not explicitly stated as a proposition. For further details the reader is referred to Chapter \ref{Chapter6}.}: 
\begin{prop} Let $G= SL(2m, \C)$, $SO(4m,\C)$ or $Sp(4m,\C)$. Consider the fibre $h^{-1}_G(p^2)$ of the $G$-Hitchin fibration, where $p(x) = x^m + a_2x^{m-2} + \ldots + a_m$, $a_i \in H^0(\Sigma , K^i)$, if $G = SL(2m,\C)$, and $p(x) = x^{2m} + b_2x^{2m-2} + \ldots + b_{2m}$, $b_{2i} \in H^0(\Sigma , K^{2i})$, in the other two cases. Also, assume that the reduced curve $S$ inside the total space of the canonical bundle of $\Sigma$ given by $\zeros (p(\lambda)) $ is non-singular. There exists a natural surjective map from the singular fibre to an abelian variety $Ab(G_0)$, such that the inverse image of $\calO_S$ intersects the irreducible component of $h_G^{-1}(p^2)$ consisting of rank $2$ semi-stable  vector bundles on $S$ in $N_0 (G)$. When $G = SL(2m,\C)$ the abelian variety $Ab(SU^*(2m))$ is the Prym variety $\Prym$. For the groups $G = SO(4m,\C)$ and $Sp(4m,\C)$ the abelian varieties $Ab(SO^*(4m))$ and $Ab (Sp(m,m))$ are the same and isomorphic to $\PS / \rho^*\Jac (\bar{S})[2]$, where $\rho : S \to \bar{S} $ is the ramified $2$-covering from $S$ to the quotient curve $\bar{S}$.
\end{prop}
Then, we argue that the dual brane is supported on the distinguished subvariety of the dual fibre $h^{-1}_{\lan{G}}(p^2)$ corresponding to the abelian varieties dual to $Ab(G_0)$. These are precisely the spectral data associated to Higgs bundles for the Nadler group $\lan{G_0}$ and our proposal is in accordance with the conjectural picture. 

In the second part of the thesis, we introduce an isotropic locus $X = X_{(\rho , K^{1/2})}$ inside $\calM (G)$ associated to any symplectic representation $\rho$ of a complex semisimple Lie group $G$ and a choice of theta-charactreistic $K^{1/2}$. To study these loci we introduce an auxiliary moduli space of $(\rho, K^{1/2})$-pairs (whose construction follows from Schmitt \cite{schmitt}; see also \cite{hk}). In particular, we have the following (which is the content of Proposition \ref{stability}).
\begin{prop} Let $P$ be a holomorphic principal $G$-bundle on $\Sigma$ and $\psi$ a global section of $P(\V)\otimes K^{1/2}$. If $(P,\mu(\psi))$ is a (semi-)stable Higgs bundle, the pair $(P,\psi)$ is (semi-)stable. 
\end{prop}
The moduli space of $(\rho, K^{1/2})$-pairs $\calS^d (\rho , K^{1/2})$ (of type $d\in \pi_1(G)$) is a quasi-projective varity and it comes equipped with a Hitchin-type map. This map is a projective morphism and from this property we give one example where $\calS^d (\rho , K^{1/2})$ is projective and another where it is not. Then, we reprove, by using a \v{C}ech version of the proof given by Hitchin in \cite{spinors}, the following (which is the content of Proposition \ref{p712}): 
\begin{prop} The subvariety $X$ of $\calM (G)$ is isotropic. Moreover, the $1$-form $\theta$, whose differential corresponds to the natural holomorphic symplectic form on $\calM (G)$, vanishes on $X$.   
\end{prop}   
We finish the chapter (Proposition \ref{p716}) by giving conditions for a certain class of symplectic representations, which we call almost-saturated, to give a Lagrangian subvariety of the Higgs bundle moduli space.  
\begin{prop} Assume $\rho$ is almost-saturated. If the Higgs bundle $(P,\mu(\psi))$ is stable and simple, then $(P,\mu(\psi))$ is a smooth point of $\calM^d (G)$ and $(P,\psi)$ is a smooth point of $\calS^d (\rho, K^{1/2})$. Furthermore, if the Petri-type map 
$$d\mu_\psi : H^0(\Sigma, P(\mathbb{V}) \otimes K^{1/2}) \to H^0(\Sigma, \Ad (P) \otimes K)$$
is injective for all points in the smooth locus of the moduli space of pairs, the subvariety $X \subseteq \calM^d (G)$ is Lagrangian.
\end{prop}

In Chapter \ref{Chapter8} we restrict our attention to the case of the standard representation of $Sp(2m,\C)$. In this case we denote the moduli space of $K^{1/2}$-twisted symplectic pairs simply by $\calS$ and call it the spinor moduli space. The theorem below is the content of Theorem \ref{resu}.  
\begin{thm} Let $(V,\psi)$ be a pair, consisting of a symplectic vector bundle on $\Sigma$ and a spinor $\psi \in H^0(\Sigma , V\otimes K^{1/2})$. 
\begin{enumerate}
\item The pair $(V,\psi)$ is (semi-)stable if and only if its associated $Sp(2n,\C)$-Higgs bundle $(V,\psi \otimes \psi)$ is (semi-)stable. Also, the pair is polystable if and only if the associated Higgs bundle is polystable.
\item The pair $(V,\psi)$ is (semi-)stable if and only if for all isotropic subbundles $0 \neq U  \subset V$ such that $\psi \in H^0(\Sigma , U^\perp \otimes K^{1/2})$ we have $\deg (U) < (\leq) \ 0$. Also, $(V,\psi) \in \calS$ is polystable if it is semi-stable and satisfies the following property. If $0 \neq U \subset V$ is an isotropic (respectively, proper coisotropic) subbundle of degree zero such that $\psi \in H^0(\Sigma , U^\perp \otimes K^{1/2})$ (respectively, $\psi \in H^0(\Sigma , U\otimes K^{1/2})$), then there exists a coisotropic (respectively, proper isotropic) subbundle $U^\prime$ of $V$ such that $\psi \in H^0(\Sigma , U^\prime \otimes K^{1/2})$ (respectively, $\psi \in H^0(\Sigma , (U^\prime)^\perp \otimes K^{1/2}) $) and $V = U \oplus U^\prime$. 
\item The spinor moduli space $\calS$ is a projective variety which embeds in $\calM (Sp(2n,\C))$ and the variety $X $, defined as the Zariski closure of 
$$\{ (V, \Phi) \in \calM^s (Sp(2n,\C)) \ | \ \Phi = \psi \otimes \psi \ \text{for some $0\neq \psi \in H^0(\Sigma, V\otimes K^{1/2})$} \}$$ 
inside the Higgs bundle moduli space $\calM (Sp(2n\C))$, has the property that the canonical $1$-form $\theta$ (see Section \ref{defotheory} for more details) vanishes when restricted to $X$. Moreover, the embedded image of $\calS$ in $\calM(G)$ is the union of $X$ and the moduli space $\calU (Sp(2n,\C))$ of symplectic vector bundles on $\Sigma$, and $\calU (Sp(2n,\C))$ intersects $X$ in the generalized theta-divisor $\Theta \subset \calU (Sp(2n,\C))$.  
\end{enumerate}
\end{thm}   
In this case, the Gaiotto Lagrangian is a component of the nilpotent cone of the $Sp(2n,\C)$-Hitchin fibration. By studying the Morse function restricted to this Lagrangian we obtain its characterization (see Theorem \ref{compgaiolag}) as follows.
\begin{thm} The Gaiotto Lagrangian $X$ corresponds to the irreducible component of the nilpotent cone $\text{Nilp}(Sp(2n, \C))$ labelled by the  chain   
$$K^{1/2}\overset{1}{\to} K^{-1/2} \overset{0}{\to} U,$$
where $U \in \calU (Sp(2n-2,\C))$.  
\end{thm}

In the last chapter of the thesis, we mention some interesting questions and points that emerge from the results and ideas in this thesis. For the convenience of the reader, we have also included a list of selected notation and notational conventions frequently used throughout the thesis. 

Moreover, we provide Appendices with some notations and basic results used throughout the text. In Appendix \ref{comoutationnad} we give a detailed description on how to obtain the Nadler group associated to $SU^*(2m)$ and we comment on the Nadler group for the real forms $SO^*(4m)$ and $Sp(m,m)$.  
 
\pagebreak

\input{Chapters/Chapter1}

\part{Higgs bundles and non-reduced spectral curves}
\input{Chapters/Chapter2}

\input{Chapters/Chapter3}

\input{Chapters/Chapter4} 
\input{Chapters/Chapter5} 
\input{Chapters/Chapter6}

\part{Gaiotto's Lagrangian}
\input{Chapters/Chapter7}
\input{Chapters/Chapter8}
\part{Further questions}
\input{Chapters/Chapter9}

\chapter*{\huge \bfseries Notational conventions}
\thispagestyle{plain}
\label{notationalconv}
\lhead{\emph{Notational conventions}}
\addcontentsline{toc}{chapter}{Notational conventions}

We gather below a list of selected notation and notational conventions frequently used throughout the thesis. 
\begin{itemize}
\item $\Sigma$ is always a compact Riemann surface (or equivalently a projective non-singular irreducible curve over $\C$) of genus $g\ge 2$ with canonical bundle $K$.
\item $C^{(k)}$ is the $k$-th symmetric product of a Riemann surface $C$.
\item $\textsf{I}, \textsf{J}$ and $\textsf{K}$ are always complex structures.
\item $G$ is a (connected) reductive group over $\C$ (in the algebraic or in the analytic category) with Lie algebra $\lie{g}$.
\item $G_0$ is a real form of $G$ (fixed by some anti-holomorphic involution on $G$) with Lie algebra $\lie{g}_0$. 
\item $\lan{G_0} \subset \lan{G}$ is the Nadler group associated to the real form $G_0$ of $G$, where $\lan{G}$ denotes the Langlands dual group of $G$.  
\item $V$ and $E$ are (holomorphic or algebraic) vector bundles.
\item We usually denote the tensor product $L_1 \otimes L_2$ of line bundles simply by $L_1L_2$. 
\item $\calE$ is a coherent sheaf. 
\item $P$ is a (holomorphic or algebraic) principal bundle.
\item $\calU_\Sigma (n,d)$ (or simply $\calU (n,d)$) is the moduli space of semi-stable vector bundles of rank $n$ and degree $d$ on $\Sigma$. Moreover, we denote by $\calU (n, L) \subset \calU (n,d)$ the locus of vector bundles with fixed determinant $L$, where $L$ is a line bundle on $\Sigma$ of degree $d$.  
\item $\calN^d (G)$ is the moduli space of semi-stable principal $G$-bundles on $\Sigma$ of topological type $d\in \pi_1 (G)$. When seen as a moduli space parametrizing semi-stable vector bundles (with extra structure) $\calN^d (d)$ is denoted by $\calU^d (G)$.  
\item $B$ is a non-degenerate $\Ad$-invariant bilinear form on $\lie{g}$. In particular, we usually denote the Killing form of $\lie{g}$ by $B_\lie{g}$.  
\item In general, when we want to refer to a Higgs bundle on $\Sigma$ as a complex of locally-free sheaves concentrated in degrees $0$ and $1$ we write the underlying bundle in Sans Serif font, e.g. the Higgs bundle $(E,\Phi)$ is denoted as $\mathsf{E} = (E \overset{\varphi}{\to} E\otimes K )$. 
\item $\calM^d (G)$ is the moduli space of polystable $G$-Higgs bundles on $\Sigma$ of type $d\in \pi_1 (G)$. In particular, when $G = GL(n,\C)$, we denote by $\calM (n,d)$ the moduli space of polystable Higgs bundles of rank $n$ and degree $d$. 
\item $\calM^d_{dR}(G)$ (respectively, $\calM^d_{gauge}(G)$) is the moduli space of flat $G$-connections (respectively, solutions to Hitchin equations) on $\Sigma$ of type $d\in \pi_1 (G)$.
\item $h_G : \calM^d(G) \to \calA (G)$ (or simply $h$) is the $G$-Hitchin map, where $\calA (G)$ is the Hitchin base.
\item $X_a$ (respectively, $\pi_a : X_a \to \Sigma$) is the spectral curve (respectively, spectral cover) associated to a point $a\in \calA (G)$ in the Hitchin base.
\item $\Pic^0 (X)$ is the Jacobian of a projective curve $X$ (i.e., the connected component of the identity of the Picard scheme). 
\item $\Nm_\pi : \Pic^0(X_1) \to \Pic^0 (X_2)$ is the Norm map associated to a finite map $\pi : X_1 \to X_2 $ between projective curves over $\C$. Moreover, $\text{P} (X_1, X_2)$ is the associated Prym variety. 
\item $\calM (\calO_X(1), P)$ is the Simpson moduli space parametrizing semi-stable sheaves on a polarized projective $\C$-scheme $(X, \calO_X(1))$ with Hilbert polynomial $P$. In particular, we denote by $\calM (a ; k)$ the Simpson moduli space for semi-stable rank $1$ sheaves on the spectral curve $X_a$ of polarized degree $k$ (with the polarization coming from a polarization on $\Sigma$).
\item $\calS^d (\rho , K^{1/2})$ is the moduli space of semi-stable $(\rho , K^{1/2})$-pairs of type $d\in \pi_1 (G)$ on $\Sigma$.
\item $\Theta(\V)$ is the theta-bundle on $\calN (G)$ associated to a linear representation $\V$. 
\item $D_{L_0} = \{ P \in \calN (G) \ | \ H^0(\Sigma , P(\V)\otimes L_0) \neq 0 \}$ is the support of the generalized theta-divisor associated to the linear representation $\V$ (and an appropriate choice of line bundle $L_0$).
\item $\partial_A$ and $\bar{\partial}_A$ are the $(1,0)$ and $(0,1)$-parts of the covariant derivative $d_A$ (or $\nabla_A$) associated to a connection $A$.

\end{itemize}

\addtocontents{toc}{\vspace{2em}} 

\appendix 


\input{Appendices/AppendixA}
\input{Appendices/AppendixB}
\input{Appendices/AppendixC}

\input{Appendices/AppendixD}

\addtocontents{toc}{\vspace{2em}} 

\backmatter


\label{Bibliography}

\lhead{\emph{Bibliography}} 

\bibliographystyle{alpha} 

\bibliography{Bibliography} 

\end{document}

%% file: Chapters/Chapter1.tex

\chapter{General background} 

\label{Chapter1} 

\lhead{Chapter 1. \emph{Background material}} 


\section{Higgs bundles}

Let $G$ be reductive Lie group. Fix a maximal compact subgroup $H \subseteq G$ and let $H^\mathbb{C}$ be its complexification. This choice, together with a non-degenerate $\Ad(G)$-invariant bilinear form $B$ on $\mathfrak{g}$ gives rise to a Cartan decomposition
$$\mathfrak{g} = \mathfrak{h} \oplus \mathfrak{m},$$
where $\mathfrak{h}$ is the Lie algebra of $H$ and $\mathfrak{m}$ its orthogonal complement with respect to $B$. The Cartan decomposition complexifies to 
$$\mathfrak{g}^\mathbb{C} = \mathfrak{h}^\mathbb{C} \oplus \mathfrak{m}^\mathbb{C}$$    
and we have
$$ [\mathfrak{h}^\mathbb{C},\mathfrak{h}^\mathbb{C}] \subseteq \mathfrak{h}^\mathbb{C}, \ \ \ \ [\mathfrak{h}^\mathbb{C},\mathfrak{m}^\mathbb{C}] \subseteq \mathfrak{m}^\mathbb{C}, \ \ \ \ [\mathfrak{m}^\mathbb{C},\mathfrak{m}^\mathbb{C}] \subseteq \mathfrak{h}^\mathbb{C} .$$
The adjoint representation induces an action of $H^\mathbb{C}$ on $\mathfrak{m}^\mathbb{C}$, the so-called \textbf{isotropy representation}
$$\iota : H^\mathbb{C} \to GL(\mathfrak{m}^\mathbb{C}).$$
Let $\Sigma$ be a compact Riemann surface and $K$ its canonical bundle. Given a principal $H^\mathbb{C}$-bundle $P$ on $\Sigma$, we denote by $P(\mathfrak{m}^\mathbb{C}) = P \times_\iota \mathfrak{m}^\mathbb{C}$ the vector bundle with fibre $\mathfrak{m}^\mathbb{C}$ associated to $P$ via the isotropy representation. 

\begin{defin} A \textbf{$G$-Higgs bundle} on a Riemann surface $\Sigma$ is a pair $(P, \Phi)$, where $P$ is a principal holomorphic $H^\mathbb{C}$-bundle on $\Sigma$ and $\Phi \in H^0(\Sigma, P(\mathfrak{m}^\mathbb{C}) \otimes K)$. The holomorphic section $\Phi$ is called the \textbf{Higgs field}.
\end{defin}  

\begin{rmks} 

\noindent 1. When $G$ is compact, a $G$-Higgs bundle is simply a holomorphic $G^\C$-bundle. 

\noindent 2. When $G$ is complex, $H^\mathbb{C} = G$ and $ \mathfrak{m} = i  \mathfrak{h}$, so $\mathfrak{m}^\mathbb{C} = \mathfrak{g}$ and the isotropy representation is just the adjoint representation of $G$. Thus, a $G$-Higgs bundle is a pair $(P, \Phi)$, where $P$ is a principal $G$-bundle and $\Phi \in H^0(\Sigma, \Ad (P) \otimes K)$. Here, $\Ad (P) = P \times_{\Ad} \lie{g}$ is the adjoint bundle of $P$.  
\end{rmks}

For a complex Lie group $G \subseteq GL(n, \mathbb{C})$ it is convenient to work with the holomorphic vector bundle associated to the principal $G$-bundle (via the standard representation). Thus, in this case, it is common to consider $G$-Higgs bundles as classical Higgs bundles\footnote{A classical Higgs bundle is a pair $(V, \Phi)$ where $V$ is a holomorphic vector bundle on $\Sigma$ and $\Phi \in H^0(\Sigma, \End V \otimes K)$.} with extra structure reflecting the group in question. 

\begin{ex} Let $G = SL(n, \mathbb{C})$. Then an $SL(n, \mathbb{C})$-Higgs bundle consists of a pair $(V, \Phi)$, where $V$ is a rank $n$ holomorphic vector bundle on $\Sigma$ with trivial determinant bundle and $\Phi$ is a holomorphic section of $\End_0 (V)$, the bundle of traceless endomorphisms of $V$, twisted by $K$.  
\end{ex}

\begin{ex} Let $G = SO(n,\C)$. A principal holomorphic $SO(n,\C )$-bundle on $\Sigma$ corresponds to a rank $n$ orthogonal vector bundle $V$, i.e., a holomorphic vector bundle $V$ together with a (fibrewise) non-degenerate symmetric bilinear form $q \in H^0(\Sigma, \Sym^2V^*)$. The adjoint bundle can be identified with the subbundle of $\End (V)$ consisting of endomorphisms which are skew-symmetric with respect to $q$. Thus, the corresponding Higgs bundle has the form $(V,\Phi)$, where $(V,q)$ is an orthogonal vector bundle and the Higgs field $\Phi \in H^0(\Sigma, \End_0 (V)\otimes K)$ satisfies $q(\Phi \cdot, \cdot) + q (\cdot ,\Phi \cdot) = 0$. 
\end{ex}

\begin{ex} Let $G = Sp(2n,\C)$ be the complex symplectic group. A principal $Sp(2n,\C)$-bundle on $\Sigma$ corresponds to a rank $2n$ symplectic vector bundle $(V,\omega)$, i.e., a holomorphic vector bundle $V$ together with a (fibrewise) non-degenerate skew-symmetric form $\omega \in H^0(\Sigma, \wedge^2V^*)$. The adjoint bundle can be identified with the subbundle of $\End (V)$ consisting of endomorphisms which are skew-symmetric with respect to $\omega$. Thus, an $Sp(2n,\C)$-Higgs bundle is a pair $(V,\Phi)$, where $(V,\omega)$ is a symplectic vector bundle of rank $2n$ and the Higgs field $\Phi \in H^0(\Sigma, \End_0 (V)\otimes K)$ satisfies $\omega (\Phi \cdot, \cdot) + \omega (\cdot ,\Phi \cdot) = 0$. 
\end{ex}

Next, we discuss Higgs bundles for the real Lie groups $G_0 = SU^*(2m), SO^*(4m)$ and $Sp(m,m)$. These are connected, semi-simple, non-compact real forms\footnote{Recall that a real Lie group $G_0$ is a real form of a complex Lie group $G$ if $G_0$ is a subgroup of the underlying real Lie group of $G$ which is given as the fixed point set of an anti-holomorphic involution of $G$.} of $G = SL(2m, \C)$, $SO(4m,\C)$ and $Sp(4m,\C)$, respectively.

\begin{ex}\label{su*} The group $SU^*(2m)$ is the fixed point set of the anti-holomorphic involution
\begin{align*}
SL(2m,\C)  \to & SL(2m,\C) \\
g  \mapsto & J_m\bar{g}J_m^{-1},
\end{align*}
where $J_m = \left( \begin{matrix} 0& 1_m\\ -1_m&0 \end{matrix} \right)$; here, $1_m$ is the $m \times m$ identity matrix. This group is also denoted by $SL(m, \K)$, where $\K$ stands for the quaternions, and sometimes called the non-compact dual of $SU(2m)$\footnote{In the sense that, in Cartan's classification of symmetric spaces, the non-compact symmetric space $SU^*(2m)/Sp(2m)$ is the dual of the compact symmetric space $SU(2m)/Sp(2m)$.}. Its maximal compact subgroup is $H = Sp(m)$, which complexifies to $H^\C = Sp(2m,\C)$. Thus, $$\mathfrak{m}^\C = \{ x \in \mathfrak{sl}(2m,\C) \ | \ \tp{x}J = Jx \}$$ and an $SU^*(2m)$-Higgs bundle is a pair $(V,\Phi)$ consisting of a symplectic vector bundle $(V,\omega)$ on $\Sigma$ of rank $2m$ and a section $\Phi \in H^0(\Sigma, \End_0 (V) \otimes K)$ which is symmetric with respect to the symplectic form of $V$ (i.e., $\omega (\Phi \cdot, \cdot) = \omega (\cdot, \Phi \cdot))$. 
\end{ex}

\begin{ex}\label{so*} The group $SO^*(2m)$ is the fixed point set of the anti-holomorphic involution
\begin{align*}
SO(2m,\C)  \to & SO(2m,\C)\ \\
g  \mapsto & J_{m}\bar{g}J_{m}^{-1}.
\end{align*}
Its maximal compact is $H = U(m)$, so $H^\C = GL(m,\C)$. This group is the non-compact dual of the real form $SO(2m) \subseteq SO(2m,\C)$ and sometimes denoted by $SO(m,\K)$. As remarked in \cite[Section 3]{bradlow2015higgs}, let $T$ be the complex automorphism of $\C^{2m}$ defined by 
$$T = \left( \begin{matrix} 1_{m}& i1_{m}\\ 1_{m}&-i1_{m} \end{matrix} \right).$$
Then, 
$$\mathfrak{so}(2m,\C) = \mathfrak{gl}(m,\C) \oplus \mathfrak{m}^\C ,$$
where, after conjugating by $T$ on has  
\begin{equation}
\mathfrak{gl}(m,\C) \cong \{  \left( \begin{matrix} x& 0\\ 0&-\tp{x} \end{matrix} \right) \  | \ x \in \text{Mat}_{m}(\C) \} 
\end{equation}
and 
$$\mathfrak{m}^\C  \cong \{  \left( \begin{matrix} 0& \beta \\ \gamma &0 \end{matrix} \right) \  | \ \beta, \gamma \in \mathfrak{so}(m,\C) \}. $$
A principal $GL(m,\C)$-bundle gives a rank $m$ vector bundle and the isotropy representation acts on $\mathfrak{m}^\C $ via conjugation by $ \iota(g)  = \left( \begin{matrix} g& 0\\ 0& \tp{g}^{-1} \end{matrix} \right)$, for $g \in GL(m,\C)$. Thus, a Higgs bundle for $SO^*(2m)$ is given by the data $(W, \beta, \gamma)$, where $W$ is a rank $m$ vector bundle, $\beta \in H^0(\Sigma, \Lambda^2W \otimes K)$ and $\gamma \in H^0(\Sigma, \Lambda^2W^* \otimes K)$. It will be convenient to us to consider such Higgs bundles as $SO(2m,\C)$-Higgs bundles. Explicitly, given an $SO^*(2m)$-Higgs bundle $(W,\Phi=(\beta, \gamma))$, take $V = W \oplus W^*$ with its natural orthogonal form $q ((w_1, \xi_1), (w_2,\xi_2)) = \xi_2(w_1) + \xi_1(w_2)$. Then, $(V, \Phi =  \left( \begin{matrix} 0& \beta \\ \gamma &0 \end{matrix} \right))$ is an $SO(2m,\C)$-Higgs bundle.      
\end{ex}

\begin{ex}\label{spmm} The group $Sp(m,m)$ is the fixed point set of the anti-holomorphic involution
\begin{align*}
Sp(4m,\C)  \to & Sp(4m,\C)\ \\
g  \mapsto & K_{m}\bar{g}K_{m}^{-1}
\end{align*}
where $K_m = \left( \begin{matrix} 0&0&-1_m&0\\0&0&0&1_m\\1_m&0&0&0\\0&-1_m&0&0 \end{matrix} \right)$. Its maximal compact subgroup is $H = Sp(m) \times Sp(m)$, so $H^\C = Sp(2m,\C) \times Sp(2m,\C)$. Similarly to the example above, $\mathfrak{m}^\C$ can be expressed as a subset of certain off-diagonal matrices and a Higgs bundle for $Sp(m,m)$ is of the form $(W_1 \oplus W_2, \Phi = \left( \begin{matrix} 0& \beta \\ -\beta^T & 0 \end{matrix} \right) )$, where $(W_i,\omega_i)$, $i=1,2$, are symplectic vector bundles of rank $2m$ and $\beta^T : W_1 \to W_2 \otimes K$ is the symplectic transpose\footnote{By definition, $\omega_1 (\beta \cdot, \cdot) = \omega_2 (\cdot, \beta^T \cdot)$.} of $\beta : W_2 \to W_1 \otimes K$. Note that, by considering the symplectic form $\omega = (\omega_1,\omega_2)$ on $V = W_1 \oplus W_2$, the pair $(V, \Phi = \left( \begin{matrix} 0& \beta \\ -\beta^T & 0 \end{matrix} \right))$ is naturally an $Sp(4m,\C)$-Higgs bundle. Now, by considering the symplectic form $\omega^\prime = (\omega_1,-\omega_2)$ on $V$, the pair $(V,\Phi)$ is naturally an $SU^*(4m)$-Higgs bundle, reflecting the fact that $Sp(m,m)$ is a subgroup of $SU^*(4m)$.
\end{ex}

The construction of the moduli space of semi-stable\footnote{As usual, the moduli space parametrizes $S$-equivalence classes of semi-stable Higgs bundles or equivalently isomorphism classes of polystable Higgs bundles (for appropriate notions of semi-stability and polystability generalizing the usual one for principal bundles.} $G$-Higgs bundles when $G$ is a complex reductive Lie group goes back to Hitchin \cite{hitchin1987self, hitchin1987stable} and Simpson \cite{si1, si3}. In the general case where $G$ is a reductive Lie group (real or complex), the existence of a moduli space of Higgs bundles $\calM^d (G)$ parametrizing isomorphism classes of polystable $G$-Higgs bundles (of topological type\footnote{The topological type of a Higgs bundle is defined as the topological type of the underlying principal bundle (see Appendix \ref{AppendixD} for more details).} $d\in \pi_1 (G)$) is guaranteed by Schmitt's GIT general construction (see \cite{schmitt}) or Kuranishi slice method (see, e.g., \cite{kob}). In particular, $\calM^d (G)$ is a complex analytic variety, which is algebraic when $G$ is algebraic. When $G$ is a complex reductive Lie group, for each topological type $d\in \pi_1 (G)$, the moduli space $\calM^d (G)$ is non-empty and connected (see \cite{MR1206154, garcia2014connectedness}).

\section{Hitchin Equations}\label{hkstructure}

Let $G$ be a complex semisimple\footnote{We restrict to semisimple groups as, apart from the general linear group, all groups considered in this thesis are of this type. Generalizations to the reductive case (and when the group is real) can be found in the given references.} Lie group and $K \subset G$ a maximal compact subgroup defined by an anti-holomorphic involution $\tau$.  Also, let $P$ be a smooth principal $G$-bundle on $\Sigma$ (of topological type $d\in \pi_1 (G)$) and fix a reduction $P_K $ of $P$ to the maximal compact subgroup $K$ (for notation and basic definitions concerning principal bundles the reader is referred to Appendix \ref{AppendixD}). We also denote by $\tau $ the isomorphism 
$$\tau : \Omega^{1,0}(\Sigma , \Ad (P)) \to \Omega^{0,1}(\Sigma , \Ad (P))$$  
given by the combination of complex conjugation on forms and $\tau$. 

Now, let $\calA$ be the set of $K$-connections on $P_K$. This is an affine space modeled on $\Omega^1 (\Sigma , \Ad (P_K))$ and, by the Chern correspondence, is in bijection with the set $\calC$ of holomorphic structures on $P$, which in turn is an affine space modeled on $\Omega^1 (\Sigma , \Ad (P))$ (see e.g. \cite{aticonn}). Thus, $T^*\calA \cong \calA \times \Omega^{1,0}(\Sigma , \Ad (P))$ is naturally an infinite-dimensional flat hyperk\"{a}hler space. Following Hitchin \cite{hitchin1987self}, given a tangent vector $(\beta , \varphi) \in \Omega^{0,1}(\Sigma , \Ad (P)) \oplus \Omega^{1,0}(\Sigma , \Ad (P))$ of $T^*\calA$, we define the complex structures 
\begin{align*}
\textsf{I}(\beta , \varphi) & = (i\beta , i\varphi)\\ 
\textsf{J}(\beta , \varphi) & = (-i\tau (\varphi), i \tau (\beta))\\
\textsf{K}(\beta , \varphi) & = (\tau (\varphi), -\tau (\beta)),
\end{align*} 
where $\textsf{K} = \textsf{IJ}$, and a Riemannian metric
\begin{equation*}
g( (\beta, \varphi) , (\beta, \varphi) ) = 2i \int_{\Sigma} B( \tau (\varphi), \varphi) - B(\tau (\beta), \beta),
\end{equation*}
which is a hyperk\"{a}hler metric.

The gauge group $\calG_K$ (i.e., the group of automorphisms of $P_K$) acts naturally on $T^*\calA$ preserving the hyperk\"{a}hler structure defined above and the associated hyperk\"{a}hler moment map $\mu = (\mu_1, \mu_2, \mu_3)$ is given by 
\[
\left\{
                \begin{array}{ll}
                  \mu_1 (A, \Phi)  = F_A - [\Phi , \tau (\Phi)]\\
                  (\mu_2+ i \mu_3)(A,\Phi)  = \bar{\partial}_A\Phi.          
                \end{array}
              \right.
\]
The equations 
\[
\left\{
                \begin{array}{ll}
                  F_A - [\Phi , \tau (\Phi)] = 0\\
                  \bar{\partial}_A\Phi = 0          
                \end{array}
              \right.
\]
are usually called \textbf{Hitchin equations} (or Higgs bundle equations). Let $\calM_{gauge}^d (G) = \mu^{-1}(0)/\calG_K$ be the moduli space of solutions to Hitchin equations. Thus we have the following Hitchin-Kobayashi theorem\footnote{When $G$ is a reductive group, a \textit{Hitchin-Kobayashi correspondence} for $G$-Higgs bundles was obtained in \cite{hk}.}. 

\begin{thm}[\cite{hitchin1987self, si1}]\label{hkhiggs} A $G$-Higgs bundle $(P,\Phi)$ (of topological type $d\in \pi_1 (G)$) admits a reduction $h$ of structure group to the maximal compact subgroup $K$ (i.e., a metric) whose curvature $F_h$ of the associated Chern connection satisfies the equation 
$$ F_h - [\Phi , \tau (\Phi)] = 0 $$
if and only if it is a polystable Higgs bundle. Moreover, there is a homeomorphism 
$$\calM^d (G) \cong \calM_{\text{gauge}}^d (G)$$ 
taking irreducible solutions in $\calM_{gauge}^d (G)$ to stable and simple Higgs bundles in $\calM^d (G)$.
\end{thm}  

In particular, $\calM_{gauge}^d (G)$ is hyperk\"{a}hler and we identify it in complex structure $\textsf{I}$ (which is complex structure coming from $\Sigma$) with the moduli space of Higgs bundles\footnote{Note that $\calM^d (G)$ can be seen as the quotient by the gauge group $\calG$ of polystable elements $(\bar{\partial}_A, \Phi) \in T^*\calA$ satisfying $\bar{\partial}_A \Phi$, where we have implicitly used the bijection between $\calC$ and $\calA$. Thus, the isomorphism stated in Theorem \ref{hkhiggs} takes $(A,\Phi)$ in complex structure $\textsf{I}$ to $(\bar{\partial}_A, \Phi)$ in the natural complex structure on the Higgs bundle moduli space coming from the curve.}. Given a solution $(A,\Phi)$ to Hitchin equations, the connection $$\nabla \coloneqq \nabla_A +\phi - \tau(\phi)$$ is a flat $G$-connection\footnote{Since $\nabla_A$ is compatible with $K$ and $\tau (\Phi - \tau(\phi)) = - (\Phi - \tau(\phi))$ (i.e., $\Phi - \tau(\Phi)$ is a 1-form taking value in $i\lie{k}$). Moreover, from the equations, we see that $\nabla$ is flat.}, so we obtain a map from $\calM_{gauge}^d (G)$ to the moduli space $\calM_{dR}^d(G)$ of flat $G$-connections (which is sometimes called the \textbf{de Rham moduli space}). This gives a homeomorphism between the moduli space of Higgs bundles and the moduli space of flat connections, which is usually referred to as the \textbf{non-abelian Hodge correspondence}\footnote{For the non-abelian Hodge correspondence when $G$ is a real reductive Lie group see \cite{hk, relativehk}.} (the proof relies heavily on two existence theorems for gauge-theoretic equations due to Donaldson \cite{do} and Corlette \cite{corlette}). Moreover, 
the map 
$$(A,\Phi) \mapsto \nabla_A +\phi - \tau(\phi)$$
sends $\calM_{gauge}^d (G)$ in complex structure $\textsf{J}$ to $\calM_{dR}^d(G)$ in the natural complex structure induced from $G$.


\section{Simpson compactified Jacobian}\label{Simpson}

Recall that for a non-singular projective curve over $\C$ of genus $g$, the Jacobian variety is a $g$-dimensional abelian variety that can be described as the moduli space of degree $0$ line bundles on the curve. More generally, let $X$ be any projective pure-dimensional $\C$-scheme of finite type and dimension $1$, hereinafter simply called a projective curve. The Picard group $\Pic (X)$ is the set of isomorphism classes of invertible sheaves on $X$, i.e., the moduli space of line bundles on $X$, and it is naturally an abelian group isomorphic to $H^1(X, \calO_X^\times)$. We denote by $\Jac (X)$ the connected component of the identity of the Picard scheme and call it the \textbf{Jacobian of $X$} (see e.g. \cite{kleim}). Clearly, in this generality, $\Jac (X)$ is not a proper $\C$-scheme (see e.g. \cite{bosch2012neron} for a very good overview on the subject). The problem of compactifying the Jacobian of a singular curve has been addressed by many authors using different methods and in various levels of generality. Some references include \cite{alt, odases, ses, est}. In this section we focus on Simpson moduli space of semi-stable sheaves, which generalizes many of the past constructions, and is intimately connected, as we will see, to the fibres of the Hitchin fibration.

Let us recall some basic definitions about sheaves. If $\calE$ is any coherent sheaf on a (Noetherian) scheme, one defines its \textbf{dimension} $d(\calE)$ as the dimension of the closed set $\supp{ (\calE)} = \{ x \in X \ | \ \calE_x \neq 0 \}$. Moreover, $\calE$ is said to be \textbf{pure of dimension $d$} if for all non-trivial subsheaves $\calF \subset \calE$, $d (\calF) = d$. Let us return to the case where $X$ is a projective curve. Generalizing the usual notion for integral curves, a coherent sheaf $\calE$ on $X$ will be called \textbf{torsion-free} if it is pure of dimension one. 

\begin{ex} If $X$ is a projective smooth curve, a coherent sheaf $\calE$ is torsion-free if and only if it is locally-free. Thus, any coherent sheaf on $X$ splits as a direct sum of a torsion sheaf (a finite collection of points, and so a sheaf of dimension $0$) and a locally free-sheaf (a sheaf of dimension $1$).  
\end{ex}   

Let $\calO_X (1)$ be an ample invertible sheaf on $X$ and $\delta$ the degree of the associated polarization. Then, for every coherent sheaf $\calE$ on $X$ there exists a polynomial, called the \textbf{Hilbert polynomial of} $\calE$, with rational coefficients and degree $d(\calE)$ defined by 
$$P(\calE , t) = \chi (X, \calE (t) ), \qquad \text{for $t>> 0$},$$
where $\calE (t) = \calE \otimes \calO_X(1)^t$. Let $\calE$ be a coherent sheaf on $X$ of dimension $d (\calE) = 1$. Then, the associated Hilbert polynomial is of the form
$$P(\calE , t) = \delta \rk_P (\calE) t + \deg_P (\calE) + \rk_P (\calE) \chi (\calO_X),$$
for some rational numbers $\rk_P (\calE)$ and $\deg_P (\calE)$, called the \textbf{(polarized) rank} and \textbf{degree} of $\calE$, resp., with respect to $\calO_X(1)$. We also denote define the \textbf{(polarized) slope} of $\calE$ by $\mu_P (\calE) = \deg_P (\calE)/ \rk_P (\calE)$. Note in particular that considering some positive power of $\calO_X(1)$ as the ample line bundle, does not alter the polarized slope. 

\begin{rmk} If $X$ is integral, the usual notions of degree and rank coincide with the ones given above (for any polarization). In particular, if $\calE$ is a $1$-dimensional coherent sheaf on $X$, there exists an open dense subset $U \subset X$, such that $\calE|_U$ is locally free and the rank of $\calE$ is the rank of the locally free sheaf $\calE|_U$. In general, however, the polarized rank and degree of a torsion-free sheaf need not be integers and might depend on the polarization.  
\end{rmk} 

\begin{ex} Let $X = \Sigma_1 \cup \Sigma_2$ be a nodal curve, where $\Sigma_1 $ and $\Sigma_2$ are smooth curves (of genus $g_1$ and $g_2$, respectively) meeting transversely at $x_0 \in X$. A polarization on $X$ can be seen as a collection of rational numbers, one for each irreducible component, whose sum is an integer. Let $\calO_X (1)$ be an ample line bundle whose corresponding polarization is given by $\{1/2 , 1/2 \}$. Also, let $E$ be a locally-free sheaf of rank $r$ and degree $d$ on $\Sigma_1$, and consider the coherent sheaf $\calE = j_* E$ on $X$, where $j: \Sigma_1 \to X$ is the natural inclusion. The sheaf $\calE$ is torsion-free and its polarized rank and degree are given by 
\begin{align*}
P(\calE, t) & = \chi (X, \calE (t)) \\
  & = \rk_P (\calE) t + \deg_P (\calE) + \rk_P (\calE) \chi (\calO_X) \\
  & = \rk_P (\calE) t + \deg_P (\calE) + \rk_P (\calE) (1-g_1-g_2).
\end{align*}
But, 
\begin{align*}
\chi (X, \calE (t) ) & = \chi (\Sigma_1, E \otimes j^* \calO_X(t)) \\
& =  \frac{r}{2}t + d + r(1-g_1),
\end{align*}
by Riemann-Roch. Thus $\rkp (\calE) = r/2$. This example illustrates two points. Taking $E$ to be an invertible sheaf on $\Sigma_1$, we obtain that the polarized rank of $ j_* E $ is not integral. Also, note that if a torsion-free sheaf has rank $1$ on each irreducible component, it will also have rank $1$ on $X$. The converse, however, does not hold. Take the rank of $E$ to be $2$. Then, $\rkp (j_* E) = 1$, but $j_*E$ restricted to $\Sigma_2$ is a torsion-sheaf supported at $x_0$.    
\end{ex}

Recall that any semi-stable sheaf $\calE$ admits a (non-unique) \textbf{Jordan-H\"{o}lder} filtration, i.e., a filtration
$$0 = \calE_0 \subset \calE_1 \subset \ldots \subset \calE_l = \calE,$$
where the factors $\calE_i / \calE_{i-1}$ are stable and have polarized slope $\mu_P (\calE)$, for $i=1, \ldots , l$. Moreover, the graded object 
$$Gr(\calE)  \coloneqq \bigoplus_{i = 1}^l \calE_i/\calE_{i-1},$$
up to isomorphism, does not depend on the choice of the Jordan-H\"{o}lder filtration (see e.g. \cite[Proposition 1.5.2]{huygeo}).  

\begin{defin} A coherent sheaf $\calE$ on $X$ is \textbf{semi-stable} (resp. \textbf{stable}) if it is torsion-free and for any proper subsheaf $\calF \subset \calE$, one has $\mu_P (\calF) \leq \mu_P (\calE)$ (resp. $<$). Also, two semi-stable sheaves $\calE_1$ and $\calE_2$ are called \textbf{S-equivalent} if $Gr (\calE_1) \cong Gr (\calE_2)$. 
\end{defin}

\begin{rmk} The stability condition defined above is equivalent to Simpson's definition \cite{si2}, where a torsion-free sheaf $\calE$ is called $\mu$-(semi-)stable if for all proper subsheaves $\calF \subset \calE$ one has 
$$\frac{\deg_P (\calF) + \rk_P (\calF) \chi (\calO_X)}{\delta \rk_P (\calF)}  \underset{(\leq)}{<}  \frac{\deg_P (\calE) + \rk_P (\calE) \chi (\calO_X)}{\delta \rk_P (\calE)}.$$
Also, for stable sheaves, $S$-equivalence is the same as isomorphism.
\end{rmk}

Simpson constructed in \textit{loc. cit.} a coarse moduli space $\calM (\calO_X (1), P)$, hereinafter called \textbf{Simpson moduli space}, for semi-stable sheaves on a polarized curve\footnote{In \textit{loc. cit.}, Simpson actually constructs a moduli space of semi-stable sheaves on any polarized projective $\C$-scheme.} $(X, \calO_X(1))$  with Hilbert polynomial 
$$P(t) = \delta t + k + \chi (\calO_X),$$
where $\delta$, as before, is the degree of the polarization. He shows in \cite[Theorem 1.21]{si2} that $\calM (\calO_X (1), P)$ is a projective scheme whose points correspond to $S$-equivalence classes of semi-stable (polarized) rank $1$ and (polarized) degree $k$ torsion-free sheaves on $X$. Moreover, the locus $\calM^s (\calO_X (1), P)$ corresponding to stable sheaves is the fine moduli space of stable sheaves, an open dense set in $\calM (\calO_X (1), P)$.

\begin{rmk} As pointed out in the beginning of this section, the Simpson moduli space is a compactification of $\Pic^0(X)$ which generalizes prior constructions. In particular, if $X$ is integral (i.e., reduced and irreducible), every rank $1$ torsion-free sheaf is stable and $\calM (\calO_X (1), P)$ coincides with the compactification constructed by Altman and Kleiman \cite{alt}. Also, when $X$ is reduced but reducible, the (semi-)stability condition introduced by Seshadri in \cite{ses} is known to be equivalent to the notion introduced in this section. Moreover, Seshadri's compactification corresponds to the connected component of Simpson moduli space consisting of sheaves which are rank $1$ on every irreducible component of the curve (for more details and a good guide to the literature see \cite{ana}).   
\end{rmk}  

\section{Hitchin fibration}\label{Section1.4}

Let $G$ be an (affine) reductive group over $\C$ with Lie algebra $\lie{g}$ and fix a topological type $d\in \pi_1 (G)$. Consider a homogeneous basis $\{ p_1, \ldots , p_r \}$ for the algebra $\C [\lie{g}]^G$ of polynomials invariant under the adjoint action of $G$. Here, $r$ is the rank of $\lie{g}$ and let $d_i$ be the degree of the invariant polynomial $p_i$, for $i=1, \ldots , r$. The \textbf{$G$-Hitchin fibration}, or simply \textbf{Hitchin map}, is given by 
\begin{align*}
h_G : \calM^d (G) & \to \calA (G) = \bigoplus_{i=1}^r H^0(\Sigma , K^{d_i}) \\
(P,\Phi) & \mapsto (p_1 (\Phi), \ldots , p_r (\Phi)), 
\end{align*} 
where $\calA (G) $ is a vector space of half of the dimension of $\calM^d (G)$ called \textbf{Hitchin base}. If no confusion is possible, we will drop the subscript and denote the morphism above simply by $h$. The Hitchin map is a surjective proper morphism for any choice of basis, with complex Lagrangian fibres. Moreover it turns the moduli space of $G$-Higgs bundles into an algebraically completely integrable Hamiltonian system\footnote{Recall that an algebraically completely integrable Hamiltonian system consists of a $2N$-dimensional symplectic complex algebraic variety $Y$ with a map $f = (f_1, \ldots , f_N): Y \to \C^N$ such that $df_1 \wedge \ldots \wedge f_N$ is generically non-zero; the Poisson bracket of $f_i$ and $f_j$ vanish for all $i,j = 1, \ldots , N$; for generic $z \in \C^N$, $f^{-1}(c)$ is an open set in an abelian variety (of dimension $N$); the Hamiltonian vector fields $X_{f_i}$, $i=1, \ldots , N$, are linear on $f^{-1}(c)$.} \cite{hitchin1987stable}. To describe the fibres of $h$ for classical groups, Hitchin introduced the notion of spectral curve\footnote{For more general groups, the notion of a cameral cover was introduced by Donagi (see e.g. \cite{donagi}).} \cite{hitchin1987self, hitchin1987stable}, which we will recall in this section.    

Note that given a linear representation $\rho : G \to GL (\V)$, one can associate to any semi-stable $G$-Higgs bundle $(P,\Phi)$ a semi-stable Higgs bundle $(P (\V), \rho (\Phi))$ and then consider a particularly convenient basis of homogeneous polynomials given by the coefficients of the characteristic polynomial (see e.g. \cite{donagi}). As our main concern will be when the complex group $G$ is a classical group, we start by introducing the spectral curve for the group $GL(n,\C)$. In this case, we use the more conventional notation $\calM (n,d)$ for the moduli space $\calM^d (GL(n,\C))$ of Higgs bundles of rank $n$ and degree $d$ on $\Sigma$. Choose the basis of invariant polynomials to be given by the coefficients of the characteristic polynomial and denote the Hitchin fibration, which is a projective fibration, by 
\begin{align*}
h : \calM (n,d) & \to \calA = \bigoplus_{i=1}^n H^0(\Sigma , K^{i}) \\
(V,\Phi) & \mapsto (p_1 (\Phi), \ldots , p_n (\Phi)), 
\end{align*} 
where $p_i (\Phi) = (-1)^i \tr (\Lambda^i \Phi)$, for $i=1, \ldots , n$. Compactify $\pi : K \to \Sigma$ to the ruled surface 
$$\bar{\pi} : Y \to \Sigma ,$$
where $Y = \CP (\calO_\Sigma \oplus K)$, i.e., the total space of $\Proj (\Sym^\bullet (\calO_\Sigma \oplus K^{-1})) $, and let $\calO (1)$ be the relatively ample invertible sheaf on $Y$, so that $\bar{\pi}_* \calO (1) \cong \calO_\Sigma \oplus K^{-1}$. Let $\sigma \in H^0(Y, \calO (1)) = H^0(\Sigma , \calO_\Sigma \oplus K^{-1})$ be the canonical section of $\calO (1)$, called the ``\textbf{infinity section}'', and let $\lambda \in H^0(\Sigma , \bar{\pi}^*K(1))= H^0(\Sigma , K \oplus \calO_\Sigma)$ be the canonical section of $ \bar{\pi}^*K \otimes \calO (1)$, called the ``\textbf{zero section}'' (both corresponding to the respective associated summand $\calO_\Sigma$). 
\begin{rmk} Note that the zero section and the infinity section are disjoint on $Y$. Also, when restricted to the total space $|K| \subset Y$ of the canonical bundle, $\lambda \in H^0(|K|, \bar{\pi}^*K)$ is the tautological section. 
\end{rmk}
Let $a = (a_1, \ldots , a_n) \in \calA$ and consider the section 
$$s_a = \lambda^n + (\bar{\pi}^*a_1)\lambda^{n-1}\sigma + \ldots + (\bar{\pi}^*a_n)\sigma^n \in H^0(Y, \bar{\pi}^*(K)^n\otimes \calO (n)).$$
We denote by $X_a$ the zeros of the section $s_a$ with the natural scheme structure and call it the \textbf{spectral curve} associated to $a\in \calA$. In particular we denote by 
$$\pi_a : X_a \to \Sigma ,$$
the restriction of $\pi : K \to \Sigma $ to $X_a$, and call it the \textbf{spectral cover} associated to $X_a$.  
\begin{rmks} 

\noindent 1. Note that $\sigma$ restricted to $X_a$ is everywhere non-zero and so trivializes $\calO (1)|_{X_a}$. Thus, $\lambda$ restricted to the spectral curve is a section of 
$$(\bar{\pi}^* (K) \otimes \calO (1))|_{X_a} \cong \pi_a^* (K).$$ 
In particular, the spectral curve $X_a$ can be seen as the zero scheme corresponding to the section
$$\lambda^n + \pi^*a_1\lambda^{n-1} + \ldots + \pi^*a_n \in H^0(|K|, \pi^*K^n).$$
 
\noindent 2. It follows from the remark above that the spectral curve $X_a$ is an element of the linear system $\CP (H^0(Y, \bar{\pi}^*(K)^n\otimes \calO (n)))$ which is contained in $|K| \subset Y$. As explained in \cite{hitchin1987stable}, this linear system is base-point free. This follows from the fact that $\lambda^n $ is a point of $\CP (H^0(Y, \bar{\pi}^*(K)^n\otimes \calO (n)))$, but the linear system $\CP H^0(\Sigma , K^n)$ is base-point free for $n \ge 2$, which we may assume as the case $n=1$ is trivial. In particular, it follows from Bertini's theorem that the generic spectral curve is a non-singular projective curve.

\end{rmks}

Since we will be dealing with non-reduced spectral curves, it is instructive to get a clearer picture of the scheme structure of $X_a$. For this, set $a_0 = 1 \in H^0(\Sigma , \calO_\Sigma)$. We may see the sections $a_i \in H^0(\Sigma , K^i)$ as embeddings  
$$a_i : K^{-n} \to K^{-n+i},$$ 
for each $i=0, \ldots , n$. Consider the ideal sheaf 
$$\calI_a = \left( \bigoplus_{i=0}^{n} a_i(K^{-n}) \right) \subset \Sym^\bullet K^{-1}$$
generated by the images of $a_i$. Then, 
$$X_a =  \Spec (\Sym^\bullet K^{-1}/\calI_a)  \subset |K| =  \Spec (\Sym^\bullet K^{-1}) .$$
In particular, the spectral cover
$$\pi_a : X_a \to \Sigma $$
is a finite morphism of degree $n$. 

Note that\footnote{This is standard (see e.g. \cite[Ex. 5.17]{hart}).}  
$$\pi_{a,*} (\calO_{X_a}) = \Sym^\bullet K^{-1}/\calI_a = \calO_\Sigma \oplus K^{-1} \oplus \ldots \oplus K^{- (n-1)}$$
and from this we can calculate the arithmetic genus $g_{_{X_a}} = p_a (X_a)$ of $X_a$ by using the fact that the Euler characteristic is invariant under direct images by finite morphisms. So,  
$$\chi (X_a ,  \calO_{X_a}) = \chi (\Sigma , \pi_{a,*} \calO_{X_a})$$
gives
\begin{equation}
g_{_{X_a}} = 1 + n^2(g-1).
\label{genus}
\end{equation}
A finite morphism is in particular affine. Let us make this explicit. Let $U  $ be an open affine of $\Sigma$ and consider a nowhere vanishing section $u \in H^0(U, K^{-1})$, so that $| K|_U | = \Spec (\calO_\Sigma (U)[u])$. Denote by 
$$\bar{a}_i = \inn{a_i}{u^i} \in \calO_\Sigma (U)$$
the local function on $U$ obtained by the natural pairing between $K^i$ and $K^{-i}$, for $i=1, \ldots , n$. Then, we obtain the affine open set 
$$\pi_a^{-1}(U) = \Spec \left(\frac{\calO_\Sigma (U)[u]}{(u^n + \bar{a}_1u^{n-1} + \ldots + \bar{a}_n)} \right),$$
on $X_a$ and the map $\pi_a$ restricted to this affine is simply the natural projection of $\calO_\Sigma (U)[u]$ onto its quotient by the ideal $(u^n + \bar{a}_1u^{n-1} + \ldots + \bar{a}_n) \subset \calO_\Sigma (U)[u]$. 

\subsection{The BNR correspondence}

Let $\calO_\Sigma (1)$ be an ample line bundle on $\Sigma$ of degree $\delta$. Given a point $a \in \calA$ of the Hitchin base for the Hitchin fibration 
$$h : \calM (n,d)  \to \calA = \bigoplus_{i=1}^n H^0(\Sigma , K^{i}),$$ 
let $\pi_a : X_a \to \Sigma$ be the spectral cover associated to the spectral curve $X_a$, as in the last section. We denote by $\calM (a; k) $ the Simpson moduli space $\calM (\calO_{X_a}(1), P)$, given by the polarization $\calO_{X_a}(1)$, obtained by the pullback $\pi_a^* \calO_\Sigma (1)$ of the ample line bundle $\calO_\Sigma (1)$ by the spectral cover, and Hilbert polynomial  
$$P (t) = n\delta t + k + \chi (\calO_{X_a}).$$
In other words, $\calM (a; k) $ is the Simpson moduli space for semi-stable rank $1$ sheaves on $X_a$ (of polarized degree $k$).  

Denote by $\calB$ the $\calO_\Sigma$-algebra $\pi_*\calO_{X_a}$. Since $\pi_a$ is an affine morphism, the direct image $\pi_{a,*}$ induces an equivalence of categories between the category of quasi-coherent $\calO_{X_a}$-modules and the category of $\calB$-modules (i.e., quasi-coherent $\calO_\Sigma$-modules with a $\calB$-module structure), which preserves the subcategories of torsion-free sheaves. If $\calE$ is a torsion-free sheaf on $X_a$ of (polarized) rank $1$ and degree $k$, then $V = \pi_{a,*} (\calE)$ is a vector bundle\footnote{We will avoid differentiating a vector bundle from its associated locally-free sheaf whenever notation becomes too cumbersome.}. Since the Euler characteristic is invariant under direct image by a finite map, by the projection formula we have $\chi (X_a , \calE \otimes \calO_{X_a}(t)) = \chi (\Sigma , \pi_{a,*} (\calE) \otimes \calO_\Sigma (t))$ and thus we obtain, as pointed out in \cite[Remark 3.8]{hausel2012prym}, the relations
\begin{align*}
n \rkp (\calE) & =  \rk (\pi_{a,*} (\calE))\\
\degp (\calE) + \rkp (\calE)\chi (\calO_{X_a}) & = \deg (\pi_{a,*} (\calE)) + \rk (\pi_{a,*} (\calE))\chi (\calO_\Sigma).
\end{align*} 
Thus, by the above and (\ref{genus}), $V$ is a vector bundle of rank $n$ and degree $k+ n(n-1)(1-g)$. A $\calB$-module structure on $V$ corresponds to an algebra homomorphism 
$$\Sym^\bullet K^{-1}/ \calI_a \to \End V,$$
which is equivalent to an $\calO_\Sigma$-module $K^{-1} \to \End V$, i.e., a Higgs field $\Phi$. Equivalently, we may see $\Phi$ as the twisted endomorphism of $V$ obtained from multiplication by the tautological section 
$$\lambda : \calE \to \calE \otimes \pi_a^*K.$$ 
In particular, we denote the Higgs bundle obtained from $\calE$ through this construction by 
$$\pi_{a,*}(\calE, \lambda) = (V,\Phi),$$
and sometimes say that $(V,\Phi)$ was obtained by the direct image construction of $\calE$.
It follows from \cite[Corollary 6.9]{si3} that $\calE$ is (semi-)stable (resp. polystable) if and only if $\pi_{a,*}(\calE, \lambda)$ is (semi-)stable (resp. polystable). Thus, since $h (V,\Phi) = a$ by \cite[Lemma 6.2]{hausel2012prym}, the fibre $h^{-1}(a)$ can be characterized in the following way. 
\begin{thm}\label{bnrcorrespondence} Let $a \in \calA$ be any point of the Hitchin base. Then, the fibre $h^{-1}(a)$ of the $GL(n,\C)$-Hitchin fibration is isomorphic to the Simpson moduli space $\calM (a; k)$, where $k =d+ n(n-1)(g-1)$.  
\end{thm}    
We will refer to the correspondence above as the \textbf{\textit{BNR} correspondence}, or spectral correspondence. The name is justified by the fact that when the spectral curve is an integral scheme, the characterization above is a theorem by Beauville, Narasimhan and Ramanan \cite[Proposition 3.6]{bnr} (in this case, Simpson moduli space simply parametrizes isomorphism classes of rank $1$ torsion-free sheaves on the spectral curve). Explicitly, given a Higgs bundle $(V,\Phi) \in h^{-1}(a)$, the cokernel of 
$$\pi_a^*\Phi - \lambda \Id: \pi_a^*V \to \pi_a^* (V\otimes K)$$ 
defines a torsion-free sheaf $\calE \otimes \pi_a^*K$ and one recovers the Higgs bundle as $\pi_{a,*}(\calE, \lambda) = (V,\Phi)$. In particular, the torsion-free sheaf fits into the exact sequence 
$$0 \to \calE \otimes  \pi_a^*K^{n-1} \to \pi_a^*V  \xrightarrow{\pi_a^*\Phi - \lambda \Id}  \pi_a^*(V\otimes K) \to \calE \otimes \pi_a^*K \to 0.$$
For general spectral curves, the characterization of the fibres is due to Schaub \cite{schaub}, with minor corrections, as pointed out in \cite{chaud, catal, hausel2012prym}.  

\begin{rmks}\label{123} 

\noindent 1. As noted, the generic spectral curve $X_a$ is non-singular (by Bertini's theorem) and the Simpson moduli space $\calM (a ; k)$ is isomorphic to $\Pic^k( \Sigma)$. In particular, these fibres are torsors for the abelian variety $\Pic^0(\Sigma)$.

\noindent 2. Note that, given a Higgs bundle $(V,\Phi)$, any proper $\Phi$-invariant subbundle $W$ defines a Higgs bundle $(W, \Phi|_W)$ and the characteristic polynomial of $\Phi_W$ must divide the characteristic polynomial of $\Phi$. In other words, the spectral curve of $(W,\Phi|_W)$ is a component of the spectral curve of $(V,\Phi)$. In particular, if $X_a$ is integral, all Higgs bundles in the fibre $h^{-1}(a)$ must be stable.

\noindent 3. The spectral correspondence for classical groups was achieved by Hitchin in \cite{hitchin1987stable} (see also \cite{schaposnik2013spectral} for spectral data for Higgs bundles for real groups). In particular, these can be seen as special loci of the Simpson moduli space. 

\end{rmks}

%% file: Chapters/Chapter2.tex

\chapter{Extensions of Higgs bundles} 

\label{Chapter2} 

\lhead{Chapter 2. \emph{Extensions of Higgs bundles}} 

In this chapter we introduce extensions of Higgs bundles as short exact sequences in the abelian category of Higgs bundles and show how these are related to the first hypercohomology group of a certain complex of sheaves, which can be seen as a Higgs bundle in its own right. After recalling the long exact sequences which will be used in subsequent chapters to obtain information about the hypercohomology groups of the Higgs bundles governing the extensions, we give a parametrization of certain symplectic and orthogonal Higgs bundles which appear as special extensions by maximally isotropic Higgs subbundles. The results will be used in chapters \ref{Chapter3} and \ref{Chapter4} when describing the singular fibres related to the real groups $SU^*(2m)$, $SO^*(4m)$ and $Sp(m,m)$.

\section{Classical Higgs bundles} \label{Sect2.1} 

A Higgs bundle on $\Sigma$ can be viewed as a complex of locally-free sheaves concentrated in degrees $0$ and $1$. Then, a homomorphism between two Higgs bundles $\mathsf{E} = (E \overset{\varphi}{\to} E\otimes K )$ and $\mathsf{F} = (F \overset{\psi}{\to} F\otimes K )$ is a homomorphism of complexes of the form
\[\xymatrix@M=0.08in{
E \ar[r]^--{\varphi} \ar[d]^{\alpha} & E\otimes K  \ar[d]^{\alpha \otimes 1} \\    
F \ar[r]^-{\psi} & F\otimes K ,}\]   
for some homomorphism $\alpha : E \to F$ of vector bundles.

Let $\mathsf{W}_i = (W_i \overset{\Phi_i}{\to} W_i\otimes K )$, $i=1,2$, be Higgs bundles on $\Sigma$. An \textbf{extension} of $\mathsf{W}_2 $ by $\mathsf{W}_1 $ is a short exact sequence (of complexes)  
\begin{equation}
0 \to \mathsf{W}_1 \to \mathsf{V} \to \mathsf{W}_2 \to 0,
\label{extensionsofhiggs0}
\end{equation} 
where $\mathsf{V} = (V \overset{\Phi}{\to} V\otimes K )$ is a Higgs bundle and the maps are homomorphisms of Higgs bundles. In particular, these can be seen as short exact sequences in the abelian category $\mathsf{Higgs}(\Sigma)$ of Higgs sheaves. We will sometimes denote such an extension by  
\begin{equation}
0 \to (W_1, \Phi_1) \to (V, \Phi) \to (W_2, \Phi_2) \to 0
\label{extensionsofhiggs}
\end{equation}
and say that $(V,\Phi)$ is an extension of $ (W_2, \Phi_2)$ by $(W_1, \Phi_1)$. This means, of course, that the vector bundle $V$ is an extension of $W_2$ by $W_1$ and $\Phi$ is a Higgs field on $V$ which projects to $\Phi_2$ and satisfies $\Phi|_{W_1} = \Phi_1$. In particular, $(W_1, \Phi_1)$ is a $\Phi$-invariant subbundle of $(V,\Phi)$. The split extension $\mathsf{W}_1 \oplus \mathsf{W}_2 = (W_1 \oplus W_2 \overset{\Phi_1 \oplus \Phi_2}{\longrightarrow} (W_1 \oplus W_2)\otimes K )$ will be called the \textbf{trivial extension}.

\begin{defin} Two extensions of $\mathsf{W}_2$ by $\mathsf{W}_1$ are \textbf{strongly equivalent} if we have an isomorphism of exact sequences
\[\xymatrix@M=0.08in{
0 \ar[r] & (W_1, \Phi_1) \ar[r] \ar[d]^{\Id} & (V,\Phi) \ar[r] \ar[d]^{\cong} & (W_2,\Phi_2) \ar[r] \ar[d]^{\Id} & 0\\    
0 \ar[r] & (W_1, \Phi_1) \ar[r] & (V^\prime,\Phi^\prime) \ar[r] & (W_2,\Phi_2) \ar[r] & 0,}\] 
where the vertical map in the middle is an isomorphism of Higgs bundles. If we allow $\Id_{W_i}$, $i=1,2$, to be more general automorphisms of Higgs bundles, we say that the extensions are \textbf{weakly equivalent} (i.e., two extensions are weakly equivalent if they are isomorphic short exact sequences in $\mathsf{Higgs}(\Sigma)$).      
\end{defin}
Extensions of Higgs bundles, modulo equivalence, are parametrized by the first hypercohomology group of a certain complex of sheaves. To see this, recall that given two Higgs bundles $\mathsf{V}_1 = (V_1 \overset{\Phi_1}{\to} V_1\otimes K )$ and $\mathsf{V}_2 = (V_2 \overset{\Phi_2}{\to} V_2\otimes K )$, their tensor product $\mathsf{V}_1 \otimes \mathsf{V}_2$ is defined as the Higgs bundle whose underlying vector bundle is $V_1\otimes V_2$ and Higgs field is $\Phi_1\otimes 1 + 1 \otimes \Phi_2$. We denote by $\mathsf{O} = (\calO_\Sigma \overset{0}{\to} K)$ the neutral Higgs bundle. The dual of a Higgs bundle $\mathsf{V} = (V \overset{\Phi}{\to} V\otimes K )$ is defined as the Higgs bundle $\mathsf{V}^*$ whose Higgs field is $-\trans{\Phi} : V^* \to V^*\otimes K$, so that the natural morphisms $\mathsf{O} \to \mathsf{V}^*\otimes \mathsf{V} \to  \mathsf{O}$ are Higgs bundle morphisms. Note also that the set of homomorphisms between two Higgs bundles $\mathsf{V}_1 $ and $\mathsf{V}_2$ corresponds precisely to the zeroth hypercohomology group of $\mathsf{V}_1^*\otimes \mathsf{V}_2$ (seen as a complex of sheaves).     

Tensoring (\ref{extensionsofhiggs0}) by $\mathsf{W_2}^*$ we obtain the short exact sequence (since $W_2$ is locally-free) 
$$0 \to \mathsf{W_2}^*\otimes \mathsf{W}_1 \to \mathsf{W_2}^*\otimes \mathsf{V} \to \mathsf{W_2}^*\otimes \mathsf{W}_2 \to 0.$$      
The exact sequence above is in particular a short exact sequence of two-term complex of sheaves. 
Considering the corresponding long exact sequence in hypercohomology (for more details, see Appendix \ref{AppendixB}) we get the connecting homomorphism  
\begin{equation}
\K^0 (\Sigma,\mathsf{W}_2^*\otimes \mathsf{W}_2) \to \K^1 (\Sigma, \mathsf{W}_2^*\otimes \mathsf{W}_1).
\label{delta}
\end{equation} 
Note that 
$$ \Id_{W_2} \in \K^0 (\Sigma, \mathsf{W}_2^*\otimes \mathsf{W}_2) \cong \ker (H^0(\Sigma, \End(W_2)) \to H^0(\Sigma, \End(W_2) \otimes K)).$$ 
Then, we define $\delta (V,\Phi)$, or $\delta (\mathsf{V})$, as the element in $\K^1 (\Sigma, \mathsf{W}_2^*\otimes \mathsf{W}_1 )$ given by the image of the identity endomorphism of $W_2$ via the map (\ref{delta}).  

Adapting Atiyah's argument \cite{aticonn} we can give a more precise description of $\delta (V,\Phi)$. One can always find an open covering $\calU = \{ U_i \}$ of $\Sigma$ such that the restriction to $U_i$ of the underlying extension of vector bundles
$0 \to W_1 \to V \overset{p}{\to} W_2 \to 0$
splits. Let $h_i : W_{2,i} \to V_i$ be the local splittings, where the subscript $i$ denotes restriction to $U_i$. Thus, $\{ h_i \}$ is a $0$-cochain of $\calO (\Hom(W_2,V))$ which lifts the $0$-cocycle $\Id_{W_2}$ of $\calO(\End(W_2))$ and $ \{ h_j - h_i \}$ is a representative for the image of $\Id_{W_2}$ under the map 
$$H^0(\Sigma , \End (W_2)) \to H^1(\Sigma, \Hom (W_2,W_1)),$$ 
which corresponds to the extension class of $V$. Applying the Higgs field of the Higgs bundle $\mathsf{W}_2^*\otimes \mathsf{V}$ to $h_i$ gives 
$$\psi_i = \Phi_i\circ h_i - h_i \circ \Phi_{2,i}$$ 
and this projects to zero under $p_i$. Thus, $\{ \psi_i \} \in C^0 (\mathcal{U}, \Hom(W_2, W_1) \otimes K) $ (since $p_i\circ \Phi_i\circ h_i = \Phi_{2,i}$) and 
$$(\{ h_j - h_i\}, \{\psi_i \}) \in C^1(\calU , \Hom (W_2,W_1)) \oplus C^0(\calU, \Hom (W_2,W_1) \otimes K)$$
is a representative for $\delta (V,\Phi) \in \K^1 (\Sigma, \mathsf{W}_2^*\otimes \mathsf{W}_1)$. Indeed, the hypercohomology groups of the complex $\mathsf{W}_2^*\otimes \mathsf{W}_1$ are calculated by the cohomology of the complex $(C^\bullet , d)$ 
\begin{eqnarray}
C^0  & = & C^{0, 0} = C^0 (\mathcal{U}, \Hom(W_2, W_1)) \nonumber \\
C^1 & = & C^{1, 0} \oplus  C^{0, 1}= C^1 (\mathcal{U}, \Hom(W_2, W_1)) \oplus C^0 (\mathcal{U}, \Hom(W_2, W_1) \otimes K) \nonumber \\
C^2 & = & C^{1, 1} = C^1 (\mathcal{U}, \Hom(W_2, W_1) \otimes K) \nonumber
\end{eqnarray}
whose differential is a combination between the \v{C}ech differential $\delta$ and the Higgs field $\phi_{21}$ of $\mathsf{W}_2^*\otimes \mathsf{W}_1$:
$$d^n \coloneqq \sum_{p+q = n}^{} \  \delta + (-1)^{q} \phi_{21} : C^n \to C^{n+1}.$$     
\begin{rmk} If we had considered local splittings $h_i : \WW_{2,i} \to \VV_i$ of the Higgs bundles, i.e., local splittings $h_i : W_{2,i} \to V_i$ such that the diagram 
\begin{equation}\label{cechdiag} 
\begin{gathered} 
\xymatrix{
W_{2,i} \ar[r]^{h_i} \ar[d]^{\Phi_{2,i}} & V_i  \ar[d]^{\Phi_i} \\    
(W_{2}\otimes K)_i \ar[r]^{h_i \otimes 1} & (V\otimes K)_i }
\end{gathered}
\end{equation}   
commutes, then $(\{ h_j - h_i \}, 0)$ would be a representative for $\delta(V,\Phi)$. In particular, the trivial extension corresponds to $0 \in \K^1 (\Sigma, \mathsf{W}_2^*\otimes \mathsf{W}_1)$. 
\end{rmk}
Let $\mathsf{V} = (V \overset{\Phi}{\to} V\otimes K )$ and $\mathsf{V}^\prime = (V^\prime \overset{\Phi^\prime}{\to} V^\prime \otimes K )$ be two weakly equivalent extensions of $\WW_2 $ by $\WW_1 $. In our applications, $\WW_1 $ and $\WW_2 $ are always stable Higgs bundles. In particular, their only automorphisms are non-zero scalars (see e.g. \cite[Thm 5.1.2]{hausel2001geometry}). In this case, we have 
\[\xymatrix@M=0.08in{
0 \ar[r] & \WW_1 \ar[r] \ar[d]^{\lambda_1} & \VV \ar[r] \ar[d]^{\cong} & \WW_2 \ar[r] \ar[d]^{\lambda_2} & 0\\    
0 \ar[r] & \WW_1 \ar[r] & \VV^\prime \ar[r] & \WW_2 \ar[r] & 0,}\] 
for some scalars $\lambda_1$, $\lambda_2 \in \C^\times$. Thus, $\lambda_2 \delta (V^\prime, \Phi^\prime) = \lambda_1\delta (V,\Phi)$. In particular, two strongly equivalent extensions define the same element in $\K^1 (\Sigma, \mathsf{W}_2^*\otimes \mathsf{W}_1)$, whereas weakly equivalent extensions define only the same class in $\mathbb{P} (\K^1 (\Sigma, \mathsf{W}_2^*\otimes \mathsf{W}_1))$.  
 
Conversely, given an element $\delta \in \K^1 (\Sigma, \mathsf{W}_2^*\otimes \mathsf{W}_1)$, its image under the map $\K^1 (\Sigma, \mathsf{W}_2^*\otimes \mathsf{W}_1) \to H^1(\Sigma , W_2^*\otimes W_1)$ (for more details, see (\ref{les1}) in the end of the section) defines an extension 
$$0 \to W_1 \to V \overset{p}{\to} W_2 \to 0$$
of $W_2$ by $W_1$. We may take an open covering $\calU = \{ U_i \}$ of $\Sigma$ such that $W_{1,i} \oplus W_{2,i}$ is isomorphic to $V_i$ via $(w_1, w_2) \mapsto w_1 + h_i(w_2)$, for holomorphic maps $h_i : W_{2,i}\to V_i$, so that $\{ h_j - h_i\}$ is a representative for $V$ in $H^1(\Sigma , W_2^*\otimes W_1)$. Then $\delta$ is represented by the cocycle $(\{h_j-h_i\}, \{\psi_i\})$, for some $\{\psi_i\} \in C^0(\calU , \Hom (W_2, W_1)\otimes K)$. We may then define $$\Phi_i  = \Phi_{1,i}\circ \Id - h_i\circ p_i + \psi_i\circ p_i + h_i\circ \Phi_{2,i}\circ p_i $$ 
and this is a Higgs field on $V$, which gives the extension of Higgs bundles. It is straightforward to check that by choosing another representative of $\delta$ we obtain a weakly equivalent extension. Thus we have proven the well-known\footnote{See e.g. \cite[Section 3]{thaddeus2000variation}.} result below.  
\begin{prop} The set of extensions of a Higgs bundle $(W_1, \Phi_1)$ by a Higgs bundle $(W_2, \Phi_2)$, modulo strong equivalence, is in bijection with $\mathbb{H}^1(\Sigma, \mathsf{W_2}^*\otimes \mathsf{W}_1)$. Moreover, if $(W_1, \Phi_1)$ and $(W_2, \Phi_2)$ are stable Higgs bundles, weak equivalence classes of such extensions are parametrized by $\mathbb{P} (\K^1 (\Sigma, \mathsf{W}_2^*\otimes \mathsf{W}_1))$. 
\end{prop}
 
It is convenient to also understand this correspondence from the Dolbeault point of view. Denote by $\bar{\partial}_V$, $\bar{\partial}_1$, $\bar{\partial}_2$ and $\bar{\partial}_{\Hom(W_2, W_1)}$ (or simply $\bar{\partial}_{\Hom}$ if no confusion is possible) the Dolbeault operator that gives the holomorphic structure on $V$, $W_1$, $W_2$ and $\Hom(W_2, W_1)$, respectively. By the Dolbeault lemma the following resolutions are acyclic 
\begin{align*}
& \calO (\Hom(W_2, W_1)) \to \Omega^{0}( \Hom(W_2, W_1)) \overset{\bar{\partial}_{\Hom}}{\longrightarrow} \Omega^{0,1}( \Hom(W_2, W_1)) \to 0,\\
&\calO (\Hom(W_2, W_1) \otimes K) \to \Omega^{1,0}( \Hom(W_2, W_1)) \overset{\bar{\partial}_{\Hom}}{\longrightarrow} \Omega^{1,1}( \Hom(W_2, W_1)) \to 0.
\end{align*}
Moreover we can define the commutative diagram
\[\xymatrix@M=0.08in{
\calO (\Hom(W_2, W_1)) \ar[r] \ar[d]^{\phi_{21}} & \Omega^{0}(\Hom(W_2, W_1)) \ar[r]^{\bar{\partial}_{\Hom}} \ar[d]^{\phi_{21}} & \Omega^{0,1}(\Hom(W_2, W_1)) \ar[r]\ar[d]^{\phi_{21}} & 0\\
\calO (\Hom(W_2, W_1) \otimes K)  \ar[r] & \Omega^{1,0}(\Hom(W_2, W_1)) \ar[r]^{\bar{\partial}_{\Hom}} & \Omega^{1,1}(\Hom(W_2, W_1)) \ar[r] & 0,}\] 
where we extend $\phi_{21}$ in the natural way
to a (0,1)-form $\alpha$ with value in $\Hom(W_2, W_1)$. 

The hypercohomology groups of the complex $\mathsf{W_2}^*\otimes \mathsf{W}_1$ can be computed via the cohomology of the associated simple complex $(C^\bullet, d)$   
\begin{eqnarray}
C^n(X) & \coloneqq & \bigoplus_{p+q = n} C^{p,q} \nonumber \\
d^n & \coloneqq & \sum_{p+q = n}^{} \  \bar{\partial}_{\Hom} + (-1)^{q} \phi_{21} \nonumber 
\end{eqnarray} 
where $C^{p,q} \coloneqq \Omega^{p,q}(\Sigma, \Hom(W_2,W_1))$.
In particular, given $\beta \in \Omega^{0,1}(\Sigma, \Hom(W_2,W_1))$ and $\psi \in \Omega^{1,0}(\Sigma, \Hom(W_2,W_1))$ satisfying 
$$\bar{\partial}_{\Hom} (\psi) + \phi_{21} (\beta) = 0, $$
the pair $(\beta, \psi)$ is a representative for a class in the first hypercohomology group of $\mathsf{W_2}^*\otimes \mathsf{W}_1$. Moreover, any other representative in the class of $(\beta, \psi)$ is of the form $(\beta, \psi) = (\beta + \bar{\partial}_{\Hom} \theta, \psi + \Phi_{21}(\theta)) \in \mathbb{H}^1(\Sigma,\mathsf{W_2}^*\otimes \mathsf{W}_1)$, where $\theta \in \Omega^0(\Sigma, \Hom(W_1,W_1))$.

In the smooth category, every short exact sequence splits (due to the existence of a  partition of unity) and so choose a smooth splitting
$$V \cong W_1 \oplus W_2 .$$
Note that this is only an isomorphism between the underlying smooth vector bundles and with respect to this $C^\infty$-splitting we can write
\[ \bar{\partial}_V = \left( \begin{array}{cc}
\bar{\partial}_1 & \beta \\
0 & \bar{\partial}_2
\end{array} \right)
\ \ \ \ \text{and} \ \ \ \  
\Phi = 
\left( \begin{array}{cc}
\Phi_1 & \psi \\
0 & \Phi_2
\end{array} \right),
\] 
where $\beta \in \Omega^{0,1}(\Sigma, \Hom(W_2, W_1))$ (second fundamental form) and $\psi \in \Omega^{1,0}(\Sigma, \Hom(W_2, W_1))$.

\begin{rmk} Any other smooth splitting differs from this one by a smooth section of $\Hom (W_2, W_1)$. Indeed, the short exact sequence 
$$0 \to W_1 \overset{i}{\to} V \overset{p}{\to} W_2 \to 0$$ 
splits if and only if there exists a map $V \to W_1$, such that composition with $i$ equals $\Id_{W_1}$. Let $f_k : V \to W_1$, $k = 1, 2$, be the maps corresponding to two different splittings. Given $w_2 \in W_2$, define $f (w_2) \coloneqq (f_1 - f_2) (v)$, for some $v \in V$ such that $p(v) = w_2$ ($p$ is surjective). This is clearly well-defined because any other $v^\prime \in V$ such that $p(v^\prime) = w_2$ satisfies $v - v^\prime \in \ker p = \im i$, but $(f_1 - f_2) \circ i = 0$.
\end{rmk}     

Now, the fact that $\Phi$ is holomorphic is translated into the following compatibility condition
\begin{equation}
\bar{\partial}_{\Hom} (\psi) + \phi_{21}(\beta) = 0.
\label{eqnhyper}
\end{equation}
Note that we chose a $C^\infty$-splitting $V = W_1 \oplus W_2$, associated to some $f_1 : V \to W_1$. Consider now another splitting $V^\prime = W_1 \oplus W_2$ given by $f_2 : V \to W_1$ and let $f \coloneqq f_1 - f_2 \in \Omega^0(\Sigma, \Hom (W_2,W_1))$. Here, we use $V^\prime$ to distinguish the second identification $V = W_1 \oplus W_2$ via $f_2$ with the first one given by $f_1$. With respect to this new splitting we write $\bar{\partial}_{V^\prime} = \left( \begin{matrix} \bar{\partial_1}& \beta^\prime\\ 0&\bar{\partial_2} \end{matrix} \right)$ and $\Phi^\prime = \left( \begin{matrix} \Phi_1& \psi^\prime\\ 0&\Phi_2 \end{matrix} \right)$. An isomorphism $F : V \to V^\prime$ is given by $$(a,b) \mapsto (a - f(a), b)$$ and we must have $(F \otimes \Id_K) \circ \Phi = \Phi^\prime \circ F$, which gives $\psi^\prime = \psi + \phi_{21} (f)$. Similarly, from the Dolbeault operators, we obtain $\beta^\prime = \beta + \bar{\partial}_{\Hom} (f)$.  
This means that $(\beta, \psi) \in \mathbb{H}^1(\Sigma,\mathsf{W_2}^*\otimes \mathsf{W}_1)$.

Conversely, given $(\beta, \psi) \in \mathbb{H}^1(\Sigma,\mathsf{W_2}^*\otimes \mathsf{W}_1)$, the element $\beta \in \Omega^{0,1}(\Sigma, \Hom (W_2, W_1)) \cong H^1(\Sigma, \Hom (W_2, W_1))$ gives an extension of bundles
$$0 \to W_1 \to V \to W_2 \to 0.$$
Choose a smooth splitting $V= W_1 \oplus W_2$ with 
$$\bar{\partial}_V = \left( \begin{matrix} \bar{\partial_1}& \beta\\ 0&\bar{\partial_2} \end{matrix} \right)$$
and define 
$$\Phi \coloneqq \left( \begin{matrix} \Phi_1& \psi\\ 0&\Phi_2 \end{matrix} \right).$$
We finish this section with the following remarks. Both spectral sequences associated to the double complex that calculates hypercohomology of the complex of sheaves $\mathsf{W_2}^*\otimes \mathsf{W}_1$ satisfy $E_2^{p,q} = 0$, unless $p,q \in \{0,1 \}$. Therefore, as discussed in Appendix \ref{AppendixB}, we have the following short exact sequence
$$0 \rightarrow E_2^{1,0} \rightarrow \mathbb{H}^1(\Sigma, \mathsf{W_2}^*\otimes \mathsf{W}_1) \rightarrow E_2^{0,1} \rightarrow 0.$$
The first spectral sequence gives
\begin{align*}
0 & \rightarrow \coker (H^0(\Sigma, \Hom (W_2, W_1) \rightarrow H^0(\Sigma, \Hom (W_2, W_1) \otimes K))) \rightarrow \mathbb{H}^1(\Sigma, \mathsf{W_2}^*\otimes \mathsf{W}_1) \\
& \rightarrow \ker (H^1(\Sigma, \Hom (W_2, W_1) \rightarrow H^1(\Sigma, \Hom (W_2, W_1) \otimes K))) \rightarrow 0.
\end{align*}
This can be written as 
\begin{align}
0 & \rightarrow \mathbb{H}^0(\Sigma,\mathsf{W_2}^*\otimes \mathsf{W}_1) \rightarrow H^0(\Sigma, \Hom (W_2, W_1) \rightarrow H^0(\Sigma, \Hom (W_2, W_1) \otimes K)) \nonumber \\
& \rightarrow \mathbb{H}^1(\Sigma,\mathsf{W_2}^*\otimes \mathsf{W}_1) \rightarrow H^1(\Sigma, \Hom (W_2, W_1) \rightarrow H^1(\Sigma, \Hom (W_2, W_1) \otimes K)) \nonumber \\
& \rightarrow \mathbb{H}^2(\Sigma,\mathsf{W_2}^*\otimes \mathsf{W}_1) \rightarrow 0. 
\label{les1}
\end{align}


Now, the second spectral sequence, which will be important for us further on, gives
\begin{equation}
0 \to H^1(\Sigma, \ker \phi_{21}) \to \mathbb{H}^1 \to H^0(\Sigma, \coker \phi_{21}) \to 0.
\label{les2}
\end{equation}

\section{The orthogonal and symplectic cases}

Let $(V,\Phi)$ be a Higgs bundle of rank $2n$ on $\Sigma$ and $\alpha$ a bilinear non-degenerate form on $V$ with respect to which the Higgs field $\Phi$ is skew-symmetric. Of course, when $\alpha$ is symmetric (respectively, skew-symmetric), $((V,\alpha),\Phi)$ is an orthogonal (respectively, symplectic) Higgs bundle.  

For any subbundle $W \subset V$, the orthogonal complement $W^{\perp_\alpha}$ of $W$ with respect to $\alpha$ is defined as the kernel of the map 
\begin{align*}
V &\to W^* \\
v & \mapsto \alpha (v, \cdot)|_W,
\end{align*}
so we have a short exact sequence
$$0 \to W^{\perp_\alpha} \to V \to W^* \to 0.$$
If $W$ is a $\Phi$-invariant maximally isotropic (i.e., $W^{\perp_\alpha} = W$) subbundle, we obtain an exact sequence of Higgs bundles of the form\footnote{Note that the Higgs field of $W^*$, say $\phi^\prime$ for now, is obtained via the map $V\to W^*$. More precisely, given $\delta \in W^*$, $\delta = \alpha (v, \cdot)$, for some (non-unique) $v \in V$, we obtain $\phi^\prime (\delta) = \alpha (\Phi v, {}\cdot{})=- \alpha (v, \phi {}\cdot{}) = - \delta (\phi {}\cdot {}) = -\tp{\phi}(\delta)$. So, as explained in the last section, this is the natural Higgs field in $W^*$ associated to the Higgs bundle $(W,\phi)$.}
$$0 \to (W, \phi) \to (V, \Phi) \to (W^*, -\trans{\phi}) \to 0,$$
where $\phi \coloneqq \Phi|_W$.
%
%
Choose a smooth splitting $V \cong W \oplus W^*$ so that $\alpha$ becomes
\begin{align*}
q = &\left( \begin{matrix} 0& 1\\ 1&0 \end{matrix} \right),
\end{align*}  
if the Higgs bundle is orthogonal and
\begin{align*}
\omega = &\left( \begin{matrix} 0& 1\\ -1&0 \end{matrix} \right),
\end{align*} 
if the Higgs bundle is symplectic. Thus, with respect to this smooth splitting,
\[ \bar{\partial}_V = \left( \begin{array}{cc}
\bar{\partial}_W & \beta \\
0 & -\trans{\bar{\partial}_W}
\end{array} \right)
\ \ \ \ \text{and} \ \ \ \  
\Phi = 
\left( \begin{array}{cc}
\phi & \psi \\
0 & -\trans{\phi}
\end{array} \right).
\] 
Since $\Phi$ is skew-symmetric with respect to $\alpha$ we have $\trans{\Phi}\alpha + \alpha \Phi = 0$. Thus, $\psi \in \Omega^{1,0}(\Sigma, W\otimes W)$ is an element of $\Omega^{1,0}(\Sigma, \Lambda^2 W)$ in the orthogonal case and of $\Omega^{1,0}(\Sigma, \Sym^2 W)$ in the symplectic case. The form $\alpha$ can be seen as a global holomorphic section of $\Hom (V,V^*)$, so
\begin{equation}
0 = (\bar{\partial}_{\Hom(V,V^*)}\alpha)v = -\trans{\bar{\partial}}_V (\alpha(v)) - \alpha(\bar{\partial}_V(v)),
\label{calculation}
\end{equation}
for any smooth section $v = (w,\xi)$ of $V \cong W\oplus W^*$. The condition (\ref{calculation}) is equivalent to $\beta \in \Omega^{0,1}(\Sigma, W\otimes W)$ being skew-symmetric (respectively, symmetric) when $\alpha$ is symmetric (respectively, skew-symmetric). This means that $(\psi, \beta)$ is a representative for a class in the first hypercohomology group of a subcomplex of the complex of sheaves $\mathsf{W}\otimes \mathsf{W}$. In the orthogonal case the two-term complex of sheaves is the Higgs bundle $\Lambda^2\mathsf{W}$, which corresponds to the complex 
\begin{align*}
\hat{\phi} : \calO(\Lambda^2 W) & \to \calO(\Lambda^2 W \otimes K )\\
w_1\wedge w_2 & \mapsto \phi (w_1)\wedge w_2 + w_1\wedge \phi (w_2)
\end{align*}  
while in the symplectic case the complex is 
\begin{align*}
\hat{\phi} : \calO (\Sym^2 W) & \to \calO (\Sym^2 W \otimes K) \\
w_1\odot w_2 & \mapsto  \phi (w_1)\odot w_2 + w_1\odot \phi (w_2),
\end{align*}
corresponding to the Higgs bundle $\Sym^2\mathsf{W}$. 

The class $\delta (V,\Phi) $ is thus an element of $\K^1 (\Sigma, \Lambda^2\mathsf{W})$ in the orthogonal case and of $\K^1 (\Sigma, \Sym^2\mathsf{W})$ in the symplectic case.  

\begin{prop}\label{themext} An extension $(V, \Phi)$ of $(W^*, -\trans{\phi})$ by $(W, \phi)$ has an orthogonal (resp., symplectic) structure turning it into an orthogonal (resp., symplectic) Higgs bundle and with respect to which $W$ is a maximally isotropic subbundle if and only if it is strongly equivalent to an extension of Higgs bundles whose hypercohomology class belongs to $\K^1 (\Sigma, \Lambda^2\mathsf{W})$ (resp., $\K^1 (\Sigma, \Sym^2\mathsf{W})$). 
\end{prop}
\begin{proof}
It remains only to prove the converse, but this follows directly from the considerations of Section \ref{Sect2.1} and \cite[Criterion 2.1]{rmks}, where the case when there is no Higgs fields is dealt with by Hitching.
\end{proof}

%% file: Chapters/Chapter3.tex

\chapter{The group $SU^*(2m)$} 

\label{Chapter3} 

\lhead{Chapter 3. \emph{The group $SU^*(2m)$}} 

Let $U$ be a complex vector space of dimension $2m$ endowed with a non-degenerate skew-form $\omega$ and $A \in \End(U)$ symmetric with respect to $\omega$. Given eigenvectors $u_i, u_j$ of $A$ with corresponding eigenvalues $c_i, c_j$, respectively, one has 
\begin{align*}
c_i \omega (u_i, u_j) & = \omega (Au_i, u_j) \\
& = \omega (u_i, Au_j) \\
& = c_j \omega (u_i, u_j).
\end{align*}   
By the non-degeneracy of $\omega$ it follows that the eigenspaces must be even-dimensional and the characteristic polynomial of $A$ is of the form $\det (x \Id - A) = p(x)^2$. In particular, as explained in Hitchin-Schaposnik \cite{hitchin2014}, the polynomial $p(x)$ may be seen as the Pfaffian polynomial of $x\Id - A$ (i.e., $p(x)$ is the polynomial defined by $p(x)\omega^m = (x\omega - \alpha)^m$, where $\alpha \in \Lambda^2U^*$ is the element corresponding to $A$ via the isomorphism $\End (U) \cong U^* \otimes U^*$ given by $\omega$). 

Let $(V,\Phi)$ be an $SU^*(2m)$-Higgs bundle (see Example \ref{su*}). The Higgs field $\Phi$ is symmetric with respect to the symplectic form $\omega$ of $V$, so the image of $(V,\Phi)$ under the $SL(2m,\C)$-Hitchin fibration 
$$h : \mathcal{M} (SL(2m, \mathbb{C})) \to \calA (SL(2m, \mathbb{C})) = \bigoplus_{i=2}^{2m}H^0(\Sigma, K^i) $$
is given by the coefficients of a polynomial of the form $p(x)^2$, where $p(x) = x^m + a_2 x^{m-2} + \ldots + a_m$, $a_i \in H^0(\Sigma, K^i)$. By abuse of notation, we will refer to such a fibre as $h^{-1}(p^2)$.  

Denote the projection of the canonical bundle by $\pi : K \to \Sigma$, its total space by $|K|$ and let $\lambda \in H^0(|K|, \pi^*K)$ be the tautological section. Note that the spectral curve corresponding to $(V,\Phi)$ is non-reduced. However, by Bertini's theorem, the reduced curve $S = \zeros (p(\lambda))  \subseteq |K|$ is generically non-singular and we assume that is the case. In particular, restricting $\pi$ to $S$ gives a ramified $m$-fold covering $$\pi : S \to \Sigma .$$
Also, by the adjunction formula, $K_S \cong \pi^*K^m$, which tells us that the genus of $S$ is 
$$g_{_S} = m^2(g-1)+1.$$
Note that from our definition of the polynomial $p(x)$, any $SU^*(2m)$-Higgs bundle $(V,\Phi)$ satisfies $p(\Phi) = 0$ and so the cokernel of the map 
$$\tilde{\pi}^*\Phi - \lambda \Id : \tilde{\pi}^*V  \to  \tilde{\pi}^* (V\otimes K)$$
is a torsion-free sheaf supported on the reduced curve $S$, where $\tilde{\pi} : X \to \Sigma$ is the spectral cover associated to the spectral curve $X = \zeros (p(\lambda)^2)$. Since $S$ is non-singular and it has miltiplicity $2$ in $X$, this sheaf is a vector bundle of rank $2$ on $S$. One then recovers the $SU^*(2m)$-Higgs bundle as $(V,\Phi) = \pi_*(E, \lambda)$, i.e., the Higgs bundle obtained by the direct image construction of $E$. According to \cite{hitchin2014} we have the following.
\begin{prop} \label{prop1} \cite{hitchin2014} Let $p(\lambda) =\lambda^m+ \pi^*a_2\lambda^{m-2}+\dots+\pi^*a_m$ be a section of the line bundle  $\pi^*K^m$ on the cotangent bundle of $\Sigma$ whose divisor is a smooth curve $S$, and let $E$ be a rank $2$ vector bundle on $S$. Then the direct image of $ \lambda :E\rightarrow E\otimes \pi^*K$ defines a semi-stable Higgs bundle on $\Sigma$ (with underlying vector bundle $\pi_*E$) for the group $SU^*(2m)$ if and only if $E$ is semi-stable with determinant line bundle $\Lambda^2E\cong \pi^*K^{m-1}$.
    \end{prop}

From the proposition above, we can identify the locus $N_0 = N_0(SL(2m,\C))$ in $h^{-1}(p^2)$ corresponding to $SU^*(2m)$-Higgs bundles with 
$$N_0 = \calU_S (2, \pi^*K^{m-1}).$$ 

Our goal is to describe the rest of the fibre $h^{-1}(p^2)$ using data from the reduced non-singular curve $S$. Let $(V,\Phi) \in h^{-1}(p^2)$. We have two cases depending on the vanishing of $p(\Phi) : V \to V \otimes K^m$.   
\begin{prop}\label{N} The locus 
$$  \{ (V,\Phi) \in h^{-1}(p^2) \ | \ p(\Phi) = 0 \}$$
is isomorphic to 
$$N = \{ E \in \calU_S(2,e) \ | \ \det(\pi_*E) \cong \calO_\Sigma\},$$ 
where $\calU_S(2, e)$ is the moduli space of semi-stable rank $2$ bundles on $S$ of degree $e = 2m(m-1)(g-1)$. 
\end{prop}
\begin{proof}
Since $S$ is smooth and $\lambda$ is a well-defined eigenvalue of $\Phi$ on $S$, the cokernel of $\pi^*\Phi - \lambda$ defines a rank $2$ vector bundle $E$ on $S$
$$0 \to E \otimes \calO_S(- R_\pi) \to \pi^*V \xrightarrow{\pi^*\Phi - \lambda \Id} \pi^*(V \otimes K) \to E \otimes \pi^*K \to 0,$$ 
where $R_\pi$ is the ramification divisor of $\pi$ (in particular, $\calO_S( R_\pi) \cong \pi^*K^{m -1}$ since we see $d\pi$ as a holomorphic section from $K_S^{-1}$ to $\pi^*K^{-1}$). We recover $(V,\Phi)$ by the direct image construction. That is, we take $V = \pi_*E$ and push-forward the map multiplication by the tautological section $\lambda : E \to E \otimes \pi^*K$, obtaining $\Phi = \pi_* \lambda : \pi_*E \to \pi_* (E \otimes \pi^*K) = \pi_* E \otimes K$. Since the Euler characteristic is invariant under direct images by finite maps, we must have $\deg (E) = 2m(m-1)(g-1)$. Suppose $L$ is a line subbundle of $E$ with $\deg(L) > e/2$. Then, $W \coloneqq \pi_* L $ is a $\Phi$-variant subbundle of $(V,\Phi)$ of rank $m$ and degree $\deg (L) - m(m-1)(g-1)$, so $\deg (W) > 0$, contradicting the stability of $(V,\Phi)$. Thus, the rank $2$ bundle $E$ must be semi-stable.

Start now with $E \in N$ and let $(V,\Phi) = \pi_*(E,\lambda)$ be the Higgs bundle obtained by the direct image construction of $E$. The constraint on the determinant makes $(V,\Phi)$ into a $SL(2m,\C)$-Higgs bundle. Note that $p(\lambda) = 0$ holds on $S$ and if we push-forward this equation to $\Sigma$ we obtain $p(\Phi) = 0$. Now, by the Cayley-Hamilton theorem, the characteristic polynomial of $\Phi$ must divide $p$, which is irreducible, therefore $\det (x - \Phi ) = p(x)$. Let us check that $(V,\Phi) = \pi_*(E,\lambda)$ is semi-stable. Let $W$ be a proper $\Phi$-invariant subbundle of $V$. Then, its characteristic polynomial $p_1(x) = \det (x\Id - \Phi|_W)$ properly divides $p(x)^2$ and, since $S$ is smooth, and in particular irreducible, it must be $p(x)$. Semi-stability of $E$ implies $\deg (W) \leq 0$ and $(V,\Phi)$ is semi-stable.
\end{proof}
 
Before we proceed we recall the definition of the \textbf{norm map}  
\begin{align}
\Nm_\pi : \Pic^0 (S) & \to \Pic^0 (\Sigma) \label{normmapsmt} \\ 
\mathcal{O}(\Sigma n_i p_i) & \mapsto \mathcal{O}(\Sigma n_i \pi (p_i)) \nonumber
\end{align}
associated to the covering $\pi : S \to \Sigma$. Here, $n_i \in \mathbb{Z}$, $p_i \in S$ and the map is well-defined as it only depends on the linear equivalence class of the divisor $\sum n_i p_i$. We will sometimes omit the subscript and denote the norm map simply by $\Nm$ if the covering to which it is related is clear. This is clearly a group homomorphism. Also, it satisfies (and can be characterized by - see e.g. \cite{arb}) 
\begin{align}
\Nm (L) &= \det (\pi_* (L)) \otimes (\det (\pi_*\calO_S))^{-1} \label{directimage}\\
&= \det (\pi_* (L)) \otimes K^{m(m-1)/2} \nonumber
\end{align}
for any line bundle $L$ on $S$. If the pullback $\pi^* : \Pic^0 (\Sigma) \to \Pic^0 (S)$ is injective, then the norm map homomorphism can also be identified with the transpose of $\pi^*$ (and hence its kernel is connected). The \textbf{Prym variety} $\Prym$ associated to the covering $\pi: S \to \Sigma$ is defined as the kernel of $\Nm_\pi : \Pic^0(S) \to \Pic^0(\Sigma)$. It is an abelian variety, which, in our case, is connected as the pullback map $\pi^* : \Pic^0(\Sigma) \to \Pic^0 (S)$ is injective. We state a technical lemma which will be needed later on in the text.

\begin{lemma}\label{nmlm} Let $M$, $E$ be vector bundles on $S$ of rank $r$ and $2$, respectively, and $L$ a line bundle on $S$. We have
\begin{enumerate}[label=\alph*),ref=\alph*]
\item $\det (\pi_* (M \otimes L)) \cong \det(\pi_*M) \otimes \Nm (L)^r$.
\item $\det (\pi_* (\det E)) \cong \det (\pi_*E)\otimes K^{m(m-1)/2}$\\
(or, equivalently, $\Nm (\det E)  \cong \det (\pi_* E)   \otimes K^{m(m-1)}$).
\end{enumerate}
\end{lemma}
\begin{proof}
a) The statement follows from induction on the rank $r$. Indeed, for $r=1$ this follows directly from the characterization of the direct image (\ref{directimage}) and the fact that the norm map is a group homomorphism. Assume this is true for rank $r-1$ and let $M$ and $L$ be as in the lemma. The locally-free sheaf $\calE = \calO (M)$ fits into a short exact sequence
$$0 \to \calL^\prime \to \calE \to \calE^\prime \to 0,$$
where $\calL^\prime$ is an invertible sheaf and $\calE^\prime$ a locally free sheaf of rank $k-1$. Note that if $M$ has a section, the statement is clear. Otherwise, choose an effective divisor on $S$ of degree $d$ such that $rd > h^1(M)$. Part of the exact sequence in cohomology associated to the short exact sequence
$$0 \to \calE \to \calE (D) \to \calE \otimes \calO_D \to 0 $$
is 
$$H^0(S, \calE (D)) \to H^0(S, \calE \otimes \calO_D) \to H^1(S, \calE).$$
As $rd > h^1(M)$, the map $H^0(S, \calE \otimes \calO_D) \to H^1(S, \calE)$ cannot be injective, thus we can find a non-zero section of $\calE (D)$, or, in other words, a map $\calO (-D) \to \calE$ whose image is an invertible sheaf. The degree of an invertible subsheaf of $\calE$ is bounded above by Riemann-Roch formula, so we take a maximal degree invertible subsheaf. This guarantees that the cokernel is also locally-free. Let $\calL$ be the invertible sheaf corresponding to the line bundle $L$. As $R^i\pi_* = 0$, $i>0$, we obtain 
$$0 \to \pi_*(\calL \calL^\prime) \to \pi_* (\calE \otimes \calL) \to \pi_*(\calE^\prime \otimes \calL) \to 0.$$ 
Thus, $\det (\pi_* (\calE \otimes \calL)) \cong \det (\pi_* (\calL \calL^\prime)) \det (\pi_* (\calE^\prime \otimes \calL))$, but $\det (\pi_* (\calL \calL^\prime)) \cong \det (\pi_* \calL^\prime) \Nm (L)$ and $\det (\pi_* (\calE^\prime \otimes \calL)) \cong \det (\pi_* E^\prime) \Nm(\calL)^{r-1}$ (by the induction hypothesis). Since $$\det (\pi_* \calL^\prime) \det (\pi_* \calE^\prime) \cong \det (\pi_* \calE) $$ the result follows.\\
b) Every rank $2$ vector bundle $E$ on a curve fits into a short exact sequence
$$0 \to L_1 \to E \to L_2 \to 0,$$
for some line bundles $L_1$ and $L_2$. Thus, $\det (E) \cong L_1L_2$ and    
\begin{align*}
\det (\pi_* (\det E)) & \cong \det (\pi_* (L_1L_2))\\
& \cong \Nm (L_1L_2) K^{-m(m-1)/2}\\
& \cong \det (\pi_* L_1) K^{m(m-1)/2} \det (\pi_* L_2) K^{m(m-1)/2}K^{-m(m-1)/2}\\
& \cong \det (\pi_* L_1)  \det (\pi_* L_2) K^{m(m-1)/2}.
\end{align*}
But $\det (\pi_*E) \cong \det (\pi_* L_1)  \det (\pi_* L_2)$ since $R^i\pi_* = 0$, $i>0$, proving the statement. 
\end{proof}

Note that since elements of $N_0$ are rank $2$ vector bundles on $S$ and they define $SU^*(2m)$-Higgs bundles in the fibre $h^{-1}(p^2)$, from Proposition \ref{N}, $N_0$ must be inside $N$. More directly, one may use the lemma above to see this. Indeed, let $E \in N_0$. Then, 
\begin{align*}
\det (\pi_* (\det E)) & = \det (\pi_* (\pi^* K^{m-1}))\\
& = \det ((\calO_\Sigma \oplus K^{-1} \oplus \ldots \oplus K^{-(m-1)})\otimes K^{m-1}) \\
& = \det (\calO_\Sigma \oplus K^{1} \oplus \ldots \oplus K^{(m-1)})\\
& = K^{m(m-1)/2}.
\end{align*}
From item b) of Lemma \ref{nmlm}, it follows that $\det (\pi_*E)\cong \calO_\Sigma $, which shows that $N_0 \subset N$.  

Consider now $(V,\Phi) \in h^{-1}(p^2)$ such that $p(\Phi) \in H^0(\Sigma, \End V \otimes K^m)$ is not identically zero. The kernel of the map 
$$p(\Phi) : \mathcal{O}(V) \to \mathcal{O}(V \otimes K^m)$$
is a locally free sheaf which corresponds to a proper $\Phi$-invariant subbundle\footnote{More precisely, the kernel of $p(\phi)$ is a locally free sheaf of the form $\calO (W)$, where $W$ is a vector bundle on $\Sigma$. Denote by $\calG$ the quotient of $\calO(V)$ by $\calO(W)$ and let $\calT$ be its torsion. Then, the kernel of the natural map $\calO(V) \to \calG /\calT$ is isomorphic to $\calO(W_1)$, where $W_1$ is a subbundle of $V$ of same rank as $W$ (and possibly higher degree). This is usually called the vector bundle generically generated by $\ker p(\Phi)$.} $W_1 \subseteq V$. Now, the characteristic polynomial of $\Phi_1 = \Phi|_{W_1}$ divides the characteristic polynomial of $\Phi$, which is $p(x)^2$. The smoothness of $S$ implies again that the characteristic polynomial of $\Phi_1$ must be $p(x)$ (by assumption, it cannot be $p(x)^2$). In particular, $W_1$ has rank $m$ and by stability $\deg (W_1) = -d \leqslant 0$. Thus, we have the following extension of Higgs bundles
$$0 \to (W_1, \Phi_1) \to (V, \Phi) \to (W_2, \Phi_2) \to 0,$$
where $W_2$ is the quotient bundle $V/W_1$ and $\Phi$ projects to $W_2$ giving the Higgs field $\Phi_2$ (this projection being well-defined thanks to the $\Phi$-invariance of $W_1$). Notice that the characteristic polynomial of $\Phi_2$ is also $p(x)$. In particular, since $S$ is smooth, $(W_1, \Phi_1)$ and $(W_2, \Phi_2)$ are stable Higgs bundles (see the second item of Remark \ref{123}). 

Thus, to analyze points $\mathsf{V} = (V,\Phi) \in h^{-1}(p^2) \setminus N$ we consider extensions of Higgs bundles  
\begin{equation}
0 \to \mathsf{W}_1 \to \mathsf{V} \to \mathsf{W}_2 \to 0,
\label{ext2}
\end{equation}
where 
\begin{enumerate}
\item $\mathsf{W}_1 = (W_1, \Phi_1)$ and $ \mathsf{W}_2 = (W_2, \Phi_2)$ are stable Higgs bundles having the same spectral curve $S$, with $\deg(W_2) = - \deg (W_1) = d \geq 0$ and  
\item $\det(W_1) \otimes \det(W_2) \cong \mathcal{O}_{\Sigma}$.
\end{enumerate}

Let us rephrase these conditions in terms of $S$. The first condition is equivalent, by the \textit{BNR} correspondence, to the existence of line bundles $L_i \in \Pic^{d_i}(S)$, $i=1,2$, on $S$ of degree
\begin{align*}
d_1 & = -d + m(m-1)(g-1) \\
d_2 & = d + m(m-1)(g-1)
\end{align*}
such that $\mathsf{W}_i = \pi_*(L_i, \lambda)$. The degrees can be computed by the Riemann-Roch theorem using that the Euler characteristic is preserved by direct image of finite maps. 

Using the norm map, the condition $\det(W_1) \otimes \det(W_2) \cong \mathcal{O}_{\Sigma}$ becomes
$$\Nm (L_1L_2) = K^{m(m-1)},$$  
or equivalently 
$$L_1L_2\pi^*K^{1-m} \in \Prym .$$
\begin{rmk}\label{d>0} If $d = 0$, $(V,\Phi)$ is strictly semi-stable and we can replace it by the corresponding polystable object, which in this case is isomorphic to $(W_1, \Phi_2) \oplus (W_2, \Phi_2)$. But $(W_1, \Phi_2) \oplus (W_2, \Phi_2) \cong \pi_*(L_1 \oplus L_2, \lambda)$ and we are back to the first case. 
Note also that, by stability, the extension never splits if $d>0$.
\end{rmk}

So, Higgs bundles $(V,\Phi)$ in this fibre lying outside the locus $N$ are non-trivial extensions of $(W_2, \Phi_2) = \pi_*(L_2, \lambda)$ by $(W_1, \Phi_1) = \pi_*(L_1, \lambda)$, for some line bundles $L_i \in \Pic^{d_i} (S)$, $i=1,2$, with $2d_2 > e$ and $L_1L_2\pi^*K^{1-m} \in \Prym$. Note that the condition $2d_2 > e$ is equivalent to $d>0$ (the integer $e$, as defined in Proposition \ref{N}, is equal to $2m(m-1)(g-1)$). 

As discussed in the previous chapter, the space of extensions of the form (\ref{ext2}) is parametrized by the first hypercohomology group of the induced Higgs sheaf $\mathsf{W_2}^*\otimes \mathsf{W}_1$. In other words, we look at the  first hypercohomology group of the two-term complex of sheaves
\begin{align*}
\phi_{21}: \calO (\Hom(W_2, W_1)) & \to \calO ( \Hom(W_2, W_1) \otimes K)\\
s & \mapsto - (s \otimes 1) \circ \Phi_2 + \Phi_1 \circ s,
\end{align*}
where $s$ is a local section of $\Hom(W_2, W_1)$. To simplify notation, if no confusion arises, we will denote $\mathbb{H}^1(\Sigma, \mathsf{W_2}^*\otimes \mathsf{W}_1)$ by $\mathbb{H}^1$.  
One of the spectral sequences associated to the double complex that calculates hypercohomology give us the following short exact sequence (see Section \ref{Sect2.1} for more details) 
\begin{equation}
0 \to H^1(\Sigma, \ker \phi_{21}) \to \mathbb{H}^1 \to H^0(\Sigma, \coker \phi_{21}) \to 0.
\label{ses2}
\end{equation} 

From the smoothness of the curve $S$, $\ker \phi_{21}$ and $\coker \phi_{21}$ turn out to be vector bundles on $\Sigma$ and we can characterize them as the direct image of some line bundles on $S$ as the next lemma shows. 
\begin{lemma} The $\calO_\Sigma$-modules $\ker \phi_{21}$ and $\coker \phi_{21}$ are locally-free sheaves of rank $m$. Moreover,
\begin{enumerate}[label=\alph*),ref=\alph*]
\item $\ker \phi_{21} \cong \pi_*(L_2^*L_1)$.
\item $\coker \phi_{21} \cong \pi_*(L_2^*L_1K_S)$.
\end{enumerate}
\end{lemma}
\begin{proof}
a) We can cover $\Sigma$ by open sets $U \subset \Sigma$ which trivialize both $W_1$ and $W_2$. Denote the trivial bundle of rank $m$ on $U$ by $W \cong W_1|_U \cong W_2|_U$ and note that $L_1$ and $L_2$ will be isomorphic on $U$, so $\phi \coloneqq \Phi_1|_U = \Phi_2|_U$. Then,
$$\phi_{21}|_U = [ \phi , -] : \End (W) \to \End (W) \otimes K|_U.$$
Now, $S$ smooth implies that $(W_i,\Phi_i)$, $i=1,2$, are everywhere regular\footnote{A classical Higgs bundle $(V,\Phi)$ of rank $n$ is said to be \textbf{everywhere regular} if, for every $x\in \Sigma$, the dimension of the centralizer of $\Phi_x$ is equal to $n$. The \textit{BNR} correspondence gives a bijection between (isomorphism classes of) rank $1$ torsion-free sheaves on an integral spectral curve $X$ and (isomorphism classes of) Higgs bundles with spectral curve $X$. Under this correspondence, line bundles on $X$ give everywhere regular Higgs bundles \cite{B} (see also Remark 2.4 in \cite{mark}). In our case $S$ is assumed to be smooth, so every torsion-free sheaf of rank $1$ is a line bundle and the Higgs bundles obtained under this correspondence are all stable and everywhere regular.}. So, kernel and cokernel of $\phi$ are locally free of rank $m$. Concretely, the kernel is the sheaf of sections of $\pi_*\calO_S \cong \calO_\Sigma \oplus K^{-1} \oplus \ldots \oplus K^{-(m-1)}$ restricted to $U$. Thus, 
$$\phi_{21}: \Hom (W_2,W_1) \to \Hom (W_2,W_1) \otimes K$$   
has kernel and cokernel vector bundles of rank $m$. The global picture corresponds to our claim that the kernel subbundle is $\pi_*(L_2^*L_1)$. 

Consider the natural map $t: \pi_*(L_2^*L_1) \to \calO (\Hom(W_2, W_1))$ given by multiplication\footnote{We use the same symbol for the vector bundle $\pi_*L$, where $L$ is a line bundle on a smooth spectral curve $S$, and its locally-free sheaf of sections.}. So, for local sections $s\in H^0(U,\pi_*(L_2^*L_1) ) = H^0(\pi^{-1} (U),L_2^*L_1 )$ and $s_2 \in H^0(U,\pi_*(L_2) ) = H^0(\pi^{-1}(U), L_2)$, $t(s)(s_2) = ss_2$. The map t is clearly injective. Moreover, we have the following diagram
\[\xymatrix@M=0.13in{
H^0(\pi^{-1}(U), L_2) \ar[r]^{s} \ar[d]^{\lambda} & H^0(\pi^{-1}(U), L_1) \ar[d]^{\lambda} \\
H^0(\pi^{-1}(U), L_2 \otimes \pi^*K) \ar[r]^{s\otimes id} & H^0(\pi^{-1}(U), L_1 \otimes \pi^*K)}\] 
which commutes since we are at the level of line bundles. Therefore we have the following commutative diagram
\[\xymatrix@M=0.13in{
H^0(U, W_2) \ar[r]^{t(s)} \ar[d]^{\Phi_2} & H^0(U, W_1) \ar[d]^{\Phi_1} \\
H^0(U, W_2 \otimes K) \ar[r]^{t(s) \otimes id} & H^0(U, W_1 \otimes K).}\]
So, $t: \pi_*(L_2^*L_1) \to \ker \phi_{21}$. Since this is not only an injective morphism of sheaves, but also an injective morphism of vector bundles of same the rank, it follows that $\ker \phi_{21} \cong \pi_*(L_2^*L_1)$.

b) Take the dual of the exact sequence 
$$0 \to \ker \phi_{21} \to \Hom(W_2,W_1) \to \Hom(W_2,W_1) \otimes K \to \coker \phi_{21} \to 0$$ 
and tensor it by $K$. This yields the exact sequence
$$0 \to (\coker \phi_{21})^* \otimes K \to \Hom(W_1,W_2) \to \Hom(W_1,W_2) \otimes K \to (\ker \phi_{21})^* \otimes K \to 0.$$
Since we are interchanging the roles of $W_1$ and $W_2$ in this sequence, by (a), we obtain $(\coker \phi_{21})^* \otimes K \cong \pi_*(L_1^*L_2)$. Now, the relative duality theorem gives $(\pi_*(M))^* \cong \pi_*(M^* \otimes K_S) \otimes K^{-1}$ for a vector bundle $M$ on $S$. In our case, $$(\pi_*(L_1^* \otimes L_2))^* \cong \pi_*(L_2^*L_1 \pi^*K^m) \otimes K^{-1} \cong \pi_*(L_2^*L_1) \otimes K^{m-1}$$ where we use that $K_S \cong \pi^*K^m$. Thus, it follows that $\coker \phi_{21} \cong \pi_*(L_2^*L_1K_S)$.    
\end{proof}  

Since the fibres of $\pi$ consist of a finite number of points, by the Leray spectral sequence, the short exact sequence (\ref{ses2}) can be rewritten as 
\begin{equation}
0 \to H^1(S, L_2^*L_1) \to \mathbb{H}^1 \overset{\zeta}{\to} H^0(S, L_2^*L_1K_S) \to 0.
\label{jk}
\end{equation}

\begin{defin}\label{l1l2} Let $d>0$. We denote by $A_d$ the space of Higgs bundles $(V,\Phi)$ given by (non-split) extensions of Higgs bundles of the form 
$$0 \to \pi_*(L_1, \lambda) \to (V,\Phi) \to \pi_*(L_2, \lambda) \to 0,$$
for some $L_i \in \Pic^{d_i}(S)$, $i=1,2$, such that $L_1L_2\pi^*K^{1-m} \in \Prym$ and $\zeta (\delta) \neq 0 \in H^0(S, L_2^*L_1K_S)$, where $\delta = \delta (V,\Phi)$ is the element in $\mathbb{H}^1$ corresponding to $(V,\Phi)$.
\end{defin}

As discussed, starting with $(V,\Phi) \in h^{-1}(p^2) \setminus N$ we get an element of $A_d$ for some $d>0$. Conversely, given an element $(V,\Phi) \in 
A_d$, clearly $\det (x - \Phi) = p(x)^2$. As we will show in the next lemma, $p(\Phi) \neq 0$ for any element $(V,\Phi) \in A_d$.   

\begin{lemma} Let $(V,\Phi)$ be an extension of $(W_2, \Phi_2) = \pi_*(L_2, \lambda)$ by $(W_2, \Phi_2) = \pi_*(L_1, \lambda)$ represented by $\delta \in \mathbb{H}^1$, where $L_1$ and $L_2$ are line bundles on $S$. Then, $p(\Phi) = 0$ if and only if $\zeta (\delta) = 0$.
\end{lemma}
\begin{proof}
Choose a $C^\infty$-splitting $V = W_1 \oplus W_2$ and write 
\[ \bar{\partial}_V = \left( \begin{array}{cc}
\bar{\partial}_1 & \beta \\
0 & \bar{\partial}_2
\end{array} \right),
\ \   
\Phi = 
\left( \begin{array}{cc}
\Phi_1 & \psi \\
0 & \Phi_2
\end{array} \right),
\] 
where $\beta \in \Omega^{0,1}(\Sigma, \Hom(W_2, W_1))$ and $\psi \in \Omega^{1,0}(\Sigma, \Hom(W_2, W_1))$. As we saw in Chapter \ref{Chapter2}, $\delta$ is represented by the pair $(\beta, \psi )$ and $\zeta (\delta)$ is simply the class of $\psi$ inside $\coker \phi_{21} $ (this makes sense since $\dbar \psi + \phi_{21} \beta = 0$). If $\zeta (\delta)=0$, the element $\psi \in \Omega^{1,0}(\Sigma, \Hom (W_2,W_1))$ is of the form $\phi_{21} (\theta)$ for some smooth section $\theta$ of $\Hom (W_2,W_1)$. This means we can find another representative in the same class of $(\beta, \psi )$ of the form $(\beta^\prime, 0)$ (corresponding to another smooth splitting $V = W_1 \oplus W_2$). The condition $(\beta^\prime, 0) \in \K^1$ implies $\beta^\prime \circ \Phi_2 = \Phi_1 \circ \beta^\prime$. Thus, since the spectral curve is the same for $(W_1, \Phi_1)$ and $(W_2, \Phi_2)$, an eigenvalue of $\Phi_1$ is also an eigenvalue of $\Phi_2$, which means that the generic eigenspace is two dimensional. It follows that $(V,\Phi) \cong \pi_*(E,\lambda)$, for some rank $2$ vector bundle $E$ on $S$, so $p(\Phi) = 0$. This, of course, can also be seen by a simple direct computation. With respect to the splitting $V = W_1 \oplus W_2$ we have 
\[
\Phi^k = 
\left( \begin{array}{cc}
\Phi_1^{k-1} & q_k(\Phi) \\
0 & \Phi_2^{k-1}
\end{array} \right),
\] 
for every integer $k > 0$, where 
$$q_k(\Phi) = \sum_{i=0}^{k-1} \Phi_1^{k-1-i}\psi \Phi_2^i.$$
If $\psi = \Phi_1 \theta - \theta \Phi_2$ for some $\theta \in \Omega^0(\Sigma, \Hom (W_2,W_1))$, 
\begin{align*}
q_k(\Phi) & = \sum_{i=0}^{k-1} \Phi_1^{k-1-i}(\Phi_1 \theta - \theta \Phi_2) \Phi_2^i\\
& = \Phi_1^{k}\theta - \Phi_1^{k-1}\theta\Phi_2
 + \Phi_1\theta \Phi_2^{k-1} - \theta\Phi_2^k + \sum_{i=1}^{k-2} \Phi_1^{k-i}\theta\Phi_2^i - \Phi_1^{k-1-i}\theta\Phi_2^{i+1} \\
 & = \Phi_1^{k}\theta - \theta\Phi_2^k + \sum_{i=1}^{k-1} \Phi_1^{k-i}\theta\Phi_2^i - \sum_{j=1}^{k-1}  \Phi_1^{k-j}\theta\Phi_2^{j} \qquad \qquad (j=i+1)\\
 & = \Phi_1^{k}\theta - \theta\Phi_2^k. 
 \end{align*}  
Thus, 
\[ p(\Phi) = \left( \begin{array}{cc}
p(\Phi_1) & q (\Phi) \\
0 & p (\Phi_2)
\end{array} \right),
\] 
where 
\begin{align*}
q(\Phi) & = q_m(\Phi) + a_2q_{m-2}(\Phi)+\ldots + a_{m-1}q_1(\Phi) + a_m\theta-a_m\theta \\
&= p(\Phi_1)\theta - \theta p(\Phi_2).
\end{align*}
Since the spectral curve of $(W_1,\Phi_1)$ and $(W_2,\Phi_2)$ is $S$, $p(\Phi_1) = 0 = p(\Phi_2)$ and $p(\Phi) = 0$.  

Conversely, suppose $p(\Phi) = 0$. Then, the cokernel of $\pi^*\Phi - \lambda \Id$ is a holomorphic vector bundle of rank $2$ on $S$. Equivalently stated, the rank $1$ torsion-free sheaf on the spectral curve whose direct image construction corresponds to $(V,\Phi)$ is supported on the reduced scheme $S$, thus it must be a locally-free sheaf on $S$ of rank $2$. Thus, $(V,\Phi)$ fits into the short exact sequence 
\begin{equation}
0 \to \pi_*(L_1, \lambda) \to \pi_*(E, \lambda) \to \pi_*(L_2, \lambda) \to 0,
\label{direita}
\end{equation}
where $E$ is a rank $2$ holomorphic vector bundle on $S$. Since $\pi$ is an affine morphism, we have an equivalence of categories between vector bundles on $S$ and vector bundles on $\Sigma$ with a $\pi_*\calO_S$-module structure. This structure of a sheaf of algebras amounts to the data of a Higgs bundle as we saw in the previous chapter. Thus, (\ref{direita}) comes from a short exact sequence 
$$0 \to L_1 \to E \to L_2 \to 0$$    
on $S$. Now, choose a smooth splitting $E = L_1 \oplus L_2$. Since the Higgs field $\Phi$ is obtained by taking the direct image of multiplication by the tautological section $\lambda$, it is block diagonal with respect to the splitting $V = W_1 \oplus W_2$ induced by the smooth splitting $E = L_1 \oplus L_2$. Thus, $\psi = 0 \in \Omega^{1,0}(\Sigma, \Hom(W_2, W_1)) $. This means that $\delta $ is represented by the class $(\beta , 0) = (\beta + \dbar \theta , \phi_{21}\theta)$, where $\theta \in \Hom (\Sigma , \Hom (W_2 , W_1))$. This shows that $\zeta (\delta) = 0$.       
\end{proof}

From the lemma above we conclude that points in $A_d$ cannot also be in $N$. Moreover they are stable Higgs bundles as the next lemma shows.

\begin{lemma} Let $(V,\Phi)$ be an extension of $(W_2, \Phi_2) = \pi_*(L_2, \lambda)$ by $(W_1, \Phi_1) = \pi_*(L_1, \lambda)$, where $L_1$ and $L_2$ are line bundles on $S$, satisfying $p(\Phi)\neq 0$. Then, $(V,\Phi)$ is a stable Higgs bundle if and only if $d>0$. 
\end{lemma}
\begin{proof}
If $d\leq 0$, $(V,\Phi)$ is either strictly-semistable (when $d=0$) or unstable (when $d<0$). Conversely, choose a $C^\infty$-splitting $V = W_1 \oplus W_2$ and write 
\[ \bar{\partial}_V = \left( \begin{array}{cc}
\bar{\partial}_1 & \beta \\
0 & \bar{\partial}_2
\end{array} \right),
\ \  
\Phi = 
\left( \begin{array}{cc}
\Phi_1 & \psi \\
0 & \Phi_2
\end{array} \right),
\] 
where $\beta \in \Omega^{0,1}(\Sigma, \Hom(W_2, W_1))$ and $\psi \in \Omega^{1,0}(\Sigma, \Hom(W_2, W_1))$. With respect to this splitting, $p(\Phi)$ can be written as 
\[ \left( \begin{array}{cc}
p(\Phi_1) & q (\Phi) \\
0 & p (\Phi_2)
\end{array} \right) =  \left( \begin{array}{cc}
0 & q (\Phi) \\
0 & 0
\end{array} \right), 
\] 
for some non-zero holomorphic section $q (\Phi)$ of $\Hom (W_2, W_1)\otimes K^m$ and we have
\[\xymatrix@M=0.1in{
0 \ar[r] & W_1 \ar[r] \ar[d] & V \ar[r] \ar[d]^{p (\Phi)} & W_2 \ar[r] \ar[d]^{q(\Phi)} & 0 \\
0  & W_2\otimes K^m \ar[l]     & V\otimes K^m \ar[l]      & W_1\otimes K^m \ar[l]             & 0. \ar[l] }\]
Now, any $\Phi$-invariant subbundle $W$ of $V$ is also invariant under $p(\Phi)$, but the image of $p(\Phi)$ is contained in $W_1\otimes K^m$. Since $S$ is irreducible, $W_1$ must be, up to isomorphism, the only non-trivial $\Phi$-invariant subbundle of $V$. Thus, $(V,\Phi)$ is stable if and only if $\deg (W_1) < 0$, i.e., $d>0$. 
\end{proof}

In particular, the lemma above shows that each $A_d$ is contained in $h^{-1}(p^2) \setminus N$ for $d>0$. Let $$\bar{d} = \deg (L_2^*L_1K_S) = -2d + 2m^2(g-1).$$ 
From any non-zero element in $H^0(S, L_2^*L_1K_S)$ we get an effective divisor $D \in S^{(\bar{d})}$, where $S^{(\bar{d})}$ is the $\bar{d}$-th symmetric product\footnote{Given an irreducible smooth complex projective curve $C$ and an integer $k >1$, the group of permutations of $\{ 1, \ldots, k \} $ acts on $C^k$. The quotient, denoted by $C^{(k)}$, is an irreducible smooth complex
projective variety of dimension $k$, called the \textbf{$k$-th symmetric product} of $C$. Moreover, $C^{(k)}$ can be naturally identified with the set of effective divisors of degree $k$ on $C$.} of the curve $S$. From (\ref{jk}) we get a natural map
$$A_d \to Z_d,$$
where
$$
Z_d = \left \{ \ (D,L_1,L_2) \in S^{(\bar{d})} \times \Pic^{d_1}(S) \times \Pic^{d_2}(S) \
\bigg |
\gathered
\begin{array}{cl}
i. & \mathcal{O}(D) = L_2^*L_1K_S, \\
ii. & \Nm (L_1L_2) = K^{m(m-1)}.\
\end{array}
\endgathered  \ \  \right
\}.$$
Note that $(V,\Phi)\in A_d$ defines a non-zero section of $L_2^*L_1K_S$ only up to a non-zero scalar, so we get a well-defined map only to the corresponding effective divisor. Of course, we can also write 
$$Z_d \cong \{ (D,L) \in S^{(\bar{d})} \times \Pic^{d_2}(S) \ | \ L^2(D) \pi^*K^{1-2m} \in \Prym \}.$$  
Moreover, the fact that $\zeta (\delta) \neq 0 \in H^0(S, L_2^*L_1K_S)$ for any $\delta$ representing an element of $A_d$ gives a constraint on $d$. Namely, $L_2^*L_1K_S$ has at least one global section, which means that we must have $\bar{d} = \deg (L_2^*L_1K_S) \geq 0$. From this we obtain 
$$1 \leqslant d \leqslant g_{_S}-1.$$

Consider the natural projection $\varphi : Z_d \to S^{(\bar{d})}$. Given two elements in the fibre of a divisor $D \in S^{(\bar{d})}$ they must differ by an element $L \in \Pic^0(S)$ such that $L^2 \in \Prym$. The connectedness of the Prym variety\footnote{The multiplication map $[r] : A \to A$ for a (connected) complex abelian variety $A$ is always surjective (with kernel $A[r]$ isomorphic to $(\Z /r \Z )^{2\dim A}$).} implies that the map $$U \in \Prym \mapsto U^2 \in \Prym$$ is surjective and there exists $U \in \Prym$ such that $L^2 = U^2$. If $U^\prime \in \Prym$ is another element satisfying $L^2 = (U^\prime)^2$, then $U^\prime U^{-1} \in \Prym[2]$ is a 2-torsion point of the Prym variety. Therefore we have a well-defined map
$$F \coloneqq \{ M \in \Pic^0(S) \ | \ M^2 \in \Prym \} \to \frac{\Pic^0(S)[2]}{\Prym[2]}$$
given by $M \mapsto [MU^{-1}]$, where $U\in \Prym$ satisfies $M^2 = U^2$. Note that the fibre of the map defined above at $[L^\prime] \in \Pic^0(S)[2]/\Prym[2] \cong (\Z / 2\Z)^{2g}$ is 
$$\{L^\prime U_0 \ | \ U_0 \in \Prym \} \cong \Prym . $$
Since the index of $\Prym[2]$ in $\Pic^0(S)[2]$ is $2^{2g}$, the fibre $\varphi^{-1}(D)$ is modeled on the disjoint union of $2^{2g}$ copies of the Prym variety. As the divisor $D$ moves, however, the base $Z_d$ becomes connected.  

\begin{lemma}\label{compzd} The space $Z_d$ is connected for $1 \leqslant d < g_{_S} - 1$ and $Z_{g_{_S}-1}$ has $2^{2g}$ connected components.
\end{lemma}
\begin{proof} First assume that $1 \leqslant d < g_{_S} - 1$. Consider the map 
\begin{align*}
n : \Pic^{\bar{d}}(S) \times \Pic^{d_2}(S) & \to \Pic^0(\Sigma) \\
(\mathcal{L}_1, \mathcal{L}_2) & \mapsto \Nm (\mathcal{L}_2^2 \mathcal{L}_1 \pi^*K^{1-2m})
\end{align*}
and the natural projections $p_1 : \ker (n) \to \Pic^{\bar{d}}(S)$ and $p_2 : \ker (n) \to \Pic^{d_2}(S)$. The fibre $p_2^{-1}(L_2)$ is given by $\mathcal{L}_1 \in \Pic^{\bar{d}}(S)$ such that $L_2^2\mathcal{L}_1 \pi^*K^{1-2m} \in \Prym$. If $\mathcal{L}_1^\prime$ is another element of this fibre, $\mathcal{L}_1^\prime \mathcal{L}_1^{-1} \in \Prym$ and the fibres of $p_2$ are modeled on $\Prym$, which is connected. Since $\Pic^{d_2}(S)$ is also connected, $\ker(n)$ is connected. Note that the fibres of $p_1$ are modeled on $ F$ and they are acted transitively by  $\pi_1 (\Pic^0(S)) = H_1(S, \mathbb{Z})$. Consider the commutative diagram 
\[\xymatrix@M=0.13in{
Z_d \ar[d]_{\varphi} \ar[r] & \ker(n) \ar[d]^{p_1} \\
S^{(\bar{d})} \ar[r]^{} & \Pic^{\bar{d}}(S). }\] 
Clearly $Z_d$ is the pullback of $\ker(n)$ by the natural map which send the effective divisor to the correspond line bundle. Since $ d < g_{_S} - 1$ gives $\bar{d}>0$, there is a natural surjective map $H_1(S^{(\bar{d})}, \mathbb{Z}) \to H_1(S, \mathbb{Z})$ and composing this with the surjective map $\pi_1(S^{(\bar{d})}) \to H_1(S^{(\bar{d})}, \mathbb{Z})$ we obtain a surjection $\pi_1(S^{(\bar{d})}) \to \pi_1(\Pic^{\bar{d}}(S)) = H_1(S,\mathbb{Z})$. Since $H_1(S,\mathbb{Z})$ acts transitively on the fibres, so does $\pi_1(S^{(\bar{d})})$ and we conclude that $Z_d$ is connected. When $d=g_{_S}-1$, $\bar{d}=0$ and by fixing an element $L_0 \in Z_{g{_S}-1}$, we obtain an isomorphism between $Z_{g{_S}-1}$ and $F$ and we are done.  
\end{proof}

\begin{lemma} Let $1 \leqslant d \leqslant g_{_S}-1$. There exists a unique vector bundle $F_d$ on $S^{(\bar{d})}$ whose fibre at a divisor $D$ is canonically isomorphic to $H^1(S, K_S^{-1}(D))$. The rank of $F_d$ is $2d + g_{_S}-1$.
\end{lemma}
\begin{proof}
The vector bundle is clearly unique up to canonical isomorphism, so let us prove existence. Consider the universal effective divisor $\triangle \subseteq S^{(\bar{d})} \times S$ of degree $\bar{d}$ . Explicitly,
$$\triangle = \{(D, x) \in S^{(\bar{d})} \times S \ | \ x \in D \}.$$  
Let $pr_1 : S^{(\bar{d})} \times S \to S^{(\bar{d})}  $ and $pr_2 : S^{(\bar{d})} \times S \to S $ be the natural projections and consider the sheaf $F_d = R^1pr_{1, *}(\mathcal{O}(\triangle)pr_2^*K_S^{-1})$. The projection $pr_1$ is a proper holomorphic map whose fibre at $D \in S^{(\bar{d})}$ is $pr_1^{-1}(D) \cong S$. The line bundle $\mathcal{O}(\triangle)pr_2^*K_S^{-1}$ restricted to this fibre is simply $K_S^{-1}(D)$, whose degree is $-2d$, so by Riemann-Roch, $h^{1}(S, K_S^{-1}(D))$ is equal to $2d + m^2(g-1)$, regardless of the effective divisor $D \in S^{(\bar{d})}$ taken. Thus, by Grauert's theorem, $F_d$ is a vector bundle whose fibre at $D$ is isomorphic to $H^1(S, K_S^{-1}(D))$. Note that when $d = g_{_S}-1$ we have $\bar{d} = 0$. In this case, $F_d$ is simply the vector space $H^1 (S , K_S^{-1})$. 
\end{proof}
Also, let $E_d$ be the total space of the vector bundle given by the pullback of $F_d$ by $\varphi$ 
\[\xymatrix@M=0.13in{
E_d \ar[d] \ar[r] & F_d \ar[d]^{} \\
Z_d \ar[r]^{\varphi} & S^{(\bar{d})}. }\]
In particular, this vector bundle has fibre at $(L_2, D)$ isomorphic to $H^1(S, K_S^{-1}(D)) = H^1(S, L_2^*L_1)$. Note that the dimension of $E_d$ is 
\begin{align*}
\dim E_d & = \dim Z_d + H^1(S, K_S^{-1}(D)) \\
&= (\bar{d} + g_{_S} - g) + (3(g_{_S} - 1) - \bar{d}) \\
&= (4m^2 - 1)(g-1),
\end{align*}
where the dimension of $H^1(S, L_2^*L_1)$ is calculated by the Riemann-Roch theorem (noticing that $\deg (L_2^*L_1) = -2d < 0$).

Given two elements $e, e^\prime \in A_d$ mapping to the same effective divisor, from the exactness of the sequence (\ref{jk}), their difference defines a unique element in $H^1(S,L_2^*L_1)$. Thus, $A_d \to Z_d$ is an affine bundle modeled on the vector bundle $E_d \to Z_d$. In particular, $A_d$ corresponds to a certain class in $H^1(Z_d, E_d)$, as this cohomology group classifies the set of isomorphism classes of affine bundles on $Z_d$ with underlying linear structure $E_d$.

\begin{thm}\label{fibreforsu*} Let $p(\lambda) = \lambda^m + \pi^*a_2\lambda^{m-2} + \ldots + \pi^*a_m$ be a section of the line bundle $\pi^*K^m$ on the total space of the cotangent bundle of $\Sigma$ whose divisor is a smooth curve $S$. The fibre $h^{-1}(p^2)$ of the $SL(2m,\C)$-Hitchin fibration is a disjoint union 
$$h^{-1}(p^2) \cong N \cup \bigcup_{d=1}^{g_{_S}-1} A_d,$$
where 
\begin{itemize}
\item $N$ consists of semi-stable rank $2$ vector bundles on $S$ with the property that the determinant bundle of $\pi_*E$ is trivial.
\item $A_d \to Z_d$ is an affine bundle (of certain class in $H^1(Z_d, E_d)$) modeled on the vector bundle $E_d \to Z_d$, whose fibre at $(L,D) \in Z_d$ is isomorphic to $H^1(S,K_S^{-1}(D))$.
\item $Z_d$ is connected for $1 \leqslant d < g_{_S}-1$ and $Z_{g_{_S}-1}$ has $2^{2g}$ connected components. Also, the natural map $Z_d \to S^{(\bar{d})}$ is a fibration, whose fibres are modeled on $2^{2g}$ copies of the Prym variety.
\item Each stratum has dimension $(4m^2 - 1)(g-1)$ and the irreducible components of $h^{-1}(p^2) $ are precisely the Zariski closures of $A_d$, $1 \leqslant d \leqslant g_{_S} -1$, and $N$.
\end{itemize} 
\end{thm} 
\begin{proof}
The only thing that remains to be proven is the last assertion. For this, consider the map 
\begin{align*}
f: N & \to \Prym \\
E & \mapsto \det (E) \pi^*K^{1-m}.
\end{align*}  
It follows from item b) of Lemma \ref{nmlm} that this map of projective varieties is well defined (in particular, all vector bundles in an $S$-equivalence class have isomorphic determinant line bundles). Also, $f$ is a surjective map. Indeed, let $L \in \Prym$. Since $\Prym$ is a complex abelian variety, the squaring map is surjective and we can choose $L_0 \in \Prym$ such that $L \cong L_0^2$. Then, for any $E \in N_0 = f^{-1}(\calO_S)$, a straightforward computation using Lemma \ref{nmlm} shows that $E^\prime = E \otimes L_0 \in N$, and $f(E^\prime) = L$. Note that, given $L \in \Prym$, the moduli space $\calU_S(2,L\pi^*K^{m-1})$ of rank $2$ bundles on $S$ with fixed determinant $L\pi^*K^{m-1}$ is contained in $N$. This follows directly from Lemma \ref{nmlm}. Therefore, every fibre $f^{-1}(L) \cong \calU_S(2,L\pi^*K^{m-1})$ is irreducible of dimension $3(g_{_S}-1)$. Since $\Prym$ is irreducible, $N$ must also be. Moreover we have  
\begin{align*}
\dim N & = \dim N_0 + \Prym \\
&= 3(g_{_S}-1) + (g_{_S} - g) \\
&=  (4m^2 - 1)(g-1). 
\end{align*}
As the $SL(2m,\C)$-Hitchin map is flat, its fibres are equidimensional of dimension $(4m^2 - 1)(g-1)$. Thus, $N$ is an irreducible component and since $A_d$, $d=1, \ldots, g_{_S}-1$, are open disjoint sets of dimension $(4m^2 - 1)(g-1)$, their Zariski closures give the other irreducible components of the fibre $h^{-1}(p^2)$.  
\end{proof}

\begin{ex} For $m=1$, $SU^*(2) = SU(2)$ and by taking $p(x)=x$, $h^{-1}(p^2)$ is the nilpotent cone $\text{Nilp}(2,\calO_\Sigma)$ associated to the $SL(2,\C)$-Hitchin fibration. In this case, $N_0 = N = \calU (2, \calO_\Sigma)$ is the irreducible component where the Higgs field is zero. Moreover, $A_d$ parametrizes extensions of Higgs bundles of the form
$$0 \to (L_1, 0) \to (V,\Phi) \to (L_2, 0)\to 0,$$
where $ L_2 = L_1^*$. Such extensions are parametrized by the two term complex of sheaves 
$$\phi_{21} : \calO (L^{-2}) \to \calO (L^{-2}K) \to 0,$$
where $L=L_2$. This induces the short exact sequence 
\begin{equation}
0 \to H^1(\Sigma, L^{-2}) \to \mathbb{H}^1 \to H^0(\Sigma, L^{-2}K) \to 0
\label{nc}
\end{equation}
with $0 < d = \deg (L) \leqslant g-1$.
We can always split the vector bundle $V = L^* \oplus L$ smoothly and with respect to this splitting we can write 
$$\bar{\partial}_V = \left( \begin{matrix} \bar{\partial}_L& \beta\\ 0&\bar{\partial}_{L^*} \end{matrix} \right)$$
and 
$$\Phi = \left( \begin{matrix} 0& \psi\\ 0&0 \end{matrix} \right),$$
where $\beta \in \Omega^{0,1}(\Sigma, L^{-2})$, $\psi \in \Omega^{1,0}(\Sigma, L^{-2})$. As discussed in the last section, the holomorphicity of $\Phi$ implies $$\bar{\partial}\psi + \phi_{21} (\beta) = 0,$$ 
so that $(\beta, \psi) \in \mathbb{H}^1$.
But $\phi_{21} = 0$, so the short exact sequence ($\ref{nc}$) splits. Thus, $A_d$ is actually a vector bundle on $Z_d = \{ (L,D) \in \Pic^d(\Sigma) \times \Sigma^{(\bar{d})} \ | \ \calO (D) = L^{-2}K \}$, where $\bar{d} = 2g-2-2d$. The map $Z_d \to \Sigma^{(\bar{d})}$ is a $2^{2g}$-fold covering corresponding to choosing a square root of the canonical bundle. In particular, $Z_{g-1}$ corresponds precisely to the $2^{2g}$ square roots of the canonical bundle. This is precisely Hitchin's description of the nilpotent cone \cite{hitchin1987self}.
\end{ex}

\begin{rmk} Note that, by dropping the constraints on the determinant we have a description of the fibre at $p(x)^2$ for the Hitchin fibration for classical Higgs bundles of rank $2m$ and zero degree. Also, we could have worked with $L$-twisted Higgs bundles $(V,\Phi)$ (for a line bundle $L$ of degree $\deg (L) \geqslant 2(g-1)$) and considered the fibre of the corresponding Hitchin fibration associated to the moduli space of $L$-twisted Higgs bundles with traceless Higgs field $\Phi : V \to V \otimes L$ and fixed determinant bundle $\det (V) \cong \Lambda$ (not necessarily the trivial bundle). The construction would carry out precisely in the same way.  
\end{rmk}

%% file: Chapters/Chapter4.tex

\chapter{Other cases} 

\label{Chapter4} 

\lhead{Chapter 4. \emph{Other cases}} 


We now consider the connected semisimple non-compact groups $G_0= SO^*(4m)$ and $Sp(m,m)$, real forms of $G = SO(4m,\C)$ and $Sp(4m,\C)$, respectively. Being both subgroups of $SU^*(4m)$, the Higgs field of such a $G_0$-Higgs bundle $(V,\Phi)$ (see Examples \ref{so*} and \ref{spmm} for the description of these objects) has characteristic polynomial of the form $p(x)^2$, for some polynomial $p(x)$. In both cases, the generic eigenspaces are two-dimensional and for each eigenvector $(w,w^\prime) \in V$ of $\Phi$ with eigenvalue $\mu$, $(w,-w^\prime) \in V$ is an eigenvector of $\Phi$ with eigenvalue $- \mu$. Thus, the reduced curve $S = \zeros  (p(\lambda)) \subset |K|$ admits an involution 
\begin{align*}
\sigma : S & \to S \\
\lambda & \mapsto - \lambda
\end{align*}
and the polynomial $p(x)$ is of the form $p(x) = x^{2m} + b_2x^{2m-2} + \ldots + b_{2m-2}x^2 + b_{2m}$, for some $b_{2i} \in H^0(\Sigma, K^{2i})$. 

Throughout this chapter we assume that the divisor $S$ corresponding to the section $p(\lambda)$ of $\pi^*K^{2m}$ is non-singular (again by Bertini's theorem, this is true for generic $b_i$). Also, consider the quotient curve $\bar{S} = S/\sigma $, which is clearly non-singular, and denote by $\rho : S \to \bar{S}$ the natural projection. Let $\bar{\pi} : K^2 \to \Sigma$ be the square of the canonical bundle of $\Sigma$. Setting $\eta = \lambda^2$ embeds $\bar{S} \subset |K^2|$ with equation $p(\eta) = 0$. Also, $\eta$ is the tautological section of $\bar{\pi}^*K^2$ and the ramified coverings fit into the following commutative diagram.

    \centerline{\xymatrix{
        S\ar[rrrr]^\rho_{2:1} \ar[rrdd]_\pi^{2m:1} && && \bar{S}\ar[lldd]^{\bar{\pi}}_{m:1} \\
                                       \\
  && \Sigma &&
    }}
    
In particular, by the adjunction formula, $K_S \cong \pi^*K^{2m}$ and $K_{\bar{S}} \cong \bar{\pi}^*K^{2m - 1}$ (the canonical bundle of $K^2$ is $\bar{\pi}^*K^{- 1}$), so that
\begin{align}
g_{_S}  &= 4m^2(g-1) + 1 ,\label{genusS} \\
g_{_{\bar{S}}}  &= m(2m-1)(g-1) + 1.\label{genusbar}
\end{align}

The difference between the two cases manifests itself on how the involution acts on the determinant bundle of the rank $2$ bundle $E$ on $S$. The spectral data for these two groups was described by Hitchin and Schaposnik.

 \begin{prop}\cite{hitchin2014}\label{so*sp} Let $p(\lambda) = \lambda^{2m}+\pi^*b_2\lambda^{2m-2}+\dots+\pi^*b_{2m}$ be a section of the line bundle  $\pi^*K^{2m}$ on the cotangent bundle of $\Sigma$ whose divisor is a non-singular curve $S$ and let  $\sigma$ be the involution  $\sigma(\lambda)=-\lambda$ on $S$. If $E$ is a rank $2$ vector bundle on $S$, then the  direct image of $\lambda : E \to E\otimes \pi^*K$ defines a semi-stable Higgs bundle on $\Sigma$ for the group $SO^*(4m)$ if and only if
 \begin{itemize} 
 \item $E$ is semi-stable, with $\Lambda^2E\cong \pi^*K^{2m-1}$ and $\sigma^*E\cong E$ where the induced action on  $\Lambda^2E=\pi^*K^{2m-1}$ is trivial
 \end{itemize} 
 and for the group $Sp(m,m)$ if and only if
 \begin{itemize}
 \item $E$ is semi-stable, with $\Lambda^2E\cong \pi^*K^{2m-1}$ and $\sigma^*E\cong E$ where the induced action on  $\Lambda^2E=\pi^*K^{2m-1}$ is $-1$.
 \end{itemize}
    \end{prop}

Let us say a few words about these loci of the fibre $h^{-1}(p^2)$, where  
$$h : \mathcal{M} (G) \to \mathcal{A} (G)$$
is the $G$-Hitchin fibration (for $G = SO(4m,\C)$ and $Sp(4m,\C)$). First, note that $\sigma$ induces a natural (algebraic) involution $\tau$ on $\calU_S (2,0)$ defined by 
$$\tau \cdot E = \sigma^* E^*.$$ 
This is clearly true in the stable locus of $\calU_S (2,0)$. Now, if $E_1$ and $E_2$ are in the same $S$-equivalence class 
$$Gr (\tau \cdot E_1) \cong \sigma^* Gr(E_1)^* \cong \sigma^* Gr(E_2)^* \cong Gr (\tau \cdot E_2).$$
The fixed point set $\calU_S (2,0)^\tau$ is then a closed subvariety of $\calU_S (2,0)$ whose smooth points correspond precisely to the locus where the bundles are stable (for more details see \cite{automor}). Recall that for a generic point of the Hitchin base, the Hitchin fibre is a torsor for an abelian variety. By choosing a square root $K^{1/2}$ of the canonical bundle of $\Sigma$, one then has a natural identification of the fibre with the abelian variety. Choosing $K^{1/2}$ in our case will allow us to identify $h^{-1}(p^2)$ with some locus inside the Simpson moduli space of semi-stable sheaves of rank $1$ and degree $0$. Thus, fix a square root $K^{1/2}$ and denote the loci consisting of $G_0$-Higgs bundles in the fibre $h_G^{-1}(p^2)$ by $N_0(G)$ (where $G_0 = SO^*(4m)$ and $Sp(m,m)$). Since any rank $2$ bundle $E$ on $S$ satisfies $\det (E) \otimes E^* \cong E$, by tensoring the elements in $N_0(G)$ by $\pi^*K^{(1-2m)/2}$ we obtain that 
$$N_0 (G) \subset \calU_S (2,0)^\tau$$ 
for both $G_0 = SO^*(4m)$ and $Sp(m,m)$. Note also that the action $\tau$ restricts naturally to an involution on $\calU_S (2, \calO_S)$, which coincides in this case with the action 
$$\sigma \cdot E = \sigma^* E,$$
and by \cite{automor}, $N_0 (SO(4m,\C))$ and $N_0(Sp(4m,\C))$ account for all such fixed points. More precisely, the determinant map gives an equivariant morphism 
$$\det : \calU_S (2,0)^\sigma \to \Pic (S)^\sigma .$$ 
An equivariant structure on a line bundle $L \in  \Pic (S)^\sigma$ is a lifted action of $\sigma$ to $L$. In other words, an isomorphism $\varphi : \sigma^* L \to L$ such that $\varphi \circ \sigma^*\varphi = \Id$. There are two possible lifts, and they differ by a sign. One can see this by looking at a fixed point $p \in S$ of $\sigma$. In this case, $\varphi$ gives a linear action on $L_p$ and $\varphi^2 = 1$, which means that $\varphi_p = \pm 1$. In cit. loc. the authors also show that each point in $\calU_S (2,0)^\sigma$ may be represented by a semi-stable bundle $E$ together with a lift $\varphi : \sigma^* E \to E$ (i.e., an isomorphism $\varphi$ such that $\varphi \circ \sigma^* \varphi = \Id$). Again, at a fixed point $p\in S$ of $\sigma$ there is a linear action $\varphi_p : E_p \xrightarrow{\cong} E_p$ and one either has that $\varphi_p$ is of type $\pm \Id$ or, in the case we have distinct $+1$ and $-1$ eigenspaces,  $ \varphi_p = \left( \begin{matrix} \pm 1& 0\\ 0& \mp 1 \end{matrix} \right)$. The former case coincides with the trivial action on the determinant line bundle (i.e., $N_0 (SO(4m,\C))$), whereas the latter corresponds to the $-1$ action on the determinant (i.e., $N_0(Sp(4m,\C))$). Note that we could have considered the map 
$$\psi : \sigma^*E \xrightarrow{\varphi} E \to E^*,$$
where the last map is given by the natural symplectic structure on an $SL(2,\C) = Sp(2,\C)$ bundle. More precisely, 
\begin{align*}
E & \xrightarrow{\cong} E^*\\
e & \mapsto (e^\prime \mapsto e \wedge e^\prime).
\end{align*}    
The condition of the induced action on the determinant being $\pm 1$ is then equivalent to $\tp{(\sigma^*\psi)} = \mp \psi.$

\begin{rmk} As remarked in \cite{hitchin2014}, $N_0 (SO(4m,\C))$ has $2^{4m(g-1)-1}$ connected components in the fibre. One can tell the components apart by the action on the fixed points of $\sigma$, which are the images of the $4m(g-1)$ zeros of $b_{2m} \in H^0(\Sigma , K^{2m})$. More precisely, the action is $+1$ at $M$ fixed points (and $-1$ at the remaining $4m(g-1) - M$ fixed points). When the bundle is stable there are only two possible lifts, which differ by the sign (and interchanges $M$ by $4m(g-1) - M$). The locus $N_0(Sp(4m,\C))$ on the other hand is connected.   
\end{rmk}


\section{The group $SO^*(4m)$} \label{sectiononso}

Let us highlight the differences from the $SU^*(4m)$ case described in the last chapter. Recall that an $SO(4m,\C)$-Higgs bundle is a pair $(V,\Phi)$ consisting of an orthogonal vector bundle $(V,q)$ of rank $4m$ and a Higgs field $\Phi$ which is skew-symmetric with respect to the orthogonal structure $q$ of $V$. In the $SU^*(4m)$ case, one of the irreducible components of the fibre $h_{SL(4m,\C)}^{-1} (p^2)$ consisted of rank $2$ semi-stable vector bundles $E$ on the non-singular curve $S = \zeros (p(\lambda))$ satisfying $\det (\pi_* E) \cong \calO_\Sigma$. As we want to compare our construction to the one described in the last chapter, let us denote this component by $N(SL(4m,\C))$ for now. Inside it, we have the locus 
$$N = N(SO(4m,\C)) \subset N(SL(4m,\C))$$ 
consisting of rank $2$ bundles $E$ on $S$ whose direct image construction yields $SO(4m,\C)$-Higgs bundles. 

\begin{prop}\label{pairingn} The locus 
$$\{ (V,\Phi) \in h^{-1}(p^2) \ | \ p(\Phi) = 0 \} \subset \calM (SO(4m,\C))$$
is isomorphic to 
$$
N = \left \{ \ E \in \calU_S (2,e) \
\bigg |
\gathered
\begin{array}{cl}
 & \text{There exists an isomorphism}\\
 & \psi : \sigma^*E \to E^* \otimes \pi^*K^{2m-1} \\
 & \text{satisfying $\tp{(\sigma^*\psi)} = - \psi$}.\
\end{array}
\endgathered  \ \  \right
\},$$
where $e= 4m(2m-1)(g-1)$.
\end{prop}
\begin{proof}
Suppose $E \in N$ and $\psi : \sigma^*E \to E^* \otimes \pi^*K^{2m-1}$ is an isomorphism satisfying $\tp{(\sigma^*\psi)} = - \psi$. We want to construct an orthogonal form $q$ on $V = \pi_*E$ with respect to which the Higgs field $\Phi$ is skew-symmetric, where $\Phi$ is obtained from the map $\lambda : E \to E\otimes \pi^*K$, which is multiplication by the tautological section $\lambda$. Let $x\in \Sigma$ be a regular value of $\pi$ with $\pi^{-1}(x) = \{ y_1, \sigma (y_1), \ldots, y_m, \sigma (y_m) \}$. Then, at this point, the fibre of $V$ decomposes as 
$$V_x = E_{y_1} \oplus E_{\sigma (y_1)} \oplus \ldots \oplus E_{y_m} \oplus E_{\sigma (y_m)}.$$
Taking $s,t \in V_x$, we define the non-degenerate pairing 
\begin{equation}\label{pairing}
q(s,t) = \sum_{y \in \pi^{-1}(x)} \frac{\inn{\psi (\sigma^*s)}{t}_y}{ d\pi_y} = \sum_{y \in \pi^{-1}(x)} \frac{\inn{\psi_{y}(s_{\sigma (y)})}{t_y}}{ d\pi_y},
\end{equation}
where the brackets $\inn{\cdot}{\cdot}$ denote the natural pairing between $E$ and its dual $E^*$. In the expression above we see $d\pi$ as a holomorphic section of $K_S\pi^*K^{-1}$, which we identify with $\pi^*K^{2m-1}$ via the canonical symplectic form on the cotangent bundle of $\Sigma$. From the condition $\tp{(\sigma^*\psi)} = - \psi$ we obtain the relation 
\begin{equation}\label{expre}
\inn{s_{\sigma (y)}}{\psi_{\sigma (y)} (t_y)} = - \inn{\psi_y (s_{\sigma (y)})}{t_y}
\end{equation} 
for all $y \in \pi^{-1}(x)$.
The summands of $q(s,t)$ and $q(t,s)$ containing elements of $E_y$ and $E_{\sigma (y)}$ for some $y\in \pi^{-1}(x)$ are
$$\frac{\inn{\psi_{y}(s_{\sigma (y)})}{t_y}}{ d\pi_y} + \frac{\inn{\psi_{\sigma (y)}(s_{y})}{t_{ \sigma (y)}}}{ d\pi_{\sigma (y)}}$$
and 
$$\frac{\inn{\psi_{y}(t_{\sigma (y)})}{s_y}}{ d\pi_y} + \frac{\inn{\psi_{\sigma (y)}(t_{y})}{s_{ \sigma (y)}}}{ d\pi_{\sigma (y)}}$$
respectively. Now, using (\ref{expre}) and the fact that the denominator $d\pi$ is anti-invariant, we obtain that the pairing (\ref{pairing}) is symmetric. Note that if 
$$s  = (s_{y_1}, s_{\sigma (y_1)}, \ldots ,s_{y_m} , s_{\sigma (y_m)}),$$
multiplying by the tautological section we obtain 
$$ \lambda (s)  = ( s_{y_1} \otimes y_1, - s_{\sigma (y_1)} \otimes y_1, \ldots , s_{y_m} \otimes y_m, - s_{\sigma (y_m)}\otimes y_m).$$
Thus, using (\ref{expre}) again we find 
$$q (\lambda (s), t) = - q (s , \lambda (t)),$$
which shows that the associated Higgs field is skew-symmetric with respect to $q$. Using precisely the same argument as in \cite[Section 4]{hitchin2013higgs} we see that $q$ can be extended to the branch points of $\pi $ and thus, $\pi_* (E,\lambda)$ is a semi-stable\footnote{Semi-stability can be seen in many ways. One, as we will explain better in Chapter \ref{Chapter5}, comes from the fact that these rank $2$ bundles on $S$ are part of the locus of Simpson moduli space of semi-stable rank $1$ sheaves on the non-reduced curve which gives the spectral data for the fibre in question. Another way is by noticing that the corresponding Higgs bundle  is in particular a semi-stable $SL(4m,\C)$-Higgs bundles by Proposition \ref{N}, and thus must also be a semi-stable $SO(4m,\C)$-Higgs bundle.} $SO(4m,\C)$-Higgs bundle.  

Conversely, let $(V,\Phi)$ be a semi-stable $SO(4m,\C)$-Higgs bundle such that $p(\Phi) = 0$. Analogously to the $SU^*(4m)$ case, this defines a semi-stable rank $2$ vector bundle $E$ on $S$, which fits into the exact sequence 
\begin{equation}
0 \to E \otimes \calO_S(- R_\pi) \to \pi^*V \xrightarrow{\pi^*\Phi - \lambda \Id} \pi^*(V \otimes K) \to E \otimes \pi^*K \to 0,
\label{exactseq}
\end{equation}  
and we identify, as before, $\calO_S( R_\pi) \cong \pi^*K^{2m-1}$. Taking $\sigma^*$ of (\ref{exactseq}) we obtain 
$$0 \to \sigma^* (E) \otimes \pi^*K^{1-2m} \to \pi^*V \xrightarrow{\pi^*\Phi + \lambda \Id} \pi^*(V \otimes K) \to E \otimes \pi^*K \to 0,$$
as $\sigma^* (\lambda ) = - \lambda$. Since the Higgs field $\Phi$ is skew-symmetric with respect to the orthogonal form $q : V \to V^*$, i.e., $q\circ \Phi = - \tp{\Phi} \circ q$, and the kernel and cokernel of the map 
$$- \pi^* (\tp{\Phi}) + \lambda \Id : \pi^* (V^*) \to \pi^* (V^*\otimes K)$$  
can be obtained by taking the dual of the exact sequence (\ref{exactseq}), we have 
\[
\begin{tikzpicture}[baseline= (a).base]
\node[scale=.85] (a) at (0,0){
\begin{tikzcd}
0 \xar{r} & \sigma^*(E)\otimes \pi^*K^{1-2m} \xar{r} \xar{d}{\psi} & \pi^* (V) \xar{r}{\pi^*\Phi - \lambda \Id} \xar{d}{\pi^*q} &[1cm] \pi^*(V\otimes K) \xar{r} \xar{d}{\pi^*q} & \sigma^*(E)\otimes \pi^* (K) \xar{r}\xar{d} & 0\\
0 \xar{r} & E^* \xar{r} & \pi^* (V^*)  \xar{r}{- \pi^* (\tp{\Phi}) + \lambda \Id} & \pi^* (V^* \otimes K) \xar{r} & E^*\otimes \pi^*K^{2m} \xar{r} & 0.
\end{tikzcd}
};
\end{tikzpicture}
\]
In particular, we obtain an isomorphism $\psi : \sigma^*E \to E^* \otimes \pi^*K^{2m-1}$. Note that the exact sequences above are really exact sequences of sheaves and the injective map taking $\sigma^*(E)\otimes \pi^*K^{1-2m}$ into $\pi^* (V)$ is given by multiplication by the section $d\pi$, whereas the injective map $E^* \to \pi^* (V^*)$ corresponds to the transpose of the natural evaluation map $ev : \pi^* (\pi_*E) \to E$. Since the section $d\pi$ is anti-invariant under $\sigma$, it follows that $\tp{(\sigma^*\psi)} = - \psi$. 
\end{proof}

Fix a square root $K^{1/2}$ and let $E$ be a stable rank $2$ vector bundle on $S$ of degree $e=4m(2m-1)(g-1)$. Also, denote by $E_0 $ the degree $0$ stable bundle $E\otimes \pi^*K^{(1-2m)/2}$. If there exists an isomorphism $ \sigma^*E \to E^* \otimes \pi^*K^{2m-1}$, then $E_0$ is anti-invariant, i.e., $ \sigma^*E_0 \cong E_0^*$. Denote this latter isomorphism by $\psi$. Since $E_0$ is stable, it is simple and $\tp{(\sigma^*\psi)} = \lambda \psi$, for some $\lambda \in \C^\times$. But $\sigma$ is an involution, so applying $\sigma^*$ and taking the transpose map gives $\psi = \lambda^2 \psi$ and we must have $\lambda = \pm 1 $.

\begin{prop}\label{connectedcompo} The closed subvariety $N \subset \calU_S (2,e)$ has dimension $$\dim N = \dim SO(4m,\C) (g-1).$$ 
Moreover, it has $2$ connected components.
\end{prop}
\begin{proof}
Choose a square root $K^{1/2}$ of the canonical bundle of $\Sigma$ and denote by $N^s $ the locus of $N$ consisting of stable vector bundles. Then, 
$$N^s \cong \{ \ E_0 \in \calU^s_S (2,e) \ | \ \exists \ \psi : \sigma^*E_0 \xrightarrow{\cong} E_0^* \ \text{s.t. }\tp{(\sigma^*\psi)} = - \psi \ \}.$$
In particular, $N^s$ sits inside the fixed point set $\calU^s_S (2,0)^\tau$ of the involution $\tau $ on the moduli space $\calU^s_S (2,0)$ of stable rank $2$ bundles on $S$, where $\tau \cdot E_0 = \sigma^* E_0^*$. The fixed point set $\calU^s_S (2,0)^\tau$ is non-singular and so is the locus $N^s$ (see e.g. \cite{edi}). By \cite[Theorem 3.12.(b)]{zelaci2016hitchin}, the dimension of $N^s$ equals $2m(4m-1)(g-1) = \dim SO(4m,\C) (g-1)$, thus the dimension of $N$ is $\dim SO(4m,\C) (g-1)$. To show that $N$ has two connected components it is enough to consider the smooth locus. Note that $N_0^s = N^s \cap \calU_S (2, \calO_S)$ is, after choosing $K^{1/2}$, naturally isomorphic to the smooth locus corresponding to $SO^*(4m)$-Higgs bundles in the fibre $h^{-1}(p^2)$ (see Proposition \ref{so*sp}). Moreover, the tensor product gives a well-defined map 
\begin{equation}
N_0^s \times \PS \to N^s.
\label{n0s}
\end{equation}
Note that the determinant line bundle of an element $E \in N^s$ satisfies 
$$\sigma^* \det (E) = \det (E)^*$$
and thus $\det (E) \in \PS$. The Prym variety $\PS$ is connected, so the squaring map 
$$[2] : \PS \to \PS$$
is surjective. This means that for any $E \in N^s$ we can find $L \in \PS$ such that $L^2 \cong \det (E) $. We may write 
$$E = (E\otimes L^*)\otimes L,$$
and $E\otimes L^*$ has trivial determinant. Thus, (\ref{n0s}) is surjective. As noted in \cite{zelaci2016hitchin}, the $2$-torsion points $\PS [2]$ of the Prym variety $\PS$ act on the connected components of $N_0^s$. The action can be understood in terms of the fixed points of $\sigma$. More precisely, if $p_i$, $i \in \{ 1, \ldots , 4m(g-1) \}$ are the fixed points of $\sigma$, $E \in N_0^s$ we can tell the connected component of $E$ by looking at $\{ \varphi_{p_i} \}$, where $\varphi : \sigma^* E \xrightarrow{\cong} E$ is the isomorphism explained in the last section. Then, given $L \in \PS [2]$ (i.e., $\sigma^* L \cong L^* \cong L$) acts by multiplication by $\pm 1$ at each $p_i$. Note that the action is free modulo the trivial action, i.e., when the action is $+1$ at every fixed point. This can only happen if the bundle comes from $\bar{S}$, so that $\PS [2] / \rho^* \Pic^0 (\bar{S})[2]$ gives a free action on the set of connected components of $N_0^s$. Since there are $2^{2(g_{_S} -2g_{_{\bar{S}}})} = 2^{4m-2}$ points in $\PS [2] / \rho^* \Pic^0 (\bar{S})[2]$, we find, from the surjectivity of (\ref{n0s}), that $N_s$ must have two connected components.         
\end{proof}

The next step is to consider extensions of Higgs bundles. Let $(V,\Phi)$ be an element of the fibre which is not in $N$ (i.e., $p(\Phi) \neq 0$). From the last chapter, we know that these points are given by certain extensions of Higgs bundles. To understand which extensions can appear note the following. If $W$ is a $\Phi$-invariant subbundle of $V$ so is its orthogonal complement $W^{\perp_q}$. Indeed, if $v \in W^{\perp_q}$, for any $w \in W$, $\Phi w \in W$ and $q(\Phi v, w) = - q(v, \Phi w) = 0$. This means that $W \cong W^{\perp_q}$ and $W$ is a maximally isotropic subbundle. Thus, $V$ is an extension 
$$0 \to W \to V \to W^* \to 0,$$    
where the projection to $W^*$ is given by $q$. As $W$ is $\Phi$-invariant, the Higgs bundle $(V,\Phi)$ is an extension
$$0 \to (W, \phi) \to (V, \Phi) \to (W^*, -\tp{\phi}) \to 0,$$
where $\phi \coloneqq \Phi|_W$. From Theorem \ref{themext}, such extensions of Higgs bundles are in correspondence with the first hypercohomology group $\K^1 (\Sigma, \Lambda^2\mathsf{W})$ of the two-term complex of sheaves 
\begin{align*}
\hat{\phi} : \calO(\Lambda^2 W) & \to \calO(\Lambda^2 W \otimes K )\\
w_1\wedge w_2 & \mapsto \phi (w_1)\wedge w_2 + w_1\wedge \phi (w_2).
\end{align*}  
For simplicity, we denote $\K^1 (\Sigma, \Lambda^2\mathsf{W})$ simply by $\K^1$. The Higgs bundles $(W, \phi)$ and $(W^*, -\tp{\phi})$ have the same spectral curve $S$ and we may assume that $\deg (W) = -d < 0$ (see Remark \ref{d>0}). Also, by the \textit{BNR} correspondence, $\pi_* (L , \lambda) \cong (W, \phi)$ and $\pi_* (L^\prime , \lambda) \cong (W^*, -\tp{\phi})$ for some line bundles $L$ and $L^\prime$ on $S$. In particular, since the Euler characteristic is invariant under direct images we have 
\begin{equation}
\deg (L)  = - d + 2m(2m-1)(g-1).
\label{degreeL}
\end{equation}  
Let us show how $L^\prime$ is related to $L$. Note that relative duality gives 
$$\pi_*(L^*K_S\pi^*K^{-1}) \cong W^*.$$
However, multiplication by the tautological section
$$\lambda : L^*K_S\pi^*K^{-1} \to L^*K_S$$
induces the Higgs field $\tp{\phi} : W^* \to W^* \otimes K$. This is clear from the natural pairing, which is the isomorphism in the relative duality theorem. Indeed, given a regular value $x \in \Sigma$ of $\pi$, $s \in (\pi_*L)_x$ and $t \in \pi_*(L^*K_S\pi^*K^{-1})_x$,
$$\inn{s}{t} = \sum_{y \in \pi^{-1}(x)} \frac{t(s)_y}{d\pi_y}$$
and $$\inn{\lambda s}{t} = \inn{s}{\lambda t} .$$ 
\begin{lemma} Under the \text{BNR} correspondence we have $\pi_* (L^\prime , \lambda ) \cong (W^*, -\tp{\phi})$, where $L^\prime = \sigma^*(L^*K_S\pi^*K^{-1})$. 
\end{lemma}
\begin{proof}
There is an isomorphism $$\pi_*(\sigma^*(L^*K_S\pi^*K^{-1})) \cong \pi_*(L^*K_S \pi^*K^{-1}),$$ which comes from pulling back sections by the involution. Also, given an open set $U \subseteq \Sigma$,
\[\xymatrix@M=0.13in{
H^0(\pi^{-1}(U), \sigma^*(L^*K_S\pi^*K^{-1})) \ar[r]^{\cong} \ar[d]_{  \lambda} & H^0(\pi^{-1}(U), L^*K_S\pi^*K^{-1}) \ar[d]^{- \lambda \ = \ \sigma^*(\lambda)} \\
H^0(\pi^{-1}(U), \sigma^*(L^*K_S)) \ar[r]^{\cong} & H^0(\pi^{-1}(U), L^*K_S).}\]  
Thus  $\pi_* (L^\prime, \lambda) = (W^*, -\tp{\phi})$, where $L^\prime = \sigma^*(L^*K_S\pi^*K^{-1})$.
\end{proof}

Before we proceed, let us introduce some notation. Let $M$ be a vector bundle of rank $r$ on $S$. Given an open set $U \subseteq \Sigma$, $\rho^{-1}(U)\subseteq S $ is invariant under $\sigma$ and the $\pm 1$ eigenspaces of the action on $H^0(\rho^{-1}(U), M)$ decompose the direct image on $\bar{S}$ as 
$$\rho_*(M) = \rho_*(M)_+ \oplus \rho_*(M)_-.$$ 
In particular, $\bar{M}_+ \coloneqq \rho_*(M)_+$ and $\bar{M}_- \coloneqq \rho_*(M)_-$ are rank $r$ bundles on $\bar{S}$ and their direct image under $\bar{\pi}$
\begin{center}
$M_+ \coloneqq \bar{\pi}_*(\bar{M}_+) = (\pi_*M)_+$\\
$M_- \coloneqq \bar{\pi}_*(\bar{M}_-) = (\pi_*M)_-$
\end{center}
are rank $mr$ vector bundles on $\Sigma$. From now on we denote by $M$ the line bundle\footnote{Making a parallel to last chapter, this corresponds to the line bundle $L_2^*L_1$.} $L\sigma^*L\pi^*K^{1-2m}$. 

\begin{lemma} From the smoothness of the spectral curve it follows that $\ker \hat{\phi}$ and $\coker \hat{\phi}$ are rank $m$ vector bundles. Denote by $M$ the line bundle $L\sigma^*L\pi^*K^{1-2m}$. We have:
\begin{enumerate}[label=\alph*),ref=\alph*]
\item $\ker \hat{\phi} \cong \pi_*(L\sigma^*L\pi^*K^{1-2m})_+$ (i.e., $M_+$) and
\item $\coker \hat{\phi} \cong \pi_*(L\sigma^*L\pi^*K^{1-2m})_- \otimes K^{2m}$ (i.e., $M_- \otimes K^{2m}$).
\end{enumerate}
\end{lemma}
\begin{proof} a) We have the following commutative diagram 
\[\xymatrix@M=0.13in{
H^0(\pi^{-1}(U), \sigma^*(L^*K_S\pi^*K^{-1})) \ar[r]^-s \ar[d]^{  \lambda} & H^0(\pi^{-1}(U), L) \ar[d]^{ \lambda} \\
H^0(\pi^{-1}(U), \sigma^*(L^*K_S)) \ar[r]^-{s\otimes 1} & H^0(\pi^{-1}(U), L\pi^*K).}\]  
where $s \in H^0(\pi^{-1}(U), L\sigma^*(LK_S^{*}\pi^*K))$. This tells us that the kernel of $\hat{\phi}$ viewed in $\Hom(W^*, W)$ is $\pi_*(L\sigma^*(LK_S^{*}\pi^*K))$. Since we want the kernel in $\Lambda^2 W$ we take the anti-invariant part of the direct image. We now use the canonical symplectic form of the cotangent bundle of $\Sigma$ to identify $K_S$ with $\pi^*K^{2m}$. Since the symplectic form on the cotangent bundle is anti-invariant under $\sigma$, which is scalar multiplication by $-1$ in the fibres, 
$$\pi_*(L\sigma^*(LK_S^{*}\pi^*K))_- \cong \pi_*(L\sigma^*L\pi^*K^{1-2m})_+.$$
b) Dualizing the sequence 
\begin{equation*}
0 \to \ker \hat{\phi} \to \Lambda^{2} W \to \Lambda^{2} W \otimes K \to \coker \hat{\phi} \to 0 \nonumber 
\end{equation*}
we obtain 
$$0 \to (\coker \hat{\phi})^*\otimes K \to \Lambda^{2} W^* \to \Lambda^{2} W^* \otimes K \to (\ker \hat{\phi})^* \otimes K \to 0.$$
Since we are exchanging the roles between $(W,\phi)$ and $(W^*, -\tp{\phi})$, we get
$$\pi_*(L^*\sigma^*L^*\pi^*K^{2m-1})_+ \cong (\coker \hat{\phi})^*\otimes K.$$

By relative duality, 
\begin{eqnarray}\nonumber
\pi_*(L^*\sigma^*L^*\pi^*K^{2m-1})^*  \cong & \pi_*(L\sigma^*L\pi^*K^{1-2m}K_S)\otimes K^{-1}\\
 \cong & M_+K^{2m-1} \oplus M_-K^{2m-1}.\nonumber
\end{eqnarray} 
To identify $K_S \otimes \pi^*K^{-1}$ with $\pi^*K^{2m - 1}$ we use the canonical symplectic form of the cotangent bundle of $\Sigma$, which is anti-invariant under $\sigma$. This gives
$$\pi_*(L^*\sigma^*L^*\pi^*K^{2m-1})_+ \cong M_-^*K^{1 - 2m}.$$
Thus we obtain $\coker \hat{\phi} \cong M_- \otimes K^{2m}$.
\end{proof}

\begin{rmks} 

\noindent 1. We can write $L\sigma^*L = \rho^*\bar{L}$, for some line bundle $\bar{L}$ on $\bar{S}$. Indeed, if we denote $L\sigma^*L$ by $\calL$, $\sigma^*\calL \cong \calL$ and at any fixed point $x \in S$ of the involution, the fibre $\calL_x$ is of the form $L_x^2$ and the action is always trivial, which means that $\calL$ is pulled-back from a line bundle on $\bar{S}$. Also, $\bar{L}$ is unique as the map $\rho^* : \Jac (\bar{S}) \to \Jac (S)$ is injective. 

\noindent 2. The direct image $\rho_*\mathcal{O}_S$ is naturally isomorphic to $\mathcal{O}_{\bar{S}} \oplus \bar{\pi}^*K^{-1}$. Given $s_0 \in H^0(U, \mathcal{O}_{\bar{S}})$, $s_1 \in H^0(U, \bar{\pi}^*K^{-1})$ we associate $$(s_0, s_1) \mapsto \rho^*s_0 + \lambda\rho^*s_1 \in H^0(\rho^{-1}(U), \mathcal{O}_S) = H^0(U, \rho_*\mathcal{O}_S).$$
By the remark above 
\begin{align*}
\rho_*(L\sigma^*L\pi^*K^{1-2m}) & =  \rho_*(\rho^*(\bar{L}\bar{\pi}^*K^{1-2m}))\\
 & =  (\mathcal{O}_{\bar{S}} \oplus \bar{\pi}^*K^{-1})\otimes \bar{L} \bar{\pi}^*K^{1-2m}. 
\end{align*}
Since the term $\bar{\pi}^*K^{-1}$ comes from multiplication by $\lambda$, we must have 
 
  \[
    \left\{
                \begin{array}{ll}
                  \bar{M}_+ = \bar{L}\bar{\pi}^*K^{1 - 2m}\\
                  \bar{M}_- = \bar{L}\bar{\pi}^*K^{- 2m},
                \end{array}
              \right.
  \]
where, as before, $\rho_*(L\sigma^*L\pi^*K^{1-2m}) = \bar{M}_+ \oplus \bar{M}_-$.

\end{rmks}

The first hypercohomology group $\mathbb{H}^1 = \K^1 (\Sigma, \Lambda^2\mathsf{W})$ fits into the short exact sequence
\begin{equation*}
0 \to H^1(\bar{S}, \bar{M}_+) \to \mathbb{H}^1 \to H^0(\bar{S}, \bar{M}_- \otimes \bar{\pi}^*K^{2m}) \to 0.
\end{equation*}  
Equivalently, 
\begin{equation*}
0 \to H^1(\bar{S}, \bar{L}K_{\bar{S}}^{-1}) \to \mathbb{H}^1 \to H^0(\bar{S}, \bar{L}) \to 0.
\end{equation*} 

Since $\deg (\sigma^* L ) = \deg (L)$, we must have $\deg (\bar{L}) = \deg (L)$, whose value was given in (\ref{degreeL}). Let 
\begin{align*}
\bar{d} & = \deg (\bar{L}) = -d + 2(g_{_{\bar{S}}}-1),\\
d^\prime & = \deg (L^\prime) = d + 2(g_{_{\bar{S}}}-1). 
\end{align*}
From a non-zero element $u \in H^0(S, \bar{L})$ we get a divisor $\bar{D} \in \bar{S}^{(\bar{d})}$. Thus, we define\footnote{Note that we could also have defined $Z_d$ in terms of the line bundle $L$, i.e., \begin{align*}
Z_d & =  \{ (\bar{D}, L) \in \bar{S}^{(\bar{d})} \times \Pic^{d_1}(S) \ | \ \rho^*\mathcal{O}_{\bar{S}}(\bar{D}) \cong L\sigma^*L \}.
\end{align*}
Even though this is more concise, we consider $L^\prime$ as this is compatible with our construction in the $SL(4m, \mathbb{C})$ case.} 
\begin{align*}
Z_d & =  \{ (\bar{D}, A) \in \bar{S}^{(\bar{d})} \times \Pic^{d^\prime}(S) \ | \ A \sigma^*A \cong \rho^* (K_{\bar{S}}^2(-\bar{D}))  \}.
\end{align*}
 
To obtain non-zero global holomorphic sections of $\bar{L}$ we must impose the constraint $\bar{d} \ge 0$. Together with the assumption that $d$ is positive, we have the constraint $0 < d \leq 2(g_{_{\bar{S}}}-1)$.

\begin{rmks} 

\noindent 1. Considering the natural map $\varphi : Z_d \to  \bar{S}^{(\bar{d})}$, its fibres are modeled on the Prym variety $$\PS = \{M \in \Jac(S) \ | \ \sigma^*M \cong M^* \}$$
corresponding to the ramified covering $\rho : S \to \bar{S}$. In particular, $Z_d$ is connected for all $0 < d \leq 2(g_{_{\bar{S}}}-1)$.   

\noindent 2. Since $SO(4m, \mathbb{C}) \subseteq SL(4m, \mathbb{C})$, our fibre is a subscheme of the fibre described in the previous chapter. Let us temporarily denote $Z_d$ for $SL(4m, \mathbb{C})$ by $Z_d(SL(4m,\C))$. Then, 
$$Z_d(SL(4m,\C)) = \{ (D,A) \in S^{(\bar{d}^\prime)} \times \Pic^{d_2}(S) \ | \ A^2(D)\pi^*K^{1-4m} \in \Prym \},$$
where $\bar{d}^\prime = \deg ((L^\prime)^*L K_S) = -2d + 8m^2(g-1)$.
Then, 
$$Z_d \hookrightarrow Z_d(SL(4m,\C)).$$
To define the map note that we have a distinguished divisor on $S$ whose corresponding line bundle is isomorphic to $\pi^*K$. This is the ramification divisor $R_\rho$ of the covering $\rho : S \to \bar{S}$, which is simply the intersection of the zero section of $K$ with $S$. We can define the map above by assigning $(\bar{D}, A) \in Z_d$ to $(D, A) \in Z_d(SL(4m,\C))$, where $D = R_\rho + \rho^*\bar{D}$ (and this is clearly injective). To see that this is indeed well-defined, note that $\calO_S(D) = \pi^*K \rho^*\calO_{\bar{S}}(\bar{D})$. Thus, $A^2(D)\pi^*K^{1-4m} = A\sigma^*A^{-1}$, which is clearly in $\PS \subseteq \Prym$. 

\end{rmks}

The constructions from Chapter \ref{Chapter3} for $SL(4m, \mathbb{C})$ can be restricted to our case. For example, $E_d \to Z_d$ is a vector bundle, whose fibres are isomorphic to $H^1(\bar{S}, \bar{L}  K_{\bar{S}}^{-1} )$.  We thus obtain the following. 

\begin{thm}\label{descriptionofthefibreforso} Let $p(\lambda) = \lambda^{2m} + \pi^*b_2\lambda^{2m-2} + \ldots + \pi^*b_{2m}$ be a section of the line bundle $\pi^*K^{2m}$ on the total space of the cotangent bundle of $\Sigma$ whose divisor is a smooth curve $S$. The fibre $h^{-1}(p^2)$ of the $SO(4m,\mathbb{C})$-Hitchin fibration is a disjoint union 
$$h^{-1}(p^2) \cong N \cup \bigcup_{d=1}^{2(g_{_{\bar{S}}}-1)} A_d$$
where 
\begin{itemize}
\item $N$ has $2$ connected components and it corresponds to the locus of rank $2$ vector bundles on $S$ (of degree $4m(2m-1)(g-1)$) which admit an isomorphism $\psi : \sigma^*E \to E^* \otimes \pi^*K^{2m-1}$ satisfying $\tp{(\sigma^*\psi)} = - \psi$.
\item $A_d \to Z_d$ is an affine bundle modeled on the vector bundle $E_d \to Z_d$, whose fibre at $(\bar{D},L)$ is isomorphic to $H^1(\bar{S}, K_{\bar{S}}(-\bar{D}) )$. 
\item The natural map $Z_d \to S^{(\bar{d})}$ is a fibration, whose fibres are modeled on the Prym variety $\PS$.
\item Each stratum has dimension $\dim SO(4m,\mathbb{C})(g-1)$ and the irreducible components of $h^{-1}(p^2) $ are precisely the Zariski closures of $A_d$, $1 \leqslant d \leqslant 2(g_{_{\bar{S}}}-1)$, and $N$.
\end{itemize} 
\end{thm} 
\begin{proof}
The only thing left to be checked is that the dimension of each $A_d$ equals $\dim SO(4m,\mathbb{C})(g-1)$. Note that 
$$\dim A_d = \deg(\bar{L}) + \dim H^1(\bar{S}, \bar{L}K_{\bar{S}}^{-1}) + \dim \PS.$$  
The dimension of the Prym variety is $g_{_S} - g_{_{\bar{S}}}$, so, by (\ref{genusS}) and (\ref{genusbar}), the dimension of the Prym is $m(2m+1)(g-1)$. Now, the degree of $ \bar{L}K_{\bar{S}}^{-1}$ is $-d < 0$. So, by Riemann-Roch 
\begin{align*}
h^1 (\bar{S},  \bar{L}K_{\bar{S}}^{-1}) & = d + g_{_{\bar{S}}} - 1\\
& = d + m(2m-1)(g-1).
\end{align*}
The degree of $\bar{L}$ was given in (\ref{degreeL}) and so we obtain a total of 
\begin{align*}
& \bar{d} + h^1(\bar{S}, \bar{L}K_{\bar{S}}^{-1}) + \dim \PS =  \\
& =  (-d + 2m(2m-1)(g-1)) + (d + m(2m-1)(g-1)) + (m(2m+1)(g-1)) \\ 
& =  (8m^2 - 2m)(g-1) \\
& =  \dim SO(4m,\mathbb{C})(g-1),
\end{align*} 
finishing the proof.
\end{proof}


\section{The group $Sp(m,m)$} \label{Sp(m,m)} 

This is very similar to the $SO^*(4m)$ case. First we consider the rank $1$ sheaves supported on the reduced curve $S$. 

\begin{prop} The locus 
$$\{ (V,\Phi) \in h^{-1}(p^2) \ | \ p(\Phi) = 0 \} \subset \calM (Sp(4m,\C))$$
is isomorphic to 
$$
N = \left \{ \ E \in \calU_S (2,e) \
\bigg |
\gathered
\begin{array}{cl}
 & \text{There exists an isomorphism}\\
 & \psi : \sigma^*E \to E^* \otimes \pi^*K^{2m-1} \\
 & \text{satisfying $\tp{(\sigma^*\psi)} =  \psi$}.\
\end{array}
\endgathered  \ \  \right
\},$$
where $e= 4m(2m-1)(g-1)$. Moreover, $N$ is connected and of dimension $\dim Sp(4m,\C) (g-1)$.
\end{prop}
\begin{proof}
The proof for the first assertion follows exactly as in Proposition \ref{pairingn}. The only difference is that the pairing (\ref{pairing}) is now skew-symmetric, which is clear from the condition $\tp{(\sigma^*\psi)} =  \psi$. Note that starting with a symplectic Higgs bundle $(V,\Phi)$ we use the symplectic form to identify the bundle $V$ with its dual and obtain an isomorphism $\psi : \sigma^* E \to E^*$. Since $d\pi$ is anti-invariant under $\sigma$ and the symplectic form is skew-symmetric, we obtain that $\tp{(\sigma^*\psi)} =  \psi$. The dimension follows from e.g. \cite[Theorem 3.12.(a)]{zelaci2016hitchin}. Now, choosing a square root $K^{1/2} $ of the canonical bundle of $\Sigma$ we may identify the stable locus of $N$ as 
$$N^s \cong \{ \ E_0 \in \calU^s_S (2,e) \ | \ \exists \ \psi : \sigma^*E_0 \xrightarrow{\cong} E_0^* \ \text{s.t. }\tp{(\sigma^*\psi)} =  \psi \ \}.$$    
By fixing the determinant we obtain the stable locus $N_0^s$ of Higgs bundles for the real form $Sp(m,m)$. Note that, since the induced action on the determinant bundle is $-1$, at any fixed point $p \in S$ of $\sigma$ we have distinct $+1$ and $-1$ eigenspaces and $N_0^s$ is irreducible (see \cite{hitchin2014} for more details). Moreover, as argued for map (\ref{n0s}) in the $SO^*(4m)$ case, the tensor product gives a surjective map 
$$N_0^s \times \PS \to N^s$$
and thus $N^s$ is irreducible.   
\end{proof}

If $W$ is an $\Phi$-invariant subbundle of $(V,\Phi) \in \calM (Sp(4m,\C))$, then so is its symplectic complement $W^{\perp_\omega}$ since $\Phi$ is skew-symmetric. Suppose $(V,\Phi) \in h^{-1}(p^2)$ is such that $p(\Phi) \neq 0$, then, as we saw, $(V,\Phi)$ is stable and given by a unique non-split extension. Thus, $W \cong W^{\perp_\omega}$ which means it is a maximally isotropic subbundle of $(V,\omega)$. Thus, we consider extensions of Higgs bundles of the form 
$$0 \to (W,\phi) \to (V,\Phi) \to (W^*, -\tp{\phi}) \to 0$$
where $(V,\Phi)$ is a symplectic Higgs bundle and $W$ a maximally isotropic subbundle of $V$. It follows from Proposition \ref{themext} that such extensions are governed by the complex 
\begin{align}
\hat{\phi} : \Sym^2 W & \to \Sym^2 W \otimes K \label{complexs2} \\ 
w_1\odot w_2 & \mapsto  (\phi w_1)\odot w_2 + w_1\odot \phi w_2 \nonumber
\end{align}  
and we have\footnote{The proof is completely analogous to the $SO^*(4m)$ case.}
\begin{enumerate}
\item $\ker \hat{\phi} \cong \pi_*(L\sigma^*L\pi^*K^{1-2m})_-$,  
\item $\coker \hat{\phi} \cong \pi_*(L\sigma^*L\pi^*K^{1-2m})_+ \otimes K^{2m}$, 
\end{enumerate}
where once again $\pi_*(L,\lambda) \cong (W,\phi)$ and $\pi_*(L^\prime, \lambda) \cong (W^*,-\phi^t) $, for $L^\prime \cong \sigma^*L^{-1}\pi^*K^{2m-1}$.

Then, denoting the first hypercohomology group $\K^1 (\Sigma, \Sym^2\mathsf{W})$ of the complex (\ref{complexs2}), the short exact sequence in question is 
\begin{equation*}
0 \to H^1(\bar{S}, \bar{M}_-) \to \mathbb{H}^1 \to H^0(\bar{S}, \bar{M}_+ \otimes \bar{\pi}^*K^{2m}) \to 0.
\end{equation*}  

Or, equivalently, 

\begin{equation*}
0 \to H^1(\bar{S}, \bar{L}\bar{\pi}^*K^{-2m}) \to \mathbb{H}^1 \to H^0(\bar{S}, \bar{L}\bar{\pi}^*K) \to 0,
\end{equation*} 
where again $\bar{L}$ is the line bundle on $\bar{S}$ defined by $L\sigma^*L = \rho^*\bar{L}$.

Set 
$$\bar{d} = \deg (\bar{L}\bar{\pi}^*K) = -d + 4m^2(g-1).$$
Then, 
\begin{align*}
Z_d & =  \{ (\bar{D}, A) \in \bar{S}^{(\bar{d})} \times \Pic^{d^\prime}(S) \ | \ A \sigma^*A \rho^*\mathcal{O}_{\bar{S}}(\bar{D}) \cong \pi^*K^{4m-1} \}.
\end{align*}

\begin{rmk} The injective map 
$$Z_d \hookrightarrow Z_d(SL(4m,\C))$$
is now given by assigning $(A,\bar{D}) \in Z_d$ to $(A,D) \in Z_d(SL(4m,\C))$, where $D$ is the effective divisor $\rho^* \bar{D}$. Note that this is well defined because  
\begin{align*}
A^2(D)\pi^*K^{1-4m} = A\sigma^*A^*  
\end{align*} 
and $A\sigma^*A^{*} \in \PS \subseteq \Prym$.
\end{rmk}
The fibres of $\varphi : Z_d \to \bar{S}^{(\bar{d})}$ are again modeled on $\PS = \{M \in \Jac(S) \ | \ \sigma^*M \cong M^* \}$, so $Z_d$ is connected and the dimension of the affine bundle $A_d$ is 
\begin{align*}
& \deg(\bar{L}\bar{\pi}^*K) + \dim H^1(\bar{S}, \bar{L}\bar{\pi}^*K^{-2m}) + \dim \PS =  \\
& =  (-d + 4m^2(g-1)) + (d + m(2m+1)(g-1)) + (m(2m+1)(g-1)) \\ 
& =  (8m^2 + 2m)(g-1)\\
&= Sp(4m,\C) (g-1).
\end{align*}
Note that, since $(\bar{D},L) \in Z_d$ we have $\bar{L}\bar{\pi}^*K^{-2m} =  K_{\bar{S}}(-\bar{D})$. Thus, we obtain the following:

\begin{thm}\label{descriptionofthesibreforsp} Let $p(\lambda) = \lambda^{2m} + \pi^*b_2\lambda^{2m-2} + \ldots + \pi^*b_{2m}$ be a section of the line bundle $\pi^*K^{2m}$ on the total space of the cotangent bundle of $\Sigma$ whose divisor is a smooth curve $S$. The fibre $h^{-1}(p^2)$ of the $Sp(4m,\mathbb{C})$-Hitchin fibration is a disjoint union 
$$h^{-1}(p^2) \cong N \cup \bigcup_{d=1}^{g_{_{S}}-1} A_d$$
where 
\begin{itemize}
\item $N$ corresponds to the locus of rank $2$ semi-stable vector bundles on $S$ (of degree $4m(2m-1)(g-1)$) which admit an isomorphism $\psi : \sigma^*E \to E^* \otimes \pi^*K^{2m-1}$ satisfying $\tp{(\sigma^*\psi)} =  \psi$.
\item $A_d \to Z_d$ is an affine bundle modeled on the vector bundle $E_d \to Z_d$, whose fibre at $(\bar{D}, L)$ is isomorphic to $H^1(\bar{S}, K_{\bar{S}}(-\bar{D}) )$. 
\item The natural map $Z_d \to S^{(\bar{d})}$ is a fibration, whose fibres are modeled on the Prym variety $\PS$.
\item Each stratum has dimension $\dim Sp(4m,\mathbb{C})(g-1)$ and the irreducible components of $h^{-1}(p^2) $ are precisely the Zariski closures of $A_d$, $1 \leqslant d \leqslant g_{_{S}}-1$, and $N$.
\end{itemize} 
\end{thm} 

%% file: Chapters/Chapter5.tex

\chapter{Alternative description} 

\label{Chapter5} 

\lhead{Chapter 5. \emph{Alternative description}} 


Consider the Hitchin fibration for $GL(2m,\C)$-Higgs bundles (of topological type $d = 0$)
$$h : \calM (2m,d)  \to \calA = \bigoplus_{i=1}^{2m} H^0(\Sigma , K^{i}).$$
Let $p(x) = x^{2m} + a_1 x^{2m-1} + \ldots + a_{2m}$ be a polynomial with $a_i \in H^0(\Sigma , K^i)$. In particular, $p(x)^2$ defines a point in the Hitchin base, which we denote simply by $p^2 \in \calA$, and by the \textit{BNR} correspondence (see Theorem \ref{bnrcorrespondence}), the fibre $h^{-1}(p^2)$ is isomorphic to Simpson moduli space $\calM (p^2 ; k)$ of semi-stable sheaves of rank $1$ and degree $k = 2m(2m-1)(g-1)$ on the spectral curve
\begin{equation}
X \coloneqq \zeros  (p(\lambda)^2) \subset |K|.
\label{dou}
\end{equation}    
We denote by $\pi: X \to \Sigma$ the natural finite morphism of degree $2m$. Moreover, we will denote the reduced scheme $X_{\text{red}}$ by $S$ and assume that 
$$S = \zeros  (p(\lambda)) \subset |K|$$
is a non-singular curve. We also denote the restriction of $\pi$ to $S$ by  
$$\pi^{\text{red}} : S \seta \Sigma,$$ 
which is a ramified $m$-fold covering of $\Sigma$. As we will see, $X$ is an example of a ribbon and, motivated by the existing literature (e.g., \cite{donagi1995non, chen2011moduli}) we give an alternative description of the fibre $h^{-1}(p^2)$ of the $G$-Hitchin fibration (where $G = SL(2m,\C), SO(4m,\C) $ and $Sp(4m,\C)$). 

\section{Ribbons}

Ribbons can be defined on any reduced connected scheme $S$ of finite type over a fixed field $k$ (see, e.g., \cite{bayer1995ribbons}). For us, however, it will be enough to consider ribbons on a non-singular irreducible projective curve $S$ over the complex numbers. From now on, $S$ will always be a compact Riemann surface.

\begin{defin} A \textbf{ribbon} $X$ on $S$ is a curve (i.e., an irreducible projective $\C$-scheme of dimension $1$) such that:
\begin{enumerate}
\item $X_{\text{red}} \cong S,$
\item The ideal sheaf $\calI$ of $S$ in $X$ satisfies
$$\calI^2 = 0,$$
\item $\calI$ is an invertible sheaf on $S$.
\end{enumerate}
\end{defin} 

\begin{rmks} 

\noindent 1. Recall that the reduced scheme $X_{\text{red}}$ has the same underlying topological space of $X$ and its structure sheaf $\calO_{X_{\text{red}}}$ is given by 
$$0 \seta \calN_X \seta \calO_X \seta \calO_{X_{\text{red}}} \seta 0,$$
where $\calN_X$ is the nilradical of $X$. Since $X_{\text{red}} \cong S$, the ideal sheaf $\calI$ is the nilradical of $X$.

\noindent 2. The condition $\calI^2 = 0$ implies that: 
\begin{itemize}
\item $\calI$ is isomorphic to the conormal sheaf of $S$ in $X$. In particular, we have a short exact sequence of $\calO_S$-modules (namely, the restricted cotangent sequence): 
$$0 \seta \calI \seta \Omega_X|_S \seta \Omega_S \seta 0.$$
\item $\calI$ may be regarded as a sheaf on $S$, so that condition \textit{3} in the definition of a ribbon makes sense. 
\end{itemize}   
\end{rmks}

The restricted cotangent sequence not only gives an element in $\ext^1_S(\Omega_S, \calI)$, but it actually classifies ribbons on $S$. More precisely, given a line bundle $\calI$ on $S$ and a class $e \in \ext^1_S(\Omega_S, \calI)$, there is a unique ribbon $X$ on $S$ whose class is $e$ (\cite{bayer1995ribbons}, Theorem 1.2). This works for any ribbon on a connected reduced $k$-scheme. For us, this tells us that ribbons on $S$ are classified by $H^1(S, K_S^{-1}\calI)$.  

The non-reduced curve $X$ defined in (\ref{dou}) is an example of a ribbon, where the ideal sheaf is isomorphic to $\calO_{|K|}(- S)|_{S} \cong K_S^{-1}$. Thus, these are classified by $H^1(S,K_S^{-2})$ (or more precisely by zero and $\mathbb{P} (H^1(S,K_S^{-2}))$). In particular, the ribbon corresponding to $0 \in H^1(S,H_S^{-2})$ is just the first infinitesimal neighborhood of the zero section $\Sigma \subset |K|$ (e.g., $m=1$, where $p(x)=x$, using our previous notation).

\section{Generalized line bundles}

Let $X$ be a ribbon on $S$ with generic point $\eta$.

\begin{defin} Let $\calE$ be a coherent sheaf on $X$.
\begin{itemize}
\item $\calE$ is a \textbf{generalized line bundle} on $X$ if it is torsion-free and generically free of rank 1. The last condition means that $\calE_\eta \cong \calO_{X,\eta}$ (equivalently, $\calE|_U$ is an invertible sheaf for some open set $U \subseteq X$).    
\item The sheaf $\bar{\calE}$ associated to a generalized line bundle $\calE$ is defined to be the maximal torsion-free quotient of $\calE \otimes \calO_S$ (i.e., $\bar{\calE} = \calE |_S / \text{torsion}$).
\item The \textbf{genus} $g_{_X}$ of $X$ is $g_{_X} = 1 - \chi(\calO_X)$.
\end{itemize}
\end{defin}

\begin{rmk} Clearly, a generalized line bundle has polarized rank equal to $1$.
\end{rmk}
Let us restrict now to our case, where $X = \zeros  (p(\lambda)^2) \subset |K|$ (and we assume that the reduced scheme $S = \zeros  (p(\lambda))$ of $X$ is non-singular). As remarked, this means that the ideal sheaf of $S \subset X$ is $K_S^{-1}$. Given a generalized line bundle $\calE$, we can canonically associate an effective divisor $D$ supported on the set where $\calE$ fails to be locally isomorphic to $\calO_X$. One way to see this is by considering the sheaf $\calG$ defined as the kernel of the natural map from a generalized line bundle to its maximal torsion-free quotient. We have a short exact sequence of $\calO_X$-modules
\begin{equation*}
0 \seta \calG \seta \calE \seta \bar{\calE} \seta 0.
\end{equation*}  
Note that $\calG$ is torsion-free, since it is a subsheaf of the torsion-free sheaf $\calE$. Furthermore, since $\calE$ is generically free of rank $1$, $\calG$ can be seen as an invertible sheaf on $S$. The natural map $\calI \otimes \bar{\calE} \seta \calG$ is injective (recall that $\calI = K_S^{-1}$ is the ideal sheaf), thus $\calG$ is naturally an invertible sheaf on $S$ isomorphic to $K_S^{-1} (D) \otimes \bar{\calE} $, for the effective divisor $D$ on $S$. We shall denote by $b(\calE)$ the degree of the corresponding effective divisor, and call it the \textbf{index} of $\calE$. Note that the index of $\calE$ is zero if and only if $\calE$ is an invertible $\calO_X$-module. 
   
\begin{thm} (\cite{eisenbud1995clifford}, Thm 1.1) Let $\calE$ be a generalized line bundle on a ribbon $X$. There is a unique effective Cartier divisor $D$ on $S$ and a unique line bundle $\calL^\prime$ on the blow-up\footnote{Here, the Cartier divisor $D$ is considered as a closed subscheme of $X$.} $f: X^\prime \coloneqq \Bl_D(X) \seta X$ such that $f_*\calL^\prime \cong \calE$.
\end{thm}

\begin{rmk} $X^\prime$ is clearly also a ribbon on $S$ and the ideal sheaf of $S$ is $K_S^{-1}(D)$. In particular, using the short exact sequence of $\calO_{X^\prime}$-modules
$$0 \to K_S^{-1}(D) \to \calO_{X^\prime} \to \calO_S \to 0,$$
we find that the genus $g_{_{X^\prime}}$ of $X^\prime$ is $g_{_X} - b = 4g_{_S} -3 - b$, where $b = \deg (D)$.  
\end{rmk}

Let us see how this works through an example (the reader is referred to the paper \cite{chen2011moduli} for more details). Let $p\in X$ and denote by $\hat{\calO}_{X,p}$ the completion of the local ring $\calO_{X,p}$ of $X$ at the point $p\in X$. Thus, 
$$\hat{\calO}_{X,p} \cong \calO_0 \coloneqq \dfrac{\C [\![ x, y ]\!]}{(y^2)},$$
where $\C [\![ x, y ]\!]$ is the ring of formal power series over $\C$. 

Consider the ideal sheaf $\calE$ of the unique closed subscheme $Z$ of length $n > 0$ supported at $p$ and contained in $S$. This is a generalized line bundle, which locally is the ideal $(x^n,y)$, and its associated divisor is $D = np$ (on $S$). Denote by 
$$f: X^\prime \seta X$$ 
the natural morphism associated to the blow-up $X^\prime = \Bl_D(X)$ of $X$ along $D$. The finite morphism $f: X^\prime \seta X$ is a set-theoretical bijection (for any effective Cartier divisor on $S$). Let $q \in X^\prime$ be the point mapping to $p  \in X$ via this map. Then,    
$$\hat{\calO}_{X^\prime,q} \cong \calO_n \coloneqq \dfrac{\C [\![ X, Y ]\!]}{(Y^2)},$$
where $\calO_n$ is the $\calO_0$-algebra via
\begin{align}
\calO_0 & \to \calO_n   \label{algebra} \\
x & \mapsto X, \nonumber \\
y & \mapsto YX^n. \nonumber
\end{align}
Then, the line bundle $\calL^\prime \in \Pic (X^\prime)$ is locally $\calO_n$ (seen as an algebra), and its direct image $\calE$ is locally $\calO_n$ seen as an $\calO_0$-module via the map (\ref{algebra}), which is exactly the map associated to $f$ at the level of stalks (the blow-up of $\Spec (\calO_0)$ along the ideal $(x^n,y)$ is $\Spec (\calO_n) \to \Spec (\calO_0)$ by Lemma 2.12 from \cite{chen2011moduli}).   

Now let $D$ be any effective Cartier divisor on $S$. Since the ideal sheaf of $S$ in $X^\prime$ is $K_S^{-1}(D)$ and 
\begin{equation}
0 \to K_S^{-1} \to K_S^{-1}(D) \to K_S^{-1}(D)\otimes \calO_D \to 0,
\label{divi}
\end{equation}
we have
$$0 \to \calO_X \to f_*\calO_{X^\prime} \to K_S^{-1}(D)\otimes \calO_D \to 0.$$     
Notice that, restricting to the invertible elements\footnote{Note that we can always write $P \in \calO_0$ as $P = p_1(x) + yp_2(x) + (y^2)$, where $p_1,p_2 \in \C [\![ x ]\!] $. Clearly, $P $ is not a zero divisor of $\calO_0$ if and only if $p_1 \neq 0$. Also, $P \in \calO_0^\times$ is an invertible element of $\calO_0$ if and only if $p_1 \in \C [\![ x ]\!]^\times = \C^\times $ is an invertible polynomial in $x$ (i.e., a non-zero scalar). }, we obtain 
$$0 \to \calO^\times_X \to f_*\calO^\times_{X^\prime} \to K_S^{-1}(D)\otimes \calO_D \to 0.$$     
Now, the associated exact sequence in cohomology gives
\begin{equation}
0 \to H^0(D, K_S^{-1}(D)|_D) \to \Pic (X) \to \Pic (X^\prime) \to 0,
\label{picses}
\end{equation}
since $H^0(X, \calO^\times_X) \cong H^0(X^\prime, \calO^\times_{X^\prime}) \cong \C^\times$,  $H^1(X, f_*\calO^\times_{X^\prime}) \cong H^1(X^\prime, \calO^\times_{X^\prime})$ and the first cohomology of sheaves supported on points vanishes. Here, $\Pic (X) = H^1(X, \calO_X^\times)$ is the Picard group of the curve $X$. Note that we have the following commutative diagram 
\[\xymatrix@M=0.08in{
  &  & 0\ar[d] & 0\ar[d] &  \\
0 \ar[r] & H^0(D,K_S^{-1}(D)) \ar[r] \ar[d]^{\Id} & H^1(S,K_S^{-1}) \ar[r] \ar[d] & H^1(S,K_S^{-1}(D)) \ar[r] \ar[d] & 0  \\
 0 \ar[r] & H^0(D,K_S^{-1}(D)) \ar[r]  & \Pic(X) \ar[r]^{f^*} \ar[d]  & \Pic(X^\prime) \ar[r] \ar[d] & 0 \\
  &   & \Pic(S) \ar[r]^{Id} \ar[d]  & \Pic(S)  \ar[d] &  \\
&   & 0   & 0 .  &  }\] 
The top horizontal sequence in the diagram comes from (\ref{divi}) by noting that $\deg K_S^{-1}(D) < 0$ and the middle horizontal one, for example, comes from the exact sequence  
$$0 \to K_S^{-1} \to \calO_X^\times \to \calO_S^\times \to 0.$$ 

\begin{rmk} Let $\calE_1$ and $\calE_2$ be generalized line bundles on $X$ with the same associated effective divisor $D = n_1 p_1 + \ldots + n_k p_k$ on $S$. Then, $\calE_i \cong f_*\calL^\prime_i$, $i=1,2$, for line bundles $\calL^\prime_i$ on the blow-up $X^\prime = \Bl_D (X)$. Since the map $f^* : \Pic (X) \to \Pic (X^\prime)$ is surjective, we can find $\calL \in \Pic(X)$ such that $f^*(\calL) = \calL^\prime_1 (\calL^\prime_2)^{-1}$, which, by the projection formula, implies 
$$\calE_2 \otimes \calL \cong \calE_1.$$
Note also that any two line bundles with this property differ by an element of $H^0(D,K_S^{-1}(D))$. In particular, denoting by $Z$ the unique closed subscheme on $S$ supported at the points $p_i$, $i=1, \ldots, k$, and with lengths $n_i$ at $p_i$, the ideal sheaf $\calI_Z$ is a generalized line bundle and all the other generalized line bundles with corresponding divisor $D$ are given by $\calI_Z \otimes \calL$, where $\calL \in \Pic(X)$ (for more details see \cite[Lemma 2.9]{chen2011moduli}).
\end{rmk}

\section{Stability}

Let us give some examples of rank $1$ torsion-free sheaves on $X$ and discuss stability for these sheaves (as points in Simpson moduli space $\calM (p^2 ; k)$, where $k= 2m(2m-1)(g-1)$). In particular, note that (polarized) slope stability as defined in section \ref{Simpson} is independent of the polarization taken. 

\begin{exs}

\noindent 1. Let $E$ be a rank $2$ vector bundle on $S$ of degree $e = 2m(m-1)(g-1)$. Consider it as an $\calO_X$-module (i.e., $j_*E$, where $j : S \hookrightarrow X$) and let us calculate its Hilbert polynomial.
\begin{align*}
P(j_*E, t) & = \chi(X, j_*E \otimes \pi^*\calO_\Sigma (t)) \\
& = \chi(X, j_*(E \otimes j^*\pi^*\calO_\Sigma (t)))\\
& = \chi(S, E \otimes \pi^{\text{red}}\calO_\Sigma (t)) \\
& = 2m\delta t + e + 2m^2(1-g). \qquad \qquad \text{(by Riemann-Roch)}
\end{align*}
Therefore, $j_*E$ has (polarized) rank $1$ and (polarized) degree $k$. Clearly, $j_*E$ is torsion-free, but not generically free (so it is not a generalized line bundle on the ribbon $X$). Moreover, semi-stability for $j_*E$ is the statement that every subsheaf $\calF$ satisfies 
$$\frac{\deg_P (\calF)}{\rk _P(\calF)} \leq \dfrac{\deg_P (j_*E)}{\rk_P (j_*E)} = k.$$
It is enough to restrict to the case where $\calF$ is of the form $j_*L$, for a line subbundle $L $ of $E$. Calculating the Hilbert polynomial of $j_*L$, we find that its (polarized) rank is $\rk (j_*L) = 1/2$ and its (polarized) degree is $\deg_P (j_*L) = \deg (L) + m^2(g-1)$. Thus, semi-stability for $j_*L$ is equivalent to 
$$\frac{\deg (L) + m^2(g-1)}{1/2} \leq k$$
for every line subbundle $L \subset E$, i.e., 
$$\deg(L)\leq e/2.$$
This proves that (semi-)stability of a rank $2$ bundle $E$ on $S$ is equivalent to (semi-)stability of the corresponding rank $1$ torsion-free sheaf $j_*E$ on $X$.

\noindent 2. Let $\calE$ be a semi-stable generalized line bundle and $b$ the degree of its corresponding effective divisor $D$ on $S$. Calculating degrees via the short exact sequence of $\calO_X$-modules 
$$0 \to j_*(\bar{\calE}(D)\otimes K_S^{-1}) \to \calE \to j_*(\bar{\calE}) \to 0 $$
we obtain 
$$\deg (\bar{\calE}) = \frac{\deg_P (\calE) - b}{2}.$$   
Moreover, $\calF = j_*(\bar{\calE}(D)\otimes K_S^{-1})$ is a subsheaf of $\calE$ and by stability, $\deg(\calF)/\rk (\calF) = \deg (\bar{\calE}) + b - m^2(g-1)/ (1/2) = \deg (\calE) + b - 2m^2(g-1) \leq \deg(\calE)$. Thus we have the constraint 
$$b \leq 2m^2(g-1).$$ 
\end{exs}

The Simpson moduli space $\calM (p^2; k)$ parametrizes semi-stable sheaves on $X$ of (polarized) rank $1$ and (polarized) degree $k$. By \cite[Theorem 3.5]{chen2011moduli} (see also \cite[Theorem 3.2]{donagi1995non}) the ones described above comprise all such sheaves. More precisely, a semi-stable rank $1$ sheaf $\calE$ on $X$ is either   
\begin{itemize}
		\item a generalized line bundle whose index satisfies $b(\calE) \leq 2m^2(g -1)$, or 
		\item a direct image $i_{*}(E)$ of some semi-stable rank $2$ vector bundles $E$ on $S$ of degree $e = 2m(m-1)(g-1)$.
	\end{itemize} 

Similarly, a generalized line bundle is stable if and only if the inequality above is strict. In particular, this means that a generalized line bundle $\calE$ is strictly semi-stable if and only if $b(\calE) = 2m^2(g -1)$. Moreover, (semi-)stability of $E \in \calU_S(2,e)$ is equivalent to (semi-)stability of $j_*E \in \calM (p^2; k)$ as we have checked. 

\section{The singular fibre}

Note that if we disregard the constraint imposed by the determinant, Theorem \ref{fibreforsu*} describes the fibre $h^{-1}(p^2)$ for the Hitchin fibration 
$$h : \calM (2m,0) \to \calA.$$
In particular, from the \textit{BNR} correspondence, 
$$h^{-1}(p^2) \cong \calM (p^2 ; k) \cong \calU_S(2,e) \cup \bigcup_{d=1}^{g_{_S} - 1}A_d .$$

Let $\bar{d} = -2d + 2(g_{_S} -1)$, $d>0$, and consider 
$$\calA_d \coloneqq \{ \calE \in \calM^s (p^2 ; k) \ | \ b(\calE) = \bar{d}  \},$$  
where $k = 2m(2m-1)(g-1)$. 

There is a natural map  
\begin{align*}
\calA_{d} & \to \Pic^{d_2}(S) \times S^{(\bar{d})} \\
\calE & \mapsto (\bar{\calE}, D_\calE),
\end{align*}
where $d_2 = (k-\bar{d})/2$. Let $(L,D) \in \Pic^{d_2}(S) \times S^{(\bar{d})}$. As discussed, every generalized line bundle $\calE$ is characterized by its effective divisor $D_\calE$ and a line bundle $\calL^\prime $ in $X^\prime = \Bl_{D_{\calE}}X$. Thus, the fibre at $(L,D)$ of the map above is given by the line bundles $\calL^\prime \in \Pic (X^\prime)$ such that the restriction $\bar{\calL^\prime}$ to $S$ is isomorphic to $L$. But
$$0 \to H^1(S,K_S^{-1}(D)) \to \Pic (X^\prime) \to \Pic (S) \to 0,$$
so that the fibre at $(L,D)$ is an affine space modeled on $H^1(S,K_S^{-1}(D))$. This makes $\calA_{d}  \to \Pic^{d_2}(S) \times S^{(\bar{d})}$ into an affine bundle modeled on the vector bundle $E_d$ on $\Pic^{d_2}(S) \times S^{(\bar{d})}$ (see Chapter \ref{Chapter3} for the definition and construction of $E_d$).    

\begin{prop}\label{osdoad} The locus $\calA_d$ is isomorphic to the total space of the affine bundle $A_d$ (see Definition \ref{defA_d}) for $ 1\leq d \leq g_{_S} -1$.
\end{prop}
\begin{proof} 
Let $\calE \in \calA_{\bar{d}}$. Then, the direct image $\pi_*\calE$ is a rank $2m$ vector bundle on $\Sigma$, as $\calE$ is torsion-free and generically free of rank $1$, and pushing forward multiplication by the tautological section $\lambda : \calE \to \calE \otimes \pi^*K$ we get a Higgs field $\Phi : V \to V \otimes K$. As $\calE$ is a sheaf on $X$, $p(\Phi)^2=0$ and $(V,\Phi) \in h^{-1}(p^2)$. Note that $p(\Phi)$ cannot be zero, as that would correspond to a rank $2$ vector bundle on $S$. Thus, $(V,\Phi)$ is an extension of the form 
$$0 \to (W_1, \Phi_1) \to (V,\Phi) \to (W_2, \Phi_2) \to 0,$$
where $(W_1, \Phi_1) \cong \pi_*(L_i,\lambda)$, $i=1,2$, and
$$0 \to L_1 \to \calE \to L_2 \to 0.$$
In other words, $L_2 = \bar{\calE}$ and $L_1 = \bar{\calE}K_S^{-1}(D_\calE)$. 

Let now $(V,\Phi) \in A_d$ which projects to $(L,D) \in \Pic^{d_2}(S) \times S^{(\bar{d})}$. Consider the blow-up $f : X^\prime \coloneqq \Bl_{D_{\calE}} \to X$ and denote by $\pi^\prime $ the composition $\pi \circ f : X^\prime \to \Sigma$. Then, (see, e.g., \cite{si1, bnr}) the cokernel of the map 
$$\pi^*\Phi - \lambda : \pi^*V \to \pi^*(V \otimes K)$$
is a torsion-free sheaf $\calE \otimes \pi^*K$ of rank $1$ (which must then be a generalized line bundle on $X$) such that $\pi_*\calE \cong V$ and the Higgs field $\Phi$ is obtained by pushing forward multiplication by the tautological section. Alternatively, we could consider the cokernel of the map pulled back to the ribbon $X^\prime$  
$$(\pi^\prime)^*\Phi - \lambda : (\pi^\prime)^*V \to (\pi^\prime)^*(V \otimes K).$$
This gives a line bundle $\calL^\prime \otimes (\pi^\prime)^*K \in \Pic (X^\prime)$. Thus, $f_*\calL^\prime = \calE \in \calA_d$ is the generalized line bundle.
\end{proof}

The proposition above gives a dictionary between our description of the fibre 
$$h^{-1}(p^2) \cong \calU_S(2,e) \cup \bigcup_{d=1}^{g_{_S} - 1}A_d $$
in terms of the reduced curve, and the description 
$$h^{-1}(p^2) \cong \calM^k(X) \cong \calU_S(2,e) \cup \bigcup_{d=1}^{g_{_S} - 1}\calA_d,$$
of the Simpson moduli space in terms of rank $2$ bundles on $S$ and generalized line bundles (see \cite{chen2011moduli} and \cite{donagi1995non} for more on the second description). Note also that the number $\bar{d}$ which appeared in Chapter \ref{Chapter3}, by considering the divisor obtained from the projection in (\ref{jk}) of the extension class of the Higgs bundle in the fibre, has now a simple interpretation. It is the degree of the effective divisor corresponding to the singularities of the generalized line bundle.    

Alternatively, $\calA_d $ can be understood by looking directly at the projection to the corresponding divisor $$\calA_d \to S^{(\bar{d})}.$$The fibre of this map at a divisor $D$ is $\Pic^{k-\bar{d}}(X^\prime)$, where again $X^\prime = \Bl_DX$. The degrees are found using the fact that $f_*\calL^\prime \cong \calE \in \calA_d$, so that, $\chi (X^\prime, \calL^\prime) = \chi (X, \calE)$. In other words, the fibres are torsors for $\Pic^0(X^\prime)$. Note also that for $d= g_{_S} -1$, the corresponding generalized line bundles have no singularities ($\bar{d}=0$) and $\calA_{g_{_S}-1} = \Pic^k (X)$. In this case, $\calA_{g_{_S}-1}$ is an affine bundle on $\Pic^{k/2} (S)$, and the affine structure comes from 
$$0 \to H^1(S,K_S^{-1}) \to \Pic (X) \to \Pic (S) \to 0.$$

\section{The case $SU^*(2m)$ revisited}
\subsection{Norm map}

We have already encountered the norm map associated to a covering between two Riemann surfaces (see (\ref{normmapsmt}) in Chapter \ref{Chapter3} ). We may generalize the construction of the norm map for a finite morphism  
$$\pi: X \to \Sigma $$
where $X$ is a more general curve (e.g., an irreducible, projective $\C$-scheme of dimension 1) by setting 
\begin{align*}
\Nm_\pi : \Pic^0(X) & \to \Pic^0(\Sigma)\\
\calL & \mapsto \det (\pi_*\calL) \otimes \det (\pi_*\calO_X)^{-1}.
\end{align*}
This is a group homomorphism (for more details see \cite[Section 6.5]{grot}) and clearly satisfies 
$$\Nm_\pi (\pi^*L) = L^{\deg (\pi)},$$
as $\det (\pi_*(\pi^*L)) = \det (\pi_*\calO_X \otimes L) = \det (\pi_*\calO_X) \otimes L^{\deg{\pi}}$ using the projection formula. 

The \textbf{Prym variety} $\PX$ is defined as the kernel of this map. Note that in this generality, $\PX$ is neither a proper $\C$-scheme, nor is it connected. The problem of finding the number of connected components of the Prym variety was addressed in \cite{hausel2012prym}.    

We will be particularly interested in the case where $X$ is a ribbon as before. 
In this case, $\Pic^0(X)$ has dimension $g_{_X} = 1 - \chi (\calO_X) = 4g_{_S} - 3$, but it is not a proper $\C$-scheme, so, in particular, it is not projective (see, e.g., \cite{bosch2012neron, chen2011moduli}). Given an effective divisor $D$ on $S$ we have three natural finite morphisms, namely $\pi^{\text{red}} : S \to \Sigma$, $\pi : X \to \Sigma$ and $\pi^\prime = \pi \circ f : X^\prime \to \Sigma$, where $X^\prime = \Bl_D (X)$ and $f$ is the blow-up map. The norms of these maps can be compared, by noting that  
\begin{align}
\det (\pi^{\text{red}}_*\calO_S) & \cong K^{- m(m-1)/2} \label{detR} \\
\det (\pi_*\calO_X) & \cong K^{- m(2m-1)}. \label{detS} 
\end{align}
This follows directly from the fact that $X$ and $S$ are subschemes of the total space of the canonical bundle of $\Sigma$ determined by the zeros of a section of $p(\lambda)^2 \in H^0(|K| , \pi^*K^{2m})$ and $p(\lambda) \in H^0(|K| , \pi^*K^{m})$ (for more details, see Section \ref{Section1.4}). The determinant of the direct image of the structure sheaf of $X^\prime$ by the morphism $\pi^\prime$ can also be computed by the simple lemma below.   

\begin{lemma}\label{xprime} $\det (\pi^\prime_* \calO_{X^\prime}) = \Nm_{\pi^{\text{red}}} (\calO_S(D)) \otimes K^{-m(2m-1)}$. 
\end{lemma}
\begin{proof}
Consider the exact sequence of $\calO_{X^\prime}$-modules
$$0 \to K_S^{-1}(D) \to \calO_{X^\prime} \to \calO_S \to 0.$$
Then, we must have 
$$\det (\pi^\prime_* \calO_{X^\prime})   \cong \det (\pi^{\text{red}}_* K_S^{-1}(D)) \otimes \det (\pi^{\text{red}}_* \calO_S),$$
and again by the projection formula
$$\det (\pi^\prime_* \calO_{X^\prime})   \cong \det (\pi^{\text{red}}_* \calO_S(D)) \otimes K^{-m^2} K^{-m(m-1)/2} \cong \det (\pi^{\text{red}}_* \calO_S(D)) \otimes K^{(-3m^2+m)/2}.$$
But, by the definition of the norm map, we find  
\begin{align*}
\det (\pi^\prime_* \calO_{X^\prime}) & \cong \Nm_{\pi^{\text{red}}} (\calO_S(D)) K^{-m(m-1)/2}K^{(-3m^2+m)/2}\\
& \cong \Nm_{\pi^{\text{red}}} (\calO_S (D))K^{-m(2m-1)},
\end{align*}
a required.
\end{proof}

By the comments above we can deduce\footnote{A more general statement can be found in \cite[Lemma 3.6]{hausel2012prym}. For convenience, however we write down a proof for the cases of importance for us.} the following.

\begin{lemma}\label{normvs} Let $j: S \hookrightarrow X$ and $j^\prime : S \hookrightarrow X^\prime$ be the natural inclusion.
\begin{enumerate}[label=\alph*),ref=\alph*]
\item $\Nm_\pi (\calL) =  \Nm_{\pi^{\text{red}}}^2 (j^*\calL)$, for any $\calL \in \Pic(X)$.
\item $\Nm_{\pi^\prime} (\calL^\prime) = \Nm_{\pi^{\text{red}}}^2 ((j^\prime)^*(\calL^\prime))$, for any $\calL^\prime \in \Pic(X^\prime)$. 
\end{enumerate}
\end{lemma}
\begin{proof}
From the exact sequence
$$0 \to K_S^{-1}\bar{\calL} \to \calL \to \bar{\calL} \to 0 $$ 
we obtain 
\begin{align*}
\det (\pi_* \calL) & \cong \det (\pi^{\text{red}}_* (\bar{\calL}\otimes K_S^{-1})) \otimes \det (\pi^{\text{red}}_* (\bar{\calL})) \\
& \cong \det (\pi^{\text{red}}_* (\bar{\calL} ) \otimes K^{-m})) \otimes \det (\pi^{\text{red}}_* (\bar{\calL}))\\
& \cong (\det (\pi^{\text{red}}_* (\bar{\calL})))^2 \otimes K^{-m^2},
\end{align*}
and a) follows from the definition of the norm map and (\ref{detS}). Alternatively, we could have used the short exact sequence 
$$0 \to \calL \to j_*j^*\calL \to \calT \to 0,$$
where $\calT$ is a torsion $\calO_X$-module. Item b) is analogous and the result follows, for example, by using the short exact sequence
$$0 \to K_S^{-1}(D)\bar{\calL^\prime} \to \calL^\prime \to \bar{\calL^\prime} \to 0 $$ 
together with the definition of the norm map and Lemma \ref{xprime}.
\end{proof}
In particular, the Prym variety $\PX$ is related to the Prym variety associated to the covering $\pi^{\text{red}} : S \seta \Sigma$ 
\[\xymatrix@M=0.08in{
  & 0\ar[d] & 0\ar[d] &  &  \\
 & H^1(S,K_S^{-1})  \ar[r]^{\Id}\ar[d] & H^1(S,K_S^{-1})  \ar[d] &  &   \\
 0\ar[r] & \PX \ar[r] \ar[d]^{j^*} & \Pic^0(X) \ar[r]^{\Nm_{\pi}} \ar[d]^{j^*}  & \Pic^0(\Sigma) \ar[r] \ar[d]^{\Id} & 0 \\
 0\ar[r]  & F\ar[r] \ar[d]  & \Pic^0(S) \ar[r]^{\Nm_{\pi^{\text{red}}}^2} \ar[d]  & \Pic^0(\Sigma) \ar[r]  & 0. \\
 &  0 & 0   &   &  }\]  
The subvariety $F$ was defined in Chapter \ref{Chapter3} as
$$F = \{ M \in \Pic^0(S) \ | \ M^2 \in \Prym \} $$
and it also appears naturally when comparing the norm maps associated with $\pi^{\text{red}}$ and $\pi^\prime$   
\[\xymatrix@M=0.08in{
  & 0\ar[d] & 0\ar[d] &  &  \\
 & \dfrac{H^1(S,K_S^{-1})}{H^0(D,K_S^{-1}(D))} \ar[r]^{\Id}\ar[d] & H^1(S,K_S^{-1}(D))  \ar[d] &  &   \\
 0\ar[r] & \PXp \ar[r] \ar[d]^{(j^\prime)^*} & \Pic^0(X^\prime) \ar[r]^{\Nm_{\pi^\prime}} \ar[d]^{(j^\prime)^*}  & \Pic^0(\Sigma) \ar[r] \ar[d]^{\Id} & 0 \\
 0\ar[r]  & F\ar[r] \ar[d]  & \Pic^0(S) \ar[r]^{\Nm_{\pi^{\text{red}}}^2} \ar[d]  & \Pic^0(\Sigma) \ar[r]  & 0 \\
 &  0 & 0   &   &  }\]  

\subsection{The fibre}

By \cite[Proposition 6.1]{hausel2012prym}, 
$$h^{-1}(p^2) \cong \{ \calE \in \calM (p^2;k) \ | \ \det(\pi_*\calE) = \calO_\Sigma \},$$
where $k=2m(2m-1)(g-1)$. 
Thus, we define 
$$\calA_d \coloneqq \{ \calE \in \calM^s (p^2;k) \ | \ b(\calE) = \bar{d} \ \text{and} \ \det(\pi_*\calE) = \calO_\Sigma \},$$
where $\bar{d} = -2d + 2(g_{_S} -1)$, $d>0$. Once again we have the natural map
\begin{equation*}
\begin{split}
\calA_d & \to \Pic^{d_2}(S) \times S^{(\bar{d})},\\  
\calE & \mapsto (\bar{\calE}, D_{\calE})
\end{split}
\end{equation*} 
where $d_2 = (k-\bar{d})/2$. This morphism maps $\calA_d$ onto a subvariety $\calZ_d \subseteq \Pic^{d_2}(X) \times S^{(\bar{d})}$. This, as we observe next, is isomorphic to 
$$Z_d \cong \{ (D,L) \in S^{(\bar{d})} \times \Pic^{d_2}(S) \ | \ L^2(D) (\pi^{\text{red}})^*K^{1-2m} \in \Prym \}$$
which was defined in Chapter \ref{Chapter3}.
 
\begin{prop}\label{basesch5} The locus $\calZ_d$ is isomorphic to the base $Z_d$ of the affine bundle $A_d$.  
\end{prop}
\begin{proof}
Let $\calE \in \calA_d$ and $\calL^\prime$ be its corresponding invertible sheaf on the blow-up $X^\prime = \Bl_{D_{\calE}}X$. Then the condition $\det (\pi_*\calE) \cong \calO_\Sigma$ is equivalent to $\det (\pi^\prime_*\calL^\prime) \cong \calO_\Sigma$, since $f_*\calL^\prime \cong \calE$ and $\pi^\prime = \pi \circ f$. But, as noted, $f_*\calO_{X^\prime}/\calO_X $ is supported on $D_\calE$, so $\det (\pi_*f_*\calO_{X^\prime}) = \det (\pi^\prime_*\calO_{X^\prime})$ is isomorphic to $\det (\calO_X)\otimes \det(\pi^{\text{red}}_*\calO_D)$. Since $\det(\pi^{\text{red}}_*\calO_D) \cong \det (\pi^{\text{red}}_* \calO_S(D)) \otimes (\det (\pi^{\text{red}}_*\calO_S))^{-1}$, it follows from the formulas in Lemma \ref{xprime} that 
$$\det (\pi^\prime_*\calO_{X^\prime}) \cong \Nm_{\pi^{\text{red}}} (\calO_S(D))\otimes K^{m(1-2m)}.$$
Thus, $\det (\pi^\prime_*\calL^\prime) \cong \calO_\Sigma$ is equivalent to $\Nm_{\pi^\prime}(\calL^\prime) \otimes \det (\pi^\prime_*\calO_{X^\prime}) \cong \calO_\Sigma$. But, by Lemma \ref{normvs} we have $\Nm_{\pi^\prime}(\calL^\prime) \cong \Nm_{\pi^{\text{red}}} (\bar{\calE }^2)$. Thus, the condition is equivalent to $\bar{\calE}^2(D)(\pi^{\text{red}})^*K^{1-2m}$ being an element in $\Prym$, which is exactly the definition of $Z_d$.    
\end{proof}

An alternative way to understand $\calA_d \cong A_d$ is by looking at the natural projection to the effective divisor
$$\calA_d \to S^{(\bar{d})}.$$

\begin{prop}\label{ch5seg} Let $1 \leq d \leq g_{_S}-1$. The fibre of $\calA_d \to S^{(\bar{d})}$ at an effective divisor $D$ is a torsor for the Prym variety $\PXp$, where $X^\prime = \Bl_DX$. 
\end{prop}
\begin{proof}
Let $\calE_i$, $i=1,2$, be two elements of the fibre at $D \in S^{(\bar{d})}$. Consider the corresponding invertible sheaves $\calL_i^\prime $ in the blow-up $X^\prime = \Bl_DX$. As we saw, $\det(\pi^\prime_*\calL_i^\prime) \cong \calO_\Sigma$. Denote by $\calU^\prime = (\calL_1^\prime)^{-1} \calL_2^\prime $. Then, since the norm map is a group homomorphism, 
$$\det(\pi^\prime_*\calL_2^\prime)  \det(\pi^\prime_*\calL_1^\prime) \Nm_{\pi^\prime} (\calU^\prime),$$ 
which means that $\calU^\prime \in \PXp$. 
\end{proof}

Checking dimensions 
\begin{align*}
\dim \calA_d & = \dim S^{(\bar{d})} + \dim \PXp \\
& = \bar{d} + (g_{_{X^\prime}} - g)\\
& = g_{X} - g \\
& = (4m^2 -1)(g-1),
\end{align*} 
as expected. In particular, we have the following: 
\begin{prop}\label{ch5ter} The Prym variety $\PX$ is isomorphic to $\calA_{g_{_S}-1} \cong A_{g_{_S} -1}$. In particular, it has $2^{2g}$ connected components. 
\end{prop}
\begin{proof}
An element of $\calA_{g_{_S}-1}$ is an invertible sheaf $\calL$ on $X$ satisfying
$$L^2 (\pi^{\text{red}})^*K^{1-2m} \in \Prym,$$
where $L = \bar{\calL}$. Thus, choosing a square root $K^{1/2}$ of the canonical bundle of $\Sigma$, tensoring by $\pi^* K^{(1-2m)/2}$ gives an isomorphism from $\calA_{g_{_S}-1}$ and $\PX$. Indeed, this follows directly from the fact that $\Nm_{\pi} (\calL) = \Nm_{\pi^{\text{red}}}(L^2)$. The number of connected components is a consequence of Lemma \ref{compzd}. 
\end{proof}

\begin{rmk} As observed in the proof above, by choosing a square root $K^{1/2}$, $\calA_{g_{_S}-1}$ can be identified with $\PX$. Also, under this isomorphism, $\calZ_{g_{_S}-1}$ gets identified with $F$. When $m=1$, for example, the spectral curve equals the non-reduced curve with trivial nilpotent structure of order $2$, $F$ equals the $2$-torsion points of $\Pic (\Sigma)$ and the Prym variety $\PX$ is a copy of $2^{2g}$ vector spaces. As proven by Hausel and Pauly \cite[Theorem 1.1(3)]{hausel2012prym}, the group of connected components of $\PX$ is, in this case, isomorphic to $\Pic^0 (\Sigma)[2]$, which is compatible with our description.     
\end{rmk}

Note that taking the tensor product of elements of the fibre with $\pi^*K^{(1-2m)/2}$, identifies $h^{-1}(p^2)$ with $\widetilde{\Nm}_\pi^{-1}(\calO_\Sigma)$, where $\widetilde{\Nm}_\pi$ is defined as
\begin{align*}
\widetilde{\Nm}_\pi : \calM (p^2 ; 0) & \to \Pic^0(\Sigma)\\
\calE & \mapsto \det (\pi_*\calE) \otimes \det (\pi_*\calO_X)^{-1}.
\end{align*}
This is well-defined and our fibre can be seen as a compactification of the Prym variety $\PX$ inside the Simpson moduli space.

\section{The other cases}

Let us give some indications on how to proceed in the other two cases. Consider the polynomial $p(x) = x^{2m} + b_2x^{2m-2} + \ldots + b_{2m}$, where $b_{2i} \in H^0(\Sigma, K^{2i})$, and assume that $S = \zeros (p(\lambda))\subset |K|$ is a non-singular curve. Once again we denote by $\bar{S} = S/\sigma $ the quotient curve, which is also non-singular, and think of it embedded in $|K^2|$ with equation $p(\eta) = 0$, where $\eta = \lambda^2$.  

The non-reduced curve $X = \zeros (p(\lambda)^2) \subset |K| $ is a  ribbon on $S$ with ideal sheaf $K_S^{-1}$. In particular,
$$g_{_X} - 1 = 16m^2(g-1)$$
and the finite map $\pi : X \to \Sigma$ provides a good local picture of the spectral curve $X$. Indeed, let $U \subset \Sigma$ be an open affine set and $u : U \to K^*$ a nowhere vanishing section. Then, $\pi^{-1}(U)$ is the affine open set of $X$ given by 
$$\Spec \left(\frac{\calO_\Sigma (U)[u]}{((u^{2m} + \bar{b}_2u^{2m-2} + \ldots + \bar{b}_{2m}))^2} \right),$$   
where $\bar{a}_i = \inn{b_i}{u^i} \in \calO_\Sigma (U)$. Thus, it is clear that we can extend $\sigma$ to an involution 
$$\tilde{\sigma} : X \to X,$$
which locally is just the map $u \mapsto -u$. From the local description, it is also clear that the categorical quotient $X/\tilde{\sigma}$ is a geometric quotient and it can be identified with 
$$\bar{X} \coloneqq \zeros (p(\eta)^2) \subset |K^2|,$$
which is a ribbon on $\bar{S}$ with ideal sheaf $\bar{\pi}^*K^{-2m}$ and genus 
$$g_{_{\bar{X}}} - 1 = 2m(4m-1)(g-1).$$
Moreover, we have a finite morphism $\bar{\pi} : \bar{X} \to \Sigma$ of degree $2m$ and 
$$\bar{\pi}^{-1}(U) = \Spec \left(\frac{\calO_\Sigma (U)[v]}{((v^{m} + \bar{b}_2v^{m-1} + \ldots + \bar{b}_{2m}))^2} \right),$$ 
where $v = u^2$ is a local nowhere vanishing section of $K^{-2}$. We have 
\[\xymatrix@M=0.13in{
S \ar[d]^{\rho} \ar@{^{(}->}[r]^{j} & X \ar[d]^{\tilde{\rho}}  \\
\bar{S} \ar@{^{(}->}[r]^{\bar{j}} & \bar{X}}\]
where $\tilde{\rho} : X \to \bar{X}$ is the finite morphism\footnote{Denote $\calO_\Sigma (U)[v]/((v^{m} + \bar{b}_2v^{m-1} + \ldots + \bar{b}_{2m}))^2$ and $\calO_\Sigma (U)[u]/((u^{2m} + \bar{b}_2u^{2m-2} + \ldots + \bar{b}_{2m}))^2 $ by $A$ and $B$, respectively. Then, $B$ is finitely generated as an $A$-algebra as $B \cong A \oplus A\cdot u$. Also, $B$ is integral over $A$ (the element $u \in B$ is the zero of the monic polynomial $T^2 - v \in A[T]$). This proves that $\tilde{\rho}$ is a finite morphism.} of degree $2$ locally given by the ring homomorphism 
$$\frac{\calO_\Sigma (U)[v]}{((v^{m} + \bar{b}_2v^{m-1} + \ldots + \bar{b}_{2m}))^2} \to \frac{\calO_\Sigma (U)[u]}{((u^{2m} + \bar{b}_2u^{2m-2} + \ldots + \bar{b}_{2m}))^2} $$ 
defined by sending $v$ to $u^2$. Moreover, just as in the non-singular case (see, e.g., \cite{hitchin1987stable}), the involution $\tilde{\sigma}$ acts on $\Pic^0 (X)$ and the Prym variety $\PRib$ may be defined as 
$$\PRib = \{ \calL \in \Pic^0 (X) \ | \ \tilde{\sigma}^*\calL \cong \calL^{-1} \}.$$

Let us now focus on the symplectic case. Let $\bar{D}$ be an effective divisor on $\bar{S}$ of degree $\bar{d}$ and consider the divisor $D = \rho^*\bar{D}$ on $S$, which, as explained in Section \ref{Sp(m,m)}, is the natural way to embed $Z_d (Sp(4m,\C))$ into $Z_d (SL(4m,\C))$. Let $X^\prime = \Bl_D X$ and $\bar{X}^\prime = \Bl_{\bar{D}} \bar{X}$ be the corresponding blow-ups, which are, in particular, ribbons on $S$ and $\bar{S}$, respectively. Note that $\deg (D) = 2 \deg (\bar{D}) = 2 \bar{d}$. For $1 \leqslant d \leqslant g_{_S} - 1$ we obtain the fibration
$$\calA_d \to \bar{S}^{(\bar{d})}.$$

\begin{prop}\label{ribbonsp} Let $1 \leq d \leq g_{_S}-1$ and consider the locus $A_d$ inside the fibre $h^{-1}(p^2)$ of the $Sp(4m,\C)$-Hitchin fibration (as defined in Section \ref{Sp(m,m)}). Then,
$$\calA_d \to \bar{S}^{(\bar{d})}$$ 
is a fibration, whose fibre at a divisor $\bar{D}$ is modeled on the Prym variety $\PBl$. In particular, choosing a square root of the canonical bundle of $\Sigma$, the locus $\calA_{g_{_S}-1}$ can be identified with $\PRib \subset \PX$. 
\end{prop}
\begin{proof}
Let us start with $\calA_{g_{_S}-1}$, whose elements are invertible sheaves on $X$. So, given $\calL \in \calA_{g_{_S}-1}$, it defines the $Sp(4m,\C)$-Higgs bundle $(V,\Phi) = \pi_* (\calL , \lambda)$, where $\lambda \in H^0(|K|, \pi^*K)$ is the tautological section. These fit together in the exact sequence
\begin{equation}
0 \to \calL \otimes \pi^*K^{1-4m} \to \pi^* (V) \xrightarrow{\pi^*\Phi - \lambda \Id} \pi^* (V \otimes K) \to \calL \otimes \pi^*K \to 0.
\label{seilaex}
\end{equation}
Now, applying $\tilde{\sigma}^*$ we obtain 
$$0 \to \tilde{\sigma}^* (\calL) \otimes \pi^*K^{1-4m} \to \pi^* (V) \xrightarrow{\pi^*\Phi + \lambda \Id} \pi^* (V \otimes K) \to \tilde{\sigma}^*(\calL ) \otimes \pi^*K \to 0$$
since $\tilde{\sigma}^* (\lambda) = -\lambda$. Since $\calL$ is an invertible sheaf, we can dualize the exact sequence (\ref{seilaex}) and identify $V $ and its dual using the symplectic form, which comes from the pairing provided by the relative duality. The Higgs field is skew-symmetric with respect to the symplectic form on $V$, so we conclude that 
$$\calL^{-1} \cong \calL \otimes \pi^*K^{1-4m}.$$
In particular, choosing a square root $K^{1/2}$ of $K$, $\calL \otimes \pi^*K^{(1-4m)/2}$ is a point in $\PRib$, and this identifies $\calA_{g_{_S}-1}$ with $\PRib$. Note also that, under this identification, the map which takes the line bundle $\calL$ to its restriction $j^*\calL$ to the reduced curve $S$ has the same content as the natural projection 
$$\PRib \to \PS,$$
which is well-defined since $j\circ \sigma = \tilde{\sigma} \circ j$. For the other cases, let $\calE_i$, $i=1,2$, be two elements of the fibre at $\bar{D} \in \bar{S}^{(\bar{d})}$. Consider the corresponding invertible sheaves $\calL_i^\prime $ in the blow-up $X^\prime = \Bl_DX$. Then we may proceed exactly in the same way as before and we find that $\calL_1^\prime (\calL_2^\prime)^{-1}$ is in the Prym variety $\PBl$.     
\end{proof}
In particular, checking dimensions we obtain:
\begin{align*}
\dim \bar{S}^{(\bar{d})} + \dim \PBl & = \bar{d} + g_{_{X^\prime}} - g_{\bar{X}^\prime} \\
& =  \bar{d} + (g_{_X} - 2\bar{d}) - (g_{_{\bar{X}}} - \bar{d})\\
& = g_{_X} - g_{_{\bar{X}}} \\
&= (8m^2+2m)(g-1) \\
&= \dim Sp(4m,\C)(g-1).
\end{align*} 

The $SO(4m,\C)$ case is almost identical. Let $\bar{D}$ be an effective divisor on $\bar{S}$ of degree $\bar{d}$ and take the corresponding divisor $D = \rho^*\bar{D} + R_\rho$ on $S$ (see Section \ref{sectiononso} for more details). In particular, $\deg (D) = 2 \bar{d} + 4m(g-1)$ and the genus of the ribbon $X^\prime = \Bl_D X$ is $g_{_X} - 2\bar{d} - 4m(g-1)$, so that 
$$\dim \bar{S}^{(\bar{d})} + \dim \PBl = \dim SO(4m,\C).$$
\begin{prop}\label{ribbonso} Let $1 \leqslant d \leqslant 2(g_{_{\bar{S}}}-1)$ and consider the locus $A_d$ inside the fibre $h^{-1}(p^2)$ of the $SO(4m,\C)$-Hitchin fibration (as defined in Section \ref{sectiononso}). Then,
$$\calA_d \to \bar{S}^{(\bar{d})}$$ 
is a fibration, whose fibre at a divisor $\bar{D}$ is modeled on the Prym variety $\PBl$. In particular, choosing a square root of the canonical bundle of $\Sigma$, the locus $\calA_{g_{_S}-1}$ can be identified with $\PRib \subset \PX$. 
\end{prop}

The proof of the proposition above is completely analogous to the proof of Proposition \ref{ribbonsp}. Note that in this case the ramification divisor $R_\rho$ of $\rho : S \to \bar{S}$ is incorporated into the picture. In particular, $R_\rho$ prevents the fibre $h^{-1}_{SO(4m,\C)}(p^2)$ from intersecting $\calA_{g_{_S}-1}(SL(4m,\C)) \cong \PX$. 

%% file: Chapters/Chapter6.tex

\chapter{Mirror symmetry} 

\label{Chapter6} 

\lhead{Chapter 6. \emph{Mirror symmetry}} 


One of the most remarkable mathematical predictions of string theory is the phenomenon of mirror symmetry which relates the symplectic geometry (A-\textit{model}) of a Calabi-Yau manifold $X$ to the complex geometry (B-\textit{model}) of its mirror Calabi-Yau manifold $\check{X}$. Among the major mathematical approaches to mirror symmetry are those of Kontsevich \cite{hms}, known as \textit{homological mirror symmetry}, and the proposal of Strominger, Yau and Zaslow \cite{syz}, hereinafter referred to as \textit{SYZ mirror symmetry}. 

In string theory there are natural geometrical objects associated to Dirichlet boundary conditions, called D-\textit{branes}. Kontsevich's formulation relates D-branes on the A-model (A-\textit{branes}) of $X$ to D-branes on the B-model of $\check{X}$ (B-\textit{branes}). Prototypical examples of A-branes are Lagrangian submanifolds equipped with a flat connection, whereas B-branes are holomorphic vector bundles on a complex submanifold (or more generally, coherent sheaves supported on complex subvarieties). The duality is then formulated as an equivalence (of triangulated categories) between the derived Fukaya category on $X$ and the derived category of coherent sheaves on $\check{X}$.  

SYZ mirror symmetry, on the other hand, based on physical reasonings, provides a more geometrical interpretation of mirror symmetry. Loosely stated, it states that if a Calabi-Yau $n$-fold $X$ has a mirror, then there exists a fibration whose general fibres are special $n$-Lagrangian tori and the mirror is obtained by dualizing these tori. Thus, the so-called \textit{SYZ mirror pair} ($X$, $\check{X}$) should admit dual special Lagrangian fibrations (over the same base) and the duality should be realized as a Fourier-type transformation. In \cite{lectslf}, Hitchin incorporated the notion of gerbes (known by physicists as $B$-fields) for describing important features of the SYZ conjecture. 

\section{Mirror symmetry for Higgs bundles}

Let $G$ be a complex reductive Lie group with Lie algebra $\mathfrak{g}$. The Higgs bundle moduli space $\calM (G)$ can be constructed as a hyperk\"{a}hler quotient, thus inheriting a hyperk\"{a}hler structure, with complex structures $\textsf{I}$, $\textsf{J}$ and $\textsf{K}=\textsf{IJ}$. For us, the complex structure $\textsf{I}$ is the one obtained from $\Sigma$. In other words, it corresponds to $G$-Higgs bundles $(P,\Phi)$ (i.e., pairs consisting of a holomorphic principal $G$-bundle $P$ on $\Sigma$ and a holomorphic global section $\Phi$ of the adjoint bundle of $P$ twisted by $K$). The complex structure $\textsf{J}$, on the other hand, refers to flat $G$-bundles. Also, $\calM (G)$ comes equipped with an algebraically integrable system through the Hitchin fibration 
\[ h : \calM (G) \to \calA (G). \]
By fixing the topological type of the underlying principal bundle (i.e., an element of $\pi_1 (G)$), the generic fibres of $h$ are connected complex Lagrangian submanifolds (with respect to $\Omega_\textsf{I} = \omega_\textsf{J} + i\omega_\textsf{K}$), torsors for an abelian variety. Moreover, applying a hyperk\"{a}hler rotation (using complex structure $\textsf{J}$ instead of $\textsf{I}$), the generic fibres become special Lagrangian submanifolds (for more detail see \cite{lagran}). The Hitchin base for $G$ can be identified with the one for its Langlands dual group $\lan{G}$
\[\xymatrix@M=0.13in{
\calM (G)  \ar[d] & \calM (\lan{G})  \ar[d] \\
\calA (G) \ar[r]_{\cong}^{\Xi} & \calA (\lan{G}).}\]
The isomorphism $\Xi$ takes discriminant locus to discriminant locus and the smooth fibres are dual abelian varieties, thus the moduli space of Higgs bundles for Langlands dual groups are examples of SYZ pairs. This was first proven for the special linear group by Hausel and Thaddeus \cite{hausel2003mirror}, who also showed, by calculating Hodge theoretical invariants, that topological mirror symmetry holds for these pairs (see also \cite{MR2927840, gothen2017topological} for an extension of these results for parabolic Higgs bundles). The case of $G_2$ was treated in \cite{hitchin2007langlands} and, in greater generality, Donagi and Pantev \cite{donagi2012langlands} extended the picture for semisimple groups. In particular, the fundamental group $\pi_1 (\lan{G})$ can be identified with $Z(G)^\vee = \Hom (Z(G), \C^\times)$. Under the SYZ duality, a choice of connected component of $\calM (\lan{G})$ corresponds to a $\C^\times$-gerbe on $\calM (G)$ (which is induced from the universal $G/Z(G)$-bundle on $\calM (G)$ after a choice of element in $\pi_1 (\lan{G})\cong Z (G)^\vee$).  

Since the moduli space of Higgs bundles for a complex group has a hyperk\"{a}hler structure one can speak of branes of type A or B with respect to any of the structures. For example, a BAA-brane on a hyperk\"{a}hler space $\calM$ (with complex structures $\textsf{I}, \textsf{J}$ and $\textsf{K}$) is a pair consisting of a complex Lagrangian with respect to the first complex structure $\textsf{I}$ and a flat bundle on it. A BBB-brane on $\calM$, on the other hand, consists of a hyperk\"{a}hler subvariety $Y \subset \calM$ together with a hyperholomorphic bundle on it (i.e., principal bundle equipped with a connection whose curvature is of type $(1,1)$ with respect to $\textsf{I}, \textsf{J}$ and $\textsf{K}$\footnote{Recall that on a complex manifold, if the $(0,2)$ part of the curvature vanishes we can endow the bundle with a holomorphic structure. Thus, we are requiring that the principal bundle has holomorphic structures compatible with $\textsf{I}, \textsf{J}$ and $\textsf{K}$.}. These objects generalize anti-self-dual connections in (real) dimension $4$.

In \cite{kapustin2006electric}, mirror symmetry for the Higgs bundle moduli space is derived by comparing $\calN = 4$ super Yang-Mills theory for the maximal compact subgroups $K \subseteq G$ and $\lan{K} \subseteq \lan{G}$ on a four-manifold $M$, which is a product of $\Sigma$ with another Riemann surface $C$, together with electric-magnetic duality. The theory in $M = \Sigma \times C$ reduces to a two-dimensional sigma model on $C$ with target space $\calM (G)$ and geometric Langlands duality can be understood in terms of electric-magnetic duality in $M$. In two-dimensions this corresponds to mirror symmetry or S-duality between the B-model in $\calM_{dR}(\lan{G^c})$ (i.e., in complex structure $\textsf{I}, \textsf{J}$ and $\textsf{K}$) and the A-model in symplectic structure $\omega_\textsf{K}$ for the gauge group $K$. Moreover, mirror symmetry in complex structure $\textsf{I}$ should be Fourier-Mukai transform \cite{kapustin2006electric, gukov}, so that a brane on $\calM (G)$ corresponds to a brane of the same type in $\calM (\lan{G})$. Thus, according to the physics literature, mirror symmetry suggests that there is a correspondence between BAA-branes on $\calM (G)$ and BBB-branes on $\calM (\lan{G})$ (this particular correspondence has been addressed by Hitchin in \cite{hitchin2013higgs} and more recently in \cite{franco2017mirror, franco2017borel}).   

A smooth fibre of the Hitchin fibration is a BAA-brane and its mirror is simply the skyscraper sheaf of a distinguished point of the corresponding fibre for the Langlands dual group \cite{kapustin2006electric, gukov}. A natural way to obtain more general branes is by considering fixed point sets of involutions (see e.g. \cite{baraglia2014higgs, baraglia2013real, heller2016branes, biswas2014anti, garcia2016involutions}). A particular case of this arises when one considers a real form $G_0$ of $G$. As we will see next, $G_0$-Higgs bundles inside $\calM (G)$ provide an example of a BAA-brane. It is natural to seek the mirror of such geometrical object. From the discussion above, the mirror should be a BBB-brane on $\calM (\lan{G})$. Such branes are interesting in themselves as they are special mathematical objects. For the real groups treated here this is even more intriguing, since the complex Lagrangians never intersect smooth fibres of the Hitchin fibration and general Fourier-type transforms are not available. 

\section{Real forms and branes}

Let $G_0$ be a non-compact real form of the complex semisimple Lie group $G$ given by the fixed point set of an anti-holomorphic involution $\sigma : G \to G$. One can always find another anti-holomorphic involution $\tau$, corresponding to a compact real form of $G$ that commutes with $\sigma$ (see e.g. \cite{cartan, onishchik2004lectures}). The composition of the two $\theta = \tau \circ \sigma$, is a holomorphic involution on $G$ and it induces an involution, which we also denote by $\theta$, on the moduli space of $G$-Higgs bundles 
\begin{align*}
\theta : \mathcal{M} (G) & \to \mathcal{M} (G) \\
(P, \Phi) & \mapsto (\theta (P), - \theta (\Phi)).
\end{align*}
Note that the involution $\theta$ on the Lie algebra of $G$ gives a decomposition into $\pm 1$ eigenspaces 
$$\lie{g} = \lie{h}^\C \oplus \lie{m}^\C,$$
where $\lie{h}$ is the Lie algebra of a maximal compact subgroup of $G_0$. Thus the fixed point set $\calM (G)^\theta$ corresponds to $G_0$-Higgs bundles in $\calM (G)$.   
 
From the expressions for the hyperk\"{a}hler metric and complex structures (see Section \ref{hkstructure}), it is clear that $\theta$ is an isometry with $\theta \circ \textsf{I}= \textsf{I} \circ \theta$ and $\theta \circ \textsf{J}= - \textsf{J} \circ \theta$. Now, let $(\beta_i, \varphi_i) \in T_{(P,\Phi)}\calM (G)^\theta = (T_{(P,\Phi)}\calM (G))^\theta$, $i=1,2$. Equivalently, $\theta (\beta_i) = \beta_i$ and $\theta (\varphi_i) = - \varphi_i$, $i=1,2$. Then, 
\begin{align*}
\omega_\textsf{J}((\beta_1, \varphi_1), (\beta_2, \varphi_2)) & = g(J(\beta_1, \varphi_1), (\beta_2, \varphi_2))\\
& = g((-i \tau (\varphi_1), i \tau (\beta_1)), (\beta_2, \varphi_2)) \\
& = g( (-i \theta (\tau (\varphi_1)), - i \theta (\tau (\beta_1))), (\theta (\beta_2), - \theta (\varphi_2)))\\
& = - \omega_\textsf{J}((\beta_1, \varphi_1), (\beta_2, \varphi_2)).
\end{align*} 
This shows that $\calM (G)^\theta$ is a complex subvariety with respect to complex structure $\textsf{I}$ and a Lagrangian with respect to $\omega_\textsf{J}$ (as the dimension is clearly half the dimension of $\calM (G)$). Thus, this is also a Lagrangian with respect to $\omega_\textsf{K}$ and it provides an example of a BAA-brane on $\calM (G)$. 

\begin{rmk} In general, this fixed point set need not be isomorphic to the moduli space of Higgs bundles for the real group. In terms of representations of the fundamental group of $\Sigma$, for example, $\calM (G)^\theta$ is homeomorphic to the space of reductive\footnote{We say that a representation $\rho$ is \textbf{reductive} if $\Ad \circ \rho$ is completely reducible} representations $\rho : \pi_1(\Sigma) \to G_0$ modulo conjugation by elements of $g\in G$ such that $g \rho g^{-1} : \pi_1(\Sigma) \to G_0$. In other words, we quotient by the normalizer $N_G (G_0)$ of $G_0$ in $G$. As explained in \cite{garcia2016involutions}, the map  
$$\calM (G_0) \to \calM (G)^\theta \subset \calM (G) $$
is a well-defined $n_\theta$-to-$1$ map, where $n_\theta$ is the order of the finite group $N_G (G_0)/G_0$. Also, if $(P,\Phi)$ is a semi-stable (respectively, polystable) $G$-Higgs which reduces to a Higgs bundle for $G_0$, then the corresponding $G_0$-Higgs bundle is (respectively, polystable). By abuse of notation, however, we will refer to the fixed point set $\calM (G_0)^\theta$ as $\calM (G_0) \subset \calM (G) $.       
\end{rmk}

\section{The Nadler group}

Given a connected reductive complex algebraic group $G$, a geometrical way to realize the Langlands dual group $\lan{G}$ is by studying perverse sheaves on the loop Grassmannian $Gr$ of $G$. A certain category of perverse sheaves on $Gr$ is equivalent to the category $Rep(\lan{G})$ of finite-dimensional representations of $\lan{G}$, and from this one can recover the Langlands dual group (see e.g. \cite{geolan}). In \cite{nadler2005perverse}, Nadler extends these ideas for a real form $G_0$ of $G$. It turns out that the natural finite-dimensional stratification of the real loop Grassmannian $Gr_0$ is real even-codimensional and some of the results for perverse sheaves on complex spaces still hold on $Gr_0$. Nadler introduces a certain subcategory of perverse sheaves on $Gr_0$ and shows that it is equivalent to $Rep(\lan{G}_0)$, for some complex subgroup $\lan{G}_0$ of $\lan{G}$, which we call the \textbf{Nadler group} associated to the real form $G_0 \subset G$. Note that the moduli space of $\lan{G}_0$-Higgs bundles sits inside the moduli space of Higgs bundles for the Langlands dual group $\lan{G}$ as a hyperk\"{a}hler subvariety. It is conjectured (see e.g. \cite{baraglia2013real}) that the mirror brane to the moduli space of $G_0$-Higgs bundles is supported on $\calM (\lan{G}_0) \subseteq \calM (\lan{G})$.  

Let us say a few words about why the conjecture is physically reasonable and expected. As explained in \cite[Section 6]{bdrcond}, the construction of the Nadler group can be interpreted from a gauge-theoretical point of view, and thus be linked to mirror symmetry. If $K$ is a maximal compact subgroup of $G$, one studies $K$ gauge theory on a four-dimensional manifold with boundary, where the boundary condition is determined by the involution associated to the real form $G_0$. It turns out that 't Hooft operators supported on the boundary are related to representations of a different compact group. The complexification of this group is the Nadler group. From the physics perspective, BAA-branes constructed from a real form form a special class of boundary conditions and the dual brane is the $S$-dual boundary condition. For real forms of the special linear group the $S$-dual branes are described in \cite[Table 3]{bdrcond} and are supported on the moduli space of Higgs bundles for the Nadler group. Some remarks about the relation of the Nadler group for other real forms and the $S$-dual brane associated the real form are made in Section 7.4 in loc. cit..             

Beyond the support of the BBB-brane, Hitchin also described in \cite{hitchin2013higgs} the hyperholomorphic bundle on the dual brane associated to the real form $U(m,m) \subset GL(2m,\C)$, which arose from the Dirac-Higgs bundle\footnote{For more details on the Dirac-Higgs operator and the Dirac-Higgs bundle see \cite{blaavand2015dirac}.}. Hitchin's description also matches the predictions from physics, as remarked in \cite{gaiotto2016s}. Let us now give some examples that provide a heuristic idea of what duality should give and compare this to the conjecture that the dual brane should be supported on the moduli space of Higgs bundles for the Nadler group.

\begin{ex} Take a generic (smooth) fibre $A$ of the $G$-Hitchin fibration and suppose it intersects $\calM (G_0)$ on an abelian variety $B$. This is the case, for example, when the real form $G_0$ is \textit{quasi-split} (see \cite{peon2013higgs}). The dual fibre in $\calM (\lan{G})$ is the dual abelian variety $\hat{A} = \Pic^0(A)$ and duality should correspond to dualizing the inclusion. In other words, we have a short exact sequence of abelian varieties
$$0 \to B \to A \to C \to 0,$$
and dualizing it 
$$0 \to \hat{C} \to \hat{A} \to \hat{B} \to 0$$
one obtains a distinguished abelian subvariety of the dual fibre which parametrizes line bundles of degree zero on $A$ which are trivial on $B$.

%
\end{ex}

Now, given any subvariety $B = \calM (G_0) \cap A$ (not necessarily an abelian variety) of a smooth fibre $A$ of $\calM_{G^c}$, we still have a distinguished subvariety on the dual fibre $\hat{A}$, namely the annihilator $B^0$ consisting of line bundles on $A$ of zero degree which are trivial on $B$.  

\begin{ex} If $G_0$ is a split real form (e.g. $SL(n,\R) \subset SL(n,\C)$), $\calM (G_0)$ intersects a smooth fibre $A$ of the $G$-Hitchin fibration in the 2-torsion points $A[2]$ of $A$ \cite{schaposnik2013spectral}. Following the idea above we expect the dual brane to intersect the fibre $\hat{A}$ on the whole $\hat{A}$, so this brane should be supported on $\calM (\lan{G})$. This is compatible with the conjecture as in this case $\lan{G_0} = \lan{G}$. On the other extreme, take $G_0$ to be a compact real form of $G$. Higgs bundles for $G_0$ are simply holomorphic principal $G$-bundles, and their moduli space is an algebraic component of the nilpotent cone $\text{Nilp}(G)$ (i.e., the zero section of the $G$-Hitchin fibration). The dual brane should intersect the corresponding nilpotent cone on a single point (in the irreducible component corresponding to principal $\lan{G}$-bundles) which is compatible with $\lan{G_0}$ being trivial.  
\end{ex}

\section{Proposal for the mirror}

We consider the real forms $G_0 = SU^*(2m)$, $SO^*(4m)$ and $Sp(m,m)$ of $G= SL(2m,\C)$, $SO(4m,\C)$ and $Sp(4m,\C)$, respectively. For these cases, $\calM (G_0)$ never intersects the smooth fibres of the Hitchin fibration $h : \calM (G) \to \calA (G)$. We will see, however, that although the singular fibres $F$ of $h$ which have non-empty intersection with $\calM (G_0)$ are complicated and have many irreducible components, they basically depend only on the locus corresponding to $G_0$-Higgs bundles and a certain abelian variety $A(G_0)$ inside the Jacobian of the reduced scheme of the spectral curve. The abelian variety in question appears naturally when considering what is morally the determinant line bundle of the rank $1$ sheaves on the spectral curve, which is non-reduced. This determinant-type map gives a surjective morphism from the fibre (which restricts to a surjective map in each of its algebraic components) to $A(G_0)$. Moreover, just as in the generic case, the fibre of the surjective map at the trivial line bundle is isomorphic to the locus of $G_0$-Higgs bundles inside $F$ (or an affine version of that for the other components which do not contain Higgs bundles for $G_0$). Motivated by the heuristic idea of duality explained in the last section, we are led to consider the dual abelian variety of $A(G_0)$. It turns out that for the three cases considered, $\widehat{A(G_0)}$ corresponds precisely to the spectral data for $\lan{G_0}$-Higgs bundles, as predicted by the conjecture.

\subsection{The $SU^*(2m)$ case}\label{mirrorforsu}

As described in Chapter \ref{Chapter3}, the locus of rank $1$ torsion-free sheaves on the spectral curve $X = \zeros (p(\lambda)^2) \subset |K|$ supported on the reduced curve $S = X_{\text{red}}$ (assumed to be non-singular) is given by the rank $2$ semi-stable vector bundles $E$ on $S$ satisfying $\det (\pi_*E)\cong \calO_\Sigma$, where $\pi : S \to \Sigma$ is the ramified $m$-covering. We called this locus $N$ and inside it we have the locus $N_0 = \calU_S (2, \pi^*K^{m-1})$ corresponding to $SU^*(2m)$-Higgs bundles in $h^{-1}(p^2)$. There is a natural map from the algebraic component $N$ of $h^{-1}(p^2)$ to an abelian variety, namely the Prym variety $\Prym$. The map is given by  
\begin{align*}
n : N & \to \Pic^0 (S)\\
E & \mapsto \det(E) \pi^*K^{1-m}.
\end{align*} 
\begin{lemma} The map $n : N \to \Pic^0 (S)$ is a surjection onto the Prym variety $\Prym$. Also, $n^{-1}(\calO_S) = N_0$. 
\end{lemma}
\begin{proof}
Firstly, 
\begin{align*}
\det (\pi_* (\det (E)\otimes \pi^*K^{1-m})) = & \det (\pi_* (\det (E))\otimes K^{1-m})) \\
= & \det (\pi_* (\det (E))) \otimes K^{m(1-m)}.
\end{align*}
Using Lemma \ref{nmlm} and the fact that $\det (\pi_* E)$ is trivial, $\det (\pi_* (\det (E))) \otimes K^{m(1-m)} = K^{m(1-m)/2}$, thus $n (E) \in \Prym$. Now, given an element $L \in \Prym$, since the Prym variety is connected, there exists $U\in \Prym$ such that $U^2 \cong L$. Then, $n (E_0\otimes U) = L$, where $E_0$ is any element in $N_0$. Indeed, $E_0\otimes U \in N$ since, by Lemma \ref{nmlm}, $\det (\pi_* (E_0\otimes U)) = \det (\pi_*E_0)\otimes \Nm (L) = \calO_\Sigma$. Moreover, $\det (E_0\otimes U) = \pi^*K^{m-1}\otimes L$. The last assertion is obvious.   
\end{proof}

Before we proceed we remark on the naturality of the map above for our purposes. Given $E \in N$, we can always write it as $E = E_0 \otimes L$, where $E_0 \in N_0$ and $L$ is a line bundle satisfying $L^2 = \det (E)\pi^*K^{1-m} \in \Prym$. This is clearly not unique as we can always tensor $L$ with a 2-torsion point of $\Pic^0 (S)$. Thus, we have a natural map $$N \to F/\Pic^0(S)[2]$$ sending $E = E_0 \otimes L$ to the equivalence class $[L]$ of $L$ in $F/\Pic^0(S)[2]$, where $F$ is the locus of line bundles on $S$ whose squares are in $\Prym$. This map is surjective as, given $[L] \in F/\Pic^0(S)[2]$, we can take any element $E_0 \in N_0$ and $E_0 \otimes L$ maps to $[L]$ (note that $E_0 \otimes L \in N$ as a direct consequence of Lemma \ref{nmlm}). Also, $\Pic^0(S)[2]$ acts on $N_0$ and the fibre of $N \to F/\Pic^0(S)[2]$ over $[\calO_S]$ is precisely $N_0$.   
Consider the map  
\begin{align}
F  & \to  \Prym. \label{wrong}\\
L  & \mapsto  L^2. \nonumber
\end{align}
Connectedness of the Prym variety implies surjectivity of this map (as $\Prym \subseteq F$) and we have 
\begin{equation}
F/\Pic^0(S)[2] \cong \Prym.
\label{F}
\end{equation}   
Under this isomorphism, (\ref{wrong}) is precisely our map $n$. Note also that choosing a square root $K^{1/2}$ of the canonical bundle of $\Sigma$ we can twist the elements of $N$ by $\pi^*K^{(1-m)/2}$ thus obtaining rank $2$ bundles on $S$ of degree $0$. Under this identification, 
$$N \cong \{ E_0 \in \calU_S (2,0) \ | \ \det (E_0) \in \Prym \}$$
and $N_0 \cong \calU_S (2, \calO_S)$. The map $n$ in this case is precisely the determinant map.

Now, given $\delta \in A_d$, let $(L,D)$ be its projection to $Z_d$. We then have a natural map 
\begin{align*}
n :A_d & \to \Prym  \\
\delta & \mapsto L^2(D)\pi^*K^{1-2m}.
\end{align*}

This is clearly surjective. Moreover note that the fibres of $A_d \to Z_d$ are mapped to the same element. Since $A_d$ is an affine bundle, the difference between two elements in the fibre of $(L_2, D)$ will give a unique element in $H^1(S, L_2^*L_1)$, where $L_1 = L_2(D)K_S^{-1}$. In other words, it gives an extension of bundles 
$$0 \to L_1 \to E \to L_2 \to 0$$
and $\det E \cong L_1L_2 = \pi^*K^{m-1}$, if this fibre is mapped to $\calO_S$. This is the analogue of $N_0$ in $A_d$.

Now, we need to check that the map is well defined as a map 
$$n : h^{-1}(p^2) \to \Prym .$$
For this, note that we only need to check how the Zariski closure $\bar{A}_d$ of $A_d$ intersects $N$. This intersection consists of strictly semi-stable Higgs bundle. In other words, let $\delta \in \bar{A}_d$. There is a natural projection 
$$\delta \mapsto (L,D) \in \Pic^{m(m-1)(g-1)}(S) \times S^{(2m^2(g-1))}.$$ 
This element $\delta$ is $Gr$-equivalent to the object  
$$L(D)K_S^{-1} \oplus L$$
in the strictly semi-stable locus of $N$. In both cases, the maps send the point in the intersection to $L^2(D)\pi^*K^{1-2m}$.

Note that the dual abelian variety of $\Prym$ is $\Prym/ \Pic^0(\Sigma)[m]$, which is the spectral data for $PGL(m, \mathbb{C})$-Higgs bundles. In particular, this corresponds to the Nadler group $\lan{SU^*(2m)}$.

The moduli space $\calM (PGL(2m,\C))$ of $PGL(2m,\C)$-Higgs bundles is an orbifold which can be seen as the quotient of $\calM (SL(2m,\C))$ by the $2m$-torsion points of the Jacobian of $\Sigma$ (see, e.g., \cite{hauglob} for more details). Thus, by Theorem \ref{fibreforsu*}, the corresponding fibre on the fibration for $\calM (PGL(2m,\C))$ is just 
$$h^{-1}(p^2)/\Pic^0(\Sigma)[2m] \cong N/ \Pic^0(\Sigma)[2m] \cup \bigcup_{d=1}^{g_S-1} A_d/\Pic^0(\Sigma)[2m],$$ 
where $\Pic^0(\Sigma)[2m]$ acts naturally (via pullback) in each stratum. Since the surjective map $n$ to $\Prym$ involves the square of a line bundle on $S$ and $\widehat{\Prym } = \Pic^0(\Prym) \cong \Prym / \Pic^0(\Sigma)[m]$, it is natural to see the injection of $\Prym / \Pic^0(\Sigma)[m]$ into the component $N/ \Pic^0(\Sigma)[2m]$ as the map sending $[L]$ to $[L\pi^*K^{\frac{m-1}{2}} \oplus L\pi^*K^{\frac{m-1}{2}}]$, which reflects the diagonal embedding of $PGL(m,\C)$ into $PGL(2m,\C)$ (for the computation of the Nadler group, see Appendix \ref{comoutationnad}).

\subsection{The other cases}

We now focus on the cases when $G_0 = SO^*(4m)$ or $Sp(m,m)$. Again we start with the locus $N(G)$ containing the locus $N_0 (G)$ consisting of $G_0$-Higgs bundles in the fibre $h_G^{-1}(p^2)$ (for more details see Chapter \ref{Chapter4}). Then, if we denote by $N_{\pm}$ the locus  
$$
N_{\pm} = \left \{ \ E \in \calU_S (2,e) \
\bigg |
\gathered
\begin{array}{cl}
 & \text{there exists an isomorphism}\\
 & \psi : \sigma^*E \to E^* \otimes \pi^*K^{2m-1} \\
 & \text{satisfying $\tp{(\sigma^*\psi)} = \pm \psi$}.\
\end{array}
\endgathered  \ \  \right
\},$$    
where $e= 4m(2m-1)(g-1)$, then $N(SO(4m,\C)) \cong N_{-}$ and $N(Sp(4m,\C)) \cong N_{+}$. Also, $N_0 (G) \subset N(G)$ is determined by the condition that the determinant of the rank $2$ bundle on $S$ is $\pi^*K^{2m-1}$ (in both cases). In particular, this means that the induced action on the determinant bundle $\pi^*K^{2m-1}$ is trivial when $G_0 = SO^*(4m)$ and $-1$ in the other case. If $E \in N(G)$, from the isomorphism $\sigma^*E \cong E^* \otimes \pi^*K^{2m-1}$ we obtain that 
$$\det (E)\otimes \pi^*K^{1-2m} \in \PS,$$
since the Prym variety $\PS$ is characterized by those line bundles $L$ on $S$ which are isomorphic to $\sigma^* (L^{-1})$. Since $\PS$ is connected, the squaring map on this abelian variety is surjective and we can always find $U \in \PS$ such that $U^2 \cong \det (E)\otimes \pi^*K^{1-2m}$. Now, if we write $E = E_0 \otimes U$, the determinant of $E_0$ is isomorphic to $\pi^*K^{2m-1}$. This means that $E_0 \in N_0(G)$ since $U\in \PS$ (and then we obtain an isomorphism $\sigma^*E_0 \cong E_0^*\otimes \pi^*K^{2m-1}$ of the same type as the one it is induced from). If we had chosen another $U^\prime \in \PS$ satisfying $(U^\prime)^2 \cong \det (E)\otimes \pi^*K^{1-2m}$, then $L = U^\prime U^{-1} \in \PS [2]$. But if $z$ is a fixed point of $\sigma$, the linear action of $\sigma$ on the fibres $U_z$ and $U^\prime_z$ is always the same. This means that at a fixed point, the linear action of $\sigma$ on $L$ is always trivial and thus it must come from the quotient curve $\bar{S}$ (i.e., $L\in \rho^*\Pic^0 (\bar{S})[2]$). Thus we have a well-defined surjective map 
$$h_G^{-1}(p^2) \to \PS/\rho^*\Jac (\bar{S})[2]$$   
whose fibre at zero is precisely $N_0 (G)$. The map is analogous to the previous case. Explicitly, 
\begin{align*}
N (G) & \to \PS/\rho^*\Jac (\bar{S})[2] \\
E & \mapsto [\det(E) \pi^*K^{1-2m}].
\end{align*}
Now, given $\delta \in A_d(G)$, let $(L,\bar{D})$ be its projection to $Z_d(G)$, we then define  
\begin{align*}
A_d & \to \PS/\rho^*\Jac (\bar{S})[2] \\
\delta & \mapsto [L\sigma^*L^{-1}].
\end{align*}
Note that we see $Z_d (G) \subset Z_d(SL(4m,\C))$ by mapping $(L,\bar{D})$ to $(L, R_\rho + \rho^* \bar{D})$ (respectively, $(L, \rho^*\bar{D})$) when $G = SO(4m,\C)$ (res. $Sp(4m,\C)$), where $R_\rho$ is the ramification divisor of the ramified $2$-covering $\rho : S \to \bar{S}$. Then, $L^2(D)\pi^*K^{1-4m} = L\sigma^*L^{-1}$ in both cases and this coincides with our map $n$ defined previously.

The dual abelian variety of $\PS/\rho^*\Jac (\bar{S})[2]$ is $\PS$, which is the spectral data for $Sp(2m, \mathbb{C})$-Higgs bundles (see \cite{hitchin1987stable}). In particular, this corresponds to the Nadler group $\lan{SO^*(4m)} = \lan{Sp(m,m)}$.

%% file: Chapters/Chapter7.tex

\chapter{Symplectic Representations} 

\label{Chapter7} 

\lhead{Chapter 7. \emph{Symplectic Representations}} 


\section{General setup}

Let $G$ be a connected complex semi-simple Lie group, with Lie algebra $\lie{g}$, acting linearly on a finite-dimensional complex vector space $\V$. Assume that $\V$ is equipped with a $G$-invariant non-degenerate skew-symmetric bilinear form $\omega$. Then, $\V$ is called a \textbf{symplectic $G$-module} and the corresponding homomorphism from $G$ to $Sp(\V , \omega )$ is a \textbf{symplectic representation} of $G$. Given a symplectic representation 
$$\rho : G \to Sp(\V , \omega )$$
of $G$, the associated moment map $\mu_0 : \V  \to \lie{g}^*$ is the function  
$$\inn{\mu_0 (v)}{\xi}  = \dfrac{1}{2} \omega (\rho (\xi) v, v ), $$
where $v \in \V$, $\xi \in \lie{g}$. In the expression above, the brackets denote the natural pairing between $\lie{g}$ and $\lie{g}^*$ and, by abuse of notation, we also denote by $\rho : \lie{g} \to \lie{sp} (\V, \omega)$ the Lie algebra homomorphism given by the derivative of the representation. Using a non-degenerate $\Ad$-invariant bilinear form $B$ on $\lie{g}$ to identify $\lie{g}$ and $\lie{g}^*$ (take for example the Killing form $B_\lie{g}$), we obtain the quadratic equivariant map $\mu : \V \to \lie{g} $.

\begin{rmk} Using the non-degenerate form $\omega$ we can identify $\End \V $ with $\V\otimes \V$, so that, given $u_1,u_2 \in \V$, $u_1\otimes u_2$ corresponds to the endomorphism $\omega (u_1, \cdot)\otimes u_2$. Under this isomorphism we have $\lie{sp}(\V , \omega) \cong \Sym^2(\V)$. 
\end{rmk}

Fix a symplectic representation $\rho : G \to Sp(\V, \omega)$, where $\V$ has dimension $2n$. Let $P$ be a holomorphic principal $G$-bundle on $\Sigma$ and denote by $V = P(\V)$ the holomorphic vector bundle (of rank $2n$) associated to $P$ via the representation $\rho$. A choice of square root $K^{1/2}$ of the canonical bundle $K$ is the same as a spin structure on $\Sigma$ and the Dirac operator 
$$\bar{\partial} : \Omega^0(\Sigma, K^{1/2}) \to \Omega^{0,1}(\Sigma, K^{1/2})$$ 
is the $\bar{\partial}$-operator (see \cite{harm}). Since $V$ is a holomorphic vector bundle we can form the coupled Dirac operator $\bar{\partial}_A : \Omega^0(\Sigma, V\otimes K^{1/2}) \to \Omega^{0,1}(\Sigma, V\otimes K^{1/2})$. Note that, from the fact that $\mu$ is a quadratic equivariant function, it induces a map 
$$\mu : V\otimes K^{1/2} \to \Ad(P) \otimes K.$$
In particular, its evaluation on any spinor field $\psi \in H^0(\Sigma , V\otimes K^{1/2})$ is a well-defined Higgs field. Note that by the Riemann-Roch theorem, $h^0 (\Sigma , V\otimes K^{1/2}) = h^{1}(\Sigma , V\otimes K^{1/2})$ and for a generic symplectic vector bundle $V$ there may not be any non-zero global holomorphic sections of $V\otimes K^{1/2}$, thus the existence of a non-zero spinor field gives a constraint on the vector bundle. 
\begin{defin}\label{definitionofX} Choose a square root $K^{1/2}$ of the canonical bundle of $\Sigma$ and let $\rho : G \to Sp(\V, \omega)$ be a symplectic representation. We define $X_{(\rho , K^{1/2})} \subset \calM (G)$, or simply $X$, as the Zariski closure of the locus parametrizing 
\begin{center}
$\{ (P,\Phi) \in \calM^s (G) \ | \ \Phi = \mu(\psi) \ \text{for some } 0 \neq \psi \in H^0(\Sigma , V\otimes K^{1/2}) \}.$
\end{center}    
\end{defin}    
\begin{rmk} The $S$-equivalence class of a stable $G$-Higgs bundle corresponds to its isomorphism class. Let $(P,\Phi)$ and $(P^\prime,\Phi^\prime)$ be stable $G$-Higgs bundles corresponding to the same point in $\calM (G)$ and let $\eta : (P, \Phi) \to (P^\prime , \Phi^\prime )$ be an isomorphism between the two (i.e., $\eta$ is an isomorphism between the underlying principal $G$-bundles whose induced isomorphism $\hat{\eta} : \Ad(P) \otimes K \to \Ad(P^\prime) \otimes K$ sends $\Phi$ to $\Phi^\prime$). If $\Phi = \mu (\psi)$ for some spinor field $\psi$, its image under $\hat{\eta}$ is of the form $\mu (\psi^\prime)$, for some $\psi^\prime \in H^0(\Sigma, V^\prime \otimes K^{1/2})$. To see this, choose an open covering $\calU = \{U_\alpha \}$ of $\Sigma$ trivializing both $P$ and $P^\prime$. The isomorphism $\eta$ corresponds to a $0$-cocycle $(\eta_\alpha)$, where $\eta_\alpha : U_\alpha \to G$ is holomorphic. Thus, the induced isomorphisms from $V \to V^\prime$ and $\Ad (P) \to \Ad (P^\prime)$ are given by the $0$-cocycles $(\rho (\eta_\alpha))$ and $(\Ad (\eta_\alpha))$, respectively. If $\psi$ is represented by the cocycle $(\psi_\alpha)$, the section $\hat{\eta}(\Phi)$ is represented by $(\Ad(\eta_\alpha) \mu(\psi_\alpha))$, which, by the equivariance of the moment map is equal to $(\mu(\rho (\eta_\alpha)\psi_\alpha))$, where $(\rho(\eta_\alpha)\psi_\alpha)$ is a $0$-cocycle representing $\psi^\prime$.
\end{rmk} 

\begin{ex}\label{ex1} Take $(\V , \omega)$ to be the symplectic direct sum of the symplectic $G$-modules $(\V_i , \omega_i)$, $i=1,\ldots , u$, whose moment maps we denote by $\mu_i : \V_i \to \lie{g}$. The moment map for $\V$ is 
$$\mu (v_1 , \ldots , v_u) = \mu_1 (v_1) + \ldots + \mu_u (v_u),$$
where $v_i \in \V_i$.
\end{ex}

\begin{ex}\label{type1} Let $\rho$ be the standard representation of $Sp(2n,\C) = Sp(\C^{2n}, J)$, where $J = \left( \begin{matrix} 0& 1\\ -1&0 \end{matrix} \right)$ is the standard skew-symmetric form on $\C^{2n}$. Taking the invariant form on $\lie{sp}(2n,\C)$ to be $B(\xi_1, \xi_2) = -\dfrac{1}{2} \tr (\xi_1 \xi_2)$, where $\xi_1,  \xi_2 \in \lie{sp}(2n,\C)$, we have $\mu (v) = v\otimes v \in \Sym^2 (\C^{2n}) \cong \lie{sp}(2n, \C)$. This follows directly from the obvious identities $(u_1\otimes u_2) \circ (v_1\otimes v_2) = \omega (u_1,v_2) v_1\otimes u_2$ and $\tr (v_1\otimes v_2) = \omega (v_1,v_2)$. Note that the associated Higgs fields $\Phi $ are of the form $\psi \otimes \psi$, where $\psi \in H^0(\Sigma , V\otimes K^{1/2})$ for some symplectic vector bundle $(V,\omega)$ of rank $2n$ on $\Sigma$. Then, $\Phi (v) = \omega (\psi , v)\psi$ which shows that such Higgs fields have rank $1$ and satisfy $\Phi^2 = 0$. In particular, the variety $X$ is contained in the nilpotent cone of the Hitchin fibration for the complex symplectic group. We will discuss this further in the next chapter.    
\end{ex}

\begin{ex}\label{type2} Let $\tau : G \to GL(\W )$ be any finite-dimensional representation of $G$. In particular, we may assume that $\tau : G \to SL(\W )$ since the group $G$ is connected semisimple (and thus does not have any non-trivial characters). We also denote by $\tau : \lie{g} \to  \lie{sl}(\W)$ the Lie algebra homomorphism, which is injective since $\lie{g}$ is semisimple. The vector space $\V = \W \oplus \W^*$ carries a natural symplectic form 
$$\omega ((u_1,\delta_1), (u_2,\delta_2)) = \inn{\delta_2}{u_1} - \inn{\delta_1}{u_2},$$
where $\inn{{}\cdot{}}{{}\cdot{}}$ is the natural pairing between $\W^*$ and $\W$. Then, the representation $\rho = \tau \oplus \tau^* : G \to GL (\V) $ is symplectic with moment map satisfying 
$$B (\mu (u,\delta),  {}\cdot{}) = \delta (\tau (\cdot) u),$$
where $u\in \W$ and $\delta \in \W^*$.

In particular, if $\W$ is the standard $SL(n,\C)$-module and $W$ is a vector bundle on $\Sigma$ of rank $n$ with trivial determinant line bundle, a spinor is a pair $\psi = (u, \delta)$, where $u \in H^0(\Sigma , W\otimes K^{1/2})$ and $\delta \in H^0(\Sigma , W^* \otimes K^{1/2})$. Note that, since the degree of $W$ is zero, it follows from Riemann-Roch theorem that $h^0(\Sigma , W^* \otimes K^{1/2}) = h^1 (\Sigma , W^* \otimes K^{1/2})$ and Serre duality gives $h^0(\Sigma , W \otimes K^{1/2}) = h^1(\Sigma , W^* \otimes K^{1/2})$. Taking $B$ to be the trace of the product between two elements of $\lie{sl}(2n,\C)$, we obtain  
$$\mu (\psi) = \delta \otimes u - \dfrac{\inn{\delta}{u}}{n}\Id \in H^0(\Sigma , \End_0 W \otimes K). $$ 
This example (combined with Example \ref{ex1}) is closely related to the $n=2$ case treated by Hitchin in \cite{hitchin2009higgs}. 
\end{ex}

\begin{ex} One can blend the two previous examples by considering the representation $\rho = \tau \oplus \tau^*$ where $\tau : G \to Sp(\V , \omega)$ is itself a symplectic representation. Note that the map 
\begin{align*}
\V & \to \V^*\\
v & \mapsto \omega (\cdot , v) 
\end{align*} 
is an isomorphism between the $G$-module $\V$ and its dual. The last example gives, for the representation $\rho = \tau \oplus \tau $, a moment map satisfying 
$$B (\mu (v_1, v_2),  {}\cdot{}) = \omega (\tau (\cdot) v_1, v_2),$$
where $v_1, v_2 \in \V$. A simple calculation shows that the moment map for $\tau$, which we will also denote by $\mu$, satisfies 
$$B (\mu (v_1 + v_2) - \mu(v_1) - \mu(v_2),  {}\cdot{}) = \omega (\tau (\cdot) v_1, v_2).$$
Thus, by the non-degeneracy of $B$, the moment map for $\rho$ is the bilinear form associated to the quadratic moment map for the representation $\tau$. 

In particular, if $\V = \C^{2n}$ is the standard representation of $Sp(2n,\C)$ we may identify $\mu (v_1,v_2) = v_1 \odot v_2 \in \Sym^2(\C^{2n})$, where $v_1 \odot v_2 = v_1\otimes v_2 + v_2\otimes v_1$. Elements of $X$ are symplectic Higgs bundles of the form $(V, \psi_1 \odot \psi_2)$, where $V$ is a symplectic vector bundle of rank $2n$ and $\psi_1, \psi_2 \in H^0(\Sigma , V \otimes K^{1/2})$. This is closely related to higher-rank Brill-Noether theory\footnote{See, e.g., \cite{brill2} for an overview of the main results of Brill-Noether theory on a curve.}, which has to do, for example, with understanding the geometry of the subvariety of the moduli space of vector bundles consisting of those bundles admitting at least $k$ independent sections. Note also that if the vector bundle $V \otimes K^{1/2}$ has more than $1$ global section, $X$ intersects fibres of the Hitchin fibration other than the nilpotent cone. Take, for example, the locus $X_2 \subset X$ consisting of Higgs bundles whose underlying vector bundle $V$ is such that $h^0 (V\otimes K^{1/2}) = 2$. Fix a basis $\{ \psi_1 , \psi_2 \}$ for $H^0 (\Sigma, V\otimes K^{1/2})$. Then, $\det (x - \psi_1\odot \psi_2) = x^{2n-2}(x - \omega (\psi_1,\psi_2))(x + \omega (\psi_1,\psi_2))$.   
\end{ex}

\begin{ex} Hitchin studied in \cite{spinors} the variety $X$ associated to higher rank irreducible symplectic representations of $SL(2,\C)$, with emphasis on the cases of genus $2$ and $3$. These are Lagrangians in $\calM (SL(2,\C))$ and are quite interesting as they are related to the Brill-Noether problem for semi-stable rank $2$ bundles on a curve and also intersect smooth fibres of the $SL(2,\C)$-Hitchin fibration in the $m$-torsion points of these Prym varieties.     
\end{ex}

\begin{rmk} As we will see later in the chapter, the locus $X$ associated to a symplectic representation of $G$ will, under certain conditions, define a Lagrangian subvariety of $ \calM (G)$. These complex Lagrangians are named after the physicist Gaiotto, who introduced many examples of these objects (BAA-branes) on the Higgs bundle moduli space and sketched in some cases physical arguments related to the dual BBB-brane on $\calM (\lan{G})$ (see \cite{gaiotto2016s}). Given a symplectic representation $\rho$ of $G$, he called the associated Lagrangian a \textit{boundary hyper-multiplet valued in $\rho$}. Some cases treated in \textit{loc. cit.} include Example \ref{type2} (\textit{full hyper-multiplets}) and \textit{matter interfaces}, i.e., Lagrangians of $\calM (G)\times \calM (G)$ defined as 
$$
X = \left \{ \ ((V_1,\Phi_1),(V_2,\Phi_2)) \in \calM (G)\times \calM (G) \
\bigg |
\gathered
\begin{array}{cl}
 & V_1 = V_2=V\text{ and }\Phi_2 - \Phi_1 = \mu (\psi)\\
 & \text{for some }\psi \in H^0(\Sigma , V\otimes K^{1/2})\
\end{array}
\endgathered  \ \  \right
\},$$
In both cases the analysis of the hyperholomorphic brane is related to the Dirac-Higgs bundle.         
\end{rmk}

A symplectic representation is called \textbf{indecomposable} if it is not isomorphic to the sum of two non-trivial symplectic representations. It follows from \cite[Theorem 2.1]{knop1} that if two symplectic representations are isomorphic as $G$-modules, then they are isomorphic as symplectic representations. Moreover, every symplectic representation is a direct sum of finitely many indecomposable symplectic representations, where the summands are unique up to permutation. Every indecomposable symplectic representation is either an irreducible $G$-module or of the form $\V = \W \oplus \W^*$ (see Example \ref{type2}), where $\W$ is an irreducible $G$-module which does not admit a symplectic structure (and the symplectic form on $\V$ is, up to a multiple, the natural one). 

The subvariety $X$ is always isotropic and under some conditions it is Lagrangian (see Section \ref{defotheory}). To obtain information about $X$ we introduce in the next section the moduli space of twisted pairs.

\section{The moduli space of $K^{1/2}$-twisted symplectic pairs} \label{moduli}

Let $L$ be a holomorphic line bundle on $\Sigma$. The notion of an $L$-twisted pair on $\Sigma$ was given in \cite{hk} and it fits into the more general framework of twisted affine $\rho$-bumps, introduced by Schmitt (see \cite{schmitt}), where $\rho : G \to GL(\V)$ is a finite-dimensional representation of $G$. Such an object consists of a pair $(P,\varphi)$, where $P$ is a holomorphic principal $G$-bundle on $\Sigma$ and $\varphi$ is a global section of the vector bundle associated to $P$ via $\rho$ twisted by the line bundle $L$. In particular, by considering $\rho$ to be the adjoint representation of $G$ and $L$ to be the canonical bundle of $\Sigma$, we recover the definition of a $G$-Higgs bundle on $\Sigma$. Fix a square-root $K^{1/2}$ of the canonical bundle $K$. We will mostly be interested in the case where $G$ is a connected semisimple Lie group, $L=K^{1/2}$ and $\V$ is a symplectic $G$-module.    

\begin{defin} Let $\rho : G \to Sp (\V, \omega)$ be a symplectic representation of $G$. A \textbf{$K^{1/2}$-twisted symplectic $\rho$-bump on $\Sigma$}, or simply a \textbf{$(\rho, K^{1/2})$-pair}, is a pair $(P,\psi)$ consisting of a holomorphic principal $G$-bundle and a spinor $\psi \in H^0(\Sigma, P(\V)\otimes K^{1/2})$, where $P(\V)$ is the symplectic vector bundle associated to $P$ via $\rho$. Two $(\rho , K^{1/2})$-pairs $(P_i,\psi_i)$, $i =1,2$, are \textbf{isomorphic} if there is an isomorphism $f : P_1 \to P_2$ such that the induced isomorphism $P(\V_1) \otimes K^{1/2} \to P(\V_2) \otimes K^{1/2}$ sends $\psi_1$ to $\psi_2$. 
\end{defin} 

Since $G$ is connected\footnote{Recall that the category of complex semisimple algebraic groups is equivalent to the category of complex semisimple Lie groups by considering $\C$-points. Moreover, we usually refer to a $G$-bundle as either a holomorphic $G$-principal bundle on the compact Riemann surface or the corresponding object in the algebraic category.}, there is a bijection between isomorphism classes of topological principal $G$-bundles on $\Sigma$ and the fundamental group $\pi_1 (G)$ of $G$ (see, e.g., \cite{ramana}). Thus, fix the topological type $d \in \pi_1 (G)$ of the underlying topological (or equivalently, $C^\infty$) $G$-bundle. Let us now discuss the notion of stability for $(\rho, K^{1/2})$-pairs. 
\begin{rmk} If no confusion is possible, we will, for convenience, refer to a $(\rho, K^{1/2})$-pair simply as a pair. 
\end{rmk}
Given an anti-dominant character $\chi :Q \to \C^\times$ of a parabolic subgroup $Q$ of $G$, with Lie algebra $\lie{q}$, and a reduction $\sigma \in H^0(\Sigma, P/Q)$ of the structure group of the principal $G$-bundle $P$ to $Q$ (for details, see Appendix \ref{Lie theory}), the degree $\deg (P)(\sigma , \chi)$ of $P$ with respect to $\sigma$ and $\chi$ is defined as the degree of the line bundle $P_\sigma \times_{\chi} \C$, where $P_\sigma = \sigma^*P$ is the $Q$-bundle on $\Sigma$ corresponding to the restriction of structure group of $P$ to $Q$. Let $\lie{l} $ be the Lie algebra of the Levi subgroup $L$ corresponding to $Q $. Then, characters of $\lie{q}$ are in bijection with elements of the dual $\lie{z}^*_\lie{l}$ of the center of $\lie{l}$ which in turn may be identified with an element of $\lie{z}_\lie{l}$ via the Killing form. We denote by $s_\chi \in \lie{z}_\lie{l}$ the element corresponding to the derivative of the character $\chi$ through this identification. If $\lie{h}$ is the Lie algebra of a maximal compact subgroup $H$ of $G$, one can check that $\lie{z}_{\lie{l}} \subset i \lie{h}$, so from an anti-dominant character $\chi$ of the parabolic $Q$ we obtain an element $s_\chi \in i \lie{h}$ (for more details see Section $2$ of \cite{mundet}). Using this we can define the complex subspace 
$$\V_\chi^- = \{ v\in \V \ | \ \rho (e^{ts_\chi})v \ \text{is bounded as $t \to \infty$}\} $$ 
of $\V$ which is invariant under $Q$. Note that, putting $V = P(\V)$, the symplectic vector bundle $V$ is naturally isomorphic to $P_\sigma \times_Q \V$ and $V_{\sigma , \chi}^- = P_\sigma \times_Q \V^-_\chi$ is a subbundle of $V$. Similarly, one can define the complex subspace 
$$\V_\chi^0 = \{ v\in \V \ | \ \rho (e^{ts_\chi})v = v \ \text{for any $t\in \R$} \} \subset \V_\chi^-$$ 
of $\V$ which is invariant under the action of $L$ and thus gives rise to a vector subbundle $V_{\sigma_L , \chi}^0 = P_{\sigma_L} \times_L \V_\chi^0 $ of $V_{\sigma,  \chi}^-$, where $\sigma_L \in H^0(\Sigma , P/L)$ is a further reduction of $P$ to $L$. 

\begin{defin} A pair $(P,\psi)$ on $\Sigma$ is 
\begin{itemize}
\item \textbf{semi-stable} if for all parabolic subgroups $Q\subset G$, all reductions $\sigma \in H^0(\Sigma , P/Q)$, and all non-trivial anti-dominant characters $\chi$ of $Q$ such that $\psi$ is a global holomorphic section of $ V_{\sigma , \chi}^-\otimes K^{1/2}$, we have
$$\deg (P)(\sigma , \chi) \geq 0.$$
\item \textbf{stable} if the inequality above is strict (for all non-trivial parabolic subgroups $Q$ of $G$, all reductions $\sigma$ and all non-trivial anti-dominant characters $\chi$ of $Q$).
\item \textbf{polystable} if it is semi-stable and for any parabolic subgroup $Q \subset G$, any reduction $\sigma$ and any non-trivial strictly anti-dominant character $\chi$ of $Q$, such that $\psi \in H^0(\Sigma , V_{\sigma , \chi}^-\otimes K^{1/2})$ and 
$$\deg (P)(\sigma , \chi) =  0, $$
there is a further reduction $\sigma_L \in H^0(\Sigma , P/L)$ to the Levi $L\subset Q$ such that $\psi \in H^0(\Sigma , V_{\sigma_L , \chi}^0\otimes K^{1/2}) \subset H^0(\Sigma , V_{\sigma , \chi}^-\otimes K^{1/2})$.
\end{itemize}
\end{defin}

\begin{rmks}

\noindent 1. The notions of stability above can be generalized in the obvious way to any $L$-twisted pair. Moreover, if one allows the group $G$ to be a reductive Lie group, stability depends on a parameter $\alpha $ in the center of the Lie algebra $\lie{g}$ (see \cite{hk}), or, as in \cite{schmitt}, on a rational character of the reductive algebraic linear group. When $G$ is semisimple, there is only one such $\alpha$ as a semisimple Lie algebra is centreless (and a connected semisimple algebraic group does not have any non-trivial characters).  

\noindent 2. When the spinor is identically zero, stability for the pair coincides with the usual stability conditions for principal bundles on $\Sigma$ due to Ramanathan \cite{ramana} (for the proof see \cite{MR2450609} or \cite[Theorem 2.4.9.3]{schmitt}). 
\end{rmks}

Using Geometric Invariant Theory, Schmitt \cite{schmitt} constructed the moduli scheme $\calS^d (\rho , K^{1/2})$ whose closed points correspond to the $S$-equivalence classes of semi-stable $(\rho, K^{1/2})$-pairs of type $d \in \pi_1 (G)$.

\begin{prop} The moduli space $\calS^d (\rho , K^{1/2})$ is a quasi-projective variety.
\end{prop} 
\begin{proof}
This is a special case of \cite[Theorem 2.8.1.2]{schmitt}. Let $\V = \V_1 \oplus \ldots \oplus \V_u$ be the decomposition of the $G$-module $\V$ into its irreducible components. In Schmitt's notation, $\calS^d (\rho , K^{1/2})$ is denoted by $\calM^{\chi-\text{ss}}(\rho, d, \underline{L})$, where $\chi = 0$ and $\underline{L} = (L_1, \ldots , L_u)$, with $L_i = K^{1/2}$, for $i= 1, \ldots , u$.  
\end{proof} 

In \cite{hk} a Hitchin-Kobayashi correspondence is established, where polystable pairs correspond to solutions of a certain gauge-theoretic equation. We will go back to this point of view in the next chapter, where we consider the standard representation of the complex symplectic group.

\subsection{The map between the moduli spaces}

A ($K$-)twisted affine $\Ad_G$-bump is the same as a $G$-Higgs bundle and the usual notion of semi-stability (see, e.g., \cite{dey}) coincides with the one for bumps (see Remark $2.8.2.5$ of \cite{schmitt}), which in turn is equivalent to the one given in \cite{hk}. One has the following relation\footnote{We thank Oscar Garc\'{i}a-Prada for pointing that out to us.} between stability of a pair and stability of the corresponding Higgs bundle.    
%

\begin{prop}\label{stability} Let $P$ be a holomorphic principal $G$-bundle on $\Sigma$ and $\psi$ a global section of $P(\V)\otimes K^{1/2}$. If $(P,\mu(\psi))$ is a (semi-)stable Higgs bundle, the pair $(P,\psi)$ is (semi-)stable. 
\end{prop}
\begin{proof}
Let $Q\subset G$ be a parabolic subgroup, $\chi$ an anti-dominant character of $Q$ and $\sigma \in H^0(\Sigma , P/Q)$ a reduction. If $\mu (\psi) \in H^0(\Sigma , \Ad (P)_{\chi , \xi }^- \otimes K)$, where $ \Ad (P)_{\chi , \xi }^-$ is the vector bundle associated to $P_\sigma$ via the representation $\Ad : Q \to GL( \lie{g}_{\chi , \xi }^-)$, then $\psi \in H^0(\Sigma , \V_{\chi , \xi }^- \otimes K^{1/2})$. We can see this by showing that if $v \in \V_{\chi , \xi }^-$, we have $\mu(v) \in \lie{g}_{\chi , \xi }^-$. But this follows directly from the $G$-equivariance of the moment map. Indeed, 
$$\Ad (e^{ts_\chi})\mu (v) = \mu (\rho (e^{ts_\chi})v)$$
and $\underset{t \to \infty}{\lim} \Ad (e^{ts_\chi})\mu (v) = \mu (\underset{t \to \infty}{\lim}\rho (e^{ts_\chi})v )$, which is bounded. Now, by (semi-)stability of the Higgs bundle we have that $\deg (P)(\sigma , \chi) \ (\geq) > 0$.
\end{proof}

Note that the obstruction for the stability of the pair and the associated Higgs bundle to be \textit{equivalent} comes from the fact that, given an element $s \in i\lie{h}$ and $v \in \V$ such that $\mu (\rho (e^{ts})v)$ is bounded as $t $ goes to infinity, it does not follow in general that $\rho (e^{ts})v$ is also bounded. An obvious example comes from considering any symplectic vector space with the trivial $G$-module structure. In this case the moment map is zero and the image of the moduli space of pairs inside $\calM(G)$ is just the moduli space of principal $G$-bundles on $\Sigma$. As in the Higgs bundle case, the underlying bundle of a semi-stable pair $(P,\psi)$ need not be semi-stable. Indeed, the pair $(K^{-1/2}\oplus K^{1/2}, 1)$, where $1 \in H^0(\Sigma , \calO_\Sigma)\subset H^0(\Sigma , (K^{-1/2}\oplus K^{1/2})\otimes K^{1/2})$, is semi-stable but $K^{-1/2}\oplus K^{1/2}$ is not. Considering the standard representation of $Sp(2n,\C)$ however, $(e^{ts}v)\otimes (e^{ts}v)$ is bounded if and only if $e^{ts}v$ is bounded (as $t$ goes to infinity). Thus we have shown the following. 

\begin{prop}\label{p78} Let $P$ be a holomorphic principal $Sp(2n,\C)$-bundle on $\Sigma$ and $\psi$ be a global section of $P(\C^{2n})\otimes K^{1/2}$, where $\C^{2n}$ is the standard $Sp(2n,\C)$-module. Then, the pair $(P, \psi )$ is (semi-)stable if and only if $(P,\psi \otimes \psi)$ is a (semi-)stable symplectic Higgs bundle. 
\end{prop}

\section{Hitchin-type map}
Let $\V$ be a symplectic $G$-module and $\V = \V_1 \oplus \ldots \oplus \V_u$ its decomposition into irreducible components. Just as in the case of Higgs bundles, one may consider a basis of the invariant ring $\C [ \V ]^G = \Sym (\V^*)^G $ and construct a map from the moduli space $\calS^d (\rho , K^{1/2})$ of pairs onto a vector space. We have a splitting of $G$-modules
$$\Sym^e (\V^*) =  \bigoplus_{e_1 + \ldots + e_u = e} \Sym^{e_1} (\V_1^*)\otimes \ldots \otimes \Sym^{e_u} (\V_u^*),$$
so we may write 
$$ \Sym^e (\V^*)^G = \bigoplus_{e_1 + \ldots + e_u = e} S(e_1, \ldots , e_u),$$
where 
$$S(e_1, \ldots , e_u) = (\Sym^{e_1} (\V_1^*)\otimes \ldots \otimes \Sym^{e_u} (\V_u^*))^G.$$
Since $\C [ \V ]^G$ is a finitely generated $\C$-algebra for any reductive complex group $G$, we can find generators $f_1, \ldots , f_q$ of the invariant ring such that $f_i \in S(e_{i,1}, \ldots , e_{i,u})$, for some $e_{i,j}\in \Z_{\geq 0}$, $i = 1, \ldots , q$, $j= 1, \ldots , u$. Evaluation of $f_i$ on a spinor $\psi \in H^0(\Sigma , V \otimes K^{1/2})$ induces a \textbf{Hitchin-type map} 
\begin{equation}
h : \calS^d (\rho , K^{1/2}) \to \calA_{pair} = \bigoplus_{i=1}^q H^0(\Sigma , K^{(e_{i,1} + \ldots + e_{u,i})/2}).
\label{hitype}
\end{equation}  

\begin{prop}(\cite[Prop. 2.8.1.4]{schmitt}) The map (\ref{hitype}) is a projective morphism.
\end{prop}

\begin{cor}\label{cor711} The moduli space of pairs $\calS^d (\rho , K^{1/2})$ for the standard representation of the complex symplectic group is a projective variety. 
\end{cor}
\begin{proof}
The standard representation $\V = \C^{2n}$ of $G = Sp(2n,\C)$ is transitive on $\C^{2n} - \{0 \}$ since, given two non-zero vectors, we can always complete them to two symplectic basis. Thus, $\C [ \V ]^G = \C$ and the result follows from the proposition above. 
\end{proof}

Note that since the moment map 
$$\mu : \V \to \lie{g}$$
is equivariant, every element $p\in \C[\lie{g}]^G$ gives rise to an element $\mu^*p \in \C[\V]^G$.

\begin{ex} Consider the Example \ref{type2}, where $\V = \W \oplus \W^*$ and $\W$ is the standard representation of $SL(n,\C)$. There is an obvious invariant function $p$ on $\V$ given by the pairing, i.e., 
$$p(u,\delta) = \inn{\delta}{u}.$$
The function $p$ is clearly invariant for any $G$-module $\W$, where $G$ is a complex group. Thus, in particular, the moduli space of pairs for the representation $\V$ is not projective. As we saw, the moment map is given by 
\begin{align*}
\mu : \V &\to \lie{sl}(n,\C)\\
(u , \delta ) &\mapsto \delta \otimes u - \frac{\inn{\delta}{u}}{n}\Id  .
\end{align*}
The functions $p_i (x) = \tr (x^i)$ form a basis for the ring $\C [\lie{sl}(n,\C)]^{SL(n,\C)}$ and their pullback by $\mu $ is given by 
\begin{align*}
\mu^*p_i (u,\delta) &= \tr\left(\delta \otimes u - \frac{\inn{\delta}{u}}{n}\Id\right)^i \\
& = \tr\left( \sum_{j=0}^{i} (-1)^j \binom{i}{j} (\delta \otimes u)^{i-j} (\frac{\inn{\delta}{u}}{n}\Id)^j\right).
\end{align*}  
But $(\delta \otimes u)^{i-j} = \inn{\delta}{u}^{i-j-1}\delta \otimes u$, for $i-j > 0$, so we obtain 
\begin{align*}
\mu^*p_i (u,\delta) &= \inn{\delta}{u}^{i} \left(\sum_{j=0}^{i-1} (-1)^j \binom{i}{j} \frac{1}{n^j} + (-1)^{i}\frac{1}{n^{i-1}} \right)\\
&= \inn{\delta}{u}^{i} \left((1 - n^{-1})^i + (-1)^{i}(n-1)n^{-i} \ \right).
\end{align*} 
In particular, all these functions are contained in the ideal generated by $p \in \C [\V]^{SL(n,\C)}$. It turns out that the ring of invariants of $\V$ is generated by $\{1, p\}$ (see Tables 1 and 2 in \citep{knop1} for the dimension of $\C [\V]^G$ for a large class of symplectic $G$-modules $\V$).
\end{ex}

\section{The isotropy condition}
\label{defotheory}

Let $(P,\psi)$ be a $(\rho, K^{1/2})$-pair, where again $\rho : G \to Sp(\V , \omega)$. The \textbf{deformation complex} of $(P,\psi)$ is the complex of sheaves $C_{pair}^\bullet (P,\psi)$ given by 
$$ \rho(\cdot)(\psi) : \calO (\Ad (P)))  \to \calO (P(\V ) \otimes K^{1/2}). $$
It follows from \cite{bis} that the space of infinitesimal deformations of a pair $(P,\psi)$ is naturally isomorphic to the first hypercohomology group $\K^1 (\Sigma , C_{pair}^\bullet (P,\psi))$. 

Also, recall that the smooth locus $\calM (G)^{smt}$ of the moduli space of $G$-Higgs bundles is a hyperk\"{a}hler manifold. The tangent space at a smooth point $(P,\Phi)$ is isomorphic to the infinitesimal deformation space of $(P,\Phi)$, which was shown in \cite{bis} to be canonically isomorphic to the first hypercohomology group $\K^1(\Sigma, C^\bullet (P,\Phi))$, where $C^\bullet (P,\Phi)$ is the two-term complex of sheaves 
$$[{}\cdot{}, \Phi] : \calO(\Ad (P)) \to \calO(\Ad (P) \otimes K).$$ 
Using an $\Ad$-invariant form $B$ on $\lie{g}$ to identify $\Ad (P)$ with its dual, the complex $C^\bullet (P,\Phi)$ becomes isomorphic to $C^\bullet (P,\Phi)^* \otimes K$ and by Serre duality for hypercohomology we obtain an isomorphism 
$$\Omega_I : \K^1(\Sigma, C^\bullet (P,\Phi)) \to \K^1(\Sigma, C^\bullet (P,\Phi))^*,$$
which gives rise to the $I$-complex symplectic structure on $\calM(G)$ (where $I$ is the complex structure of the moduli space determined by the complex structure of $\Sigma$, see Section \ref{hkstructure}). 

In \cite{spinors}, Hitchin proves that $X$ (see Definition \ref{definitionofX}) is an isotropic subvariety of $\calM(G)$ using the infinite-dimensional approach to the construction of the Higgs bundle moduli space. Let us show this from the \v{C}ech point of view. To this end we recall an explicit description of $\Omega_I$ (see \cite{bis}).  

Let $(P,\Phi)$ be a $G$-Higgs bundle and $C^\bullet = C^\bullet (P,\Phi)$ its deformation complex. The complex $C^\bullet \otimes C^\bullet $ is given by
$$\Ad (P) \otimes \Ad (P) \overset{\beta}{\to} \Ad (P)\otimes (\Ad (P)\otimes K) \oplus (\Ad (P) \otimes K)\otimes \Ad (P) \overset{\alpha}{\to} (\Ad (P) \otimes K)\otimes (\Ad (P) \otimes K),$$ 
where 
\begin{align*}
\beta (x,y) & = ( x\otimes [y,\Phi] , [x,\Phi]\otimes y), \\
\alpha(x\otimes y, y^\prime \otimes x^\prime) & = [x,\Phi]\otimes y - y^\prime \otimes [x^\prime , \Phi].
\end{align*}
Denote by $K[-1]$ the complex whose degree $1$ entry is $K$ and all others are zero. There is a map of cochain complexes $f^\bullet : C^\bullet \otimes C^\bullet \to K[-1]$, where $f_i = 0$, for $i\neq 1$, and in degree $1$ we have $$f^1 : (x\otimes y, y^\prime \otimes x^\prime) \mapsto B(x, y^\prime) - B(y, x^\prime).$$
Also, there is a natural map $\K^1(\Sigma, C^\bullet ) \otimes \K^1(\Sigma, C^\bullet ) \to \K^2(\Sigma, C^\bullet \otimes C^\bullet)$ and the composition 
$$\K^1(\Sigma, C^\bullet ) \otimes \K^1(\Sigma, C^\bullet) \to \K^2(\Sigma, C^\bullet \otimes C^\bullet) \to \K^2(\Sigma, K[-1]) = H^1(\Sigma, K) \cong \C $$
is precisely the complex symplectic form $\Omega_I$, where the isomorphism $H^1(\Sigma, K) \cong \C$ is given by the residue map. 

Fix an open cover $\calU = \{ U_i \}$ of $\Sigma$. Let $v$, $v^\prime \in \K^1(\Sigma, C^\bullet )$ be infinitesimal deformations represented by the pairs $(\{ s_{ij} \}, \{ t_i \})$ and $(\{ s_{ij}^\prime \}, \{ t_i^\prime \})$, respectively. This means that $s_{ij} \in H^0(U_{ij}, \Ad (P))$ and $t_i \in H^0(U_i, \Ad (P) \otimes K)$ satisfying 
\begin{align}
& s_{ij} + s_{jk}  = s_{ik} \ \ \text{in $U_{ijk}$}, \label{s} \\
& t_i - t_j + [s_{ij}, \Phi]  = 0 \ \ \text{in $U_{ij}$}, \label{t} 
\end{align} 
where $U_{ij} = U_i \cap U_j$ and $U_{ijk} = U_i \cap U_j \cap U_k$ (similarly for $(\{ s_{ij}^\prime \}, \{ t_i^\prime \})$). Let $\Phi = \mu (\psi)$, for some non-zero spinor $\psi \in H^0(\Sigma , V \otimes K^{1/2})$, and take $v$ and $v^\prime$ to be infinitesimal deformations of $X$ in $\K^1(\Sigma, C^\bullet )$ as before. Let $(\{ u_i \})$ and $(\{ u_i^\prime \})$, with $u_i, u_i^\prime  \in H^0(\Sigma, V\otimes K^{1/2})$, represent infinitesimal deformations of $\psi $ (associated to the infinitesimal deformations $v$ and $v^\prime$, respectively). Thus, we have the constraint 
\begin{equation}
u_i - u_j + \rho (s_{ij})\psi = 0 
\label{constraint2}
\end{equation}
in $U_{ij}$. Also, since   
$$B(\mu (\psi), {}\cdot{}) = \dfrac{1}{2} \omega (\rho (\cdot) \psi, \psi ),$$
we have, using the skew-symmetry of $\omega$ and the fact that the representation is symplectic, 
\begin{align}
B(t_i, {}\cdot{}) & = \dfrac{1}{2} (\omega (\rho (\cdot) u_i, \psi ) +  \omega (\rho (\cdot) \psi, u_i )) \label{constraint1}\\
& = \omega (\rho (\cdot) \psi, u_i )
\end{align}
in $U_i$. Analogous statements hold for $(\{ u_i^\prime \} , \{ s_{ij}^\prime \})$.

Thus, the class $\Omega_I(v,v^\prime) \in H^1(\Sigma , K)$ is represented by the following cocycle
\begin{align*}
\Omega_I( (\{ s_{ij} \} , \{ \mu (u_i)\}) , (\{ s_{ij}^\prime \} , \{ \mu (u_i^\prime ) \}) ) & = \{ B(s_{ij}, t_j^\prime) - B(s_{ij}^\prime, t_i) \} & \\
& = \{ \omega (\rho (s_{ij})\psi , u_j^\prime) - \omega (\rho (s_{ij}^\prime)\psi, u_i) \} \qquad \qquad & \text{by (\ref{constraint1}) }\\
& = \{ \omega(u_j, u_j^\prime) - \omega (u_i, u_i^\prime)\} &\text{by (\ref{constraint2})}.
\end{align*}
But $\omega (u_i, u_i^\prime) \in H^0(U_i, K)$, so $\Omega_I(v,v^\prime) = 0 \in H^1(\Sigma, K) \cong \C$.

The complex symplectic form $\Omega_I$ is exact. More precisely, consider the natural map 
$$q : \K^1 (\Sigma, C^\bullet) \to H^1(\Sigma, \Ad (P))$$ sending a cocycle $v = (\{ s_{ij} \}, \{ t_i \})$ to the infinitesimal deformation $\{ s_{ij} \}$ of the bundle $P$. Let 
$$\theta_{(P, \Phi)} (v) = \mathsf{S}\mathsf{D} (q(v)\otimes \Phi),$$
where 
$$\mathsf{S}\mathsf{D} : H^1(\Sigma, \Ad (P)) \otimes H^0(\Sigma, \Ad (P)\otimes K) \to H^1(\Sigma, K) \cong \C$$
is the non-degenerate pairing given by Serre duality. Then, Biswas and Ramanan show \cite[Section 4]{bis} that $\Omega_I = d\theta$ (see also \cite{bohr}) and $\theta$ restricted to the cotangent bundle to the moduli space of stable $G$-bundles corresponds to the tautological $1$-form. If $\Phi = \mu (\psi)$ is the Higgs field of a smooth point of $X$ and $v = (\{ s_{ij} \} , \{ \mu (u_i)\})$ a tangent vector of $X$ in $\calM^{smt}(G)$ as before, we have 
\begin{align*}
\theta_{(P,\mu (\psi))} (v) & = \{ B (s_{ij}, \mu (\psi)) \} &\\
& = \{ \dfrac{1}{2} \omega (\rho (s_{ij}) \psi, \psi )\} \qquad \qquad &  \text{by (\ref{constraint1})}\\
& = \{ \dfrac{1}{2} \omega (u_i, \psi) - \dfrac{1}{2} \omega (u_j, \psi)\} & \text{by (\ref{constraint2})}\\
& = 0 \in H^1(\Sigma , K). &
\end{align*}

Note that the spinor moduli space has a $\C^\times$-action and $X$ is invariant under the $\C^\times$-action on the Higgs bundle moduli space. Denote by $\xi \in H^0(\calM (G), T\calM (G))$ the vector field defined by the $\C^\times$-action on $\calM (G)$. Thus, $\xi$ is tangential to $X$. Since $X$ is isotropic, $i_\xi \Omega_I$ restricted to $X$ is zero. But, by a simple calculation, one obtains $i_\xi d\theta = \theta$ (see, e.g., \cite[Lemma 2.1]{bohr}). Thus, that is another way to see that the $1$-form $\theta$ vanishes when restricted to $X$. 

\begin{prop}\label{p712} The subvariety $X$ of $\calM (G)$ is isotropic. Moreover, the $1$-form $\theta$ defined above vanishes on $X$.   
\end{prop}

\section{Gaiotto's Lagrangian}

We now relate infinitesimal deformations (and smoothness) of pairs to the associated Higgs bundles and show that under certain conditions the isotropic subvariety $X$ of $ \calM(G)$ is actually Lagrangian.

\begin{lemma}\label{simple} $\Aut (P, \psi) \subseteq \Aut (P, \mu (\psi))$.
\end{lemma}
\begin{proof}
An automorphism of $(P, \psi)$ is given by a holomorphic map $f: P \to G$ satisfying $gf(p\cdot g) = f(p)g$, for all $p\in P$ and $g\in G$, and $\rho (f) (\psi) = \psi$. But, from a direct computation 
\begin{align*}
B(\Ad(f) (\mu (\psi)), \cdot)) & = B(\mu (\psi) , \Ad (f^{-1} \cdot)) \\
& = \dfrac{1}{2} \omega (d\rho (\Ad (f^{-1} \cdot))\psi , \psi) \\
& = \dfrac{1}{2} \omega (\rho (f^{-1}) d\rho (\cdot ) \rho(f)\psi , \psi ) \\
& = B(\mu (\psi), \cdot).
\end{align*}
From the non-degeneracy of the Killing form the result follows. 
\end{proof}

\begin{defin} A pair $(P,\psi)$ is \textbf{simple} if $\Aut (P,\psi) = Z (G) \cap \ker \rho $ and a $G$-Higgs bundle $(P,\Phi)$ is \textbf{simple} if $\Aut (P,\Phi) = Z (G)$, where $Z (G)$ is the center of $G$.
\end{defin}

From the lemma above we conclude that if $(P,\mu (\psi))$ is simple, $$Z(G) \cap \ker \rho \subseteq \Aut (P,\psi) \subseteq Z(G).$$
To obtain simplicity of the pair from simplicity of the associated Higgs bundle we restrict to a certain class of symplectic representations. Let    
$$\V = \bigoplus_{i=1}^u \V_i$$ 
be the decomposition of the symplectic $G$-module $\V$ into indecomposable symplectic modules. We say that $\V$ is \textbf{almost-saturated}\footnote{The nomenclature almost-saturated is inspired by the notion of a saturated representation. In our context, saturated representations are almost-saturated representations which do not admit indecomposable summands of type $\W\oplus \W^*$ (see, e.g., \cite{knop1}).} if $\V_i$ are symplectic $G$-modules pairwise non-isomorphic and if whenever $\V_j = \W \oplus \W^*$ for some $j \in \{ 1, \ldots , u \}$, then $\W \ncong \W^*$.

\begin{lemma}\label{simplic} Assume $\rho $ is almost-saturated. If $(P,\mu (\psi))$ is a simple Higgs bundle, the pair $(P,\psi)$ is simple.   
\end{lemma}
\begin{proof}
Since $(P,\mu (\psi))$ is simple, given an automorphism $f\in \Aut (P, \psi)$, we may see it as a map $f: P \to Z(G)$. Thus, using Schur's lemma in each component, $\rho (f) : V \otimes K^{1/2} \to V\otimes K^{1/2}$ restricts to a non-zero multiple of the identity at each component $V_i \otimes K^{1/2} \to V_i \otimes K^{1/2}$, where $V_i = P (\V_i)$. The induced map $\rho (f)$ however sends $\psi$ to itself and thus $\rho (f)$ is the identity map.  
\end{proof}

Equivariance of the moment map $\mu(\rho(g)\psi) = \Ad(g)(\mu(\psi))$, for all $g \in G$, gives the following commuting diagram   
\[\xymatrix@M=0.13in{
\calO (\Ad (P))) \ar[r]^-*+{d\rho(\cdot)(\psi)} \ar[d]^{\Id} & \calO (P(\V)) \otimes K^{1/2}) \ar[d]^{d\mu_\psi} \\
\calO (\Ad (P)) \ar[r]^-*+{[\cdot, \mu(\psi)]} & \calO (\Ad (P)\otimes K).}\]

From the map $C^\bullet_{pair}(P,\psi) \to C^\bullet (P,\mu (\psi))$ above we obtain a map between the two associated long exact sequences 
\[\xymatrix@M=0.05in{
H^0(\Sigma, \Ad (P))\ar[d]_{\Id} \ar[r] & H^0(\Sigma, P(\mathbb{V}) \otimes K^{1/2}) \ar[d]^{d\mu_\psi} \ar[r] &  \mathbb{H}^1(\Sigma, C_{\text{pair}}^\bullet(P,\psi)) \ar[d]_{j} \ar[r] & H^1(\Sigma, \Ad (P)) \ar[d]_{\Id} \\
H^0(\Sigma, \Ad (P)) \ar[r] & H^0(\Sigma, \Ad (P) \otimes K) \ar[r] &  \mathbb{H}^1(\Sigma, C^\bullet(P,\mu(\psi)))  \ar[r] & H^1(\Sigma, \Ad (P)). }\]
Motivated by the Petri map, which plays an important role in classical Brill-Noether theory\footnote{Classically, given a complex subspace $\W$ of the space of global sections of a holomorphic vector bundle $E$ on $\Sigma$, the Petri map $\W \otimes H^0(\Sigma , E^*) \to H^0(\Sigma , \End E \otimes K)$ is given by the natural cup-product of sections. The smoothness of the Brill-Noether scheme (respectively, the moduli space of coherent systems) at a stable point $E$ (respectively, $(E,\W)$) is equivalent to injectivity of the Petri map (see, e.g., \cite{brill, brill2}). For the standard representation of $Sp(2n,\C)$ one could take $E = V\otimes K^{1/2}$, where $V$ is a stable symplectic vector bundle, and take $\W$ to be multiples of a given section of $E$ or the whole space of global sections. The Petri map, for $\W = H^0(\Sigma , E)$, for example, is $\Sym^2H^0(\Sigma , E) \to H^0(\Sigma , \Sym^2 E)$. Note that $\lie{sp}(2n,\C) \cong \Sym^2 (\C^{2n})$ and $d\mu_\psi (\dot{\psi}) = \psi\otimes \dot{\psi} + \dot{\psi} \otimes \psi$. The analogy of the classical Petri map with the derivative of the moment map (acting on global sections) motivates calling it a Petri-type map.}, we will refer to the map
$$d\mu_\psi : H^0(\Sigma, P(\mathbb{V}) \otimes K^{1/2}) \to H^0(\Sigma, \Ad (P) \otimes K)$$
as a \textbf{Petri-type map}. Assuming that the Petri-type map above is injective, by the five lemma, we get 
\begin{equation}
j: \K^1(\Sigma , C^\bullet_{pair} (P,\psi)) \hookrightarrow \K^1(\Sigma , C^\bullet (P,\mu(\psi))).
\label{lagran}
\end{equation}  
  
One of the spectral sequences (see (\ref{les1}) for more details) gives the long exact sequence
\begin{align*}
0 &\to \mathbb{H}^0(\Sigma , C^\bullet_{pair} (P,\psi)) \to H^0(\Sigma ,\Ad (P)) \to H^0(\Sigma ,P(\mathbb{V}) \otimes K^{1/2}) \to \mathbb{H}^1(\Sigma , C^\bullet_{pair} (P,\psi)) \to \\
&\to  H^1(\Sigma , \Ad (P)) \to H^1(\Sigma , P(\mathbb{V}) \otimes K^{1/2}) \to \mathbb{H}^2(\Sigma , C^\bullet_{pair} (P,\psi)) \to 0.
\end{align*}
So, 
$$\chi(C^\bullet_{pair} (P,\psi)) - \chi(\Ad (P)) + \chi(P(\mathbb{V}) \otimes K^{1/2}) = 0$$
which implies that 
$$- \chi(C^\bullet_{pair} (P,\psi)) = \dim G (g-1).$$

\begin{prop}\label{p716} Assume $\rho$ is almost-saturated. If the Higgs bundle $(P,\mu(\psi))$ is stable and simple, then $(P,\mu(\psi))$ is a smooth point of $\calM^d (G)$ and $(P,\psi)$ is a smooth point of $\calS^d (\rho, K^{1/2})$. Furthermore, if the Petri-type map is injective for all points\footnote{Injectivity is an open condition and in particular the points where the Petri map is injective will define a Lagrangian on the Higgs bundle moduli space.} in the smooth locus of the moduli space of pairs, the subvariety $X \subseteq \calM^d (G)$ is Lagrangian.
\end{prop}
\begin{proof}
First note that $(P,\mu(\psi))$ stable implies $\mathbb{H}^0(\Sigma , C^\bullet_{pair} (P,\mu (\psi))) = 0$ (see, for example, \cite[Proposition 3.11]{hk}). Serre duality for hypercohomology gives $$\mathbb{H}^2(\Sigma , C^\bullet (P,\mu (\psi))) \cong \mathbb{H}^0(\Sigma , C^\bullet (P,\mu (\psi)))^*,$$ so, by \cite[Theorem 3.1]{bis} (see also \cite[Proposition 3.18]{hk}), $(P,\mu(\psi))$ is a smooth point of $\calM^d (G)$. Now, $(P,\mu(\psi))$ is stable, simple and $\mathbb{H}^2(\Sigma , C^\bullet (P,\mu (\psi))) = 0$, so $(P,\psi)$ is stable (\ref{stability}), simple (Lemma \ref{simplic}) and from the map between the two long exact sequences coming from the map between the two infinitesimal deformation complexes, we obtain $\mathbb{H}^2(\Sigma , C^\bullet_{pair} (P,\psi)) = 0$. Note that, in this case, $\mathbb{H}^0(\Sigma , C_{pair}^\bullet (P,\psi))$ is also zero as there are no non-trivial infinitesimal automorphisms of $(P,\psi)$ (see \cite[Proposition 2.14]{hk}) and $- \chi(C^\bullet_{pair} (P,\psi)) = \dim \mathbb{H}^1(\Sigma , C^\bullet_{pair} (P,\psi)) = \dim G (g-1)$. The last assertion follows from (\ref{lagran}). 
\end{proof}

%% file: Chapters/Chapter8.tex

\chapter{The standard representation} 

\label{Chapter8} 

\lhead{Chapter 8. \emph{The standard representation}} 


In this chapter we study the Gaiotto Lagrangian $X$, introduced in the last chapter, for the standard representation of the symplectic group. Our primary goal is to identify $X$ with a particular irreducible component of the nilpotent cone of the Hitchin fibration for the symplectic group. It turns out that $X$ behaves in a similar fashion to the moduli space of holomorphic triples on $\Sigma$ \cite{holtri}. A holomorphic triple consists of two holomorphic vector bundles and a map between them (e.g. $(V, K^{-1/2}, \psi)$, where $\psi : K^{-1/2} \to V$). The moduli space for these objects was constructed in \cite{stabletri} (see also \cite{Schmitt2004} for an algebro-geometric construction) and depends on a real parameter $\alpha$, which is bounded when the ranks of the two vector bundles are different. Take an $\alpha$-polystable triple $(V, K^{-1/2}, \psi)$, for example. The authors show (\textit{loc. cit.}) that for $\alpha$ near the lower bound, the vector bundle $V$ is polystable, whereas near the upper bound, $\psi $ is non-vanishing. As we shall see, the norm squared of the spinor, which is a Morse function on our auxiliary moduli space, will play a very similar role to the $\alpha$-parameter.

\section{Spinor moduli space}

Fix a square-root $K^{1/2}$ of the canonical bundle of $\Sigma$. When $\rho $ is the standard representation of $Sp(2n,\C)$ we denote the moduli space of $K^{1/2}$-twisted symplectic pairs $\calS (\rho, K^{1/2})$ (see Section \ref{moduli}) simply by $\calS$ and call it the \textbf{spinor moduli space}. We summarize in the theorem below the results obtained in Chapter \ref{Chapter7} for the spinor moduli space. In particular, we give a precise characterization of stability for pairs $(V,\psi) \in \calS$.   

\begin{thm}\label{resu} Let $(V,\psi)$ be a pair, consisting of a symplectic vector bundle on $\Sigma$ and a spinor $\psi \in H^0(\Sigma , V\otimes K^{1/2})$. 
\begin{enumerate}
\item The pair $(V,\psi)$ is (semi-)stable if and only if its associated $Sp(2n,\C)$-Higgs bundle $(V,\psi \otimes \psi)$ is (semi-)stable. Also, the pair is polystable if and only if the associated Higgs bundle is polystable.
\item The pair $(V,\psi)$ is (semi-)stable if and only if for all isotropic subbundles $0 \neq U  \subset V$ such that $\psi \in H^0(\Sigma , U^\perp \otimes K^{1/2})$ we have $\deg (U) < (\leq) \ 0$. Also, $(V,\psi) \in \calS$ is polystable if it is semi-stable and satisfies the following property. If $0 \neq U \subset V$ is an isotropic (resp., proper coisotropic) subbundle of degree zero such that $\psi \in H^0(\Sigma , U^\perp \otimes K^{1/2})$ (resp., $\psi \in H^0(\Sigma , U\otimes K^{1/2})$), then there exists a coisotropic (resp., proper isotropic) subbundle $U^\prime$ of $V$ such that $\psi \in H^0(\Sigma , U^\prime \otimes K^{1/2})$ (resp., $\psi \in H^0(\Sigma , (U^\prime)^\perp \otimes K^{1/2}) $) and $V = U \oplus U^\prime$. 
\item The spinor moduli space $\calS$ is a projective variety which embeds in $\calM (Sp(2n,\C))$ and the variety $X $, defined as the Zariski closure of 
$$\{ (V, \Phi) \in \calM^s (Sp(2n,\C)) \ | \ \Phi = \psi \otimes \psi \ \text{for some $0\neq \psi \in H^0(\Sigma, V\otimes K^{1/2})$} \}$$ 
inside the Higgs bundle moduli space $\calM (Sp(2n\C))$, has the property that the canonical $1$-form $\theta$ (see Section \ref{defotheory}) vanishes when restricted to $X$. Moreover, the embedded image of $\calS$ in $\calM(G)$ is the union of $X$ and the moduli space $\calU (Sp(2n,\C))$ of symplectic vector bundles on $\Sigma$, and $\calU (Sp(2n,\C))$ intersects $X$ in the generalized theta-divisor $\Theta \subset \calU (Sp(2n,\C))$.  
\end{enumerate}
\end{thm}
\begin{proof}
The first assertion is precisely Proposition \ref{p78} (polystability is a direct consequence of the fact that the moment map $\mu (v) = v\otimes v$ is injective). Note that if $U$ is an isotropic $(\psi \otimes \psi)$-invariant subbundle of $V$, for all $u\in U$ we have $$\psi \otimes \psi (u) = \omega (\psi , u) \psi \in U\otimes K,$$
so, at a point of the curve, $\psi \in U\otimes K^{1/2}$ or $\omega (\psi , u) = 0$. In both cases, $\psi \in H^0(\Sigma , U^\perp \otimes K^{1/2})$. Thus, the second item follows from the characterization of stability for $Sp(2n,\C)$-Higgs bundles \cite[Theorem 4.2]{hk}. In particular, an $Sp(2n,\C)$-Higgs bundle $(V,\Phi)$ is polystable if it is semi-stable and for all non-zero isotropic (resp., coisotropic) $\Phi$-invariant subbundles $U \subset V$ of degree $0$, there exists a $\Phi$-invariant coisotropic (resp., isotropic) subbundle $U^\prime \subset V$ such that $V = U\oplus U^\prime$. Thus, when $\Phi = \psi \otimes \psi$, if $U$ is an isotropic subbundle of $V$ of degree zero with $\psi \in H^0(\Sigma , U^\perp \otimes K^{1/2})$, there exists a coisotropic subbundle $U^\prime$ which is $(\psi \otimes \psi)$-invariant and satisfies $V = U\oplus U^\prime$. Since $U^\prime$ is coisotropic, $(\psi \otimes \psi)$-invariance means that $\psi \in H^0(\Sigma , U^\prime \otimes K^{1/2})$. For the last item, $\calS$ is a projective variety from Corollary \ref{cor711}. Also, for any non-zero spinor, the derivative of the (complex) moment map $$d\mu_\psi (\dot{\psi}) = \psi \otimes \dot{\psi} + \dot{\psi}\otimes \psi $$is injective, so, by Proposition \ref{p716}, $X$ is Lagrangian, and $\theta |_X$ vanishes by Proposition \ref{p712}. Note in particular that if the symplectic Higgs bundles $(V_i, \psi_i\otimes \psi_i)$, for $i=1,2$, are isomorphic, there exists an isomorphism $f$ between the symplectic vector bundles $(V_1, \omega_1)$ and $(V_2, \omega_2)$ satisfying $\omega_1 (\psi_1 , \cdot) f(\psi_1) = \omega_2 (\psi_2 , f(\cdot))\psi_2$. The non-degeneracy of the symplectic forms implies that $f$ takes zeros of $\psi_1$ to zeros of $\psi_2$ and for all other points, $\psi_1$ equals to $\psi_2$ (up to a sign). In particular $-\Id $ is an automorphism of $V$ which identifies the pairs $(V,\psi) = (V, -\psi)$ in $\calS$. The rest of the statement follows from the last item and from the fact that $\Theta$ corresponds to the locus of symplectic vector bundles $V$ such that $H^0(\Sigma , V\otimes K^{1/2}) \neq 0$ (c.f. Appendix \ref{thetadivisor}), thus $X$ intersects $\calU (Sp(2n,\C))$ in the generalized theta-divisor. In particular, if $(V,\psi \otimes \psi )$ is not a fixed point of the $\C^\times$-action on $\calM(Sp(2n,\C))$ and $V$ is semi-stable, the pair at infinity over $(V,\psi \otimes \psi )$ (i.e., the point $(V, \lambda^{-1}\psi \otimes \psi)$ as $|\lambda|$ goes to infinity) is $(V,0) \in \Theta$.     
\end{proof}

Note that when $\psi = 0$ the polystability condition above for a pair $(V,0)$ is precisely the polystability condition for the symplectic vector bundle $V$ (see Appendix \ref{assoc}). Now, given a pair $(V,\psi)$, one may see the spinor as a homomorphism $\psi : K^{-1/2} \to V$. We will denote by $L_\psi$ the line bundle generated by the image of $\psi$. For any semi-stable pair $(V, \psi)$, stability gives a constraint on the degree of the line bundle $L_\psi$. More precisely, 
\begin{equation}
1-g \leq \deg (L_\psi) \leq 0.
\label{deg}
\end{equation}
Clearly, the degree of $L_\psi$ is $1-g$ if and only if the spinor is non-vanishing (and so $L_\psi = K^{-1/2}$). For the other extreme, if the pair is polystable and $\deg (L_\psi) = 0$, there exists a coisotropic subbundle $U^\prime$ of $V$ such that $\psi \in H^0(\Sigma , U^\prime \otimes K^{1/2})$ and $V = L_\psi \oplus U^\prime$. This can only happen when $\psi = 0$, since $L_\psi \subset U^\prime$. 

\begin{ex} Let $V = K^{-1/2}\oplus K^{1/2} \oplus U$, where $U$ is a polystable symplectic vector bundle. Then the pair $(V,1)$, where $1 \in H^0(\Sigma , \calO_\Sigma) \subset H^0(\Sigma , V\otimes K^{1/2})$, is polystable. This follows directly from polystability of $U$ and the fact that $L_\psi = K^{-1/2}$ and $L_\psi^\perp = K^{-1/2} \oplus U$. Note that if $U$ is a stable symplectic vector bundle, so is the pair $(V,1)$.  
\end{ex}

The theorem above motivates the question we will address in this chapter. Recall that a \textbf{Bohr–Sommerfeld Lagrangian} on $\calM (Sp(2n,\C))$ is an irreducible compact complex analytic subset of the Higgs bundle moduli space on which the canonical algebraic $1$-form $\theta$ of $\calM (Sp(2n,\C))$ vanishes identically. It is proven\footnote{The statement actually holds for any reductive complex algebraic group \cite[Theorem 3.4]{bohr}.} in \cite{bohr} that Bohr–Sommerfeld Lagrangians on $\calM (Sp(2n,\C))$ are precisely the irreducible components of the zero fibre (nilpotent cone) of the Hitchin fibration for $Sp(2n,\C)$. Thus, we ask which component(s) correspond to $X$.

\subsection{The symplectic vortex equation}\label{vortex}

Let us recall a Hitchin-Kobayashi correspondence for pairs \cite{hk, mundet, relativehk, bansf}. Let $\V$ be a (finite-dimensional) complex vector space equipped with a Hermitian metric $h$. Then, we have a decomposition $h = g - \omega_\R i$, where $g$ is a flat metric and $\omega_\R$ its associated K\"ahler form. 

Let $H$ be a compact Lie group and $\rho : H \to U(\V , h )$ be a unitary representation. We also denote by $\rho$ the associated Lie algebra homomorphism $\lie{h} \to \lie{u}(\V, h)$. The action of $H$ on $\V$ clearly respects the K\"ahler structure and admits a moment map $\mu_0 : \V \to \lie{h}^*$ given by 
\begin{align*}
\inn{\mu_0 (v)}{\xi}  & = \dfrac{1}{2} \omega_\R (\rho (\xi) v, v )\\
& = \dfrac{i}{2} h (\rho (\xi) v, v ),
\end{align*}
where $v \in \V$, $\xi \in \lie{h}$, and the brackets stand for the natural pairing between $\lie{h}$ and $\lie{h}^*$.

\begin{rmk} Alternatively, we can write
$$\mu_0(v) =  \tp{\rho}(\dfrac{i}{2} v\otimes v^*).$$
In the equation above, $\tp{\rho} : \lie{u}(\V , h)^* \to \lie{h}^*$ is the transpose of the map $\rho : \lie{h} \to \lie{u}(\V , h)$ and we denote by $v^{*_h}$, or simply $v^*$, the element $ h(\cdot , v) \in \V^*$. Note that $\V^*$ has a natural induced Hermitian form $h^* (v_1^*,v_2^*) \coloneqq h(v_2,v_1)$, where $v_1,v_2 \in \V$ and we can identify $\lie{u}(\V , h)^* = \lie{u}(\V^* , h^*) \subset V\otimes V^*$. It is straightforward to check that $iv\otimes v^* \in \lie{u}(\V^* , h^*)$. Now, assume, without loss of generality, that $\rho (\xi) \in \lie{u}(\V , h)$ is of the form $u_1^* \otimes u_2$, for some $u_1,u_2 \in \V$. We have 
\begin{align*}
\inn{\tp{\rho}(\dfrac{i}{2} v\otimes v^*)}{\xi} & = \dfrac{i}{2} h(v,u_1)h(u_2,v)\\
& = \dfrac{i}{2} h(u_1^* \otimes u_2 (v), v)\\
& = \dfrac{i}{2} h(\rho (\xi)v , v).
\end{align*}  
\end{rmk}  

Choose an invariant non-degenerate form $B$ on $\lie{h}$ and denote by $\mu : \V \to \lie{h}$ the moment map (after identifying $\lie{h}^* \cong \lie{h}$ through $B$). In this chapter we will restrict to the case where $\rho$ is the standard representation of the quaternionic unitary group. As in Example \ref{type1}, assume that $\V$ is equipped with a complex symplectic form $\omega$. Then, we consider the quaternionic structure $j$ on $\V$ defined by 
\begin{equation}
\omega (u, j (v)) = h(u,v), \label{quatrel}
\end{equation} 
where $u,v \in \V$ (for more details, see Appendix \ref{AppendixC}). Thus, let $H$ be the quaternionic unitary group $Sp(\V, j) = Sp(\V , \omega) \cap U(\V , h)$, and take the invariant form on its Lie algebra $\lie{sp}(\V , j) = \lie{sp}(\V, \omega) \cap \lie{u}(\V , h)$ to be the inner product $B(\xi_1, \xi_2) = -\dfrac{1}{2} \tr (\xi_1 \xi_2)$, where $\xi_1,  \xi_2 \in \lie{sp}(\V , j)$.

\begin{claim}\label{claim} $\mu (v) = \dfrac{i}{2} (v\otimes j(v) + j(v)\otimes v)$.
\end{claim}
\begin{proof}
It follows from (\ref{quatrel}) that $\omega (j (u_1), j(u_2)) = \overline{\omega (u_1, u_2)}$, where $u_1,u_2 \in \V$. From this relation it is straightforward to check that $x \coloneqq \dfrac{i}{2}(v\otimes j(v) + j(v)\otimes v) \in \lie{sp} (\V, j) $. Now, let $\xi \in \lie{sp}(\V , j)$. So, $ \xi =  \sum\limits_{l \in L} u_l^\prime \otimes u_l $, for some finite set $L$. It follows that 
$$\tr (x\xi) = \dfrac{i}{2} (\sum \omega (v,u_l^\prime)\omega (u_l, v) + \omega(j(v), u_l^\prime)\omega(u_l, v) ),$$
but since $\xi \in \lie{sp}(\V,\omega)$, we have $\sum \omega (v,u_l^\prime)\omega (u_l, v) = \sum \omega(j(v), u_l^\prime)\omega(u_l, v)$. Thus,
$$\tr (x\xi) = i \sum \omega (v,u_l^\prime)\omega (u_l, v).$$ 
On the other hand, 
$$h(\xi (v), v) = \omega (\xi , j (v) ) = \sum \omega (u_l, v)\omega(u_l^\prime , jv).$$
From this, we find $- \dfrac{1}{2} \tr (x \xi) = \dfrac{i}{2} h(\xi (v), v)$, which concludes the proof as $x \in \lie{sp} (\V, j)$.   
\end{proof}

The Hitchin-Kobayashi correspondence enables us to describe the spinor moduli space $\calS$ as a gauge-theoretic moduli space. For that, choose a Hermitian metric $h_\Sigma \in \Omega^0(\Sigma  , K\bar{K})$ on $\Sigma$, let $d\text{vol}$ be the associated K\"ahler form and denote by $h_0 = h_\Sigma^{-1/2}$ the induced Hermitian metric on $K^{1/2}$. Let $(V,\psi)$ be a pair consisting of a symplectic vector bundle $(V,\omega)$ on $\Sigma$ and a global holomorphic section $\psi$ of $V \otimes K^{1/2}$. A reduction of structure group from $Sp(2n,\C)$ to the maximal compact subgroup $Sp(n)$ means a quaternion structure $j : V \to V$ and the Hermitian form is $h (u_1, u_2) = \omega (u_1, j(u_2))$ (this is completely analogous to the pointwise notions given in Appendix  \ref{AppendixC}, which justifies also denoting by $h$ and $\omega$ the Hermitian metric and the symplectic form on $V$, resp.). The moment map $\mu_0$ applied to $\psi$ is the restriction of $\dfrac{i}{2} \psi\otimes \psi^{*_{h,h_0}}$ to $\lie{sp}(n)$. Thus, it follows from Claim \ref{claim} that $\mu (\psi) = \dfrac{i}{2} (\psi\otimes j(\psi) + j(\psi)\otimes \psi)h_0$. The Hitchin-Kobayashi theorem relates stability for pairs $(V,\psi)$ to solutions of the \textbf{symplectic vortex equation} 
\begin{equation}
\Lambda (F_h) + \dfrac{i}{2} (\psi\otimes j(\psi) + j(\psi)\otimes \psi)h_0 = 0,
\label{vortexeq}
\end{equation}
where $F_h \in \Omega^{1,1}(\Sigma, \lie{sp}(n) ) \subset \Omega^{2}(\Sigma, \End (V) )$ is the curvature of the Chern connection associated to $h$ and $\Lambda : \Omega^2_\Sigma \to \Omega^0_\Sigma$ is the contraction with the K\"ahler form $d\text{vol}$. The left hand side of the symplectic vortex equation can be interpreted as the moment map for the Hamiltonian natural action of the gauge group on the infinite dimension K\"ahler submanifold $\chi \subset \calA \times \Omega^0(\Sigma , V\otimes K^{1/2})$, which consists of pairs $(A,\psi)$ such that $\psi$ is holomorphic with respect to the holomorphic structure induced by the $Sp(n)$-connection $A$ on $V$ (for more details see \cite{relativehk, mundet}). By the remarks above, the general result in \cite[Theorem 2.24]{hk} applied to our cases gives the following.   

\begin{thm}\label{gauge} Let $(V,\psi)$ be a polystable pair. There is a quaternionic structure $j$ on $V$ such that the Hermitian metric $h(\cdot , \cdot) = \omega (\cdot , j (\cdot))$ on $V$ solves the symplectic vortex equation  
$$\Lambda (F_h) + \dfrac{i}{2} (\psi\otimes j(\psi) + j(\psi)\otimes \psi)h_0 = 0.$$
Furthermore, if $(V,\psi)$ be stable, the quaternionic structure is unique. Conversely, if $(V,\psi)$ admits a solution to the symplectic vortex equation, then $(V,\psi)$ is polystable.
\end{thm}  

\begin{rmk} From now on we use the algebro-geometric and gauge-theoretic points of view interchangeably and denote both moduli spaces by $\calS$. 
\end{rmk}


\section{Morse theory}

Morse theory has been used as an important tool in the study of the topology of gauge-theoretic moduli spaces, dating back to the work of Atiyah and Bott on the moduli space of holomorphic vector bundles \cite{atbo} and Hitchin on the moduli space of Higgs bundles \cite{hitchin1987self}. Even though our primary goal does not involve computing topological quantities, introducing a natural Morse function on $\calS$ will help us describe the locus $X$ inside the nilpotent cone for the $Sp(2n,\C)$-Hitchin fibration.

There is a natural $\C^\times$-action on $\calS$ given by the map $\lambda \cdot (V,\psi) = (V,\lambda \psi)$. Interpreting the spinor moduli space $\calS$ from the gauge point of view, the natural action is well-defined only when we restrict to complex number of norm $1$, i.e., to $S^1 \cong U(1) \subset \C^\times$, thus we obtain a circle action
\begin{equation}
e^{i\theta} \cdot (A,\psi) = (A,e^{i\theta} \psi),
\label{circle}
\end{equation}  
where $A$ is the Chern connection associated to a metric that solves the symplectic vortex equation (\ref{vortexeq}) for a polystable pair $(V,\psi)$. Completely analogous to the Higgs bundles case \cite[Section 6]{hitchin1987self}, the map 
\begin{align*}
f : \calS & \to \R \\
(A,\psi) & \mapsto \|\psi\|_2^2  
\end{align*}
is a (multiple of the) moment map for the $S^1$-action (\ref{circle}) on the smooth locus of $\calS$, where $\|\psi\|_2^2 = \int_\Sigma |\psi |^2d\text{vol}$ is the $L^2$-norm of $\psi$. The spinor moduli space is a projective variety (see Theorem \ref{resu}), thus the associated analytic space is compact and we have the following result.
\begin{prop} The map $f : \calS \to \R$ is closed and proper.  
\end{prop} 
Let us give an upper bound for the map $f$. Take the Hermitian metric $h_\Sigma$ on $\Sigma$ to have constant constant sectional curvature $-1$ (by the Uniformization theorem such a metric exists) and denote the associated Hermitian metric $h_\Sigma^{-1/2}$ on $K^{1/2}$ by $h_0$. Let $h$ be a Hermitian metric solving the vortex equation for a pair $(V,\psi)$ and denote by $A$ the Chern connection on $V$. Also, let $B$ be the connection on $V\otimes K^{1/2}$ obtained from $A$ and the Chern connection on $K^{1/2}$ associated to $h_0$. Then we have the decomposition $\nabla_B = \partial_B + \bar{\partial}_B$ into first order differential operators and using the metric $\tilde{h} = hh_0$ on $V\otimes K^{1/2}$ and the metric on $\Sigma$ we define the formal adjoints 
\begin{align*}
\partial_B^* : \Omega^{p,q}(V\otimes K^{1/2}) &\to \Omega^{p-1,q}(V\otimes K^{1/2}),\\
\bar{\partial}_B^* : \Omega^{p,q}(V\otimes K^{1/2}) &\to \Omega^{p,q-1}(V\otimes K^{1/2}).
\end{align*}   
The Dolbeault Laplacian $\Delta_{\bar{\partial}} = \bar{\partial}_B^* \bar{\partial}_B + \bar{\partial}_B\bar{\partial}_B^* $ is a second order differential operator on $\Omega^\bullet (V\otimes K^{1/2})$ and its restriction to $\Omega^0(V\otimes K^{1/2})$ is $\bar{\partial}_B^* \bar{\partial}_B$. Using the K\"ahler identities, one establishes the following Weitzenb\"ock formula (see e.g. \cite[Lemma 6.1.7]{donald})
$$\bar{\partial}_B^* \bar{\partial}_B = \dfrac{1}{2}\nabla_B^*\nabla_B - \dfrac{i}{2}F_B.$$ 
Note that 
\begin{align*}
\Lambda F_A\psi &= (-\dfrac{i}{2} (\psi\otimes j(\psi) + j(\psi)\otimes \psi)h_0)\psi \\
& = -\dfrac{i}{2} \omega(j(\psi ) , \psi)h_0\psi \\
& = \dfrac{i}{2} | \psi |^2\psi,
\end{align*} 
where we have used that $\omega(j(\psi ) , \psi) = - h(\psi , \psi)$. But $F_B\psi = F_A\psi - \dfrac{i}{2}\psi d\text{vol}$, so 
$$\tilde{h} (F_B\psi , \psi) = \dfrac{i}{2} (| \psi |^4 - | \psi |^2)d\text{vol} .$$
The Weitzenb\"ock formula for the holomorphic section $\psi$ gives $\nabla_B^*\nabla_B\psi = iF_B\psi$, from which we conclude that 
$$\int_\Sigma | \nabla_B\psi|^2 =  -\dfrac{1}{2}\int_\Sigma (| \psi |^4 - | \psi |^2)d\text{vol}.$$
The left hand side of the equation above is always non-negative. Thus, we have the following relation between the $L^4$ and the $L^2$-norm of $\psi$: 
\begin{equation}
\| \psi \|^4_4 \leq \| \psi \|^2_2.
\label{44}
\end{equation}  
Using the Cauchy-Schwarz inequality we obtain $| \inn{|\psi|^2}{1}_{L^2}  |^2 \leq  \| \psi \|_4^4 \  \| 1 \|^2_2$, i.e.,
\begin{equation}
\| \psi \|_2^4 \leq \text{vol}(\Sigma) \| \psi \|_4^4.
\label{42}
\end{equation}
Since $h_\Sigma$ has constant sectional curvature $-1$, the curvature of the associated Chern connection on $K^{-1}$ is $F_\Sigma = i d\text{vol}$. From this we have
$$\text{vol}(\Sigma) = \int_\Sigma d\text{vol} = 4\pi (g-1).$$ 
\begin{prop} Let $h_\Sigma$ be a Hermitian metric on $\Sigma$ with constant curvature $-1$ and $h_0 = h^{-1/2}$ the induced Hermitian metric on $K^{1/2}$. Then, 
\begin{equation}
0 \leq f (V,\psi) = \| \psi \|_2^2 \leq 4\pi (g-1),
\label{bound}
\end{equation}
for all $(V,\psi) \in \calS$. 
\end{prop}
\begin{proof}
The result follows directly from (\ref{42}) and (\ref{44}). 
\end{proof}
Using the symplectic vortex equation we obtain, from the bound above, a uniform $L^2$ bound on the curvature of the solutions to the symplectic vortex equation. Analogously to the Higgs bundle case \cite{hitchin1987self}, we can apply Uhlenbeck's weak compactness theorem \cite{MR648356}. This gives an analytic proof that $\calS$ is compact.

\subsection{Critical points}

By a general result of Frankel \cite{frank}, a proper moment map for a Hamiltonian $S^1$-action on a K\"ahler manifold is a perfect Morse function. Applying his ideas to the smooth locus $\calS^{smt}$ of $\calS$, we have the following. 

\begin{prop}[{\cite[Proposition 3.3]{rank3}}] The critical points of $f : \calS^{smt} \to \R$ are precisely the fixed points of the circle action. Moreover, the eigenvalue $l$ subspace of the Hessian of $f$ is the same as the weight $-l$ subspace for the infinitesimal circle action on the tangent space. In particular, the Morse index of $f$ at a critical point equals the dimension of the positive weight space of the circle action on the tangent space. 
\end{prop}

Let us characterize the fixed points of the circle action on $\calS$. A pair with zero spinor is clearly a fixed point. Let $(V,\psi) \in \calS$, with $\psi \neq 0$, be a fixed point of $S^1$-action on $\calS$ and denote by $A$ the Chern connection on $V$ associated to a metric that solves the symplectic vortex equation, so that the class $(A,\psi)$ represents the pair. Then, there exists a one-parameter family of $Sp(n)$-gauge transformations $\{ g_\theta \}$, $\theta \in \R$, such that 
\begin{equation}
g_\theta \cdot (A,\psi)  = (A, e^{i\theta} \psi).
\label{s1}
\end{equation}
Let 
$$\left.\Upsilon = \frac{ dg_\theta}{d \theta}\right|_{\theta=0}$$
be the infinitesimal gauge transformation generating the family $\{g_\theta \}$. Since $g_\theta \cdot A = g_\theta \circ A \circ g_\theta^{-1}$, differentiating (\ref{s1}) we obtain 
\begin{align}
d_A\Upsilon & = 0,\label{cov}\\ 
\Upsilon \psi &= i\psi.\label{v1}
\end{align} 

\begin{prop} A pair $(V,\psi) \in \calS$ is a fixed point of the circle action if and only if the vector bundle $V$ can be decomposed as a direct sum   
$$V = \bigoplus_{k=1}^{m} (V_k \oplus V_k^*) \oplus V_0,$$
where $V_k$ are the eigenbundles for a covariantly constant infinitesimal gauge transformation $\Upsilon$, with $\Upsilon|_{V_k} = ia_k$, for some $a_k \in \R$. Moreover, 
$$\psi \in H^0(\Sigma ,  V_1 \otimes K^{1/2})$$
and 
$$V_{-k} \cong jV_k \cong V^*_k,$$
where $j$ is the quaternionic structure on $V$ coming from the metric that solves the symplectic vortex equation. Also, $V_0$ and $V_k \oplus V_k^*$, for $k>1$, are polystable. 
\end{prop}
\begin{proof}
As discussed, being a fixed point of the circle action is equivalent to the existence of an infinitesimal gauge transformation $\Upsilon$ satisfying (\ref{cov}) and (\ref{v1}). Since $\Upsilon$ is covariantly constant, $V$ must decompose globally into a direct sum of eigenbundles of the endomorphism $\Upsilon$. An infinitesimal gauge transformation is locally in $\lie{sp}(n)$, so its eigenvalues are of the form $\pm i a_k$, where $a_k \in \R$. Thus, $V = \bigoplus V_k$, where $\Upsilon |_{V_k} = i a_k$, for some $a_k \in \R$. If $v_r \in V_r$ and $v_s \in V_s$   
$$ia_r\omega (v_r, v_s) = \omega (\Upsilon v_r, v_s) = -\omega (v_r, \Upsilon v_s) = -ia_s\omega (v_r, v_s),$$
so $V_r$ and $V_s$ are orthogonal unless $r+s=0$. In that case, the symplectic form $\omega$ on $V$ gives an isomorphism $V_r \cong V^*_{-r}$. Note also that $\Upsilon$ commutes with the quaternionic structure $j$ on $V$, so that if $v \in V_r$, $jv \in V_{-r}$ (since $j$ is anti-linear), which gives $jV_r \cong V_{-r} \cong V_r^*$. Thus, we can write
$$V = \bigoplus_{k=1}^{m} (V_k \oplus V_k^*) \oplus V_0$$
and $\psi \in V_1$, since (\ref{v1}) holds. From this we conclude that $\psi \otimes j\psi + j\psi\otimes \psi$ is an endomorphism of $V_1\oplus V_1^*$. Thus, $V_0$ and $V_k \oplus V_k^*$, for $k>1$, are flat and thus polystable.
\end{proof}

Let $\calF$ be the fixed point set corresponding to the $S^1$-action on $\calS$. Note that any maximum of $f $ is an element in $\calF$. In particular, they form a component of $\calF$ which will play an important role in this chapter. 

\begin{prop}\label{maximum} The maximum locus of $f$ is a connected component of the fixed point set $\calF$ of the circle action, which is singular, irreducible and can be identified with  
$$\calZ = \{ ( K^{-1/2} \oplus K^{1/2} \oplus U, 1) \ | \ U\in \calU(Sp(2n-2,\C))\}, $$ 
where we denote by $1$ the spinor given by the canonical inclusion of $K^{-1/2}$ into $K^{-1/2} \oplus K^{1/2} \oplus U$. 
\end{prop} 
\begin{proof}
First note that $\calZ $ is contained in $\calS$, since $U$ is a polystable symplectic vector bundle, and by the characterization of fixed points it must lie in $\calF$. Existence of strictly polystable symplectic bundles and irreducibility of $\calU(Sp(2n-2,\C))$ shows that $\calZ $ is a singular irreducible variety. It follows from \cite[Lemma 9.2]{hausel2003mirror} that the ranks and degrees of the eigenbundles which appear in the decomposition of a fixed point are locally constant on the fixed point set. Thus, $\calZ $ is a connected component of $\calF$. Moreover, note that the spinor $1$ has $L^2$-norm $4\pi(g-1)$ and thus, from (\ref{bound}), all pairs in $\calZ $ are maximums. We show that the converse also holds. 

Let $(V,\psi) \in \calS$ be a maximum. The spinor $\psi : K^{-1/2} \to V$ defines a line bundle $L_\psi \subset V$ whose degree $d_\psi$, by (\ref{deg}), satisfies $1-g \leq d_\psi \leq 0$. Let $h$ be the Hermitian metric on $V$ solving the symplectic vortex equation as before. From $h$ we obtain a smooth splitting $V = L_\psi \oplus Q$, where $Q$ is the quotient $V/L_\psi$. We denote by $\beta \in \Omega^{0,1}(\Sigma , \Hom (Q,L_\psi))$ the second fundamental form associated to the smooth splitting of $V$ and by $\pi$ the orthogonal projection operator associated to $L_\psi$. The Hermitian metric $h$ induces in a natural way Hermitian metrics on $L_\psi$ and $Q$, and we can write the Chern connection $A$ on $V$ as
$$A = \left( \begin{matrix} A_\psi& \beta\\ -\beta^*&A_Q \end{matrix} \right),$$
where $A_\psi$ and $A_Q$ are the associated Chern connections on $L_\psi$ and $Q$, resp., and $\beta^* \in \Omega^{1,0}(\Sigma , \Hom (Q,L_\psi))$ is obtained from the metric and conjugation on the form part. From this we find that the curvature associated to $A$ can be written as   
$$F_A = \left( \begin{matrix} F_\psi - \beta \wedge \beta^*& \partial_A \beta\\ - \bar{\partial}_A \beta^*& F_Q - \beta^* \wedge \beta \end{matrix} \right),$$   
where $F_\psi$ and $F_Q$ are the curvatures associated to the connections $A_\psi$ and $A_Q$, respectively. Recall that given any holomorphic vector bundle $E$ on $\Sigma$ with a Hermitian metric, Chern-Weil theory enables one to calculate the degree of $E$ as
$\deg (E) = \frac{i}{2\pi} \int_\Sigma  \tr (F) $, where $F$ is the curvature of the associated Chern connection\footnote{This is standard and can be found in several textbooks, such as \cite{kob} and \cite{wells}.}. In our context, this gives the formula
$$d_\psi = \dfrac{1}{2\pi} \int_\Sigma \tr (\pi i\Lambda F_A)d\text{vol} -  \dfrac{1}{2\pi} \int_\Sigma |\beta |^2 d\text{vol}.$$   
But $i\Lambda F_A = \dfrac{1}{2} (j\psi \otimes \psi + \psi \otimes j\psi)h_0$, thus projecting it onto $L_\psi$ and taking the trace gives $\omega (j\psi , \psi)h_0/2 = -|\psi |^2/2$ and we obtain
$$d_\psi = - \dfrac{1}{4\pi} \| \psi \|_2^2 - b,$$ 
where $b =   \dfrac{1}{2\pi} \int_\Sigma |\beta |^2 d\text{vol} \ge 0$.
Now, the inequality $d_\psi \ge 1-g$ gives 
$$ \| \psi \|_2^2 \leq 4\pi (g-1) - 4\pi b.$$
Thus $(V,\psi)$ is a maximum if and only if $b=0$. In this case, $d_\psi = 1-g$ and $L_\psi$ must be isomorphic to $K^{-1/2}$. The vanishing of the second fundamental form means that the splitting $V = K^{-1/2} \oplus Q$ is holomorphic. Moreover, any maximum is a fixed point of the circle action, which proves that $V$ splits (holomorphically) as a direct sum $K^{-1/2} \oplus K^{1/2} \oplus U$, where $U$ is a polystable symplectic vector bundle, and thus the pair must be in $\calZ $. 
\end{proof}

\begin{rmk} Given a fixed point $(V,\psi)$ of the circle action on $\calS$, the vector bundle $V$ decomposes as a direct sum of eigenbundles for an infinitesimal gauge transformation $\Upsilon$ satisfying (\ref{v1}) and (\ref{cov}). In particular, 
\begin{align*}
[ \Upsilon , \psi \otimes \psi ]v & = \omega (\psi, v)\Upsilon \psi - \omega (\psi, \Upsilon v)\psi \\
& = i\omega (\psi, v)\psi  + \omega (\Upsilon \psi, v)\psi \\
& = 2i \psi \otimes \psi (v).
\end{align*}
Thus, the Higgs bundle $(V,\psi \otimes \psi)$ is fixed by the circle action on $\calM (Sp(2n,\C))$ and the covariantly constant infinitesimal gauge transformation $\frac{1}{2}\Upsilon $ induces the same decomposition as a direct sum of subbundles as before.       
\end{rmk}

\section{The component of the nilpotent cone}

Let $(V,\psi\otimes \psi) \in \calM (Sp(2n,\C))$ be a Higgs bundle such that the norm of the spinor $\psi$ is sufficiently large (see Remark \ref{norms} on the next page). Pairs in $\calZ$ are non-vanishing, which is an open condition. Since the pair $(V,\psi)$ is near a maximum, its spinor must also be non-vanishing. In particular, the image of $\psi$ defines a subbundle of $V$, rather than just a subsheaf. Dualizing the injection $\psi : K^{-1/2} \to V$ and using the symplectic form $\omega : V \to V^*$ to identify $V$ with its dual, we obtain a short exact sequence 
\begin{equation}
0 \to E \to V \to K^{1/2} \to 0.
\label{ext1}
\end{equation} 
Note that the projection $V \to K^{1/2}$ is given by the map $\tp{\psi}\circ \omega$, so $E$ is isomorphic to the symplectic orthogonal $(K^{-1/2})^\perp$ of $K^{-1/2}$. In particular, we have  
\begin{equation}
0 \to K^{-1/2} \to E \to U \to 0
\label{ext22}
\end{equation}
and the skew-form $\omega|_{E\otimes E} $ induces a non-degenerate form on the quotient $U$, which is then a symplectic vector bundle of rank $2n-2$. Recall that a maximum of $f$ has an underlying vector bundle of the form $K^{-1/2} \oplus K^{1/2} \oplus W$, where $W$ is a polystable symplectic vector bundle of rank $2n-2$ (see Proposition \ref{maximum}). Since $(V,\psi)$ is near a maximum, since stability is an open condition, $U$ is also polystable. In particular, for a sufficiently small $\epsilon > 0$, the subset of Higgs bundles of the form $(V, \psi \otimes \psi)$ satisfying $f_{max} - \epsilon < \| \psi \|^2_2 \leq f_{max}$ is contained in  
$$
\calZ^- = \left \{ \ (V, \psi \otimes \psi) \in \calM (Sp(2n,\C)) \
\bigg |
\gathered
\begin{array}{cl}
i. & 0 \neq \psi \in H^0(\Sigma , V\otimes K^{1/2}), \\
ii. & \psi\text{ is non-vanishing},\\   
iii. & (K^{-1/2})^\perp / K^{-1/2}\in \calU (Sp(2n-2, \C)). \
\end{array}
\endgathered  \ \  \right
\},$$
which is thus a non-empty open subset of $\calM (Sp(2n,\C))$. Here, $f_{max}$ denotes the maximum value of the function $f: \calS \to \R$. When we take the Hermitian connection on $K^{1/2}$ associated to a metric of constant negative curvature on $\Sigma$, $f_{max} = 4\pi (g-1)$ (c.f. (\ref{bound})). Note that if $(V, \psi\otimes \psi) \in \calZ^-$, the limit of $(V, z \cdot \psi \otimes \psi)$ as the norm of $z\in \C$ goes to infinity exists (as the nilpotent cone is compact) and is a Higgs bundle in the maximum\footnote{The locus $\calZ$ corresponds to the maximums of $f$. If seen in the Higgs bundles moduli space, it corresponds to the maximum of $\| \Phi \|^2$ on $X$ (see Remark \ref{norms}).} $\calZ$. 

\begin{rmk}\label{norms} The Hitchin equation, differently from the symplectic vortex equation, does not depend on the choice of a Hermitian metric on $\Sigma$. By fixing a Hermitian metric $h_\Sigma$ and denoting the induced metric $h_\Sigma^{-1/2}$ on $K^{1/2}$ by $h_0$, the Hitchin equation 
$$F_h + [\Phi , \Phi^*] = 0$$
can be written (using the identifications from Section \ref{vortex}) as  
$$\Lambda (F_h) + \dfrac{i}{2}|\psi|^2 (\psi\otimes j(\psi) + j(\psi)\otimes \psi)h_0 = 0,$$
since $(\psi \otimes \psi )^*$ can be identified with $- j\psi \otimes j\psi$ and, for example, $(\psi \otimes \psi) \circ (j\psi \otimes j\psi) = |\psi|^2 j\psi \otimes \psi$. Thus, near a maximum of $f$, the metric solving the Hitchin equation for $(V,\psi \otimes \psi)$ is essentially the same as the one solving the symplectic vortex equation for the pair $(V,\psi)$, and so $| \psi|_{pair}^4 \approx |\psi \otimes \psi|^2_{Higgs}$. The same phenomenon, of course, also happens in the other extreme. If $(V,\psi)$ has a spinor whose norm is close to zero, the same is true for the norm of the associated Higgs field (in the Higgs bundle metric). This tells us that the underlying vector bundle is polystable.  
\end{rmk}

\begin{ex}\label{rank2} The simplest example occurs when $n=1$, i.e., for $SL(2,\C) = Sp(2,\C)$, and this has been worked out by Hitchin in \cite[Section 4.3]{spinors}. In particular, in this case $\calZ$ corresponds to a point, where the Higgs bundle has an underlying vector bundle of the form $K^{-1/2}\oplus K^{1/2}$ (i.e., the intersection of the Hitchin section with the nilpotent cone for our particular choice of theta-characteristic). Recall that the irreducible components of the nilpotent cone for the $SL(2,\C)$-Hitchin fibration are the Zariski closures of the total spaces of some vector bundles $E_m$, $1\leq m \leq g-1$, over the $(2g-2-2m)$-symmetric product of $\Sigma$ (and $\calU (2, \calO_\Sigma)$). Given $(V,\Phi) \in \calM (SL(2,\C))$ with a nilpotent Higgs field $\Phi \neq 0$, let $M$ be the subbundle of $V$ generated by the kernel of $\Phi$. Then we have an extension  
$$0 \to M \to V \to M^* \to 0$$
defined by some extension class $\delta \in H^1(\Sigma , M^2)$. Consider $\phi \in H^0(\Sigma , M^2K)$ given by 
\[\xymatrix@M=0.1in{
0 \ar[r] & M \ar[r] \ar[d] & V \ar[r] \ar[d]^{\Phi} & M^* \ar[r] \ar[d]^{\phi} & 0 \\
0  & M^*K \ar[l]     & V\otimes K \ar[l]      & MK \ar[l]             & 0. \ar[l] }\]
Then, $(V,\Phi)$, can be seen as a class $(M, \delta, \phi) = (M, \lambda \delta , \lambda \phi)$, $\lambda \in \C^\times$, and the vector bundle $E_m$ consists of such classes where $m = \deg M^*$. In particular, our choice of theta-characteristic identifies $\calZ^-$ with $ H^1(\Sigma , K^*)$. In other words, we see a point in $\calZ^-$ as a class $(K^{-1/2}, \delta , 1)$, where $\delta \in H^1(\Sigma , K^*)$ is the extension class of the exact sequence determined by $\psi$, which in this case is non-vanishing. The closure of the submanifold $\calZ^-$ is the Gaiotto Lagrangian $X$. 
\end{ex}

To reconstruct this open set we start with a polystable symplectic vector bundle $(U,\theta )$ of rank $2n-2$, where $\theta : U \to U^*$ is the symplectic form\footnote{Recall that any two symplectic forms on a polystable symplectic vector bundle $U$ are related by an automorphism of $U$.}, and an extension class $\delta \in H^1(\Sigma , \Hom (U, K^{-1/2}))$. The extension class gives  
\begin{equation}
0 \to K^{-1/2} \to E \to U \to 0
\label{FF}
\end{equation}
and there exists trivializations of the vector bundle $E$ over an open cover $\{ U_i \}$ of $\Sigma$ whose transition functions are of the form 
$$e_{ij} =  \left( \begin{matrix} l_{ij}& \delta_{ij}\\ & u_{ij} \end{matrix} \right),$$
where $l_{ij}$ and $u_{ij}$ are transition functions for $K^{-1/2}$ and $U$, resp., and the extension class $\delta$ is determined by the $\Hom (U, K^{-1/2})$-valued 1-cocycle $\{ l_{ij}^{-1} \delta_{ij} \} $. Also, the symplectic form $\theta $ on $U$ is given by a cochain $\{ \theta_i \}$ of skew-symmetric and non-degenerate $(2n-2)\times (2n-2)$ matrices which agree on the intersections $U_{ij}$, i.e., 
$$\theta_j = \tp{u}_{ij} \theta_i u_{ij}.$$ 
Since $\Hom (K^{1/2}, U)$ is a semi-stable vector bundle of negative degree, it does not have any non-trivial global section. Thus, applying the functor $\Hom (K^{1/2}, -)$ to (\ref{FF}) and considering the corresponding long exact sequence in cohomology we obtain 
$$0 \to H^1(\Sigma , K^{-1}) \to H^1(\Sigma , \Hom (K^{1/2}, E)) \to H^1(\Sigma , \Hom (K^{1/2}, U)) \to 0.$$   
Then, from the exact sequence above, $\delta$ defines an extension $V$ of $K^{1/2}$ by $E$, which is a rank $2n$ vector bundle. More precisely, to be consistent with (\ref{ext1}) and (\ref{ext22}), we choose an extension class $\tilde{\delta} \in H^1(\Sigma , \Hom (K^{1/2}, E))$ lifting the dual extension class of $\delta$, which, after using the identification $\theta : U \to U^*$, is represented by the $\Hom (K^{1/2}, U)$-valued 1-cocycle $\{ - \theta^{-1}_j  \tp{\delta}_{ij}l_{ij}^{-1} \}$. From such a choice of $\tilde{\delta}$ we obtain 
$$0 \to E \to V \to K^{-1/2} \to 0$$
and the transition functions for $V$ have the form 
$$v_{ij} =  \left( \begin{matrix} e_{ij}& \tilde{\delta}_{ij}\\ & l_{ij}^{-1} \end{matrix} \right),$$
where $\tilde{\delta}$ is represented by the 1-cocycle $\{ e^{-1}_{ij}\tilde{\delta}_{ij} \}$. We may think of $e^{-1}_{ij}\tilde{\delta}_{ij} $ as the $(2n-1)\times 1$ matrix representing the corresponding homomorphism $K^{1/2} \to E$ in the trivialization $U_j$. Thus we may write 
$$e^{-1}_{ij}\tilde{\delta}_{ij} = \left( \begin{matrix} a_{ij} \\ - \theta^{-1}_j  \tp{\delta}_{ij}l_{ij}^{-1}  \end{matrix} \right),$$ 
where $\{ a_{ij} \} \in H^1(\Sigma, K^{-1})$. Note that the symplectic form $\theta$ on $U$ extends naturally to a degenerate skew-form $\tilde{\theta}$ on $E$. Indeed, dualizing (\ref{FF}) we obtain 
$$0 \to U^* \to E^* \to K^{1/2} \to 0,$$
so  
$$0 \to \Lambda^2U^* \to \Lambda^2E^* \to U^*\otimes K^{1/2} \to 0.$$ 
Now we need to further extend the form on $E$ to a symplectic form on $V$. 
\begin{lemma} The skew-form $\tilde{\theta}$ extends uniquely to a symplectic form $\omega$ on $V$.   
\end{lemma}
\begin{proof}
First note that we may represent $\tilde{\theta}$ by $\{ \tilde{\theta}_i \}$, where 
\begin{equation}
\tilde{\theta}_i =  \left( \begin{matrix} 0& \\ & \theta_i \end{matrix} \right).
\label{thetatil}
\end{equation}
Thus, to define a symplectic form on $V$ extending $\tilde{\theta}$ we need to construct $2n\times 2n$ non-degenerate skew-symmetric matrices $\omega_i$ for each open set $U_i$ such that on the overlaps $U_{ij}$, they are compatible, i.e., 
\begin{equation}
\omega_j = \tp{v}_{ij} \omega_i v_{ij}.
\label{omega}
\end{equation} 
Since $\omega$ must extend $\tilde{\theta}$ it must be of the form 
$$\omega_i =  \left( \begin{matrix} \tilde{\theta}_i & \tilde{b}_i\\ - \tp{\tilde{b}}_i & 0 \end{matrix} \right),$$
and we may further decompose the $(2n-1)\times 1$ matrix $\tilde{b}_i$ as 
\begin{equation}
\tilde{b}_i = \left( \begin{matrix} b_i^0 \\ b_i  \end{matrix} \right),
\label{btil}
\end{equation}
where $b_i^0$ represents a local function and $b_i$ a homomorphism $K^{1/2} \to U^*$ on $U_i$. Condition (\ref{omega}) then becomes
$$\tp{e}_{ij}\tilde{\theta}_i\tilde{\delta}_{ij} + \tp{e}_{ij}\tilde{b}_il_{ij}^{-1} = \tilde{b}_j,$$
or equivalently,
\begin{equation}
\tilde{\theta}_j\tp{e}_{ij}^{-1}\tilde{\delta}_{ij} + \tp{e}_{ij}\tilde{b}_il_{ij}^{-1} = \tilde{b}_j.
\label{expr}
\end{equation}   
This is simply the statement that the cup product $\tilde{\theta} \cdot \tilde{\delta} = 0 \in H^1(\Sigma , E^*\otimes K^{-1/2})$. Now, substituting (\ref{thetatil}) and (\ref{btil}) in the expression above gives $b_i^0 = b_j^0$ on $U_{ij}$ and so, $\{ b_i^0 \} \in H^0(\Sigma , \calO_\Sigma) = \C$. Note that the determinant of $\omega_i$ equals $(b_i^0)^2\det \theta_i$. Since we may assume that $\omega_i$ and $\theta_i$ have determinant $1$, without loss of generality\footnote{Minus the identity is an automorphism of the symplectic bundle and it identifies any pair $(V,\psi)$ with $(V,-\psi)$.} we take $b_i^0$ to be $1$. The expression (\ref{expr}) also gives the relation
$$ \tp{u}_{ij}b_il_{ij}^{-1} = b_j$$ 
on $U_{ij}$. But $\tp{u}_{ij}b_il_{ij}^{-1}$ is just the expression of $b_i$ in the coordinate $U_j$, thus $\{ b_i \} $ defines a global holomorphic section of $U\otimes K^{-1/2}$. The vector bundle $U\otimes K^{-1/2}$ is semi-stable of negative degree, so $b_i = 0$, and that is the only way to extend $\tilde{\theta}$ to a symplectic form on $V$.         
\end{proof}
Note that we have reconstructed $\calZ^-$ from a polystable symplectic vector bundle $U$ of rank $2n-2$, an extension class $\delta \in H^1(\Sigma , \Hom (U, K^{-1/2}))$ and a choice of $\{a_k \} \in H^1(\Sigma , K^{-1})$. Thus, using Riemann-Roch, we have that the dimension of the open set is 
\begin{align*}
&\dim \calU (Sp(2n-2, \C) + h^1 (\Hom (U, K^{-1/2})) + h^1(K^{-1}) \\ 
& = (n-1)(2n-1)(g-1) + 4(n-1)(g-1) + 3(g-1) \\
& = n(2n+1)(g-1)\\
& = \dim Sp(2n,\C) (g-1).
\end{align*}
This is half of the dimension of the moduli space of Higgs bundles, so the Zariski closure of this open set corresponds to Gaiotto's Lagrangian $X$. Let us show that this is irreducible and thus identifies $X$ with this particular algebraic component of the nilpotent cone for the $Sp(2n,\C)$-Hitchin fibration. 

\begin{lemma}\label{down} The open set $\calZ^-$ intersects the conormal bundle of the generalized theta-divisor $D_{K^{1/2}}$ of $\calU (Sp(2n,\C))$.  
\end{lemma} 
\begin{proof}
Let $0 \neq s \in H^0(\Sigma , K)$ and write its divisor of zeros $D$ as a sum of two divisors $D_1$ and $D_2$ of degree $g-1$ without common zeros. The line bundles corresponding to these divisors are $\calO_\Sigma (D_1) = LK^{1/2}$ and $\calO_\Sigma (D_2) = L^*K^{1/2}$, for some line bundle $L$ of degree $0$. The line bundles $\calO_\Sigma (D_1)$ and $\calO_\Sigma (D_2)$ come with distinguished (up to a non-zero scalar) sections $s_1$ and $s_2$ vanishing precisely on $D_1$ and $D_2$. From the fact that the supports of $D_1$ and $D_2$ have empty intersection, the map $s_1 + s_2 : K^{-1/2} \to L\oplus L^*$ is injective. Thus, we have an extension 
$$0 \to K^{-1/2} \to L\oplus L^* \to K^{1/2} \to 0.$$
Note that $L\oplus L^*$ is a (strictly) semi-stable rank $2$ vector bundle with trivial determinant, which is non-very-stable, i.e., $L\oplus L^*$ admits a non-zero nilpotent Higgs field $\Phi$ (take the Higgs field corresponding to $s_1^2 \in H^0(\Sigma , L^2K)$ for example). Now, if $K^{-1/2}\oplus K^{1/2}\oplus U$ is the underlying vector bundle of a Higgs bundle in $\calZ$, then $(L\oplus L^* \oplus U, s_1^2)$ is in the conormal bundle of the generalized theta-divisor $D_{K^{1/2}}$ of $\calU (Sp(2n,\C))$ (since the underlying vector bundle is semi-stable). Also, $(K^{-1/2})^\perp \cong K^{-1/2} \oplus U$, where $U \in \calU (Sp(2n,\C))$, thus $(L\oplus L^* \oplus U, s_1^2)$ is a Higgs bundle in the intersection of $\calZ^-$ and the conormal bundle of $D_{K^{1/2}}$.      
\end{proof}

In greater generality, Thaddeus \cite{thad} has proven that the locus of Higgs bundles whose underlying vector bundle is stable is non-empty in every irreducible component of the nilpotent cone for Higgs bundles of rank $2$ and fixed determinant. In any case, there are points in $\calZ^-$ flowing down directly to the generalized theta-divisor. 

\begin{lemma} The Zariski closure of $\calZ^-$ is an irreducible component of the nilpotent cone for the $Sp(2n,\C)$-Hitchin fibration. 
\end{lemma}
\begin{proof}
Consider the locus $\calZ_1^-$ of $\calZ^-$ consisting of Higgs bundles $(V,\psi \otimes \psi)$ where the underlying vector bundle $V$ is semi-stable and $H^0(\Sigma , V\otimes K^{1/2})$ has dimension $1$. By Lemma \ref{down}, this is non-empty (as $H^0(\Sigma , U\otimes K^{1/2}) = 0$ for a generic polystable symplectic vector bundle $U$) and it constitutes an open set of $\calZ^-$. Moreover, there is a natural morphism from this open subset to the open set $D_1$ of the generalized theta-divisor $D_{K^{1/2}}$ consisting of those $V$ such that $h^0(\Sigma , V\otimes K^{1/2}) = 1$. Note that $\calZ_1^-$ is pure-dimensional, as every component of the nilpotent cone has the same dimension, and $D_1$ is irreducible, since the generalized theta-divisor is irreducible (c.f. Appendix \ref{irred}). The fibres of the map $\calZ_1^- \to D_1$ are isomorphic to $\C^\times$, which shows that the closure of $\calZ^-$ is irreducible.
\end{proof}

Recall that, when $(n,d)\in \Z_{>0}\times \Z$ are coprime, the moduli space $\calM (n,d)$ of Higgs bundles of rank $n$ and degree $d$ is smooth and the connected components of the fixed point set $\calF$ corresponding to the $\C^\times$-action on $\calM (n,d)$ are in bijection with the irreducible components of the associated nilpotent cone. The connected components of $\calF$ have a modular interpretation, the so-called moduli space of chains (see e.g. \cite{chains, hein}). For general $(n,d)$, or other groups, it is not known if the chains parametrize the components of the nilpotent cone. As explained in \cite[Section 6]{bradlow2017irreducibility}, every connected component $\calZ$ of the fixed point set $\calF$ gives at least one irreducible component, which arises as the closure of the set of Higgs bundles $(V,\Phi)$ such that $\underset{|\lambda| \to \infty}{\lim} (V, \lambda \Phi) \in \calZ$. 

In our case, the locus $\calZ$ corresponding to maximums of $f$ (c.f. Proposition \ref{maximum}), seen inside the Higgs bundle moduli space, is a connected component of the fixed point set of the circle action on $\calM(Sp(2n,\C))$, which, in terms of chains can be written as 
$$K^{1/2}\overset{1}{\to} K^{-1/2} \overset{0}{\to} U,$$
where $U \in \calU (Sp(2n-2,\C))$. As discussed, this is the maximum of $\| \Phi \|_2^2 $ on the Gaiotto Lagrangian. 

\begin{thm}\label{compgaiolag} The Gaiotto Lagrangian $X$ corresponds to the irreducible component of the nilpotent cone $\text{Nilp}(Sp(2n, \C))$ labeled by the  chain   
$$K^{1/2}\overset{1}{\to} K^{-1/2} \overset{0}{\to} U,$$
where $U \in \calU (Sp(2n-2,\C))$.  
\end{thm}

Note that when the norm of the spinor is sufficiently small, by Hitchin equation, the underlying vector bundle is polystable. The theta-bundle $\Theta$ has a section vanishing on the generalized theta-divisor $D_{K^{1/2}}$ and at a point $V \in D_{K^{1/2}}$, the fibre of $\Theta$ is identified with 
$$\Lambda^{top}H^0(\Sigma , V\otimes K^{1/2})^* \otimes \Lambda^{top}H^1(\Sigma , V\otimes K^{1/2})$$  
(for more details see Appendix \ref{thetadivisor}). Since $V$ is symplectic, Serre duality gives an identification between $H^1(\Sigma , V\otimes K^{1/2})$ and $H^0(\Sigma , V\otimes K^{1/2})^*$. It thus follows that restricting to the locus of $D_{K^{1/2}}$ consisting of semi-stable symplectic vector bundle $V$ which have $h^0(\Sigma , V\otimes K^{1/2}) = 1$, the conormal bundle gets identified with $\Theta^{-2}$, as explained in \cite[Section 3]{spinors}. The Gaiotto Lagrangian intersects the cotangent bundle $T^*\calU (Sp(2n,\C))$ in the conormal bundle of $D_{K^{1/2}}$.

%% file: Chapters/Chapter9.tex

\chapter{Further questions} 

\label{Chapter9} 

\lhead{Chapter 9. \emph{Further questions}} 


In this final chapter we mention some interesting questions that emerge from the results and ideas in this thesis.  

\begin{enumerate}
\item \textbf{The hyperholomorphic bundle}: Given a complex reductive group $G$, mirror symmetry should transform a BAA-brane on the moduli space $\calM (G)$ of $G$-Higgs bundles to a BBB-brane on the moduli space of Higgs bundles for the Langlands dual group $\lan{G}$. Moreover, this duality should be realized by a Fourier-type transformation. For any real form $G_0$ of $G$ there is a natural associated BAA-brane on $\calM (G)$. In Chapter \ref{Chapter6} we give a proposal for the support of the dual BBB-brane on $\calM (\lan{G})$ associated to the real forms $G_0 = SU^*(2m)$, $SO^*(4m)$ and $Sp(m,m)$ of $G = SL(2m,\C)$, $SO(4m,\C)$ and $Sp(4m,\C)$, respectively. For these real forms, the BAA-branes are supported on the singular locus (the associated spectral curves are non-reduced schemes) and a full Fourier-Mukai transform is not known to exist. Recently, Gaiotto \cite{gaiotto2016s} considered several examples of BAA-branes in the moduli space of Higgs bundles and discussed the associated dual BBB-brane, which always arose from the Dirac-Higgs bundle. In particular, he considers examples of BAA-branes associated to symplectic representations of $G$ (as in Chapter \ref{Chapter7}) and he also remarks that Hitchin's proposal in  \cite{hitchin2013higgs} for the dual brane corresponding to the real form $U(m,m) \subset GL(2m,\C)$ is in accordance with the ideas developed in his paper. Let us recall the construction of the Dirac-Higgs bundle and explain some difficulties that arise for our real forms. A good reference for this material is \cite{blaavand2015dirac}, which we follow closely together with Section 7 from \cite{hitchin2013higgs}. On a manifold with a spin structure we have natural spinor bundles $S^{\pm}$ and the ordinary Dirac-operator can be seen as an operator taking smooth sections of $S^+$ to smooth section of $S^-$. Let $K^\C$ be a semisimple Lie group, where $K$ is compact, and consider a $K^\C$-Higgs bundle $(P,\Phi)$. Also, let $A$ be the $K$-connection on $P$ solving the Higgs bundle equations. The \textbf{Dirac-Higgs operator} is then defined as the ordinary Dirac operator coupled with $(A,\Phi)$. In particular, given a representation $\V$ of $K$ we can consider the operator $D_\V : \Omega^0(S^+ \otimes V) \to \Omega^0(S^-\otimes V)$, where $V$ is the vector bundle associated to $P$ via $\V$. More precisely,
$$D_{\V} = \left( \begin{matrix} \partial_A& -\Phi\\ \Phi^*&-\bar{\partial}_{A} \end{matrix} \right): \Omega^0(V)^{\oplus 2} \to \Omega^{1,0}(V) \oplus \Omega^{0,1}(V)$$  
and considering the $L^2$-norm on $\Omega^{1}(V)$, the adjoint of $D_{\V}$ is the elliptic operator 
$$D^*_{\V} = \left( \begin{matrix} \bar{\partial}_A& \Phi\\ \Phi^*& \partial_{A} \end{matrix} \right): \Omega^{1,0}(V) \oplus \Omega^{0,1}(V) \to \Omega^{1,1}(V).$$  
As remarked by Hitchin in \textit{loc. cit.}, the Higgs bundle equation yields a vanishing theorem for irreducible connections and the dimension of $\ker (D^*_{\V})$ is $2 \rk (V) (g-1)$ (this can also be found in \cite[Proposition 2.1.10]{blaavand2015dirac}). Hodge theory allows one to view the operator in complex structure $I$ as the two-term complex of sheaves $\mathsf{V}$ given by  
$$\Phi : \calO (V) \to \calO (V\otimes K)$$ 
and the kernel of $D^*_{\V}$ is identified with the first hypercohomology $\K^1 (\Sigma , \mathsf{V})$. Moreover, if $\det (\Phi)$ has simple zeros $z_1, \ldots, z_N$, one has the identification 
$$\K^1 (\Sigma , \mathsf{V}) \cong \bigoplus_{i=1}^N \coker (\Phi_{z_i}),$$ 
where $N = 2 \rk (V) (g-1)$ (see e.g. Lemmas 2.4.1 and 2.4.3 in \cite{blaavand2015dirac}). The moduli space of Higgs bundles $\calM = \calM (K^\C)$ is coarse and not fine, thus there is no universal Higgs bundle on $\calM \times \Sigma$. However, one can always find an open covering $\{ U_i \}$ of $\calM$ and local universal Higgs bundles $(\calV_i , \Theta_i)$, where $\calV_i$ is a vector bundle on $U_i \times \Sigma$ and $\calV_i \to \calV\otimes p_1^*K$, where $p_1 : \calM \times \Sigma \to \Sigma$ is the projection onto the first factor. On each intersection $U_{ij} = U_i \cap U_j$ the bundles are related by $\calV_i|_{U_{ij}} = L_{ij}\otimes \calV_j|_{U_{ij}}$, for some flat unitary line bundles $L_{ij}$ satisfying $L_{ij} = L_{ji}^{-1}$ on $U_{ij}$ and $L_{ij}L_{jk}L_{ki}= \calO$ on the triple intersection $U_{ijk} = U_i \cap U_{j} \cap U_k$. These line bundles define a so-called \textit{gerbe} (for more on gerbes, see e.g. \cite{lectslf}). Moreover, if $p_1 : \calM^s \times \Sigma \to \calM^s$ is the projection onto the stable locus of $\calM$, the sheaf $\calD_i = R^1p_{1,*}\calV_i$ is a vector bundle whose fibre at a stable Higgs bundle $(V,\Phi)$ is isomorphic to $\K^1 (\Sigma , \mathsf{V})$. These vector bundles $\calD_i$ are related to each other on the intersections by the line bundles $L_{ij}$. One may embed $\calD_i$ in the trivial bundle $\Omega = \Omega^{1,0}(V)\oplus \Omega^{0,1}(V) \times \calM^s$ on $\calM^s$, where $(V,\Phi) \in \calM^s$ (denote by $A$ the metric solving the Higgs bundle equation). The Hermitian metric on $V$ induces an $L^2$-metric on $\Omega$ and $\ker (\calD_{(A,\Phi)}^*) \subset \Omega^{1,0}(V)\oplus \Omega^{0,1}(V)$ has a natural unitary connection (coming from the orthogonal projection of the trivial connection on $\Omega$). Using the different interpretations of the operator $\calD_{(A,\Phi)}^*$ in different complex structures (e.g. in complex structure $J$ it can be viewed as the de Rham complex for the flat connection $\nabla_A + \Phi + \Phi^*$), one can prove that the unitary connection constructed in each $\calD_i$ is of type $(1,1)$ with respect to all complex structures of $\calM$, thus a hyperholomorphic vector bundle (see e.g. \cite[Theorem 2.6.3]{blaavand2015dirac}).  

We want to find a hyperholomorphic bundle supported on the moduli space of Higgs bundles for the Nadler group $\lan{G_0}$, where $G_0 = SU^*(2m)$, $SO^*(4m)$ and $Sp(m,m)$. First we look at the Dirac-Higgs operator associated to the representations of the Nadler group coming from the embedding $\lan{G_0} \hookrightarrow \lan{G}$. Consider the embedding 
$$g\in SL(2m,\C) \mapsto \diag (g,g) \in SL(4m,\C) .$$
Take a vector bundle $V$ of rank $2m$ with trivial determinant and assume that $(V,\Phi)$ lies in a generic fibre of the $SL(2m,\C)$-Hitchin fibration. Let $S = \zeros (p(\lambda))$ be the (smooth) spectral curve of $(V, \Phi)$, where $p(x) = x^{2m} + a_2 x^{2m-2}+\ldots+a_{2m}$, and denote by $Z$ the distinct zeros of $\det (\Phi) = a_{2m}$. Thus, the first hypercohomology at $(V,\Phi)$ gives         
\begin{align*}
\K^1 (\Sigma , \mathsf{V}) & \cong \bigoplus_{z\in Z} \coker (\Phi\oplus \Phi)_{z}\\
&  \cong \bigoplus_{z\in Z} (L\pi^*K \oplus L\pi^*K)_z ,
\end{align*}
where $\pi : S \to \Sigma$ is the spectral cover and $L \in \Prym$ (with $\pi_*L \cong V$). There is a similar picture for the group $SU^*(2m)$, but now the dual brane should be supported on the strictly semi-stable locus of the moduli space of $PGL(4m,\C)$-Higgs bundles and so, there is no natural candidate for the hyperholomorphic bundle (or sheaf). For the other two cases, one also encounters the difficulty of the dual brane being contained in the strictly semi-stable locus of the Higgs bundle moduli space. For $SO^*(4m)$, the Nadler group $\lan{SO^*(4m)} = Sp(2m,\C)$ sits inside $\lan{SO(4m,\C)} = SO(4m,\C)$ as $\diag (g, \tp{g}^{-1})$. Let $(V,\Phi)$ be a symplectic Higgs bundle (of rank $2m$) associated, via the \textit{BNR} correspondence, to the line bundle $L \in \PS$. If $Z$ is the fixed point set of the involution $\sigma$ on the spectral curve $S $ of $(V,\Phi)$, the hypercohomology at this point gives 
\begin{align*}
\K^1 (\Sigma , \mathsf{V}) & \cong \bigoplus_{z\in Z} (L\pi^*K^{(2m+1)/2} \oplus \sigma^*L\pi^*K^{(2m+1)/2})_z
\end{align*}
and the obstruction to having an universal bundle lies in $H^2(\calM (Sp(2m,\C), \Z_2))$. As in the $U(m,m)$ case treated by Hitchin in loc. cit., the fibre of the integrable system for $SO^*(4m)$ has several connected components and it is not clear which hyperholomorphic bundle should be taken for each component. Finally, for $G_0 = Sp(m,m)$ we consider the embedding 
\begin{align*}
\lan{Sp(m,m)} = Sp(m,m) & \to \lan{Sp(4m,\C)} = SO(4m+1, \C)\\
g & \mapsto \diag (g, \tp(g)^{-1}, 1).
\end{align*}
As remarked by Hitchin in \cite{hitchin2007langlands}, the zero eigenspace (associated to $1$ in the embedding) links $\lie{so}(4m+1, \C)$ and $\lie{sp}(4m,\C)$ in the duality and brings further technical difficulties.
 
\item \textbf{Vector bundles with many sections}: In Chapter \ref{Chapter8} we have studied the Lagrangian inside $\calM (Sp(2m,\C))$ obtained by Higgs bundles of the form $(V,\psi\otimes \psi)$, where $V$ was a symplectic vector bundle such that $V\otimes K^{1/2}$ admitted sections (for a fixed square root $K^{1/2}$ of the canonical bundle of $\Sigma$). In particular, this Lagrangian intersected the cotangent bundle to the moduli space of stable symplectic bundles on the conormal bundle of the determinant divisor. A natural generalization is to look at the locus inside $\calM (Sp(2m,\C))$ consisting of those $V$ such that $h^0(\Sigma , V\otimes K^{1/2}) = k > 1$. This is a natural generalization of Gaiotto Lagrangian and is closely related to lower-dimensional subvarieties contained in the determinant divisor of $\calU (Sp(2m,\C))$ (and conormal bundles of those). Recall that for any symplectic representation of the semisimple complex group $G$, we obtained a quadratic moment map. In particular, $\mu$ corresponds to a bilinear form and if we have a $k$-dimensional space of sections $H^0(\Sigma , V\otimes K^{1/2})$, then we may look at the image of the map 
$$\Sym^2 H^0(\Sigma , V\otimes K^{1/2}) \to H^0(\Sigma , \Sym^2 (V) \otimes K^{1/2})$$
inside the space of Higgs fields. By the same reasons given in Chapter \ref{Chapter7}, we expect to obtain an isotropic subvariety and we may ask if it is a Lagrangian inside the Higgs bundle moduli space and how it intersects the Gaiotto Lagrangian. To get an idea about the dimension of this subvariety $X_k$ inside $\calM (Sp(2m,\C))$, note that the $\bar{\partial}$-operator 
$$\bar{\partial} : \Omega^0 (\Sigma , V\otimes K^{1/2}) \to \Omega^{0,1} (\Sigma , V\otimes K^{1/2})$$ 
is formally symmetric (see \cite{spinors}). The finite-dimensional analog of this is the space of symmetric matrices. In that case, it is known that the locus of points where the kernel has dimension greater than or equal to $k$ is of codimension $k(k+1)/2$. Using Fredholm operators one should expect the same to hold in this infinite-dimensional setting, provided $V$ is generic. Now, from a $k$-dimensional space of spinors we can construct a $k(k+1)/2$-dimensional space of Higgs fields, thus obtaining that $X_k$ is also a Lagrangian. Another way to think of this is by shifting the problem to a problem of Brill-Noether locus for line bundles, where much more is known. To that end we consider the question of abelianisation for $\calU (Sp(2m,\C))$, which has to do with finding an open set of an abelian variety with a dominant rational map to $\calU (Sp(2m, \C))$ (see \cite{bnr, hitchin1987stable, donagi}). Thus, we consider a Prym variety $\PS$ given by a ramified $2$-covering 
$$\rho : S \to \bar{S},$$ 
where $\bar{S}$ is the quotient of $S$ by an involution $\sigma : S \to S$. Denote by $\pi : S \to \Sigma$ the corresponding $2m$-ramified covering and note that $H^0(\Sigma, V\otimes K^{1/2}) = H^0(S, U\pi^*K^m)$, where $V$ is a symplectic vector bundle obtained by the direct image of the line bundle $L= U\pi^*K^{(2m-1)/2}$, with $U \in \PS$. In \cite{MR1429335}, Kanev studies Brill-Noether loci related to Prym varieties which arise in this way and obtains that if the locus 
$$W^r = \{ U \in \PS \ | \ h^0(S, U\pi^*K^m) > r \}$$
is non-empty, every irreducible component of $W^r$ has dimension at least $\dim Sp(2m,\C)(g-1) - (r+2)(r+1)/2$. Also, the Zariski open set of $W^r$ containing line bundles $U$ with  $h^0(S, U\pi^*K^m) = r+1$ and whose kernel of the Petri map is $\Lambda^2 H^0(S, U\pi^*K^m)$ is smooth of dimension $\dim Sp(2m,\C)(g-1) - (r+2)(r+1)/2$, which is what we expect to get a Lagrangian. We finish by mentioning an example when the genus of $\Sigma$ is $3$ and the rank of the bundle is $2$. In that case, the moduli space $\calU_\Sigma (2,K)$ of semi-stable rank $2$ bundles of fixed determinant $K$ (this is isomorphic to $\calU_\Sigma (Sp(2,\C))$) is known to be embedded in $\CP^7$ as a so-called \textit{Coble quartic}. Moreover, the Brill-Noether locus $W= W^0$ is such that $W^1$ is a Veronese cone with vertex $W^2$, a unique point. So there is a unique vector bundle with $3$-dimensional space of sections and it turns out that this is actually a smooth point in this moduli space (see \cite{MR1467474}). Thus, the subvariety $X_3$ inside the moduli space of $SL(2,\C) = Sp(2,\C)$-Higgs bundles in this case is indeed a Lagrangian, which is the closure of the fibre of the cotangent bundle inside the Higgs bundle moduli space at this stable bundle $W^2$. Moreover, the usual Gaiotto Lagrangian is contained in $X_3$ as a quadric cone. 

\item \textbf{An orthogonal analogue}: In the second part of the thesis we have considered a distinguished (Lagrangian) subvariety of the Higgs bundle moduli space obtained from a symplectic representation of the group. In particular, we describe in detail in Chapter \ref{Chapter8} this subvariety for the standard representation of the symplectic group. Let us briefly describe a natural analogue of this construction for the orthogonal group. 
Fix a square root $K^{1/2}$ of the canonical bundle of $\Sigma$. Recall that for any vector bundle $W$ on $\Sigma$ with $O(n,\C)$-structure we can assign 
an analytic mod $2$ index
$$\varphi_\Sigma (W) =  h^{0}(\Sigma , W\otimes K^{1/2}) \ \text{mod }2$$ 
which, by Theorem $1$ in \cite{hitchin2013higgs}, satisfies 
$$w_2 (W) = \varphi_\Sigma (W) + \varphi_\Sigma (\det (W)),$$
where $w_2$ is the second Stiefel-Whitney class. So, let $V$ be an orthogonal vector bundle. If $\varphi_\Sigma (V)$ is one, there are always spinors (i.e., global holomorphic sections of $V\otimes K^{1/2}$). If it is zero, the generic stable orthogonal vector bundle does not admit any non-trivial spinor and the constraint $h^0(\Sigma , V \otimes K^{1/2}) \neq 0$ gives a divisor $D_{K^{1/2}}$ on the moduli space of stable orthogonal bundles. Moreover, the associated $\bar{\partial}$-operator
$$\bar{\partial} : \Omega^0 (\Sigma , V\otimes K^{1/2}) \to \Omega^{0,1} (\Sigma , V\otimes K^{1/2})$$
is formally skew-adjoint and the Quillen determinant line bundle for this family of $\bar{\partial}$-operators is the square of a Pfaffian. In particular, if $V$ is on the divisor and $\psi_1$ and $\psi_2$ are two non-zero spinors, we obtain $\psi_1 \wedge \psi_2 \in H^0(\Sigma , \Lambda^2V \otimes K)$. Since $\Lambda^2V$ is isomorphic to the adjoint bundle of $V$, this is a Higgs field for the orthogonal group. Although this may no longer be maximally isotropic, it provides a distinguished subvariety of the Higgs bundle moduli space for the orthogonal group which appears as an analogue of the construction for the symplectic group. Apart from the interesting task of studying the geometry of this subvariety, one could try and see what is the corresponding locus inside the symplectic group under the duality between Higgs bundles for the odd orthogonal group and Higgs bundles for the symplectic group (for more details on this duality see \cite{hitchin2007langlands}).         
\end{enumerate}

%% file: Appendices/AppendixA.tex

\chapter{Lie theory} 

\label{Lie theory} 

\lhead{Appendix A. \emph{Lie theory}} 

In this subsection we recall some basic facts in Lie theory and establish notation used throughout the text. 

Let $\mathfrak{g}$ be a finite-dimensional Lie algebra over the field of real or complex numbers. Consider the algebraic group $\Aut (\lie{g})$ of automorphisms of $\lie{g}$. We denote by $\Int (\lie{g})$ the (normal) subgroup of $\Aut (\lie{g})$ generated by elements of the form $\exp (\ad x)$, $x \in \lie{g}$, and call it the \textbf{group of inner automorphisms} of $\lie{g}$. Also, the quotient $\Out (\lie{g}) = \Aut (\lie{g}) / \Int (\lie{g})$ is called the \textbf{group of outer outer isomorphisms} of $\lie{g}$. In particular, the Lie algebra of $\Int (\lie{g})$ is the \textbf{adjoint algebra} $\ad (\lie{g})$, an ideal of the Lie algebra of $\Aut (\lie{g})$, which is the algebra $\der (\lie{g})$ of derivations of $\lie{g}$. Moreover, the image $\Ad (G)$ of the adjoint representation of $G$ is a normal subgroup of $\Aut (\lie{g})$ called the \textbf{adjoint group} of the Lie group $G$. When $G$ is connected, $\Ad (G) = \Int (\lie{g})$ and $\ker \Ad = Z (G)$ is the center of $G$.

\begin{rmk} If $\lie{g}$ is semisimple, $\der (\lie{g}) = \ad (\lie{g})$ and $\Int (\lie{g})$ is the connected component $\Aut (\lie{g})^\circ$ containing the identity of $\Aut (\lie{g})$. 
\end{rmk}

Let $\lie{g}$ be a Lie algebra over the complex numbers. A real subalgebra $\lie{g}_0 \subseteq \lie{g}$ is called a \textbf{real form} of $\lie{g}$ if the natural homomorphism of complex Lie algebras $ \lie{g}_0\otimes_\R \C \to \lie{g}$ is an isomorphism. Note that the real form $\lie{g}_0$ defines an anti-linear involution on $\lie{g} \cong \lie{g}_0\oplus i\lie{g}_0$ by
$$\sigma : x + iy \to x + -iy.$$
Conversely, given an anti-linear involution $\sigma : \lie{g} \to \lie{g}$, the subalgebra of fixed points $\lie{g}^\sigma$ defines a real form of $\lie{g}$. This correspondence gives a bijection between isomorphism classes of real forms and isomorphism classes of anti-linear involutions, or \textbf{conjugations}, of $\lie{g}$ (where two conjugations $\sigma_1$ and $\sigma_2$ are equivalent if $\sigma_1 = \alpha \sigma_2 \alpha^{-1}$, for some $\alpha \in \Aut (\lie{g})$).

Complex semisimple Lie algebras always have special real forms. A real form $\lie{g}_0$ of $\lie{g}$ is called \textbf{compact} if $\Int (\lie{g}_0)$ is a compact group. For the semisimple case, this is equivalent to $\lie{g}_0$ being a Lie algebra of a compact Lie group, or that its Killing form $B_{\lie{g}_0}(x,y) = \tr (\ad(x) \ad(y))$ is negative definite. There exists a unique compact real form, say $\lie{k}$, up to conjugacy by an inner automorphism, of a complex semisimple Lie algebra $\lie{g}$ and one can use this to get a correspondence\footnote{The correspondence between equivalence classes of anti-linear involutions and equivalence classes of linear involutions restricts to a bijection when the equivalence is considered by inner automorphisms.} between real forms and isomorphism classes of $\C$-linear involutions of $\lie{g}$. The correspondence goes as follows. Given an anti-linear involution $\sigma$, Cartan showed (see, e.g., \cite{onishchik2004lectures}) that one can always find another anti-linear involution $\tau$ corresponding to a compact real form of $\lie{g}$ which commutes with $\sigma$. One then assigns $\sigma$ to $\theta \coloneqq \sigma \circ \tau = \tau \circ \sigma$. 

\begin{defin} Let $\lie{g}_0$ be a real Lie algebra. An involution $\theta : \lie{g}_0\to \lie{g}_0$ is called a \textbf{Cartan involution} if the symmetric bilinear form on $\lie{g}_0$ given by $B_\theta (x, y) \coloneqq - B(x, \theta y) $ is positive definite.
\end{defin}

The involution $\theta$ corresponding to a non-compact real form $\sigma$ of a complex semisimple Lie algebra $\lie{g}$ satisfies this property. 

\begin{claim} $\theta|_{\lie{g}_0}$ is a Cartan involution.
\end{claim}
\begin{proof} First notice that both $\tau$ and $\sigma$ commute with $\theta$. In particular, this means that $\theta|_{\lie{g}_0} : \lie{g}_0\to \lie{g}_0$ and $\theta|_\mathfrak{k} : \mathfrak{k} \to \mathfrak{k}$. Moreover, the eigenspace decomposition defined by the involution $\theta$ induces eigenspace decompositions
\begin{align*}
\lie{g}_0& = \lie{g}^+ \oplus \lie{g}^- \\
\mathfrak{k} & = \mathfrak{k}^+ \oplus \mathfrak{k}^- 
\end{align*}
where the superscripts $\pm$ indicate the $\pm 1$-eigenspaces with respect to $\theta$. Let $\mathfrak{h} \coloneqq \lie{g}^+$ and $\mathfrak{m} \coloneqq \lie{g}_0^-$. Since $\theta|_{\lie{g}_0} = \sigma|_{\lie{g}_0}$ and $\theta|_{\mathfrak{k}} = \tau|_{\mathfrak{k}}$,
$$\begin{array}{ccccc}
\mathfrak{k}^+ &=& \mathfrak{h} &=& \lie{g} \cap \mathfrak{k} \nonumber \\
\mathfrak{k}^- &=& i \mathfrak{m} &=& i \lie{g} \cap \mathfrak{k}. \nonumber 
\end{array}$$
It is now straightforward to show, from the fact that the Killing form of $\mathfrak{k} = \mathfrak{h} \oplus i \mathfrak{m}$ ($B_\mathfrak{k} = B_{\lie{g}}|_{\mathfrak{k} \times \mathfrak{k}}$) is negative definite, that the restriction of $\theta$ to $\lie{g}_0$ is a Cartan involution.   
\end{proof}
Note that,
\begin{itemize}
\item The decomposition $\lie{g}_0= \mathfrak{h} \oplus  \mathfrak{m}$ is orthogonal with respect to $B$ and $B_\theta$ (this follows directly from the fact that the Killing form is invariant under automorphisms).
\item This decomposition satisfies $$ [\mathfrak{h},\mathfrak{h}] \subseteq \mathfrak{h}, \ \ \ \ [\mathfrak{h},\mathfrak{m}] \subseteq \mathfrak{m}, \ \ \ \ [\mathfrak{m},\mathfrak{m}] \subseteq \mathfrak{h} .$$
\item The Killing form is negative definite on $\mathfrak{h}$ and positive definite on $\mathfrak{m}$. In particular, $\mathfrak{h}$ is a compact Lie algebra. If $G_0$ is a connected semisimple Lie group with Lie algebra $\lie{g}_0$, then the Lie subgroup $H \coloneqq < \exp \mathfrak{h}>$ of $G_0$ generated by the exponential map is a maximal compact subgroup of $G_0$.
\item It can be shown that any two Cartan involutions of $\lie{g}_0$ are conjugate via an inner automorphism (\cite{knapp2002lie}).
\end{itemize}

\begin{defin} A direct sum of the form 
$$\lie{g}_0= \mathfrak{h} \oplus  \mathfrak{m}$$
is called a \textbf{Cartan decomposition} for the Lie algebra $\lie{g}_0$ if
$$ [\mathfrak{h},\mathfrak{h}] \subseteq \mathfrak{h}, \ \ \ \ [\mathfrak{h},\mathfrak{m}] \subseteq \mathfrak{m}, \ \ \ \ [\mathfrak{m},\mathfrak{m}] \subseteq \mathfrak{h}$$
and the Killing form is negative definite on $\mathfrak{h}$ and positive definite on $\mathfrak{m}$. 
\end{defin}

Now, starting with a Cartan decomposition $\lie{g}_0= \mathfrak{h} \oplus  \mathfrak{m}$, one can define the involution $\theta$ whose $\pm$1-eigenspace is $\mathfrak{h}$ and $\mathfrak{m}$, respectively. Extending this to an involution on the complexification of $\lie{g}_0$ (denoted by the same symbol) and considering the anti-involution $\sigma$ corresponding to $\lie{g}$, we find that they commute. Thus, $\tau \coloneqq \sigma \circ \theta$ is an anti-involution which commutes with $\sigma$ and is easily checked to be compact. 

We conclude that any Cartan decomposition of a real semisimple Lie algebra $\lie{g}_0$ can be obtained using a compact real form of the complexification of $\lie{g}_0$, as described previously.

\section{The Nadler group}\label{comoutationnad}

Let $G_0$ be a real form of a connected reductive complex algebraic group $G$. In \cite{nadler2005perverse}, Nadler constructs from $G_0$ a complex subgroup of the connected dual group $\lan{G}$, extending the idea that the geometry of a connected group over $\C$ is reflected in the representation theory of its dual group. This complex subgroup will be denoted by $\lan{G}_0$ and hereinafter called the \textbf{Nadler group} of the real form $G_0 \subset G$. On page 3 loc. cit. one can find a table containing the Lie algebra $\lan{\lie{g}_0}$ of the Nadler group associated to a real form $G_0$ of a simple group $G$ and in Section 10.7 loc. cit. a more self-contained description of the Nadler group is given. In this section we will describe in some detail the Nadler group associated to the real form $SU^*(2n)$ of $SL(2n,\C)$. We also mention how the Nadler group of $SO^*(4m)$ and $Sp(m,m)$ can be realized inside $SO(4m,\C)$ and $SO(4m+1, \C)$, respectively. 

We begin with some structure theory. Let $\lie{g}_0$ be a real Lie algebra with $\lie{g}_0 = \lie{h}\oplus\lie{m}$ a Cartan decomposition. Then $\lie{k} = \lie{h} \oplus i\lie{m}$ is a compact form of the complexification $\lie{g}$ of $\lie{g}_0$. Let $\sigma$ and $\tau$ denote the conjugations of $\lie{g}$ with respect to $\lie{g}_0$ and $\lie{k}$, respectively. Let $\theta = \sigma\tau = \tau\sigma$ be the corresponding automorphism of $\lie{g}$.

One can choose real subalgebras 
\[\lie{g}_0 \supseteq \lie{p}\supseteq \lie{l} \supseteq \lie{t} \supseteq \lie{a} \]
that dictates the structure of $\lie{g}_0$. Here, $\lie{p}$ is minimal parabolic, $\lie{l}$ is its Levi factor, $\lie{t}$ is a maximal toral subalgebra and $\lie{a}$ is a maximally split toral subalgebra. The recipe to find them is as follows. First, from the Cartan decomposition, pick 
\[\lie{a} \subseteq \lie{m} \quad\textup{ any maximal abelian subalgebra.}\]
Another choice $\lie{a}'$ satisfies $\lie{a}' = \Ad(h)\lie{a}$ for a $h\in H (=$ analytic subgroup with Lie algebra $\lie{h}$). Then, define
\[\lie{l} :=\lie{a}\oplus Z_\lie{h}(\lie{a}).\]
Next, there is a root space decomposition $\lie{g}_0 = \lie{l} \oplus\bigoplus_{\lambda\in\Sigma}\lie{g}_\lambda.$
Here $\Sigma =\{\lambda\in\lie{a}^* \ | \ \lie{g}_\lambda\neq 0\}$, where $\lie{g}_\lambda = \{X\in \lie{g}_0 \ | \  [H,X]=\lambda(H)\textup{ for all }H\in\lie{a}\}.$ The set $\Sigma = \Sigma(\lie{g}_0,\lie{a})$ forms a root system (not necessarily reduced) called the \textbf{restricted roots}. Put an order on $\lie{a}^*$ (e.g., the lexicographic order with respect to a fixed basis of $\lie{a}^*$) and denote by $\Sigma^+$ the set of positive elements on $\Sigma$. Then, the algebra $\lie{p}$ is given by \[\lie{p}= \lie{l} \oplus\bigoplus_{\lambda\in\Sigma^+}\lie{g}_\lambda.\]
So far we have $\lie{g}_0 \supseteq \lie{p}\supseteq \lie{l} \supseteq \lie{a}$.

Further, let $\Phi = \Phi(\lie{g},\lie{c}^\C)$ be the root system of the complexified algebra $\lie{g}$ with respect to any Cartan subalgebra $\lie{c}^\C$. We let $(\lie{c}^\C)_\R = \{H\in\lie{c}^\C\mid\alpha(H)\in\R$ for all $\alpha\in\Phi\}$. Now choose 
\[\lie{t}\subseteq \lie{l} \quad\textup{ any maximal abelian subalgebra containing } \lie{a}.\]
Then, $\lie{t}^\C$ is a Cartan subalgebra of $\lie{g}^\C$ \cite[VI, Lemma 3.2]{helgason1979differential} with $(\lie{t}^\C)_\R = \lie{a} + i(\lie{t} \cap \lie{h})$ (as $\lie{t}^\C$ is a Cartan, we could pick any $\lie{t}'\subseteq\lie{g}$ containing $\lie{a}$ and conjugate to $\lie{a}\subseteq\lie{t}\subseteq\lie{l}$).  In this situation, the inclusion $\lie{a}\hookrightarrow (\lie{t}^\C)_\R$ induces a map $\res:(\lie{t}^\C)_\R^* \to \lie{a}^*$ by restriction. Then, $\res$ induces a surjection (of sets) \cite[pp.260--264]{helgason1979differential}
\[\res:\Phi(\lie{g},\lie{t}^\C)\twoheadrightarrow\Sigma(\lie{g}_0,\lie{a})\cup\{0\}\]
and we can consider $\Sigma\subseteq (\lie{t}^\C)_\R^*$. If we denote by $\alpha^\theta$ the functional $\alpha^\theta(H) = \alpha(\theta H)$ of $(\lie{t}^\C)_\R$, one has $\res(\alpha) = \tfrac{1}{2}(\alpha-\alpha^\theta)$ \cite[p.530]{helgason1979differential}. Putting compatible orderings on $\lie{a}^*$ and $(\lie{t}^\C)_\R$ (e.g., lexicographic with respect to a basis of $(\lie{t}^\C)_\R$ in which the first elements form a basis of $\lie{a}$), the map $\res$ sends $\Phi^+\to\Sigma^+\cup\{0\}$. Moreover
\[\dim_\R\lie{g}_\lambda = \#\{\alpha\in\Phi\mid\res(\alpha) = \lambda\}.\] 

Now, let us focus on the Lie algebra $\lie{g}_0 = \lie{su}^*(2n)$. This algebra has an explicit realization as
\[\lie{su}^*(2n) = \left\{
\begin{pmatrix}
\alpha&\beta\\-\bar\beta&\bar\alpha
\end{pmatrix}
\right\}\cap \lie{sl}(2n,\C).\]
With respect to the (real) involution $\tau:X\mapsto -X^*$ of $\lie{g} = \lie{sl}(2n,\C)$ we have $\lie{g}_0 = \lie{h}\oplus\lie{m}$ with
\[\lie{h}=\lie{sp}(n) \quad\textup{and}\quad\lie{m}=\left\{
\begin{pmatrix}
\alpha&\beta\\-\bar\beta&\bar\alpha
\end{pmatrix} \ | \ \alpha=\alpha^*,\tr(\alpha)=0,\beta\in\lie{so}(n,\C)
\right\}.\]
One checks that $\tau$ is the conjugation of $\lie{g}$ with respect to $\lie{k} = \lie{su}(2n)$. Also, the (real) involution $\sigma:X\mapsto J\bar X J^{-1}$, where $J=\left(\begin{smallmatrix} 0 & 1\\-1&0 \end{smallmatrix}\right)$ is the conjugation of $\lie{g}$ with respect to $\lie{g}_0$. Then, $\theta=\sigma\tau = \tau\sigma$ is the involution of $\lie{g}$ given by $X\to -J \ \tp{X}J^{-1}$. 

Pick $\lie{a}\subseteq\lie{m}$ to be the diagonal matrices, so that
\[\lie{a}=\left\{
\begin{pmatrix}
\delta&0\\0&\delta
\end{pmatrix} \ | \ \delta=\diag(d_1,\ldots,d_n)\in\R^n,\tr(\delta)=0
\right\}.\]
One checks that
\begin{align*}
Z_\lie{h}(\lie{a}) &=\left\{ X=
\begin{pmatrix}
\alpha&\beta\\-\bar\beta&\bar\alpha
\end{pmatrix}\in\lie{sp}(n) \ | \  [X,H] =
\begin{pmatrix}
[\alpha,\delta]&[\beta,\delta]\\
[\delta,\bar\beta]&[\bar\alpha,\delta]
\end{pmatrix}=0 \textup{ for all } H\in\lie{a}
\right\} \\
&= \left\{
\begin{pmatrix}
i\alpha&\beta\\-\bar\beta&-i\alpha
\end{pmatrix} \ | \ \alpha=\diag(a_1,\ldots,a_n)\in \R^n,\beta=\diag(b_1,\ldots,b_n)\in\C^n
\right\}.
\end{align*}
Hence,
\[\lie{l} =\lie{a}\oplus Z_\lie{h}(\lie{a}) = \left\{
\begin{pmatrix}
\delta+i\alpha & \beta\\-\bar\beta&\delta-i\alpha
\end{pmatrix} \ | \ \delta,\alpha\in\R^n,\tr(\delta)=0,\beta\in\C^n
\right\}.\]
Note that $\dim_\R\lie{l} = 4n-1$ and that $\lie{l} \cong \lie{su}(2)^n\oplus\R^{n-1}$. Now we look at the roots. Let $\epsilon_i$ be the coordinate functionals in $(\R^{2n})^*$ and let 
\[\lambda_{i,j} :=\tfrac{1}{2}(\epsilon_i - \epsilon_j +\epsilon_{i+n}-\epsilon_{j+n}),\] 
with $1\leq i\neq j \leq n$. It follows that $\lie{a}^* = \R(\lambda_{1,2})\oplus\cdots\oplus\R(\lambda_{n-1,n})\subseteq(\R^n)^*$. Let this induce the ordering of $\lie{a}^*$. We then have
\[\Sigma = \Sigma(\lie{g}_0,\lie{a}) = \{\lambda_{i,j}\mid 1\leq i \neq j\leq n\},\quad\textup{and}\quad\Sigma^+=\{\lambda_{i,j}\mid1\leq i < j\leq n\}.\]
Indeed, note that the space
\[\lie{g}_{\lambda_{i,j}} = 
\R\begin{pmatrix}
E_{i,j}&0\\0&E_{i,j}
\end{pmatrix}
\oplus 
\R\begin{pmatrix}
iE_{i,j}&0\\0&-iE_{i,j}
\end{pmatrix}
\oplus 
\R\begin{pmatrix}
0&E_{i,j}\\-E_{i,j}&0
\end{pmatrix}
\oplus 
\R\begin{pmatrix}
0&iE_{i,j}\\iE_{i,j}&0
\end{pmatrix}
\]
is contained in $\lie{g}_0$ and satisfies $[H,X] = (d_i-d_j)X$ for all $H\in\lie{a}$ and $X\in\lie{g}_{\lambda_{i,j}}$. Hence, certainly,
\[\lie{l}\oplus\bigoplus_{i\neq j}\lie{g}_{\lambda_{i,j}}\subseteq\lie{g}_0.\]
As the left hand side has  dimension $4n - 1 + 4(n(n-1)) = 4n^2 - 1 = \dim_\R\lie{g}$, this must be the root space decomposition of $\lie{g}_0$. With the present ordering, we have $\Sigma^+ = \{\lambda_{i,j}\mid i < j\}$. Hence,
\[\lie{p} = \lie{l} \oplus\bigoplus_{\lambda\in\Sigma^+}\lie{g}_\lambda = 
\left\{
\begin{pmatrix}
A&B\\-\bar B&\bar A
\end{pmatrix}
 \ | \  A,B\in\lie{gl}(n,\C)\textup{ strictly upper triangular }\right\}.
\]
Finally, take $\lie{t}\subseteq\lie{l}$ to be the diagonal matrices. Hence,
\[
\lie{t} = \left\{
\begin{pmatrix}
\delta+i\alpha&0\\0&\delta-i\alpha
\end{pmatrix} \ | \ \delta,\alpha\in\R^n,\tr(\delta)=0\right\},
\]
from which 
\[ \lie{t}^\C = \{\textup{diagonal in }\lie{gl}(2n,\C)\}\cap \lie{sl} (2n,\C) \quad\textup{and}\quad (\lie{t}^\C)_\R = \{\textup{diagonal in }\lie{gl}(2n,\R)\}\cap \lie{sl}(2n,\R ). \]
Hence, $(\lie{t}^\C)_\R^*$ is naturally identified with the subspace of $(\R^{2n})^*$ that vanishes on $\R(1,1,\ldots,1)$.

Recall that the restriction function $\res:(\lie{t}^\C)_\R^*\to\lie{a}^*$ is given by $\alpha\mapsto\frac{1}{2}(\alpha - \alpha^\theta)$, where $\theta: X\mapsto-J \ \tp{X}J^{-1}$.  We note that since $\theta$ and $\sigma$ coincide on $(\lie{t}^\R)_\R^*$, we have $\alpha^\theta = \alpha^\sigma$. We can thus define $\res$ with respect to $\sigma$. The root system 
$$\Phi=\Phi(\lie{g},\lie{t}^\C) = \{\alpha_{i,j}=\epsilon_i-\epsilon_j\mid 1\leq i\neq j\leq 2n\}.$$
If $1\leq i\neq j\leq 2n$ note that
\[(\alpha_{i,j})^\theta = -\alpha_{\overline{i+n},\overline{j+n}},\]
where an index $\overline{k}$ means $k \mod 2n$. Hence, as 
$$\lambda_{i,j} = \tfrac{1}{2}(\alpha_{i,j} + \alpha_{i+n,j+n}) = \tfrac{1}{2}(\alpha_{i+n,j}+\alpha_{i,j+n}),$$
we have
\[\{\alpha\in\Phi\mid\res(\alpha) = \lambda_{i,j}\} = \{\alpha_{i,j},\alpha_{n+i,n+j},\alpha_{i+n,j},\alpha_{i,j+n}\mid 1\leq i\neq j\leq n\}\]
and the roots $\{\alpha_{i,n+i},\alpha_{n+i,i}\mid 1\leq i \leq n\}$ are sent to zero. We note that to have compatible orderings, we take $\Phi^+$ to be induced by the simple roots 
\[\{\alpha_{1,n+1},\alpha_{n+1,2},\alpha_{2,n+2},\alpha_{n+2,3},\ldots,\alpha_{n-1,2n-1},\alpha_{2n-1,n},\alpha_{n,2n}\}.\]

Now that we know the structure of the real Lie algebra $\lie{g}_0$, we realize the Lie algebra $\lan{g}_0$ of the associated Nadler group as in \cite{nadler2005perverse}. We start by noting that we have 
\[\lie{g}\supseteq\lie{p}^\C\supseteq\lie{l}^\C\supseteq\lie{t}^\C\supseteq \lie{a}^\C \]
\begin{align*}
\lie{a}^{\C} &=\left\{
\begin{pmatrix}
A&0\\0& A
\end{pmatrix}
 \ | \  A=\diag(a_1,\ldots,a_n)\in \C^n,\tr(A) =0\right\}\\
\lie{t}^{\C} &=\{ H = \diag(h_1,\ldots h_{2n})\in\C^{2n} \ | \ \tr(H) = 0\}\\
\lie{l}^{\C} &= \lie{a}^\C \oplus \left\{
\begin{pmatrix}
D&B\\C& -D
\end{pmatrix}
 \ | \  D, B, C \in\C^n\right\} \cong \lie{a}^\C\oplus\lie{sl}(2,\C)^n\\
\lie{p}^\C &= \lie{l}^{\C}\oplus \left\{
\begin{pmatrix}
N_1&N_2\\N_3&N_4
\end{pmatrix}
 \ | \  N_j\in\lie{gl}(n,\C)\textup{ strictly upper triangular}\right\}\\
\lie{g} &= \lie{sl}(2n,\C).
\end{align*}
Note that $\lie{a}^\C$ is the center of $\lie{l}^\C$. Hence, $\Phi_{\lie{l}} = \Phi(\lie{l}^\C,\lie{t}^\C) = \{$roots of $\lie{l}\}$ equals $\{\alpha_{i,n+i},\alpha_{n+i,i}\mid 1\leq i\leq n\}$. The positive are $\Phi_{\lie{l}}^+=\{\alpha_{i,n+1}\}$. The dual Lie algebra is $\lan{\lie{g}} = \lie{sl}(2n,\C) $. Identify $\lie{g}$ and $\lan{\lie{g}}$, and, by a slight abuse of notation, we denote by $\lan{\lie{t}}$ the dual algebra associated to $\lie{t}^\C$. Also, let 
$$\Psi = \Psi(\lan{\lie{g}},\lan{\lie{t}})= \{\psi_{i,j}\mid 1\leq i\neq j\leq 2n\}$$ 
be the root system dual to $\Phi$. Everything next will depend on these identifications. 

A root of $\lie{g}$ is viewed as a coroot of $\lan{\lie{g}}$. Hence, $\alpha = \alpha_{i,j}:\lie{t}^\C \to\C$ is viewed as a function $\check\alpha_{i,j}:\C\to \lan{\lie{t}}$. Explicitly, $\check\alpha_{i,j}(z) = z(E_{i,i} - E_{j,j})\in \lan{\lie{t}} \subseteq\lan{\lie{g}}$. Hence, the coroot $2\check\rho_M = \sum_{\Phi_{\lie{l}}^+}\check\alpha$ is the function
\[2\check\rho_{\lie{l}}:z\mapsto\begin{pmatrix}
z1_n&0\\0& -z1_n
\end{pmatrix}, \]
where $1_n$ is the identity. The subalgebra $\lan{\lie{l}}_0$ is defined as $\lan{\lie{l}}_0= Z_{\lan{\lie{g}}}(2\check\rho_M)$. It is easy to check that
\[\lan{\lie{l}}_0 = \left\{
\begin{pmatrix}
A&0\\0& B
\end{pmatrix} 
 \ | \  A,B\in\lie{gl}(n,\C),\tr(A)+\tr(B) = 0\right\}\]
The roots of $\lan{\lie{l}}_0$ are $\Psi_0 = \{\psi \in \Psi \ | \ \inn{\alpha}{2\check\rho_M} = 0\}$ which are \[\Psi_0=\{\psi_{i,j} \ | \ 1\leq i\neq j\leq n\}\cup\{\psi_{i,j}\mid n\leq i\neq j\leq 2n\}.\] 
The conjugation $\sigma$ induces an involution on $\Hom (\C , \lie{t}^\C)$ by sending a coweight $\lambda : \C \to \lie{t}^\C$ to the coweight defined by the composition 
$$\sigma (\lambda) : \C \xrightarrow{\textbf{c}} \C \xrightarrow{\lambda} \lie{t}^\C \xrightarrow{\sigma} \lie{t}^\C,$$
where $\textbf{c}$ is the complex conjugation. This can be extended to the whole algebra, which we also denote by $\sigma$, and equals the dual involution $\lan{\sigma}$ of $\lan{\lie{g}}$. More precisely, 
$$\lan{\sigma} : x \mapsto J x J^{-1}.$$   
Note that the Dynkin diagram of $\lan{\lie{l}}_0$ has two connected components (of type $A_{n-1}$) and $\lan{\sigma}$ acts on the diagram by permuting these components. Thus, $\psi_{i,j} + \sigma (\psi_{i,j})$ is not a root. This means that, for any choice of $\{X_\alpha \ | \ X_\alpha\in \lie{g}\}$, we have $[X_\alpha,X_{\sigma (\alpha)}]=0$. Hence, we have 
\[\lan{\lie{l}}_0 = \lan{\lie{l}}_1 \]
by \cite[Remark 10.7.1]{nadler2005perverse}.  

\begin{rmk}
Nadler says in Section 10.2 loc. cit. that the dual group comes equipped with a distinguished choice of Borel subgroup. For our particular example, however, this will not matter, as we have $ \lan{\lie{l}}_0 =  \lan{\lie{l}}_1$.
\end{rmk} 

Thus, the Lie algebra of the corresponding Nadler group is realized as the fixed points of $\lan{\lie{l}}_0$ with respect to $\lan{\sigma}$. Therefore
\begin{equation}
\lan{\lie{su}}^*(2n) = \left\{
\begin{pmatrix}
A&0\\0& A
\end{pmatrix}
 \ | \  \tr(A) = 0\right\} \cong \lie{sl}(n,\C) \subset \lie{sl}(2n,\C)
 \label{nadlersu*}
\end{equation}
and the Nadler group of $SU^*(2n) \subset SL(2n,\C)$ is $PGL(n,\C)$, which sits inside $PGL(2n,\C)$ as two equal blocks in the diagonal.
Following this recipe we find that the Nadler group of $SO^*(4m) \subset SO(4m, \C)$ and $Sp(m,m) \subset Sp(4m, \C)$ is $Sp (2m,\C)$ (inside $SO(4m,\C)$ and $SO(4m+1, \C)$, respectively). These can be realized as
  \begin{equation}
\lan{\lie{so}}^*(4m) = \left\{
\begin{pmatrix}
A&0\\
0& -\tp{A}
\end{pmatrix}
 \ | \  A\in \lie{sp}(2m,\C) \right\} \cong \lie{sp}(2m,\C) \subset \lie{so}(4m,\C) 
 \label{nadlerso*}
\end{equation}
and 
\begin{equation}
\lan{\lie{sp}}(m,m) = \left\{
\begin{pmatrix}
A&0&\\
0& -\tp{A}&\\
 &        & 0
\end{pmatrix}
 \ | \  A\in \lie{sp}(2m,\C) \right\} \cong \lie{sp}(2m,\C) \subset \lie{so}(4m+1,\C) . 
 \label{nadlerspmm}
 \end{equation}

%% file: Appendices/AppendixB.tex

\chapter{Spectral sequences and hypercohomology} 

\label{AppendixB} 

\lhead{Appendix B. \emph{Spectral sequences and hypercohomology}} 

In order to set up notation and state some basic facts that are used throughout the text, we recall the notion of hypercohomology for a complex of sheaves. Also, we introduce spectral sequences obtained from filtrations and explain how one can use these to extract information about the cohomology of the complex. This material is standard and we refer to \cite{loop, lectures, gh} for more details.      

\section{Preliminary notions} 

Let $\textbf{Ab}$ be the category of abelian groups\footnote{In this section, there is no harm in substituting $\textbf{Ab}$ for any abelian category, such as the category of $R$-modules, where $R$ is any ring, the category of sheaves of abelian groups, or $\mathcal{O}_X$-modules, on a topological space $X$, the category of (quasi-)coherent sheaves on an algebraic variety, etc.}.  

\begin{defin} A \textbf{spectral sequence} is a sequence $\{ E_r , d_r\}$, $r \geq 0$, of bigraded objects in $\textbf{Ab}$, $E_r = \{ E_r^{p,q} \}$, $p, q \geq 0$, together with morphisms $d_r = \{ d_r^{p,q} \}$, $p, q \geq 0$,
\begin{center}
$d_r^{p,q} : E_r^{p,q} \to E_r^{p + r, q - r +1}, \qquad d_r^{ p + r, q - r +1 } \circ d_r^{p,q} = 0$
\end{center}
such that the cohomology of $E_r$ is $E_{r+1}$, i.e., $E_{r+1}^{p, q} \cong  \dfrac{\ker d_r^{p,q}}{ \im d_r^{ p - r, q + r -1 }}$.
\end{defin}

Although this definition may seem rather arbitrary at first, we will see that it appears naturally. Before that, however, let us show how one can obtain spectral sequences from filtrations of a complex. 

Let $(K^\bullet , d^\bullet)$  
$$K^0 \overset{d}{\to} K^1 \overset{d}{\to} K^2 \cdots$$
be a complex in $\textbf{Ab}$ (very often we will simply write $K^\bullet$). A \textbf{subcomplex} $L^\bullet$ of $K^\bullet$ consists of a family $L^p \subseteq K^p$ of subobjects such that $d(L^p) \subseteq L^{ p+1}$ for every $p$.

\begin{rmk} Given a subcomplex $L^\bullet$ of $K^\bullet$, one can define a new complex $K^\bullet / L^\bullet$, whose $n$-th object is the quotient $(K^\bullet / L^\bullet)^n = K^n / L^n$, and the differential is the induced one.  
\end{rmk}

\begin{defin} A \textbf{filtration}\footnote{Note that we are only considering complexes bounded below and filtrations can be defined in  a more general setting, i.e., without assuming $F^t(K^\bullet) = \{ 0 \}$ and $F^0(K^\bullet) = K^\bullet$. These filtrations are examples of \textit{simple} and \textit{exhaustive} filtrations (see \cite{loop}), and these are the ones which will be important to us.} of a complex $K^\bullet$ consists of a nested family 
$$ F^t(K^\bullet) = \{ 0 \} \subseteq F^{t-1}(K^\bullet) \subseteq \ldots \subseteq  F^p(K^\bullet) \subseteq  \ldots \subseteq F^0(K^\bullet) = K^\bullet$$
of subcomplexes $F^{p}(K^\bullet)$ of $K^\bullet$. We denote the degree $n$ of the complex $F^{p}(K^\bullet)$ by $F^{p}K^n$. 
\end{defin}

\begin{rmk} Let $\{ F^{p}(K^\bullet) \}$ be a filtration of a complex $K^\bullet$. Each inclusion of complexes $F^{p}(K^\bullet) \hookrightarrow K^\bullet$ induces a homomorphism in cohomology
$$H^n(F^{p}(K^\bullet)) \to H^n(K^\bullet)$$
and we denote the image of this homomorphism by $F^p(H^n(K^\bullet))$. This yields a filtration of the cohomology with
\begin{eqnarray}
F^0(H^n(K^\bullet)) &=& H^n(K^\bullet) \nonumber \\
F^p(H^n(K^\bullet)) &=& 0, \ \text{for $p>n$.} \nonumber 
\end{eqnarray}  
\end{rmk}

Note that we start with objects in the first quadrant, i.e., $E_0^{p,q} $, $p, q \geq 0$, and take successive cohomologies from 
$$E_r^{p - r, q + r -1} \to E_r^{p, q} \to E_r^{p + r, q - r +1}.$$
This means that for a large $r$ taking cohomology will have no effect on the object considered. More precisely, there exists $r_0$ such that $E_r = E_{r+1} = \ldots \ $, for every $r \geqslant r_0$. One says that the spectral sequence \textbf{degenerates} at $E_{r_0}$ and calls $E_\infty \coloneqq E_{r_0} $ the \textbf{abutment} of the spectral sequence. It is also common said that the spectral sequence \textbf{converges} to $E_\infty$. The following proposition can be found in \cite[Chapter 3.5]{gh}.

\begin{prop} Let $K^\bullet$ be a filtered complex of abelian groups. Then there exists a spectral sequence $\{E_r \}$ with
\begin{eqnarray}
E_0^{p,q} & = & \dfrac{F^p K^{p+q}}{F^{p+1} K^{p+q}} \nonumber \\
E_\infty^{p,q} & = & \dfrac{F^p (H^{p+q}(K^\bullet))}{F^{p+1} (H^{p+q}(K^\bullet))} \nonumber 
\end{eqnarray} 
\end{prop}

Let us now see the relationship between the cohomology of a filtered complex $K^\bullet$ and a special type of spectral sequence which will play an important role for us. Let $\{E_r\}$ be the spectral sequence given by the proposition above and assume further that $E_2^{p,q} = 0$ unless $q=0$ or $q=1$. As mentioned, the cohomology acquires a filtration from the filtration of the complex
$$\{ 0 \} = F^{n+1}(H^n(K^\bullet)) \subseteq F^{n}(H^n(K^\bullet)) \subseteq \ldots \subseteq F^{0}(H^n(K^\bullet)) = H^n(K^\bullet).$$  
From the proposition we have the following short exact sequence
$$0 \to F^{n -q +1}(H^n(K^\bullet)) \to F^{n - q}(H^n(K^\bullet)) \to E_\infty^{n-q, q} \to 0.$$
In particular, $F^{n}(H^n(K^\bullet)) \cong E_\infty^{n, 0}$. Since $E_2^{p,q} = 0$ unless $q=0$ or $q=1$, $F^{n -1}(H^n(K^\bullet)) \cong \ldots \cong F^{0}(H^n(K^\bullet)) = H^n(K^\bullet)$. We can then put together the short exact sequence
\begin{equation}
0 \to E_\infty^{n, 0} \to H^n(K^\bullet) \to E_\infty^{n- 1, 1} \to 0.
\label{sesmid}
\end{equation}
Moreover, the only non-zero terms in the third page are
\begin{eqnarray}
E_3^{p,0} &=& \coker (d_2 : E_2^{p-2,1} \to E_2^{p,0}) \nonumber \\
E_3^{p,1} &=& \ker (d_2 : E_2^{p,1} \to E_2^{p+2,0}). \nonumber 
\end{eqnarray} 
Since $d_r : E_r^{p,q} \to E_r^{p+r,q-r+1}$, the spectral sequence degenerates at $E_3$ and we can put together ($\ref{sesmid}$) for different $n$. This yields a long exact sequence 
\begin{equation*}
\ldots \rightarrow H^{n-1}(K^\bullet) \rightarrow E_2^{n-2,1} \rightarrow E_2^{n,0} \rightarrow H^n(K^\bullet) \rightarrow E_2^{n -1,1} \rightarrow E_2^{n+1,0} \rightarrow H^{n+1}(K^\bullet) \rightarrow \ldots   
\end{equation*}


\section{Hypercohomology}

Let $\mathcal{F}$ be a sheaf of abelian groups over a topological space $X$. Recall that one can define sheaf cohomology in two steps. First consider an injective resolution for $\mathcal{F}$  
$$0 \to \mathcal{F} \to \mathcal{I}^0 \to \mathcal{I}^1 \to \ldots $$ 
which always exists since one can always embed $\mathcal{F}$ into an injective sheaf. Then apply the functor of global sections to the complex of sheaves $\mathcal{I}^\bullet$. Sheaf cohomology of $\calF$ is the cohomology of the resulting complex
$$H^n(X, \mathcal{F}) \coloneqq H^n(\mathcal{I}^\bullet(X)).$$  
Moreover, given two injective resolutions, there exists a morphism of resolutions, which induces a morphism between the complexes obtained by considering global sections. This morphism of complexes is unique up to homotopy and the sheaf cohomology groups do not depend on the injective resolution considered. It is a classical result that acyclic resolutions (i.e., given by a complex of sheaves $\mathcal{G}^\bullet$ such that $H^n(X, \mathcal{G}^k) = 0$ for $n \geq 1$, $k \geq 0$) also calculate sheaf cohomology, i.e., $H^n(X, \mathcal{F}) = H^n(\mathcal{G}^\bullet(X))$ (for more details see, e.g. \cite{loop}).  

Now consider a complex $\mathcal{G}^\bullet$ of sheaves of abelian groups over $X$ 
\begin{equation*}
\mathcal{G}^0 \to \mathcal{G}^1 \to \ldots \to \mathcal{G}^n \to \ldots
\end{equation*}
We want to construct an injective resolution for this complex. First a remark. 
\begin{itemize}
\item Let $0 \to \mathcal{A} \to \mathcal{B} \to \mathcal{C} \to 0 $ be a short exact sequence of sheaves, $\mathcal{A} \rightarrow \mathcal{I_A}^\bullet$ an injective resolution of $\mathcal{A}$ and $\mathcal{C} \rightarrow \mathcal{I_C}^\bullet$ an injective resolution of $\mathcal{C}$. Then we have the commutative diagram 
\[\xymatrix@M=0.13in{
  & 0\ar[d] & 0\ar[d] & 0\ar[d] & \\
0 \ar[r] & \mathcal{A} \ar[r] \ar[d] & \mathcal{B}\ar[r]\ar[d] & \mathcal{C}\ar[r] \ar[d] & 0\\
0 \ar[r] & \mathcal{I_A}^0\ar[r] \ar[d]& \mathcal{I_A}^0 \oplus \mathcal{I_C}^0\ar[r]\ar[d] & \mathcal{I_C}^0\ar[r] \ar[d] & 0\\
0 \ar[r] & \mathcal{I_A}^1\ar[r] \ar[d]& \mathcal{I_A}^1 \oplus \mathcal{I_C}^1\ar[r]\ar[d] & \mathcal{I_C}^1\ar[r] \ar[d] & 0.\\
& \vdots & \vdots & \vdots & }\] 
Here the horizontal morphisms are the obvious ones and the injectivity of the sheaves allows the construction of an injective resolution of $\mathcal{B}$ such that each square in the diagram commutes. In particular, such construction of an injective resolution of the middle sheaf enables one to prove easily that there is a natural long exact sequence of sheaf cohomology groups.  
\end{itemize}

Let us return to our complex $\mathcal{G}^\bullet$ and denote the sheaf given by the $n$-th cohomology of this complex by $H^n(\mathcal{G}^\bullet) = \mathcal{Z}(\mathcal{G}^n)/\mathcal{B}(\mathcal{G}^n) = \ker (\mathcal{G}^n \rightarrow \mathcal{G}^{n+1})/ \im (\mathcal{G}^{n-1} \rightarrow \mathcal{G}^n)$ (i.e., the sheafification of the presheaf $U \mapsto \ker (\mathcal{G}^n(U) \rightarrow \mathcal{G}^{n+1}(U))/ \im (\mathcal{G}^{n-1}(U) \rightarrow \mathcal{G}^n(U))$). We construct an injective resolution of $\mathcal{G}^\bullet$ using the following short exact sequences  
\begin{center}
$0 \to \mathcal{B}(\mathcal{G}^n) \to \mathcal{Z}(\mathcal{G}^n) \to  H^n(\mathcal{G}^\bullet) \to 0, $\\ [\baselineskip]
$0 \to \mathcal{Z}(\mathcal{G}^n) \to \mathcal{G}^n \to \mathcal{B}(\mathcal{G}^{n+1}) \to 0.$
\end{center}
For this we choose injective resolutions $\mathcal{B}(\mathcal{G}^n) \rightarrow \mathcal{I}_B^{n, \bullet}$ and $H^n(\mathcal{G}^\bullet) \rightarrow \mathcal{I}_H^{n, \bullet}$. We can find an injective resolution of $\mathcal{G}^0$, since $\mathcal{Z}(\mathcal{G}^0) = H^0(\mathcal{G}^\bullet)$. Then we use the second short exact sequence to find an injective resolution of $\mathcal{G}^1$ using the one from $\mathcal{Z}(\mathcal{G}^1)$ (obtained by the first short exact sequence) and so on. As a result we get an injective resolution of the complex $\mathcal{G}^\bullet$ 
\[\xymatrix@M=0.13in{
  & 0\ar[d] & & 0\ar[d] & 0\ar[d] & \\
0 \ar[r] & \mathcal{G}^0 \ar[r] \ar[d] & \ldots \ar[r] & \mathcal{G}^n \ar[r] \ar[d] & \mathcal{G}^{n+1} \ar[r] \ar[d] & \ldots \\
0 \ar[r] & \mathcal{I}^{0,0} \ar[r] \ar[d] & \ldots \ar[r] & \mathcal{I}^{n,0} \ar[r] \ar[d] & \mathcal{I}^{n+1,0} \ar[r] \ar[d] & \ldots\\
0 \ar[r] & \mathcal{I}^{0,1} \ar[r] \ar[d] & \ldots \ar[r] & \mathcal{I}^{n,1} \ar[r] \ar[d] & \mathcal{I}^{n+1,1} \ar[r] \ar[d] & \ldots\\
& \vdots & & \vdots & \vdots & }\] 

\begin{rmk} From the construction above, $\mathcal{I}^{p,q} = \mathcal{I}_B^{p,q} \oplus \mathcal{I}_H^{p,q} \oplus \mathcal{I}_B^{p+1,q}$.
\end{rmk}

Apply the functor global sections and denote the differentials by 
\begin{eqnarray}
^\prime d^{p,q} & : & \mathcal{I}^{p,q}(X) \to \mathcal{I}^{p + 1,q}(X) \nonumber \\
^{\prime \prime} d^{p,q} & : & \mathcal{I}^{p,q}(X) \to \mathcal{I}^{p, q + 1}(X) \nonumber 
\end{eqnarray} 
Since they commute, $(\mathcal{I}^{\bullet , \bullet}(X),^\prime d,^{\prime \prime} d  )$ is a double complex. From this double complex we get a complex of abelian groups $(\mathcal{I}^\bullet(X), d)$ defined by 
\begin{eqnarray}
\mathcal{I}^n(X) & \coloneqq & \bigoplus_{p+q = n} \mathcal{I}^{p,q}(X) \nonumber \\
d^n & \coloneqq & \sum_{p+q = n}^{} \  ^\prime d^{p,q} + (-1)^p \ ^{\prime \prime} d^{p,q}.   \nonumber 
\end{eqnarray}

\begin{defin} Let $\mathcal{G}^\bullet$ be a bounded below complex of sheaves over a space $X$. The \textbf{$n$-th hypercohomology group} $\mathbb{H}^n (X, \mathcal{G}^\bullet)$ (or simply $\mathbb{H}^n (\mathcal{G}^\bullet)$) is the $n$-th cohomology of the complex of abelian groups $(\mathcal{I}^\bullet(X), d)$, where $\mathcal{I}^{\bullet , \bullet}$ is an injective resolution of $\mathcal{G}^\bullet$.  
\end{defin}
A similar argument to the sheaf cohomology case shows that this does not depend on the injective resolution of the complex and therefore is well defined (see e.g. \cite{loop}). Furthermore, given a short exact sequence of bounded below complexes of sheaves, we can construct an injective resolution for the complex in the middle using given injective resolutions of the other two complexes (as we have done previously). Therefore we have the following. 

\begin{prop} Let $0 \to \mathcal{G}^\bullet \to ^\prime\mathcal{G}^\bullet \to ^{\prime   \prime}\mathcal{G}^\bullet \to 0 $ be a short exact sequence of bounded below complexes of sheaves. Then, there exists a long exact sequence 
$$\ldots \to \mathbb{H}^n (\mathcal{G}^\bullet) \to \mathbb{H}^n (^\prime \mathcal{G}^\bullet) \to \mathbb{H}^n (^{\prime \prime} \mathcal{G}^\bullet) \to \mathbb{H}^{n+1} (\mathcal{G}^\bullet) \to \ldots$$ 
\end{prop}

\begin{ex} Let $\mathcal{G}$ be a sheaf over $X$ and consider it as a complex of sheaves (i.e., $\mathcal{G} \rightarrow 0 \rightarrow \ldots$). Then, hypercohomology recovers the sheaf cohomology of $\mathcal{G}$. Now, denote by $\mathcal{G}[-k]$ the complex which is $\mathcal{G}$ in degree $k$ and $0$ otherwise. Then, an injective resolution of this complex will give a complex, which in degree $n$ is $\mathcal{I}^n(X) = \mathcal{I}^{k,n}(X)$. Thus, $\mathbb{H}^n (\mathcal{G}[-k]) = H^{n-k}(X, \mathcal{G})$.    
\end{ex}

Consider again the complex of sheaves $\mathcal{G}^\bullet$ and let $\mathcal{I}^{\bullet , \bullet}$ be the injective resolution of this complex constructed previously. There are two natural filtrations of the complex $(\mathcal{I}^\bullet(X), d)$ 
\begin{eqnarray}
^\prime F^p(\mathcal{I}^\bullet(X)) &=& \mathop{\bigoplus_{r \geqslant p}}_{s \geqslant 0} \mathcal{I}^{r,s}(X), \nonumber \\
^{\prime \prime} F^p(\mathcal{I}^\bullet(X)) &=& \mathop{\bigoplus_{s \geqslant p}}_{r \geqslant 0} \mathcal{I}^{r,s}(X). \nonumber 
\end{eqnarray} 
From proposition 15, each of these filtrations gives rise to a spectral sequence whose initial term is $E_0^{p,q}  =  \dfrac{F^p \mathcal{I}^{p+q}(X)}{F^{p+1} \mathcal{I}^{p+q}(X)}$. Let us describe the first and second page of each spectral sequence.
\begin{enumerate}
\item $F^p(\mathcal{I}^\bullet(X)) = \ ^{\prime}F^p(\mathcal{I}^\bullet(X))$: \\ [\baselineskip]
We have $E_0^{p,q} = \dfrac{\mathcal{I}^{p,q}(X) \oplus \mathcal{I}^{p + 1,q - 1}(X) \oplus \ldots}{\mathcal{I}^{p + 1,q - 1}(X) \oplus \ldots} \cong \mathcal{I}^{p,q}(X)$. To find the $E_1^{p,q}$ term we have to calculate cohomology of the piece $\mathcal{I}^{p, q -1}(X) \to \mathcal{I}^{p, q}(X) \to \mathcal{I}^{p, q+1}(X)$. Since $\mathcal{I}^{p, \bullet}$ gives an injective resolution of $\mathcal{G}^p$ we have $E_1^{p,q} = H^q(X, \mathcal{G}^p)$. Now, $d_1 : E_1^{p,q} \to E_1^{p + 1,q}$, so we are looking at the $p$-th term of the complex $H^q(X, \mathcal{G}^\bullet)$
$$H^q(X, \mathcal{G}^{p-1}) \to H^q(X, \mathcal{G}^p) \to H^q(X, \mathcal{G}^{p+1}).$$
In other words, $E_2^{p,q} = H^p(H^q(X, \mathcal{G}^\bullet))$.
\item $F^p(\mathcal{I}^\bullet(X)) = \ ^{\prime \prime}F^p(\mathcal{I}^\bullet(X))$: \\ [\baselineskip]
Here, $E_0^{p,q} = \dfrac{\mathcal{I}^{q,p}(X) \oplus \mathcal{I}^{q - 1, p + 1}(X) \oplus \ldots}{\mathcal{I}^{q - 1, p + 1}(X) \oplus \ldots} \cong \mathcal{I}^{q, p}(X)$ and $d_0 : \mathcal{I}^{q,p}(X) \to \mathcal{I}^{q + 1,p}(X)$ is the map
\begin{center}
$\mathcal{I}_B^{q,p}(X) \oplus \mathcal{I}_H^{q,p}(X) \oplus \mathcal{I}_B^{q+1,p}(X) \to \mathcal{I}_B^{q+1,p}(X) \oplus \mathcal{I}_H^{q+1,p}(X) \oplus \mathcal{I}_B^{q+2,p}(X)$\\
$(x, y, z) \mapsto (z, 0, 0).$
\end{center}
Therefore, $E_1^{p,q} \cong \dfrac{\mathcal{I}_B^{q,p}(X) \oplus \mathcal{I}_H^{q,p}(X)}{\mathcal{I}_B^{q,p}(X)} \cong \mathcal{I}_H^{q,p}(X)$
and we obtain $E_2^{p,q}$ by considering the cohomology from the piece
$$E_1^{p-1,q} \to E_1^{p,q} \to E_1^{p+1,q},$$
which in this case is 
$$\mathcal{I}_H^{q,p-1}(X) \to \mathcal{I}_H^{q,p}(X) \to \mathcal{I}_H^{q,p+1}(X).$$
Since $H^q(\mathcal{G}^\bullet) \rightarrow \mathcal{I}_H^{q, \bullet}$ is an injective resolution, we find $E_2^{p,q} = H^p(X, H^q(\mathcal{G}^\bullet))$.
\end{enumerate}

Throughout the text we will refer to these as the \textbf{first} and \textbf{second} spectral sequences associated to the corresponding double complex. One could also define hypercohomology of a bounded below complex of sheaves over $X$ using the double complex obtained by considering the groups of \v{C}ech cochains. For reasonable spaces (e.g. $X$ paracompact) this turns out to be isomorphic to our definition via resolutions. Since we will be working with sheaves over a Riemann surface $\Sigma$, there is no harm in exchanging the notions (see \cite{loop} for more details). 

%% file: Appendices/AppendixC.tex
\chapter{Quaternionic vector spaces} 

\label{AppendixC} 

\lhead{Appendix C. \emph{Quaternionic vector spaces}} 

Let $\K $ be the (real non-commutative) algebra of quaternions, whose elements are of the form $q = a_0 + a_1i + a_2j + a_3k$, where $a_l \in \R$, $ l = 0, 1,2 , 3 $, and $i, j , k$ satisfy the usual quaternion identities $$i^2 = j^2 = k^2 = ijk = -1.$$
The prototypical quaternionic vector space is $\K^n$, which will be regarded as a left $\K$-module (by left multiplication). Using the decomposition $\K \cong \C \oplus \C j$, consider the isomorphism of left $\C$-vector spaces $\K^n \to \C^{2n}$ given by 
$$(q_1, \ldots , q_n) \mapsto (u_1, \ldots , u_n, u_1^\prime , \ldots , u_n^\prime),$$
where $q_l = u_l + u_l^\prime j \in \C \oplus \C j = \K$. Note that under this isomorphism, multiplication on the left by $j$ in $\K^n$ is identified with an anti-linear map\footnote{This follows immediately from the relation $jz = \bar{z}j$ in $\K$, where $z \in \C$.\label{relation}}, which we also denote by $j$, satisfying $j^2=-\Id$. Concretely,
\begin{align*}
j : \C^{2n} & \to \C^{2n}\\
u & \mapsto -J \bar{u},
\end{align*}
where $u$ is seen as a column vector and $J = \left( \begin{matrix} 0& 1\\ -1&0 \end{matrix} \right)$ is the standard skew-symmetric form on $\C^{2n}$. 

We define the standard quaternionic Hermitian product on $\K^n$ by
$$\tilde{h}(p,q) = \sum p_l \bar{q_l},$$
where $p = (p_1, \ldots, p_n) \in \K^n$ and $q = (q_1, \ldots, q_n) \in \K^n$. Using the relation described in footnote\footref{relation}, we have the following decomposition 
\begin{equation}
\tilde{h} = g - i\omega_1 - j\omega_2 - k\omega_3 = (g - i\omega_1) - (w_2 + i\omega_3)j = h - \omega j,
\label{decomp}
\end{equation}
where $g$ is the flat metric, $\omega_1, \omega_2, \omega_3$ are the K\"ahler forms (associated to the complex structures $i,j,k$, respectively), $h$ is the standard Hermitian form on $\C^{2n}$ and $\omega$ the standard (complex) symplectic form on $\C^{2n}$. The quaternionic unitary group $Sp(n)$ is defined as the subgroup of $GL(n, \K)$ preserving the standard quaternionic Hermitian form\footnote{If we view $q\in \K^n$ as a row vector, $g \in GL(n,\K)$ acts on $q$ as $qg^{-1}$. Thus, $Sp(n)$ consists of elements $g\in GL(n,\K)$ satisfying $gg^\ddagger = \Id$.}, thus, by (\ref{decomp}), we have the usual identification 
$$Sp(n) \cong Sp(2n,\C) \cap U(2n).$$ 

More generally we can define a quaternionic Hermitian form on any left $\K$-module $\V$. That is, an $\R$-bilinear map $\tilde{h} : \V \otimes_\R \V \to \K$, which is $\K$-linear in the first coordinate, the conjugate of $\tilde{h}(p,q)$ equals $\tilde{h}(q, p)$, and is such that $\tilde{h}(p,p) \in \R_{\geqslant 0}$ with equality if and only if $p=0$. 

Now, let $\V$ be a complex vector space equipped with a (complex) symplectic form $\omega$ and a Hermitian form $h$. A \textbf{quaternionic structure} on (any complex vector space) $\V$ is an anti-linear map $j: \V \to \V$ such that $j^2 = -\Id$. Given a quaternionic structure $j$ on $\V$, we can regard it as a left $\K$-module in the obvious way. Note that, if $\omega$ and $h$ satisfy the compatibility condition $$\omega (u, j(v)) = h(u,v),$$
$\tilde{h} \coloneqq h - \omega j$ is a quaternionic Hermitian form on $\V$.

%% file: Appendices/AppendixD.tex
\chapter{Principal bundles} 

\label{AppendixD} 

\lhead{Appendix D. \emph{Principal bundles}} 

In this appendix we define principal bundles and recall some necessary background related to these objects which will be needed in the text. A classic reference for principal bundles and their moduli space is  \cite{serre}. In the case of $G$-bundles on a curve, which will be our main focus, we refer to \cite{ramana, ramana1, sorger}.

\section{General notions and conventions}\label{principalbdls}

Let $G$ be an affine algebraic group over $\C$ and $X$ a smooth variety (or more generally a $\C$-scheme). 

\begin{defin} A \textbf{principal $G$-bundle} on $X$ is a variety (or more generally a $\C$-scheme) $P$ with a right $G$-action and a $G$-invariant morphism $P \to X$, which is locally trivial in the \'{e}tale topology. A \textbf{morphism} between two principal $G$-bundles $\pi : P \to X$ and $\pi^\prime : P^\prime \to X$ is a morphism of varieties $\varphi : P \to P^\prime$ (or more generally of $\C$-schemes) such that $\pi = \pi^\prime \circ \varphi$.   
\end{defin}  

Being locally trivial in the \'{e}tale topology means that there is an open cover $\{ h_i : U_i \to X \}$ by \'{e}tale maps $h_i$ and $G$-equivariant isomorphisms 
$$\varphi_i : h_i^*U \to U_i \times G$$  
to the trivial $G$-bundle on $U_i$. Then, on $U_{ij} = U_i \cap U_j$ the isomorphism $\varphi_i\circ \varphi_j^{-1}$ defines transition functions $p_{ij}$, whose class in the pointed set $H^1_{\acute{e}t}(X, G)$ corresponds to the isomorphism class of the $G$-bundle $P$. The base point of the \'{e}tale cohomology set $H^1_{\acute{e}t}(X, G)$, which we will denote simply by $H^1(X, G)$, is the principal trivial $G$-bundle $X\times G$. 
  
\begin{rmks}

\noindent 1. In general, being locally trivial in the Zariski sense turns out to be too restrictive (e.g., when $G$ is a finite group). As observed by Serre, locally triviality in the Zariski sense is stronger than locally triviality in the \'{e}tale sense. A group for which the notions are equivalent is called \textbf{special} (e.g., the general linear group is special \cite{milne}). When $G$ is semisimple, by the work of Grothendieck, special groups are exactly direct products of special linear groups and symplectic groups.  

\noindent 2. We will be mostly interested in the case $X$ is a connected smooth projective curve and in that case locally triviality on the Zariski topology is the same as in the \'{e}tale topology \cite[Theorem 1.9]{stein}. The study of principal $G$-bundles on $X$, however, usually involves considering $G$-bundles on families. Thus the \'{e}tale topology is more appropriate.

\noindent 3. If $X$ is a smooth projective curve, there is an equivalence of categories of principal $G$-bundles on $X$ and analytic principal $G$-bundles on $X$ (seen as a compact Riemann surface). We will frequently pass from one point of view to the other without further mention. \label{rmk3}   
\end{rmks}
Let $P$ be a $G$-bundle on a variety $X$ and suppose $Y$ is a quasi-projective variety acted on the left by $G$. We form the \textbf{associated bundle} to $P$ with fibre $Y$ by taking the quotient of $P\times Y$ by the $G$-action defined by
$$g \cdot (p,y) = (p\cdot g,  g^{-1}\cdot y).$$
This quotient is well-defined since $Y$ is a quasi-projective variety and will be denoted by $P(Y)$ (or $P\times_G Y$).  
    
\begin{exs}\label{gaga}

\noindent 1. Let $\rho : G \to GL(\V)$ be a linear (finite-dimensional) representation. Then $P(\V) = P\times_G \V$ is a vector bundle with fibre $\V$ called the \textbf{vector bundle associated to $P$} via $\rho$. If we want to emphasize the particular representation the associated bundle is obtained from we write $P\times_{\rho} \V$. In terms of transition functions, $P(\V)$ is given by $\{ \rho \circ p_{ij} \}$, where $\{ p_{ij} \}$ are transition functions of $P$, i.e., the isomorphism class of the vector bundle $P(\V)$ correspond to the image of $\{ p_{ij} \} $ via the natural map of pointed sets  
$$H^1(X,G) \to H^1(X,GL(\V))$$  
associated to $\rho$.

\noindent 2. More generally, let $\rho : G \to G^\prime$ be a homomorphism of groups. Then $G$ acts on $G^\prime$ by $g\cdot g^\prime = \rho (g)g^{\prime}$ and $P(G^\prime)$ is a principal $G^\prime$-bundle called \textbf{extension of structure group} of $P$ from $G$ to $G^\prime$.
\end{exs}

Note that giving a section $s : X \to P(Y)$ of an associated bundle to $P$, we have a morphism $s^\prime : P \to Y$ such that $s^\prime (p\cdot g ) = g^{-1}\cdot s^\prime (p)$. Conversely, given such a morphism $s^\prime$, we may define a section of $P(Y)$ by setting $s(x) = (p, s^\prime (p))$, where $p \in P$ is any point in the fibre of $x \in X$. Moreover, if $Y$ and $Y^\prime$ are two quasi-projective varieties acted on the left by $G$ and $f : Y \to Y^\prime$ is a $G$-equivariant morphism, there is a natural morphism $P(f): P(Y) \to P(Y^\prime)$ between the two associated bundles.

Let $H \subset G$ be a subgroup and $P$ a principal $G$-bundle on $X$. A \textbf{reduction of structure group} of $P$ to $H$ is a section 
$$\sigma : X \to P/H \cong P(G/H).$$ 
In particular, if $H$ is a maximal compact subgroup, a reduction of structure group to $H$ is called a \textbf{metric} of $P$. Now, given any reduction of structure group $\sigma : X \to P/H$, we call the pullback bundle $\sigma^*P$ a \textbf{restriction of structure group} of $P$ to $H$ and denote it by $P_\sigma$ (or $P_H$). Note that $P_\sigma \subset P$ is a principal $H$-bundle and its extension of structure group $P_\sigma(G)$ (via the inclusion) is naturally isomorphic to the principal $G$-bundle $P$.

\section{Principal bundles on a curve}

We now focus on principal $G$-bundles on a compact Riemann surface $\Sigma$ (see the third item in Remark \ref{gaga}), where $G$ is a connected complex reductive group. Recall that, topologically, principal $G$-bundle on $\Sigma$, are classified by a discrete invariant in $\pi_1(G)$ (\cite[Proposition 5.1]{ramana}). This is usually called the \textbf{topological type} of the bundle and can be defined as follows. Let $\tilde{G}$ be the universal covering group of $G$. The exact sequence 
$$1\to \pi_1(G) \to \tilde{G} \to G \to 1$$   
determines a long exact sequence in cohomology. In particular we have a map 
$$H^1(\Sigma , G) \to H^2(\Sigma , \pi_1(G)) \cong \pi_1(G).$$
The image under this map of the transition functions of a principal $G$-bundle gives its topological type. 

Since an affine algebraic group $G$ can always be realized as a (Zariski) closed subgroup of a general linear group, the theory of principal $G$-bundles is a natural generalization of the classical theory of vector bundles on a curve. Let us briefly recall the basic definitions for $GL(n,\C)$-principal bundles, i.e., holomorphic (or equivalently algebraic) vector bundles of rank $n$ on $\Sigma$. Fix the topological type $d\in \pi_1(GL(n,\C)) = \Z$, i.e. the degree of the vector bundle. Then, a vector bundle $V$ is said to be \textbf{semi-stable} if for any non-zero subbundle $W \subset V$, 
\begin{equation}
\mu (W) \leq \mu (V),
\label{slope}
\end{equation}
where the \textbf{slope} $\mu $ of a vector bundle on $\Sigma$ is defined as the quotient of the degree by the rank of the bundle. Also, $V$ is said to be \textbf{stable} if the inequality (\ref{slope}) is strict for all proper and non-zero subbundles $W$ of $V$. To obtain a good moduli space, one introduces an equivalence condition slightly weaker than isomorphism called $S$-equivalence. Let $V$ be a strictly semi-stable vector bundle of slope $\mu (V) = \mu$ (in our case, $\mu = n/d$). Then, there exists a subbundle $W \subseteq V$ with $\mu (W) = \mu$. Among those, choose the one of least rank and denote it by $W_1$. It follows that any proper subbundle of $W_1$ must have slope less than $\mu (W_1)$, i.e., $W_1$ is stable. Now consider the vector bundle $V/W_1$. One has $\mu (V/W_1) = \mu$ and $V/W_1$ is semi-stable. If $V/W_1$ is strictly semi-stable, one repeats the process to get a subbundle $W_2/W_1 \subseteq V/W_1$, which is stable and of same slope. By induction one obtains a filtration
$$0  = W_0 \subseteq W_1 \subseteq \ldots \subseteq W_k = V$$
by holomorphic subbundles such that:
\begin{enumerate}
\item $\mu (W_j/W_{j-1}) = \mu$, for $1 \leq j \leq k$.
\item The bundles $W_j/W_{j-1}$ are stable. 
\end{enumerate}

This is called the \textbf{Jordan-H\"{o}lder decomposition} of $V$. Although it is clearly not unique, the graded object 
$$Gr(V)  \coloneqq \bigoplus_{j = 1}^k W_j/W_{j-1}$$ 
is unique up to isomorphism. Two semi-stable vector bundles $V$ and $V^\prime$ of rank $n$ and degree $d$ are said to be \textbf{$S$-equivalent} if $Gr(V) \cong Gr(V^\prime)$. Finally, a vector bundle $V$ is \textbf{polystable} if $V \cong Gr(V) $. For the theorem below the reader is referred to \cite{pot, ses}.

\begin{thm}\label{gl} Fix $(n,d)\in \Z_{+}\times \Z$. There exists a coarse moduli space $\calU_\Sigma (n,d)$ for semi-stable vector bundles on $\Sigma$ of rank $n$ and topological type $d$. It is an irreducible projective variety whose points correspond to $S$-equivalence classes of semi-stable vector bundles of rank $n$ and degree $d$ on $\Sigma$.
\end{thm}     

If no confusion arises, we will drop the subscript $\Sigma$ and simply denote the moduli space of semi-stable vector bundles of rank $n$ and degree $d$ by $\calU (n,d)$.

\subsection{The moduli space of principal $G$-bundles}

Now, let $G$ be a connected reductive group over $\C$. A natural way to generalize the semi-stability condition for principal $G$-bundles on $\Sigma$ from the previous notion is to consider certain characters of parabolic subgroups and reductions of the structure group to these subgroups. 

Given a proper parabolic subgroup $Q$ of $G$, then $G \to G/Q$ is naturally a principal $Q$-bundle on the projective variety $G/Q$. Let $\chi : Q \to \C^\times$ be a character of $Q$ and consider the line bundle
$$L_\chi = G\times_\chi \C \to G/Q$$    
on $G/Q$ associated to $G \to G/Q$ via $\chi$. The character $\chi$ is called \textbf{anti-dominant} if it is trivial on the center $Z(Q)$ and the line bundle $L_\chi$ is ample\footnote{It is also possible to define the notion of anti-dominant character in terms of fundamental weights (see, e.g., \cite{hk, ramana} for more details).}. Now, if $P \to \Sigma$ is a principal $G$-bundle on the curve, $\chi : Q \to \C^\times$ is an anti-dominant character of a proper parabolic subgroup $Q \subset G$ and $\sigma : \Sigma \to P/Q$ is a reduction to $Q$, we define the line bundle 
$$L_{\sigma , \chi} = P_\sigma \times_\chi \C \to \Sigma$$
to be the associated bundle to the reduction of structure group $P_\sigma$ via $\chi$. The degree of $L_{\sigma , \chi}$ is called the \textbf{degree of $P$ with respect to $\sigma$ and $\chi$} and we denote it by $\deg (P) (\sigma , \chi)$. We are now ready to state the (semi-)condition for $P$ (see Ramanathan's paper \cite{ramana1} for a further discussion and alternative characterizations).

\begin{defin} A $G$-bundle $P$ on $\Sigma$ is \textbf{stable} (respectively, \textbf{semi-stable}) if for every parabolic subgroup $Q \subset G$, every non-trivial anti-dominant character $\chi : Q \to \C^\times$ of $Q$ and every reduction $\sigma : \Sigma \to P/Q$ of structure group to $Q$, we have $$\deg(P)(\sigma , \chi) \ge 0.$$ 
\end{defin} 

Let us indicate why this seemingly technical condition is a natural generalization of the vector bundle case.

\begin{ex} Let $P$ be a $GL(n,\C)$-principal bundle on $\Sigma$ and denote its associated vector bundle by $V$. Consider the parabolic subgroup $Q \subset GL(n, \C)$ consisting of matrices of the form $g = \left( \begin{matrix} g_1& g_3\\ 0&g_2 \end{matrix} \right)$, where $ g_1 \in GL(k,\C)$ and $ g_2 \in GL(n-k, \C)$, for some $0 < k < n$. The quotient $GL(n,\C)/Q$ is then the Grassmannian variety of $(n-k)$-dimensional quotients of $\C^n$. Thus, a reduction $\sigma : \Sigma \to P/Q$ defines a subbundle $W \subset V$ of rank $k$. The parabolic subgroup $Q$ is maximal and it follows from the general theory that any anti-dominant character $\chi$ of $Q$ is a positive multiple of 
\begin{align*}
\chi_0 : Q& \to \C^\times \\
g & \mapsto (\det g_1)^{k-n}(\det g_2)^k.
\end{align*} 
The associated line bundle is then a positive power of $L_{\sigma, \chi_0} = (\det W)^{k-n}(\det V/W)^k $, whose degree is 
$$\deg (P)(\sigma , \chi_0) = (k-n)\deg (W) + k\deg (V/W).$$
Thus, the condition $\deg (P)(\sigma , \chi) \ge 0$ is equivalent to 
$$\mu (W) \leq \mu (V).$$

\end{ex}
         
Just as one can associate a graded object $Gr(V)$ to a semi-stable vector bundle, Ramanathan showed \cite{ramana1} that there is a corresponding notion for semi-stable principal $G$-bundles. He proved that every semi-stable principal $G$-bundle $P$ on $\Sigma$ admits a special reduction of structure group $P_Q$ to a parabolic $Q \subset G$ such that the associated bundle $P_Q(M)$ is a stable principal $M$-bundle, where $M \subset Q$ is the Levi factor of $Q$. Such special reductions, called \textbf{admissible}, are characterized by the property that the degree of $L_\chi$ is zero for every character $\chi$ of $Q$ which is trivial on the center $Z(Q)$. Then, the $G$-bundle $Gr(P) \coloneqq (P_Q(M))(G)$ obtained by extension of structure group via the inclusion $M \subset G$ is called the \textbf{associated graded bundle of $P$}. Two semi-stable principal $G$-bundles $P$ and $P^\prime$ are said to be $S$-equivalent if $Gr(P) \cong Gr(P^\prime)$.     
\begin{thm}[\cite{ramana1, ramana2}] Let $G$ be a complex connected reductive group and $d\in \pi_1(G)$. 
\begin{enumerate}[label=(\roman*)]
\item There exists a coarse moduli scheme $\calN_\Sigma^d (G)$, or simply $\calN^d (G)$, for semi-stable principal $G$-bundles on $\Sigma$ of topological type $d$.
\item It is an irreducible normal variety of dimension 
$$\dim \calN^d (G) = \dim G (g-1) + \dim Z(G),$$
whose points correspond to $S$-equivalence classes of semi-stable principal $G$-bundles. 
\item If $P \in \calN^d (G)$ is stable, then, near $P$, the moduli space $\calN^d (G)$ is isomorphic to a neighborhood of $0$ in 
$$\dfrac{H^1(\Sigma , \Ad (P))}{\Aut (P)/Z(G)}.$$
In particular, for a generic stable $G$-bundle $P$, we have 
$$T_P (\calN^d (G)) \cong H^1(\Sigma , \Ad (P)).$$ 
\end{enumerate}   
\end{thm}

\subsection{The symplectic group}\label{assoc}

Given a homomorphism $G_1 \to G_2$ between reductive groups, it is natural to consider the induced map between the moduli spaces. A particularly important case is when one considers the standard representation of a classical group, which turns the study of principal bundles into the study of vector bundles. It follows from the work of Ramanathan (loc. cit.) that the moduli space $\calN^d(GL(n,\C))$ for principal $GL(n,\C)$-bundles of topological type $d$ is naturally isomorphic to the moduli space $\calU (n,d)$ for vector bundles of rank $n$ and degree $d$. Moreover, under this identification, $\calN (SL(n,\C))$ corresponds precisely to the locus $\calU (n, \calO_\Sigma)$ of $\calU (n,0)$ consisting of semi-stable vector bundles $V$ with trivial determinant line bundle. Note in particular that $SL(n,\C)$ is simply connected, so the only topological type (in this case the degree) is zero, which we have omitted. 

Assume $G$ is semisimple and $\rho : G \to GL(\V)$ is a faithful representation. Then a principal $G$-bundle $P$ on $\Sigma$ is semi-stable if and only if the associate vector bundle $P(\V)$ is semi-stable \cite[Proposition 3.17]{ramana1}. Actually, as shown in \cite{rr}, $P$ is semi-stable if and only if $P(\V)$ is a semi-stable vector bundle for any finite-dimensional representation $\V$. The same phenomenon is not true for stability though. Let us give as an example the moduli space of symplectic principal bundles on $\Sigma$, which plays an important role in this thesis. 

\begin{thm}[{\cite[Section 2.1]{hitching2005moduli}}] A principal $Sp(2n,\C)$-bundle $P \in \calN (Sp(2n,\C))$ is stable if and only if the symplectic vector bundle $(V = P(\C^{2n}), \omega )$ associated to the standard representation is an orthogonal direct sum
$$(V,\omega ) = \bigoplus (W_r, \omega_r),$$
where $W_r$ are stable vector bundles and $(W_r, \omega_r)$ are mutually non-isomorphic symplectic bundles on $\Sigma$. Moreover, $P$ is a singular point of the stable locus of $\calN (Sp(2n,\C))$ if and only if $r>1$ (i.e., $W = \bigoplus W_r$ is a strictly semi-stable vector bundle). This characterizes the smooth locus of $\calN (Sp(2n,\C))$ as the open set of stable symplectic $Sp(2n,\C)$-bundles $P$ whose automorphism group $\Aut (P)$ equals the center $Z (Sp(2n\C)) = \{\pm \Id  \}$.  
\end{thm}          

The moduli space $ \calN (SL(2n,\C))$ can be thought of as the moduli space of symplectic vector bundles, which we denote by $\calU (Sp(2n,\C))$, where such an object $(V,\omega)$ is called \textbf{semi-stable} (respectively, \textbf{stable}) if for all isotropic subbundles $W \subset V$, its slope $\mu (W)$ is non-positive (respectively, negative). Moreover, $(V,\omega )$ is \textbf{polystable} if it is semi-stable and for all for any strict isotropic (respectively, coisotropic) subbundle $0 \neq W \subset V$ of degree $0$, there is a coisotropic (respectively, isotropic) subbundle $W^\prime \subset V$ such that $V = W \oplus W^\prime$. It follows from the work of Serman \cite{serman} that the forgetful map 
\[\xymatrix@M=0.05in{
\calN (Sp(2n,\C)) \ar[d]_{\cong} \ar[r] & \calN (SL(2n,\C)) \ar[d]^{\cong}  \\
\calU (Sp(2n,\C)) \ar[r] & \calU (2n, \calO_\Sigma) }\]
associated to the extension of structure group $Sp(2n,\C) \subset SL(2n,\C)$ is an embedding, a phenomenon which does not hold in general for arbitrary reductive groups. Note that we may identify the image of this map as the locus of $\calU (2n, \calO_\Sigma)$ consisting of vector bundles which admit a symplectic form. 

\begin{rmk} For classical groups we will usually identify the moduli space $\calN (G)$ with its vector bundle counterpart (with the correct notion of stability, as in $\calU (Sp(2n,\C))$), which will be denoted by $\calU (G)$.
\end{rmk}  

\subsection{Generalized theta-divisor}\label{thetadivisor}

In order to generalize the canonically defined theta divisor on $\Pic^{g-1}(\Sigma) = \calU (1, g-1)$ to the moduli space of vector bundles of higher rank, and ultimately to the moduli space of principal $G$-bundles on $\Sigma$, we recall the determinant of the cohomology of a family of vector bundles. The idea is that there exists a natural line bundle on $\calU (n,d)$, whose fibre over a vector bundle $V$ can be naturally identified with $\Lambda^{top} H^0(\Sigma , V)^* \otimes \Lambda^{top} H^1(\Sigma , V)$. Moreover, when $d = n(g-1)$, one can construct a section of this line bundle and the corresponding divisor generalizes the classical theta divisor. We will closely follow \cite{kumar}. For more details, the reader is also referred to \cite{sorger, mum, dn} (or, from a more analytic point of view, to Quillen's determinant construction \cite{Quillen1985}).  
 
Let $T$ be a variety and $\calV$ be a vector bundle on $\Sigma_T = \Sigma \times T$, which we think of as a family of vector bundles $\{ \calV_t \}_{t\in T}$ parametrized by $T$. There exists a two-term complex $\calV^\bullet$ of vector bundles on $T$ 
$$ \calV_0 \to \calV_1 $$
such that for any base change $f: Z \to T$, we have 
$$R^ip_{2,*}((\Id \times f)^*\calV) = H^1(f^*\calV^\bullet),$$
where $p_2 : \Sigma \times Z \to Z$ is the projection onto the second factor and $f^*\calV^\bullet$ is the complex obtained by pulling-back $\calV^\bullet$ to $Z$. We define the \textbf{determinant of cohomology of $\calV$} as the line bundle on $ T$ given by
$$\calD (\calV) = \Lambda^{top} (\calV_0^*)\otimes \Lambda^{top} (\calV_1).$$
In particular, its fibre on $t\in T$ is canonically isomorphic to 
$$\Lambda^{top} H^0(\Sigma , \calV_t)^* \otimes \Lambda^{top} H^1(\Sigma , \calV_t), $$
where $\calV_t = \calV|_{\Sigma \times \{ t \}}$. From the base change property, it follows that $$\calD ((\Id \times f)^*\calV) \cong f^*(\calD (\calV)),$$
for any morphism $f: Z \to T$. Moreover, if $L$ is a line bundle on a connected variety $T$ and $p : \Sigma \times T \to T$ is the projection onto the second factor, it follows from the projection formula that 
$$\calD (\calV \otimes p^*L) \cong \calD (\calV) \otimes L^{-\chi (\calV_t)},$$  
where $\chi (\calV_t) = h^0(\Sigma , \calV_t) - h^1(\Sigma , \calV_t)$ is the Euler characteristic for any $\calV_t$ with $t\in T$. Since $T$ is connected, $\chi (\calV_t)$ does not depend on the choice of $t\in T$.  

\begin{rmks} 

\noindent 1. Take a universal line bundle $\calL$ on $\Sigma \times \Pic^{g-1}(\Sigma)$, for example. Then, fix a divisor $D$ on $\Sigma$ of degree large enough and consider the exact sequence 
$$0 \to \calL (-D) \to \calL \to \calL|_D \to 0.$$
If $p : \Sigma \times \Pic^{g-1}(\Sigma) \to \Pic^{g-1}(\Sigma)$ is the projection, by pushing forward the short exact sequence above to $\Pic^{g-1}(\Sigma)$ we obtain the complex 
$$p_*(\calL|_D)\to R^1p_*(\calL (-D)))$$
which can be shown to satisfy the properties required. 

\noindent 2. In general, there is no universal bundle on moduli spaces of higher dimensional vector bundles. It follows, however, from Grothendieck \cite{grot3}, that the complex $Rp_{2,*}(\calV)$ is perfect and induces such a complex $\calV^\bullet$. More concretely, there exists an exact sequence 
$$0 \to p_*\calV \to \calV_0 \overset{\gamma}{\to} \calV_1 \to R^1p_*\calV \to 0,$$
where $p: \Sigma \times T \to T$ is the natural projection onto the second factor and $\calV_0$ and $\calV_1$ are vector bundles whose ranks satisfy $\rk (\calV_0) = \rk (\calV_1) + \chi (\calV_t)$ (see, e.g., \cite{arb}). Here $\chi (\calV_t)$ is the Euler characteristic of $\calV_t$ for some closed point $t\in T$. The complex $\calV_0 \overset{\gamma}{\to} \calV_1$ satisfies the desired properties.       
\end{rmks} 
      
Once a point of the curve is fixed, we have canonical divisors on $\Pic^d(\Sigma)$ for every $d\in \Z$. Analogously to the case of the Picard variety, fix a point $p \in \Sigma$. Given a family of rank $n$ and degree $0$ vector bundles $\calV \to \Sigma \times T$ parametrized by a variety $T$, we define the \textbf{theta-bundle} by 
$$\Theta (\calV) = \calD (\calV) \otimes \det (\calV_p)^{\chi (\calV_t)/n},$$
where $\det (\calV_p)$ is the usual determinant line bundle of $\calV_p = \calV|_{\{p\} \times T}$. It follows from the properties of the determinant of cohomology that 
$$\Theta ((\Id \times f)^*\calV) \cong f^*(\Theta (\calV)),$$
for any morphism $f: Z \to T$. Also,
$$\Theta (\calV) = \Theta (\calV \otimes p^*L)$$
for any line bundle $L$ on $T$, where $p: \Sigma \times T \to T$ is the projection. 

\begin{thm}[\cite{dn}] There exists a line bundle $\Theta$ on $\calU (n, 0)$ such that for any family $\calV$ of rank $n$ and degree $0$ semi-stable vector bundles on $\Sigma$ parametrized by $T$, 
$$f^* (\Theta) \cong \Theta (\calV),$$
where $f: T \to \calU (n, 0)$ is the morphism given by the coarse moduli property of $\calU (n, 0)$.
\end{thm}

Now suppose $G$ is a connected semisimple group and $\V$ is a finite-dimensional representation of dimension $n$. As discussed in Appendix \ref{assoc}, from a family of semi-stable principal $G$-bundles on $\Sigma$ parametrized by $T$, we can associate a family of semi-stable vector bundles of rank $n$ and degree $0$ (since $G$ is semisimple). This, in turn gives a canonical morphism $f: T \to \calU (n,0)$. Now, using the coarse moduli property of $\calN (G) $ we have a canonical morphism $\phi_\V : \calN (G) \to \calU (n,0)$. The \textbf{theta-bundle on $\calN (G)$ associated to $\V$} is defined as the line bundle  
$$\Theta (\V) = \phi_\V^*(\Theta)$$ 
on $\calN (G)$. In particular, if $\calP$ is a family of semi-stable principal $G$-bundles parametrized by $T$ and $\calP (\V)$ is the induced family of semi-stable vector bundles, it follows that 
$$f^* (\Theta (\V)) \cong \Theta (\calP (\V)),$$
where we denote by $f : T \to \calN (G)$ the morphism given by the coarse moduli property of $\calN (G)$.
 
Given $L_0 \in \Pic^{g-1}(\Sigma)$, we can identify $\calU (n, 0)$ with $\calU (n,n(g-1))$. Thus, given a family of semi-stable principal $G$-bundles parametrized by a variety $T$, we can consider the associated family $\calV$ of semi-stable vector bundles (defined by $L_0$ and $\V$) of rank $n = \dim \V$ and degree $n(g-1)$ parametrized by $T$. Taking a complex $\calV^\bullet$, as before, given by 
$$\calV_0 \overset{\gamma}{\to} \calV_1,$$ 
note that $\rk (\calV_0) = \rk (\calV_1)$, since the Euler characteristic is $\chi (\calV_t) = 0$ by Riemann-Roch, and so the determinant of the map $\gamma$ defines a section of $\calD (\calV) = \Theta (\calV)$ (the equality coming from the fact that the Euler characteristic is zero). These patch together to give a a canonical global section of the theta-bundle on $\calU (n, n(g-1))$, and so of the theta-bundle $\Theta (\V)$ on $\calN (G)$. By \cite{dn}, the zeros of this section define a Cartier divisor on $\calN (G)$ whose support is 
$$D_{L_0} = \{ P \in \calN (G) \ | \ H^0(\Sigma , P(\V)\otimes L_0 ) \neq 0 \}.$$
Moreover, this is independent of the choice of $L_0$ and the divisor is called the \textbf{generalized theta-divisor}, or \textbf{determinant divisor}, in $\calN (G)$ associated to the linear representation $\V$. Moreover, the line bundle associated to $D_{L_0}$ corresponds to the theta-bundle $\Theta (\V)$. The study of the Picard group of $\calN (G)$ has been carried out by many authors. Generalizing the results of Dr\'{e}zet and Narasimhan for $SL(n,\C)$, the case where $G$ is a simply-connected complex affine algebraic group was considered in \cite{las1, pic} and the semisimple case was treated in \cite{las2}. When $G$ is simply-connected, the Picard group is an infinite cyclic group. It turns out that for classical groups, or when $G$ is of type $G_2$, the moduli space $\calN (G)$ is locally factorial precisely when the group is special (cf. Remark \ref{rmk3}). The groups $Sp(2n,C) $ and $SL(2n,\C)$ are particularly important for us and we state the results for these cases. 
\begin{thm}[\cite{dn, pic, las1}]\label{irred} Let $G$ be $SL(2n,\C)$ or $Sp(2n,C)$. Then, the moduli space of principal $G$-bundles on $\Sigma$ is locally factorial and its Picard group $\Pic (\calN (G)) $ is an infinite cyclic group generated by the theta-bundle corresponding to the standard representation of $G$. In particular, the support of the generalized theta-divisor is irreducible.  
\end{thm}